\documentclass[11pt]{amsart}
\usepackage[leqno]{amsmath}
\usepackage{amssymb}
\usepackage{enumerate}

\usepackage{subfigure}
\usepackage{graphicx}
\usepackage{mathtools}
\usepackage{mathabx}
\usepackage[all]{xy}
\usepackage[usenames,dvipsnames]{xcolor}

\usepackage[
colorlinks=true,linkcolor=NavyBlue,urlcolor=RoyalBlue,citecolor=PineGreen,bookmarks=true,bookmarksdepth=3,bookmarksopen=true,bookmarksopenlevel=2,unicode=true,linktocpage]{hyperref}

\usepackage{microtype}
\usepackage{centernot}
\usepackage{stmaryrd}

\usepackage{comment}

\usepackage{bm}

\numberwithin{equation}{section}

\newtheorem{theorem}{Theorem}[section]
\newtheorem{remark}[theorem]{Remark}
\newtheorem{lemma}[theorem]{Lemma}
\newtheorem{proposition}[theorem]{Proposition}

\newtheorem{definition}[theorem]{Definition}

\newcommand{\C}{\mathbf{C}}
\newcommand{\D}{\mathbf{D}}
\newcommand{\E}{\mathbf{E}}

\newcommand{\h}{\mathbf{H}}
\newcommand{\N}{\mathbf{N}}
\newcommand{\Z}{\mathbf{Z}}
\newcommand{\p}{\mathbf{P}}
\newcommand{\Q}{\mathbf{Q}}
\newcommand{\R}{\mathbf{R}}
\newcommand{\s}{\mathbf{S}}

\newcommand{\Fh}{\mathfrak {h}}

\newcommand{\CA}{\mathcal {A}}
\newcommand{\CB}{\mathcal {B}}
\newcommand{\CC}{\mathcal {C}}
\newcommand{\CD}{\mathcal {D}}
\newcommand{\CE}{\mathcal {E}}
\newcommand{\CF}{\mathcal {F}}

\newcommand{\CK}{\mathcal {K}}
\newcommand{\CL}{\mathcal {L}}
\newcommand{\CM}{\mathcal {M}}
\newcommand{\CN}{\mathcal {N}}

\newcommand{\CS}{\mathcal {S}}
\newcommand{\CT}{\mathcal {T}}

\newcommand{\CW}{\mathcal {W}}

\newcommand{\CZ}{\mathcal {Z}}

\newcommand{\CG}{\mathcal {G}}
\newcommand{\CH}{\mathcal {H}}

\newcommand{\SLE}{{\rm SLE}}
\newcommand{\CLE}{{\rm CLE}}

\newcommand{\dist}{\mathrm{dist}}

\newcommand{\diam}{\mathrm{diam}}
\newcommand{\var}{\mathrm{var}}
\newcommand{\im}{\mathrm{Im}}
\newcommand{\re}{\mathrm{Re}}

\newcommand{\cen}{\mathrm{cen}}
\newcommand{\len}{\mathrm{len}}

\newcommand{\confrad}{{\rm CR}}

\newcommand{\one}{{\bf 1}}

\newcommand{\wt}{\widetilde}
\newcommand{\wh}{\widehat}

\newcommand{\ol}{\overline}
\newcommand{\ul}{\underline}

\newcommand{\giv}{\,|\,}

\definecolor{ddred}{rgb}{0.5,0,0}
\definecolor{dred}{rgb}{0.8,0,0}

\definecolor{ddgreen}{rgb}{0,0.4,0}
\definecolor{dgreen}{rgb}{0,0.8,0}

\definecolor{ddblue}{rgb}{0,0,0.4}
\definecolor{dblue}{rgb}{0,0,0.8}

\definecolor{dorange}{rgb}{0.7,.3,0}

\newcommand{\BES}{\mathrm{BES}}

\usepackage{mathrsfs}

\newcommand{\strip}{{\mathscr S}}

\newcommand{\IG}{\mathrm{IG}}

\newcommand{\disthyp}{\mathrm{dist}_{\mathrm{hyp}}}
\newcommand{\Bhyp}{B_{\mathrm{hyp}}}
\newcommand{\pin}{\partial^{\mathrm{in}}}
\newcommand{\pout}{\partial^{\mathrm{out}}}

\newcommand{\Bin}{\mathrm{Bin}}

\newcommand{\closure}[1]{\mathrm{cl}(#1)}

\newcommand{\cut}{\mathrm{cut}}

\newcommand{\cutmeasure}[1]{\mu_{#1}^\cut}
\newcommand{\qcutmeasure}[2]{\mu_{#1,#2}^\cut}
\newcommand{\qmeasure}[1]{\mu_{#1}}
\newcommand{\qbmeasure}[1]{\nu_{#1}}
\newcommand{\qintmeasure}[3]{\mu_{#1,#2,#3}^\cap}
\newcommand{\intmeasure}[2]{\mu_{#1,#2}^\cap}

\newcommand{\UR}{\mathrm{UR}}

\newcommand{\udimM}{\ol{\dim}_M}

\begin{document}

\title[Conformal removability of non-simple Schramm-Loewner evolutions]{Conformal removability of\\non-simple Schramm-Loewner evolutions}

\author{Konstantinos Kavvadias, Jason Miller, Lukas Schoug}

\newcommand{\adjcon}{\CK}

\begin{abstract}
We consider the Schramm-Loewner evolution ($\SLE_\kappa$) for $\kappa \in (4,8)$, which is the regime that the curve is self-intersecting but not space-filling.  We let $\adjcon$ be the set of $\kappa \in (4,8)$ for which the adjacency graph of connected components of the complement of an $\SLE_\kappa$ is a.s.\ connected, meaning that for every pair of complementary components $U, V$ there exist complementary components $U_1,\ldots,U_n$ with $U_1 = U$, $U_n = V$, and $\partial U_i \cap \partial U_{i+1} \neq \emptyset$ for each $1 \leq i \leq n-1$.  It was proved by Gwynne and Pfeffer \cite{gp2020adj} that this set is non-empty. We show that the range of an $\SLE_\kappa$ for $\kappa \in \adjcon$ is a.s.\ conformally removable, which answers a question of Sheffield.  As a step in the proof, we construct the canonical conformally covariant volume measure on the cut points of an $\SLE_\kappa$ for $\kappa \in (4,8)$ and establish a precise upper bound on the measure that it assigns to any Borel set in terms of its diameter.
\end{abstract}

\date{\today}
\maketitle

\setcounter{tocdepth}{1}

\tableofcontents

\parindent 0 pt
\setlength{\parskip}{0.20cm plus1mm minus1mm}
 
\section{Introduction}
\label{sec:intro}

The Schramm-Loewner evolution ($\SLE_\kappa$) is a one parameter family of probability measures on curves which connect two boundary points in a simply connected domain.  It was introduced by Schramm~\cite{s2000sle} as a candidate for the scaling limit of the interfaces of discrete planar lattice models at criticality.  Since Schramm's original work, many such models have been conjectured to converge to $\SLE_\kappa$ in the scaling limit.  In the setting where the underlying graph is a planar lattice, this has been established in a number of different cases~\cite{s2001cardy,lsw2004lerw,ss2009contours,s2010ising} and more recently several works have proved convergence results in the setting where the underlying graph is given by a random planar map~\cite{s2016inventory,kmsw2019bipolar,gm2017perc,lsw2017wood,gm2021saw}.  The parameter $\kappa \geq 0$ determines the roughness of an $\SLE_\kappa$ curve.  An $\SLE_0$ is a smooth curve and $\SLE_\kappa$ curves become more fractal as $\kappa$ increases.  There are three regimes of $\kappa$ values which are of particular importance.  Namely, $\SLE_\kappa$ curves are simple for $\kappa \leq 4$, self-intersecting but not space-filling for $\kappa \in (4,8)$, and space-filling for $\kappa \geq 8$~\cite{rs2005basic}.  Moreover, the a.s.\ dimension of the range of an $\SLE_\kappa$ curve is given by $\min(1+\kappa/8,2)$~\cite{rs2005basic,beffara2008dimension}.

Schramm's original construction of $\SLE_\kappa$ is given in terms of the Loewner equation which we recall is defined as follows.  Recall that a set $A \subseteq \h$ is called a compact $\h$-hull if $\closure{A}$ is compact and $\h \setminus A$ is simply connected.  It was proved by Loewner that if one has a family $(K_t)_{t \geq 0}$ of compact $\h$-hulls with $K_0 = \emptyset$ which is non-decreasing, parameterized by half-plane capacity, and locally growing then there exists a continuous function $W \colon \R_+ \to \R$ so that if
\begin{equation}
\label{eqn:chordal_loewner_ode}
\partial_t g_t(z) = \frac{2}{g_t(z) - W_t},\quad g_0(z) = z
\end{equation}
then for each $t \geq 0$, $g_t$ is the unique conformal map $\h \setminus K_t \to \h$ with $g_t(z) -z \to 0$ as $z \to \infty$.  That is, one can encode a growing family of compact $\h$-hulls in terms of an $\R$-valued continuous function.  Conversely, if one solves~\eqref{eqn:chordal_loewner_ode} for a given continuous function $W \colon \R_+ \to \R$ and lets $K_t$ be the complement in $\h$ of the domain of $g_t$ then $(K_t)_{t \geq 0}$ defines a family of $\h$-hulls satisfying the assumptions above and $g_t$ is the unique conformal map $\h \setminus K_t \to \h$ with $g_t(z) -z \to 0$ as $z \to \infty$.  Schramm realized that if one assumes that the interfaces in a planar lattice model have a scaling limit, satisfy a spatial version of the Markov property, and are conformally invariant and one uses the chordal Loewner equation to encode their scaling limit then $W = \sqrt{\kappa} B$ where $\kappa \geq 0$ and $B$ is a standard Brownian motion.  It is not obvious from~\eqref{eqn:chordal_loewner_ode} that this defines a continuous curve and this was proved in \cite{rs2005basic} for $\kappa \neq 8$ and for $\kappa = 8$ in \cite{lsw2004lerw} (see also \cite{am2022sle8} for a proof which does not make use of discrete models).

Since Schramm's original work, several other representations of $\SLE_\kappa$ have been discovered.  In particular, there has been a substantial amount of work in recent years based on and centered around its representation as a random conformal welding, which we recall is defined as follows. Suppose that $\D_1$, $\D_2$ are copies of the unit disk $\D$ and $\phi \colon \partial \D_1 \to \partial \D_2$ is a homeomorphism.  A \emph{conformal welding} with \emph{welding homeomorphism $\phi$} consists of a simple loop $\eta$ on $\s^2$ and a pair of conformal maps $\psi_i$ from $\D_i$ to the two sides of $\s^2 \setminus \eta$ so that $\phi = \psi_2^{-1} \circ \psi_1$.  It is in general not obvious whether a conformal welding exists for a given welding homeomorphism $\phi$ and, if so, whether it is unique.  It was shown by Sheffield \cite{she2016zipper} that if the welding homeomorphism is defined using the boundary length measure for a certain type of Liouville quantum gravity (LQG) surface then the conformal welding exists and the welding interface is an $\SLE_\kappa$ curve for $\kappa < 4$.  Sheffield's welding result was extended to the case $\kappa = 4$ in \cite{hp2021welding} and a number of welding results for $\kappa > 4$ were proved in \cite{dms2021mating}.  In particular, the $\SLE_\kappa$ curves for $\kappa \in (4,8)$ arise as the welding interface when one glues together two independent stable looptrees $\CT_1$, $\CT_2$ where each loop is filled with a conditionally independent quantum disk.  A stable looptree consists of a countable collection of loops and the loops of $\CT_1$ (resp.\ $\CT_2$) correspond to the complementary components that the $\SLE_\kappa$ surrounds on its left (resp.\ right) side.  The welding in this case is defined as follows.  We first note that the stable looptrees $\CT_i$ each come with a measure $\nu_i$ on their outer boundary (in the same way that the continuum random tree comes with a measure).  In the finite volume case (so that the welding is homeomorphic to $\s^2$) we assume that we have conditioned $\nu_1$, $\nu_2$ to have total mass $1$ (or more generally have  sampled $\ell$  from some distribution on $\R_+$ and then take $\nu_1$, $\nu_2$ to have total mass $\ell$).  A pair of homeomorphisms $\psi_i \colon \CT_i \to \s^2$ for $i = 1,2$ give a welding of $\CT_1$, $\CT_2$ if they are conformal in each of the loops of the~$\CT_i$ (hence giving an embedding of each of the quantum disks in the $\CT_i$ into $\s^2$), the interiors of the loops of $\CT_1$, $\CT_2$ are embedded into disjoint subsets of $\s^2$, and $\psi_2^{-1} \circ \psi_1$ (which is only defined on the boundary of $\CT_1$) takes $\nu_1$ to $\nu_2$.  We refer the reader to \cite[Figures~1.7 and~1.8]{dms2021mating} for further explanation as well as a visualization of the definition of a welding in this setting as well as \cite[Section~1.3]{dms2021mating} for a discussion as to how one can use Moore's theorem to show that the topological gluing of independent trees results in a topological sphere without reference to $\SLE_\kappa$.

Recall that a set $K \subseteq \C$ is said to be \emph{conformally removable} if every homeomorphism $\varphi \colon \C \to \C$ which is conformal on $\C \setminus K$ is conformal on $\C$.  If a curve that arises as the welding interface is conformally removable, then it is the only curve that can arise as the welding curve corresponding to that welding homeomorphism, that is, the conformal welding is unique. In order to check that a simple curve is conformally removable, one often makes use of a criterion due to Jones and Smirnov \cite{js2000removability}.  We will not describe the Jones-Smirnov criterion in detail here, but in \cite{js2000removability} it is shown that if one views the simple curve as the boundary of a simply connected domain $D$ in $\C$ then the condition is implied by the modulus of continuity $\omega$ of the associated Riemann map $\varphi \colon \D \to D$ decaying to~$0$ sufficiently quickly as $\delta \to 0$.  For example, it is sufficient if~$\varphi$ is H\"older continuous up to $\partial \D$ meaning that~$\omega$ decays to~$0$ as $\delta \to 0$ as a power of~$\delta$.  In this case, $D$ is called a H\"older domain.  In fact, it is explained in \cite{js2000removability} that one can assume a much weaker hypothesis on $\omega$ and this was further refined in \cite{kn2005remove}.

In the case that $\kappa \in (0,4)$, it was proved in \cite{rs2005basic} that an $\SLE_\kappa$ curve forms the boundary of a H\"older domain so one obtains the conformal removability from \cite{js2000removability}.  The optimal H\"older exponent for the uniformizing map was computed in \cite{gms2018multifractal} and it is equal to $0$ for $\kappa = 4$ so that $\SLE_4$ does not form the boundary of a H\"older domain.  It was later shown in \cite{kms2021regularity} that the modulus of continuity of the $\SLE_4$ uniformizing map decays to $0$ as $(\log \delta^{-1})^{-1/3+o(1)}$ as $\delta \to 0$ and in fact that the condition of \cite{js2000removability} does not hold for $\SLE_4$.  The conformal removability of $\SLE_4$, however, was recently proved in the companion work \cite{kms2022sle4remov}.  $\SLE_\kappa$ curves for $\kappa \geq 8$ cannot be conformally removable because they are space-filling and any set of positive Lebesgue measure is not conformally removable.  This leaves the regime $\kappa \in (4,8)$, which is the focus of the present work.

The complement of an $\SLE_\kappa$ curve with $\kappa \in (4,8)$ consists of a countable collection of simply connected components, which means that it falls outside of the scope of \cite{js2000removability}.  In general, conformal removability in the context of domains which are not connected is more subtle.  For example, it is not difficult to check that the standard Sierpinski carpet is not conformally removable.  (See also the work \cite{n2021carpet} which proves that all topological Sierpinski carpets are not conformally removable.)  However, it is highly non-trivial to prove that the Sierpinski gasket is not conformally removable \cite{n2019nonremove}.  The main difference between the Sierpinski carpet and the Sierpinski gasket is that the boundaries of the complementary components in the latter intersect each other while they do not for the former.

Suppose that $D \subseteq \C$ is open and $K \subseteq D$ is closed in $D$.  We say that the adjacency graph of components of $K$ in $D$ is connected if for every pair of components $U,V$ of $D \setminus K$ there exist components $U_1,\ldots,U_n$ of $D \setminus K$ so that $U = U_1$, $V = U_n$, and $\partial U_i \cap \partial U_{i+1} \neq \emptyset$ for each $1 \leq i \leq n-1$.  Let $\adjcon$ be the set of $\kappa \in (4,8)$ so that the adjacency graph of complementary components of an $\SLE_\kappa$ in $\h$ is a.s.\ connected.

\begin{theorem}
\label{thm:sle_removable}
Fix $\kappa \in \adjcon$ and suppose that $\eta$ is an $\SLE_{\kappa}$ in $\h$ from $0$ to $\infty$.  It a.s.\ holds that if $f \colon \h \to \h$ is a homeomorphism which is conformal on $\h \setminus \eta$ then $f$ is conformal on all of $\h$.  In particular, the range of $\eta$ is a.s.\ conformally removable.
\end{theorem}

Let us now make some comments about the statement of Theorem~\ref{thm:sle_removable}.
\begin{itemize}
\item It was proved by Gwynne and Pfeffer \cite{gp2020adj} that there exists $\kappa_0 \in (4,8)$ so that $(4,\kappa_0) \subseteq \adjcon$, so in particular $\adjcon \neq \emptyset$.  It is currently an open problem to determine $\adjcon$.
\item The assertion of Theorem~\ref{thm:sle_removable} holds if we replace $\h$ by an arbitrary simply connected domain~$D$ because we can first apply a conformal transformation which takes~$D$ to~$\h$.
\item One can think of an $\SLE_{\kappa}$ for $\kappa \in (4,8)$ as a random analog of the Sierpinski gasket in the sense that its range is a random fractal and the boundaries of its complementary components can intersect each other.  However, in contrast to \cite{n2019nonremove}, we find that the range of such a curve is conformally removable, at least for $\kappa \in \adjcon$.  One significant difference between the range of an $\SLE_\kappa$ for $\kappa \in (4,8)$ and the Sierpinski gasket is that the boundaries of two distinct complementary components with non-empty intersection of the latter intersect in exactly one place whereas in the case of the former the intersection set is uncountable and has dimension $3-3\kappa/8$ \cite{mw2017intersections}.
\end{itemize}

As a useful technical tool for proving things about $\SLE_\kappa$, in the proof of Theorem~\ref{thm:sle_removable} we will make use of the \emph{conformal loop ensembles} ($\CLE_\kappa$) \cite{s2009cle,sw2012markovian} .  Recall that $\CLE_\kappa$ is the loop analog of $\SLE_\kappa$.  They are defined for $\kappa \in (8/3,8)$ and consist of a countable collection of loops which do not cross each other and each locally look like an $\SLE_\kappa$ curve.  The phases of $\CLE_\kappa$ are analogous to $\SLE_\kappa$.  Namely, the loops are simple curves and do not intersect the domain boundary or each other for $\kappa \in (8/3,4]$ while they do intersect the domain boundary and each other for $\kappa \in (4,8)$.  The closure of the set of points not surrounded by a loop is called the \emph{$\CLE_\kappa$ carpet} if $\kappa \in (8/3,4]$ (in analogy with the Sierpinski carpet) and the \emph{$\CLE_\kappa$ gasket} if $\kappa \in (4,8)$ (in analogy with the Sierpinski gasket).  The dimension of the $\CLE_\kappa$ carpet (resp.\ gasket) is given by $2-(8-\kappa)(3\kappa-8)/(32\kappa)$ \cite{ssw2009radii,nw2011dimension,msw2014gasket}, so is a much larger set than the range of an $\SLE_\kappa$.  Both the $\CLE_\kappa$ carpet ($\kappa \in (8/3,4]$) and $\CLE_\kappa$ gasket ($\kappa \in (4,8)$) are in fact topological Sierpinski carpets, so are not conformally removable by \cite{n2021carpet}.  (We note that the non-removability of the the $\CLE_\kappa$ does not impact the validity of our result on the removability of $\SLE_\kappa$ since in the present paper it is used only as a technical tool.)

Let us now discuss briefly the differences between proving the conformal removability of $\SLE_4$ and the $\SLE_\kappa$ for $\kappa \in \CK$.  At a high level, the key difference between $\SLE_4$ and $\SLE_\kappa$ for $\kappa \in \CK$ is that the former can be viewed as the boundary of a Jordan domain while the latter is the boundary of a countable collection of Jordan domains.  Previously developed tools for removability focused on proving the removability of boundaries of Jordan domains but required that they satisfy strong regularity hypotheses which are not satisfied by $\SLE_4$ and did not even address the latter setting.  

In \cite{kms2022sle4remov}, we formulated a new removability condition which in \cite{kms2022sle4remov} and the present paper we show is applicable in both settings.  We review the exact conditions which we will verify in order to prove that a set $K \subseteq \C$ is conformally removable in Section~\ref{subsec:removability_theorem}. What it boils down to is checking that for every $z \in K$ and at a sufficiently dense set of integers $k$ (subject to certain parameter choices) there exists a ``nice'' curve $\gamma$ circling $z$ in an annular region of size approximately $2^{-k}$.  Here, ``nice'' roughly means the curve $\gamma$ only passes through a finite number of subdomains $U_1,\ldots,U_m$ of $\C \setminus K$ (condition~\eqref{it:components}; see Figure~\ref{fig:annulus_condition}) with some control on the relationship between the hyperbolic metric and Euclidean metric as $\gamma$ approaches each $\partial U_i$.  In practice, the $U_1,\ldots,U_m$ will not be components of $\C \setminus K$ but rather complementary components of an enlargement of $K$.  In order to ensure the existence of a good choice of $\gamma$, we require that there are sufficiently many ``good'' points in each of the intersections $\partial U_i \cap \partial U_{i+1}$ for $1 \leq i \leq m$ where $U_{m+1} = U_1$.  Goodness of a point $w \in \partial U_i$ is determined by the structure of the nested bottlenecks as one approaches $w$ along a curve from a fixed reference point in $U_i$ and this is quantified precisely in terms of the relationship between the Euclidean and hyperbolic metrics on $U_i$.  We quantify the occurrence of good points using measures on (parts of) the intersections $\partial U_i \cap \partial U_{i+1}$.  These measures should be viewed as the natural Hausdorff measure on the intersection, but the construction and properties of the measures is a non-trivial matter.  Altogether, these assumptions serve to restrict the way that the $U_i$ ``harmonically communicate'' with each other.  More good points in each intersection $\partial U_i \cap \partial U_{i+1}$ translates into more ``harmonic rigidity''; this is not present in the case of the Sierpinski gasket as in that case components whose boundaries intersect do so at precisely one point and this is what makes it possible to construct the exceptional homeomorphism in \cite{n2019nonremove}.

The difficulties for the cases of $\SLE_4$ and $\SLE_\kappa$ for $\kappa \in \CK$ are largely complementary. Indeed, in the $\SLE_4$ case, each measure is taken to be the natural parameterization on parts of $\SLE_4$-type curves which serve to enlarge the original $\SLE_4$ and are described in terms of level lines of the GFF.  The construction and precise estimates for the natural parameterization were studied extensively before \cite{ls2011natural,lz2013natural,lr2015minkowski,b2018natural,rz2017higher} so we are able to use these results as a black box.  Considerable effort in \cite{kms2022sle4remov} was aimed at showing that most points (according to the natural parameterization) on an $\SLE_4$ have the correct bound on the growth of hyperbolic distance relative to the Euclidean distance as one approaches $\SLE_4$ curve; this is non-trivial since $\SLE_4$ curves do not cut out H\"older domains and form nested and tight bottlenecks.  Much of the effort in \cite{kms2022sle4remov} is also put into localizing the behavior of the $\SLE_4$ near a point in its range in order to make certain events approximately independent across scales.  On the other hand, to prove the removability of $\SLE_\kappa$ for $\kappa \in \CK$, it is necessary first to construct and prove the necessary (optimal) moment bounds for the natural measure on the cut points of an $\SLE_\kappa$.  Controlling the growth of the hyperbolic metric relative to the Euclidean metric in each $U_i$ is a quick matter, however, as the bubbles cut out by $\SLE_\kappa$ for $\kappa \in (4,8)$ are $\SLE_{16/\kappa}$ type curves hence H\"{o}lder domains.  Moreover, the localization procedure makes use of flow lines of the GFF starting from interior points, which is not possible in the case of $\SLE_4$.  Finally, it requires some effort to prove a localized version of the connectivity of the adjacency graph of the complementary components of an $\SLE_\kappa$ for $\kappa \in \CK$ which results in uniformity in the number and geometry of components at a sufficiently dense set of scales.

\subsection*{Outline}

The remainder of this article is structured as follows.  In Section~\ref{sec:preliminaries}, we will collect a number of preliminaries.  Next, in Section~\ref{sec:natural_measure_on_cut_points} we construct the canonical conformally covariant measure on the cut points of an $\SLE_{\kappa}$ for $\kappa \in (4,8)$.  In the case that $\kappa \in (4,8)$, this is a strict but non-empty subset of the range of an $\SLE_\kappa$ curve.  For $\kappa \in (0,4]$, an $\SLE_\kappa$ curve is simple so every point in its range is a cut point and $\SLE_\kappa$ for $\kappa \geq 8$ does not have cut points.  The measure that we construct is for $\kappa \in (4,8)$ the cut point analog of the so-called natural parameterization for $\SLE$ which was first constructed in \cite{ls2011natural,lz2013natural} and later shown to be equivalent to the Minkowski content in \cite{lr2015minkowski}.  The construction that we will give here makes use of the Gaussian free field (GFF) based approach to the natural parameterization developed in \cite{b2018natural} and also employed in \cite{ms2022volume}.  We will also establish a precise upper bound on the measure assigned to any Borel set in terms of its diameter, in analogy with the results proved for the natural parameterization in \cite{rz2017higher}.  The proof of this upper bound is based on GFF techniques and we expect that similar methods could lead to precise upper bounds for natural measures on $\SLE$-type fractals.  Next, in Section~\ref{sec:good_cube_lemmas} we will prove some harmonic measure estimates which we will later apply in the case of $\SLE_\kappa$ curves for $\kappa \in (4,8)$ near cut points.  Finally, in Section~\ref{sec:pocket_argument} we will complete the proof of Theorem~\ref{thm:sle_removable}.  This will involve showing that if $\kappa \in \adjcon$ then one has a stronger version of the adjacency graph of complementary components being connected which will allow us to localize the behavior of an $\SLE_\kappa$ in an annulus so that we can show that it satisfies the sufficient condition for conformal removability given in \cite{kms2022sle4remov}.

\subsection*{Notation}
We let $\Z$ denote the set of integers, $\N$ the set of positive integers, $\N_0 = \N \cup \{ 0 \}$, $\R$ the real numbers, $\C$ the complex plane, $\h$ the upper half-plane, $\D$ the unit disk and $\strip = \R \times (0,\pi)$. For a set $A \subseteq \C$, we denote by $\closure{A}$ the closure of $A$.

For two quantities $a,b$, we write $a \lesssim b$ if there is a constant $C$, independent of the parameters of interest, such that $a \leq C b$. Moreover, we write $a \gtrsim b$ if $b \lesssim a$ and we write $a \asymp b$ if $a \lesssim b$ and $a \gtrsim b$.

\subsection*{Acknowledgements}
J.M.\ and L.S.\ were supported by ERC starting grant 804166 (SPRS).  K.K.\ was supported by the EPSRC grant EP/L016516/1 for the University of Cambridge CDT (CCA) as well as by ERC starting grant 804166 (SPRS).

\section{Preliminaries}
\label{sec:preliminaries}

In this section, we will review some preliminaries which will be used throughout the rest of the paper.  We start in Section~\ref{subsec:measures} by recalling some basics of random measures before collecting some of the basic facts about the $\SLE$ processes and their variants in Section~\ref{subsec:sle}.  We will then give a review of the GFF in Section~\ref{subsec:gff} and then the coupling of $\SLE$ with the GFF in Section~\ref{subsec:ig}.  The purpose of Section~\ref{subsec:cle} is to review the construction of $\CLE$ and its relationship with space-filling $\SLE$.  Next, in Section~\ref{subsec:lqg} we will go over the basics of LQG and the aspects of its relationship with $\SLE$ which will be important for this work.  In Section~\ref{subsec:good_annuli} we will review the independence properties of the GFF across nested annuli and then, finally, in Section~\ref{subsec:removability_theorem} we review the criteria for conformal removability that we will use.

\subsection{Random measures}
\label{subsec:measures}
We now recall some basic facts about random measures. For more, we refer the reader to~\cite{kallenberg2017rmbook}. A random measure $\xi$ is a random element taking values in a space of measures on a Borel metric space $X$. In this paper, we will consider the case $X = \C$. Let $\CB(X)$ be the Borel $\sigma$-algebra on $X$. For a random measure $\xi$, we denote by $\E[\xi]$ its intensity, that is, the measure given by $\E[\xi](A) = \E[\xi(A)]$ for $A \in \CB(X)$ and given a $\sigma$-algebra $\CG$ contained in the Borel $\sigma$-algebra we define the conditional intensity given $\CG$ to be the random measure defined by $\E[ \xi \giv \CG](A) = \E[ \xi(A) \giv \CG]$ for $A \in \CB(X)$.  For a function $f: X \to \R$, we write $\xi(f) = \int_X f d\xi$.

A sequence of measures $(\nu_n)$ is said to converge vaguely or in the vague topology to a measure $\nu$ if for each continuous function $f$ with bounded support, we have $\nu_n(f) \to \nu(f)$.  In the case that the measures are locally finite and the space $X$ is separable and complete,  we have that a sequence $(\xi_n)$ of random measures converge a.s.\ in the vague topology to a measure $\xi$ if, on an event of probability one, $\xi_n \to \xi$ vaguely, or equivalently (see~\cite[Lemma~4.8]{kallenberg2017rmbook}), if for each fixed continuous function $f$ on $X$ with bounded support, $\xi_n(f) \to \xi(f)$ a.s.

Finally we remark that our definition of a random measure differs slightly from that of~\cite{kallenberg2017rmbook}, in that we do not require a random measure to be locally finite. It turns out, however, that the measures of interest are locally finite and we will prove that (see Section~\ref{sec:natural_measure_on_cut_points}) and hence the results of~\cite{kallenberg2017rmbook} apply to the measures we consider.

\subsection{Schramm-Loewner evolution}
\label{subsec:sle}

The starting point for the definition of the Schramm-Loewner evolution ($\SLE_\kappa$) is the so-called Loewner equation.  Suppose that $W \colon \R_+ \to \R$ is a continuous function.  For each $z \in \h$ we let $g_t(z)$ be the unique solution to the ODE
\begin{equation}
\label{eqn:sle_def}
\partial_t g_t(z) = \frac{2}{g_t(z) - W_t}, \quad g_0(z) = z
\end{equation}
up to the time $\tau(z) = \inf\{t \geq 0 : \im(g_t(z)) = 0\}$.  We let $K_t = \{z \in \h : \tau(z) \leq t\}$ and $\h_t = \h \setminus K_t$.  Then $g_t$ is the unique conformal map $\h_t \to \h$ with $g_t(z) - z \to 0$ as $z \to \infty$.  In the case of $\SLE_\kappa$ one fixes $\kappa \geq 0$ and then takes $W = \sqrt{\kappa} B$ where $B$ is a standard Brownian motion.  It was proved by Rohde-Schramm  \cite{rs2005basic} that for each $\kappa \neq 8$ there exists a continuous curve $\eta$ in $\h$ from $0$ to $\infty$ so that for each $t \geq 0$ we have that $\h_t$ is the unbounded component of $\h \setminus \eta([0,t])$.  This result was extended to the case $\kappa = 8$ by Lawler-Schramm-Werner \cite{lsw2004lerw} as a consequence of the convergence of the uniform spanning tree to $\SLE_8$.  See also \cite{am2022sle8} for a proof which does not make use of discrete models.  As we mentioned in Section~\ref{sec:intro}, the dimension of the range of an $\SLE_\kappa$ is $\min(1+\kappa/8,2)$ \cite{rs2005basic,beffara2008dimension} and $\SLE_\kappa$ curves are simple for $\kappa \in (0,4]$, self-intersecting but not space-filling for $\kappa \in (4,8)$, and space-filling for $\kappa \geq 8$ \cite{rs2005basic}.  Chordal $\SLE_\kappa$ in a simply connected domain $D \subseteq \C$ connecting distinct boundary points $x,y$ is defined by taking the conformal image of a chordal~$\SLE_\kappa$ in~$\h$ from~$0$ to~$\infty$ where the conformal map takes $0$ to $x$ and $\infty$ to $y$.

In this work, will also need to consider the $\SLE_\kappa(\rho)$ processes, which are an important variant of $\SLE$ in which one keep tracks of extra marked points \cite[Section~8.3]{lsw2003confres} called force points.  Fix $x_1,\ldots,x_n \in \closure{\h}$ and $\rho_1,\ldots,\rho_n \in \R$.  Let $W$ be the solution to the SDE
\begin{equation}
\label{eqn:sle_kappa_rho_def}
\begin{split}
 dW_t &= \sqrt{\kappa} dB_t + \sum_{i=1}^n \re\left( \frac{\rho_i}{W_t - V_t^i} \right) dt\\
 dV_t^i &= \frac{2}{V_t^i - W_t} dt,\quad V_0^i = x_i \quad\text{for}\quad 1 \leq i \leq n.
 \end{split}
\end{equation}
The $\SLE_\kappa(\ul{\rho})$ process with force points at $\ul{x} = (x_1,\ldots,x_n)$ and weights $\ul{\rho} = (\rho_1,\ldots,\rho_n)$ is defined by replacing $W$ in~\eqref{eqn:sle_def} with the solution to~\eqref{eqn:sle_kappa_rho_def}, at least as long as the solution to~\eqref{eqn:sle_kappa_rho_def} exists.  Chordal $\SLE_\kappa(\ul{\rho})$ in a simply connected domain $D \subseteq \C$ connecting distinct boundary points $x,y$ is defined by taking the conformal image of a chordal $\SLE_\kappa(\ul{\rho})$ in $\h$ from $0$ to $\infty$ where the conformal map takes $0$ to $x$ and $\infty$ to $y$.

In the case that the force points are in $\partial \h$, we will use the notation $\SLE_\kappa(\ul{\rho}_L; \ul{\rho}_R)$ where the weights $\ul{\rho}_L = (\rho_{1,L},\ldots, \rho_{\ell,L})$ correspond to the force points $x_{\ell,L} < \dots < x_{1,L} \leq 0$ and the weights $\ul{\rho}_R = (\rho_{1,R},\ldots,\rho_{k,R})$ correspond to the force points $0 \leq x_{1,R} < \dots < x_{k,R}$.  In this case, we let $V_t^{i,q}$ for $q \in \{L,R\}$ denote the evolution of the force points under the Loewner flow.  It was shown in \cite[Section~2]{ms2016imag1} that there is a unique solution to~\eqref{eqn:sle_kappa_rho_def} up until the so-called \emph{continuation threshold}, which is $\inf\{t \geq 0 : \sum_{i,q : V_t^{i,q} = W_t} \rho_i \leq -2\}$.  The continuity of $\SLE_\kappa(\rho)$ in this case was proved in \cite{ms2016imag1}.  The continuity in the case that $\rho \in (-2-\kappa/2,-2)$ (with a single force point) was proved in \cite{msw2017cleperc,ms2019gfflightcone}.  It is not possible to define $\SLE_\kappa(\rho)$ with $\rho \leq -2-\kappa/2$ after $W$ collides with the evolution of the force point $V$.  (We will not go into the details of constructing the $\SLE_\kappa(\rho)$ processes in the case $\rho \in (-2-\kappa/2,-2)$ as there are some extra technicalities and they will not play a role in this paper.)

In this work, we will also need to consider the radial and whole-plane $\SLE_\kappa$ and $\SLE_\kappa(\rho)$ processes.  The starting point for the definition of these processes is the radial form of the Loewner equation.  Suppose that $W \colon \R_+ \to \partial \D$ is a continuous function.  For each $z \in \D$ we let $g_t(z)$ be the unique solution to the ODE
\begin{equation}
\label{eqn:radial_loewner}
\partial_t g_t(z) = g_t(z)\frac{W_t + g_t(z)}{W_t - g_t(z)},\quad g_0(z) = z
\end{equation}
up to the time $\tau(z) = \inf\{t \geq 0 : |g_t(z)| = 1\}$.  We let $K_t = \{z \in \D : \tau(z) \leq t\}$ and $\D_t = \D \setminus K_t$.  Then for each $t \geq 0$ we have that $g_t$ is the unique conformal map $\D_t \to \D$ with $g_t(0) = 0$ and $g_t'(0) > 0$.  In fact, due to the form of~\eqref{eqn:radial_loewner} time is parameterized so that $\log g_t'(0) = t$ for each $t \geq 0$.  Radial $\SLE_\kappa$ is defined by taking $W_t = e^{i \sqrt{\kappa} B_t}$ where $B$ is a standard Brownian motion.  That radial $\SLE_\kappa$ corresponds to a continuous curve follows from the chordal case described above and that radial $\SLE_\kappa$ is locally absolutely continuous with respect to chordal $\SLE_\kappa$.  A radial $\SLE_\kappa(\rho)$ process is a variant of radial $\SLE_\kappa$ where one keeps track of an extra marked boundary point.  Let
\[ \Psi(w,z) = -z\frac{z+w}{z-w} \quad\text{and}\quad \wt{\Psi}(z,w) = \frac{\Psi(z,w) + \Psi(1/\ol{z},w)}{2}.\]
Radial $\SLE_\kappa(\rho)$ with force point at $x \in \partial \D$ is defined by solving~\eqref{eqn:radial_loewner} where $W$ is taken to be the solution to the SDE
\begin{equation}
\label{eqn:radial_sle_kappa_rho_driving}
\begin{split}
d W_t &= \left(\frac{\rho}{2} \wt{\Psi}(O_t,W_t)-\frac{\kappa}{2} W_t\right) dt + i \sqrt{\kappa} W_t dB_t\\
d O_t &= \Psi(W_t,O_t)dt,\quad O_0 = x.
\end{split}
\end{equation}
where $B$ is a standard Brownian motion.  By a comparison with a Bessel process, one can see that~\eqref{eqn:radial_sle_kappa_rho_driving} has a solution when $\rho > -2$.  (As in the case of chordal $\SLE_\kappa(\rho)$ one can also consider $\rho \in (-2-\kappa/2,-2)$ but we will not need this in the present article.)  Radial $\SLE_\kappa(\rho)$ in a simply connected domain $D \subseteq \C$ connecting a boundary point $x$ to an interior point $z$ is defined by taking the conformal image of a radial $\SLE_\kappa(\rho)$ in $\D$ from $1$ to $0$ where the conformal map takes $1$ to $x$ and $0$ to $z$.

It is explained in \cite[Section~2]{ms2017ig4} that using the local absolute continuity of a radial $\SLE_\kappa(\rho)$ process with respect to a chordal $\SLE_\kappa(\rho)$ process one has the continuity of the radial $\SLE_\kappa(\rho)$ processes when $\rho > -2$.

We now turn to give the definition of whole-plane $\SLE_\kappa$ and $\SLE_\kappa(\rho)$ processes for $\rho>-2$.  
Whole-plane $\SLE_{\kappa}(\rho)$ is a random growth process $(K_t)_{t \in \R}$ where,  for all $t \in \R$,  $K_t \subseteq \C$ is compact with $\C \setminus K_t$ simply connected (viewed as a subset of the Riemann sphere).  For each $t \in \R$,  we let $g_t : \C_t \to \C \setminus \closure{\D}$ be the unique conformal transformation with $g_t(\infty) = \infty$ and $g_t'(\infty) > 0$ (i.e.,  $\lim_{z \to \infty} \frac{g_t(z)}{z} > 0$).  Then,  $g_t(z)$ solves 
\begin{equation*}
\partial_t g_t(z) = g_t(z)\frac{W_t + g_t(z)}{W_t - g_t(z)}
\end{equation*}
up to time $\tau(z) = \sup\{t \in \R : |g_t(z)| > 1\}$,   where $(W,O)$ is given by the time-stationary solution to 
\begin{equation*}
\begin{split}
d W_t &= \left(\frac{\rho}{2} \wt{\Psi}(O_t,W_t)-\frac{\kappa}{2} W_t\right) dt + i \sqrt{\kappa} W_t dB_t\\
d O_t &= \Psi(W_t,O_t)dt
\end{split}
\end{equation*}
which is defined for all $t \in \R$.  Whole-plane $\SLE_{\kappa}$ corresponds to the case that $\rho = 0$,  i.e.,  $W_t = e^{i \sqrt{\kappa}B_t}$ with $B_t$ being a two-sided standard Brownian motion.  It follows from \cite[Proposition~2.5]{ms2017ig4} that whole-plane $\SLE_{\kappa}(\rho)$ is a.s.  generated by a continuous curve when $\rho>-2$.

\subsection{Gaussian free field}
\label{subsec:gff}

We now recall the definition and properties of the Gaussian free field (GFF).  For more on the GFF, see~\cite{she2007gff}. Let $D \subseteq \C$ be a simply connected domain with harmonically non-trivial boundary, let $H_0(D)$ denote the Hilbert space completion of $C_0^\infty(D)$ with respect to the Dirichlet inner product
\begin{align*}
	(f,g)_\nabla = \frac{1}{2\pi}\int_D \nabla f(z) \cdot \nabla g(z) dz,
\end{align*}
and let $\| \cdot \|_\nabla$ be the norm associated with $(\cdot,\cdot)_\nabla$. Let $(\phi_n)_{n \geq 1}$ be a $(\cdot,\cdot)_\nabla$-ONB of $H_0(D)$ and let $(\alpha_n)_{n \geq 1}$ be a sequence of i.i.d.\ $N(0,1)$ variables. Then, the \emph{zero-boundary GFF} on $D$ is the random distribution defined by $h = \sum_{n \geq 1} \alpha_n \phi_n$ (distribution because while the partial sums do not converge in any nice space of functions, they do converge in the Sobolev space $H^{-\epsilon}(D)$ for each $\epsilon > 0$). Moreover, for $f,g \in H_0(D)$, $(h,f)_\nabla \sim N(0,\| f \|_\nabla^2)$ and the covariance of $(h,f)_\nabla$ and $(h,g)_\nabla$ is $(f,g)_\nabla$.

An important property of the GFF is the domain Markov property, which states that on any subdomain $U \subseteq D$, we can decompose a zero-boundary GFF $h$ as $h = h_U^0 + \Fh_U$ where $h_U^0$ is a zero-boundary GFF on $U$ and $\Fh_U$ is a random distribution which restricts to a harmonic function on $U$. Moreover, with this in mind, a GFF with boundary data $f$ is defined as the sum of a zero-boundary GFF $h$ with the harmonic extension of $f$ to the interior of $D$.

Next, we recall that $H_0(D)$ is the Cameron-Martin space of the GFF. That is, if $h$ is a (not necessarily zero-boundary) GFF and $F \in H_0(D)$, then the laws of $h$ and $h+F$ are mutually absolutely continuous and the Radon-Nikodym derivative of the latter with respect to the former is given by 
\begin{align}\label{eq:RN_derivative}
	\CD_{h,F} \coloneqq \exp((h,F)_\nabla - \| F \|_\nabla^2/2).
\end{align}

\begin{remark}\label{rmk:RN_derivative}
We note that for each $p$, $\E[ \CD_{h,F}^p] = \exp((p^2-p) \| F\|_\nabla^2/2)$. This is particularly useful in comparing the laws of different GFFs, as it is often easy to find some $F$ with $\| F \|_\nabla < \infty$ such that the laws that we are interested in comparing are given by, say, $h|_U$ and $(h+F)|_U$ for some $U \subseteq D$. Then, this implies that the moments are finite and quantities of one of the fields are comparable to the corresponding ones of the other field, by applying H\"older's inequality.
\end{remark}

Finally, we recall that one can construct the GFF on the whole-plane or with free boundary conditions in an analogous manner, except one replaces the space $H_0(D)$ with the space $H(D)$ which is the Hilbert space closure with respect to $\| \cdot \|_\nabla$ of those functions in $f \in C^\infty(D)$ which are integrable with $\int f(x) dx = 0$ and which satisfy $\| f \|_\nabla < \infty$.  One then views the whole-plane or free boundary GFF as a distribution modulo additive constant; the additive constant can be fixed in many different ways but there is often a convenient choice based on the context.

\subsection{Imaginary geometry}
\label{subsec:ig}

A tool that will play an important role in this paper is the coupling between $\SLE$ and the GFF \cite{ms2016imag1,ms2017ig4}.  Throughout, we fix $\kappa \in (0,4)$ and let $\kappa'=16/\kappa \in (4,8)$.  Let
\[ \lambda = \frac{\pi}{\sqrt{\kappa}},\quad \lambda' = \frac{\pi}{\sqrt{\kappa'}}, \quad\text{and}\quad \chi = \frac{2}{\sqrt{\kappa}} - \frac{\sqrt{\kappa}}{2}.\]

Fix a vector of weights $(\ul{\rho}_L; \ul{\rho}_R)$ and points $(\ul{x}_L; \ul{x}_R)$ in $\partial \h$.  We assume that $|\ul{\rho}_L| = |\ul{x}_L| = \ell$ and $|\ul{\rho}_R| = |\ul{x}_R| = k$ where the $\ul{x}_L$ (resp.\ $\ul{x}_R$) are to the left (resp.\ right) of $0$ and given in decreasing (resp.\ increasing order).  We set $x_{0,L} = 0^-$, $x_{0,R} = 0^+$, $x_{\ell+1,L} = -\infty$, $x_{k+1,R} = +\infty$, and $\rho_{0,L} = \rho_{0,R} = 0$.  Suppose that $h$ is a GFF on $\h$ with boundary conditions given by
\begin{align*}
  -\lambda\left( 1 + \sum_{i=1}^j \rho_{i,L} \right)& \quad\text{in}\quad (x_{j+1,L},x_{j,L}] \quad\text{for each}\quad 0 \leq j \leq \ell \quad\text{and}\\ 
  \lambda\left( 1 + \sum_{i=1}^j \rho_{i,R} \right)& \quad\text{in}\quad (x_{j,R},x_{j+1,R}] \quad\text{for each}\quad 0 \leq j \leq k.
\end{align*}
By \cite[Theorem~1.1]{ms2016imag1}, there exists a coupling of $h$ with an $\SLE_\kappa( \ul{\rho}_L; \ul{\rho}_R)$ process $\eta$ in $\h$ from $0$ to $\infty$ so that the following is true.  Let $(g_t)$ be the Loewner flow associated with $\eta$, $W$ its Loewner driving function, and let $f_t = g_t - W_t$ be its centered Loewner flow.  Then for each stopping time $\tau$ for $\eta$ we have that the conditional law of $h \circ f_\tau^{-1} - \chi \arg( f_\tau^{-1})'$ given $\eta|_{[0,\tau]}$ is that of a GFF on $\h$ with boundary conditions given by $-\lambda$ (resp.\ $\lambda$) in $(f_\tau(x_{0,L}),0^-]$ (resp.\ $(0^+,f_\tau(x_{0,R})]$) and
\begin{align*}
  -\lambda\left( 1 + \sum_{i=1}^j \rho_{i,L} \right)& \quad\text{in}\quad (f_\tau(x_{j+1,L}),f_\tau(x_{j,L})] \quad\text{for each}\quad 0 \leq j \leq \ell \quad\text{and}\\ 
  \lambda\left( 1 + \sum_{i=1}^j \rho_{i,R} \right)& \quad\text{in}\quad (f_\tau(x_{j,R}),f_\tau(x_{j+1,R})] \quad\text{for each}\quad 0 \leq j \leq k.
\end{align*}
Moreover, in this coupling we have that $\eta$ is a.s.\ determined by $h$ \cite[Theorem~1.2]{ms2016imag1}.  In this coupling, $\eta$ is called the \emph{flow line} of $h$ from $0$ to $\infty$.  We note that in the case that the boundary conditions are given by $-\lambda$ (resp.\ $\lambda$) on $\R_-$ (resp.\ $\R_+$) the flow line of $h$ from $0$ to $\infty$ is an $\SLE_\kappa$ curve.

We can also define flow lines starting from other points $x \in \partial \h$ and with other angles $\theta$.  For each $x \in \partial \h$ and $\theta \in \R$ we let $\eta_\theta^x$ be the flow line of $h + \theta \chi$ starting from $x$.  Then \cite[Theorem~1.5]{ms2016imag1} describes how the flow lines starting from different boundary points and with different angles interact with each other, which we restate here.
\begin{theorem}
\label{thm:interaction}
Suppose that $x_1,x_2 \in \partial \h$ with $x_1 \leq x_2$ and angles $\theta_1, \theta_2$. Then the following a.s.\ hold.
\begin{enumerate}[(i)]
\item If $\theta_1 > \theta_2$, then $\eta_{\theta_1}^{x_1}$ stays to the left of $\eta_{\theta_2}^{x_2}$.
\item If $\theta_1 = \theta_2$, then $\eta_{\theta_1}^{x_1}$ merges with $\eta_{\theta_2}^{x_2}$ upon intersecting and does not subsequently separate.
\item If $\theta_2 - \pi < \theta_1 < \theta_2$, then $\eta_{\theta_1}^{x_1}$ crosses $\eta_{\theta_2}^{x_2}$ upon intersecting and does not subsequently cross back. 	
\end{enumerate}
\end{theorem}

We can also couple an $\SLE_{\kappa'}(\ul{\rho}_L; \ul{\rho}_R)$ process (note here the change to $\kappa' = 16/\kappa$) with force points at $(\ul{x}_L;\ul{x}_R)$ with the GFF.  In this case, one takes the boundary data to be
\begin{align*}
  \lambda'\left( 1 + \sum_{i=1}^j \rho_{i,L} \right)& \quad\text{in}\quad (x_{j+1,L},x_{j,L}] \quad\text{for each}\quad 0 \leq j \leq \ell \quad\text{and}\\ 
  -\lambda'\left( 1 + \sum_{i=1}^j \rho_{i,R} \right)& \quad\text{in}\quad (x_{j,R},x_{j+1,R}] \quad\text{for each}\quad 0 \leq j \leq k.
\end{align*}
Then there exists a coupling of $h$ with an $\SLE_{\kappa'}( \ul{\rho}_L; \ul{\rho}_R)$ process $\eta'$ in $\h$ from $0$ to $\infty$ so that the following is true.  Let $(g_t)$ be the Loewner flow associated with $\eta$, $W$ its Loewner driving function, and let $f_t = g_t - W_t$ be its centered Loewner flow.  Then for each stopping time $\tau$ for $\eta$ we have that the conditional law of $h \circ f_\tau^{-1} - \chi \arg (f_\tau^{-1})'$ given $\eta|_{[0,\tau]}$ is that of a GFF on $\h$ with boundary conditions given by $\lambda'$ (resp.\ $-\lambda'$) in $(f_\tau(x_{0,L}),0^-]$ (resp.\ $(0^+,f_\tau(x_{0,R})]$) and
\begin{align*}
  \lambda'\left( 1 + \sum_{i=1}^j \rho_{i,L} \right)& \quad\text{in}\quad (f_\tau(x_{j+1,L}),f_\tau(x_{j,L})] \quad\text{for each}\quad 0 \leq j \leq \ell \quad\text{and}\\
  -\lambda'\left( 1 + \sum_{i=1}^j \rho_{i,R} \right)& \quad\text{in}\quad (f_\tau(x_{j,R}),f_\tau(x_{j+1,R})] \quad\text{for each}\quad 0 \leq j \leq k.
\end{align*}
Moreover, in this coupling we have by \cite[Theorem~1.2]{ms2016imag1} that $\eta'$ is a.s.\ determined by $h$.  In this coupling, $\eta'$ is called the \emph{counterflow line} of $h$ from $0$ to $\infty$. We note in the case that the boundary conditions are given by $\lambda'$ (resp.\ $-\lambda'$) on $\R_-$ (resp.\ $\R_+$) the counterflow line of $h$ from $0$ to $\infty$ is an $\SLE_{\kappa'}$ curve.  If the boundary conditions are given by $\lambda'$ (resp.\ $\lambda'-2\pi \chi$) on $\R_-$ (resp.\ $\R_+$) the counterflow line of $h$ from $0$ to $\infty$ is an $\SLE_{\kappa'}(\kappa'-6)$ curve where the force point is located at $0^+$. As we will see in Section~\ref{subsec:cle}, this choice will be important because this is the one which corresponds to $\CLE_{\kappa'}$.

Suppose that $D \subseteq \C$ is simply connected, $x,y \in \partial D$ are distinct, and $\psi \colon \h \to D$ is a conformal transformation.  With $\wt{h} = h \circ \psi^{-1} - \chi \arg( \psi^{-1})'$ one can consider the flow and counterflow lines of~$\wt{h}$ which are defined to be the image under~$\psi$ of the flow and counterflow lines of~$h$.

It is also possible to define GFF flow lines starting from interior points \cite{ms2017ig4}.  We will first describe this in the case of the whole-plane GFF. Since adding $2\pi \chi$ to the field does not change its flow lines, it suffices to consider the whole-plane GFF modulo $2\pi \chi$ (i.e., viewed as a random variable in the space of distributions modulo the equivalence relation where two distributions are equivalent if they differ by an integer multiple of $2\pi \chi$).  To be concrete, we can sample a representative from the law of the whole-plane GFF modulo $2\pi \chi$ as follows.  We first sample $U \in [0,2\pi \chi]$ uniformly and then sample a whole-plane GFF with its additive constant fixed so that its average on $\partial \D$ is equal to $U$.  \cite[Theorem~1.1]{ms2017ig4} implies that we can generate flow lines of $h$ starting from different points and with different angles and that the marginal law of such a flow line is that of a whole-plane $\SLE_\kappa(2-\kappa)$ process from its starting point to $\infty$.  Moreover, \cite[Theorem~1.2]{ms2017ig4} implies that these flow lines are a.s.\ determined by the field.  The manner in which these flow lines interact with each other is described in \cite[Theorem~1.7]{ms2017ig4}.  The flow line interaction is governed by the angle at which two flow lines intersect each other and is the same as in the case that the flow lines start from the boundary. In particular, flow lines with the same angle eventually merge and form a space-filling tree.  It is shown in \cite{ms2017ig4} that the Peano curve associated with this space-filling tree is a.s.\ continuous and this process is called space-filling $\SLE_{\kappa'}$.

If one has a GFF $h$ on a domain $D \subseteq \C$ then viewed as a distribution modulo $2\pi \chi$ it is locally absolutely continuous with respect to the law of a whole-plane GFF modulo $2\pi \chi$.  Consequently, one can make sense of the flow lines of $h$ starting from interior points.  Let us mention one important result in this direction \cite{ms2017ig4}, namely that one can define a counterflow line $\eta'$ targeted at an interior point.  Suppose that we have a GFF $h$ on $\h$ whose boundary conditions are compatible with a coupling of an $\SLE_{\kappa'}(\ul{\rho}_L;\ul{\rho}_R)$ as a counterflow line.  Fix $z \in \h$ and let $\eta_L$ (resp.\ $\eta_R$) be the flow line of $h$ starting from $z$ with angle $-\pi/2$ (resp.\ $\pi/2$).  Then $\eta_L$ (resp.\ $\eta_R$) gives the left (resp.\ right) boundary of $\eta'$ targeted at $z$. For specifics on this, see \cite[Theorem~4.1]{ms2017ig4}.

\subsection{Conformal loop ensembles}
\label{subsec:cle}

We are now going to review the construction and basic properties about conformal loop ensembles ($\CLE_\kappa$) \cite{s2009cle, sw2012markovian}.  We will focus on the regime $\kappa' \in (4,8)$ since it is the one which is relevant for this paper.  See \cite{msw2014gasket,msw2017cleperc} for other thorough but brief introductions.  The starting point for the construction of $\CLE_{\kappa'}$ is the so-called \emph{exploration tree} developed in \cite{s2009cle}.  Suppose that $D \subseteq \C$ is a simply connected domain and $x, y \in \partial D$ are distinct.  Let $\eta'$ be an $\SLE_{\kappa'}(\kappa'-6)$ in $D$ from $x$ to $y$ which has its force point located at $x^+$.  The law of $\eta'$ has the special property that it is \emph{target invariant} \cite{sw2005coordinate}.  This means that if we fix $z \in \partial D \setminus \{x\}$ then the law of $\eta'$ up until it first disconnects $y$ from $z$ is the same as that of an $\SLE_{\kappa'}(\kappa'-6)$ in $D$ from $x$ to $z$ with its force point at $x^+$.  Suppose that $(y_n)$ is a countable dense subset of $\partial D$.  This implies that we can construct a coupling of $\SLE_{\kappa'}(\kappa'-6)$ processes $(\eta_n')$ in $D$ where each $\eta_n'$ is an $\SLE_{\kappa'}(\kappa'-6)$ in $D$ from $x$ to $y_n$ with force point at $x^+$ and any finite collection $(\eta_{n_1}',\ldots,\eta_{n_k}')$ agree up until disconnecting their target points after which they evolve independently.  That is, the $(\eta_n')$ have the structure of a tree which is rooted at $x$.  This tree is used to define the loops of $\Gamma$ which intersect~$\partial D$.

Fix $n \in \N$ and a point~$z$ in the clockwise arc of $\partial D$ from $x$ to $y_n$.  Let~$\tau_n$ be the first time that~$\eta_n'$ disconnects $z$ from $y_n$ and let~$\sigma_n$ be the last time that~$\eta_n'$ hits $\partial D$ before the time~$\tau_n$.  Then $\eta_n'|_{[\sigma_n,\tau_n]}$ is part of a loop $\CL$ of $\Gamma$ which intersects $\partial D$. The rest of the loop is found as follows. Let $(y_{k_m^n})$ be a subsequence of $(y_k)$ which lie in the counterclockwise arc of $\partial D$ from $\eta_n'(\sigma_n)$ to $y_n$, such that for each $m \in \N$, $y_{k_{m+1}^n}$ lies in the counterclockwise arc of $\partial D$ from $\eta_n'(\sigma_n)$ to $y_{k_m^n}$, and $y_{k_m^n} \to \eta_n'(\sigma_n)$ as $m \to \infty$. Let $\CL_m$ denote the part of $\eta_{k_m^n}'$ which is traced between the time it first hits $\eta_n'(\sigma_n)$ and the first time it separates $y_{k_m^n}$ from $\eta_n'(\sigma_n)$. Clearly, $\eta_n'([\sigma_n,\tau_n]) \subseteq \CL_m$ and $\CL_m \subseteq \CL_{m+1}$. Consequently, we find the sought loop $\CL$ to be the limit of the $\CL_m$. If we carry out this procedure for each $y_k$ and consider a dense set of $z \in \partial D$ (so that each $y_k$ is separated from some such $z$, as above) then we find all the loops of the $\CLE_{\kappa'}$ which intersect $\partial D$. To generate the rest of the $\CLE_{\kappa'}$, we iterate the same procedure inside of each of the complementary components which are not surrounded by a boundary touching loop.  That the loops of a $\CLE_{\kappa'}$ correspond to continuous curves was stated conditionally in \cite{s2009cle} on the continuity of $\SLE_{\kappa'}(\kappa'-6)$ which was proved in \cite{ms2016imag1}.  Also, that the law of the resulting ensemble of loops does not depend on the choice of $x \in \partial D$ was stated in \cite{s2009cle} conditionally on the reversibility of $\SLE_{\kappa'}$ for $\kappa' \in (4,8)$ which was proved in \cite{ms2016imag3}.  We also note that we can consider the nested version of $\CLE_{\kappa'}$ where in each complementary component of a $\CLE_{\kappa'}$ we sample an independent $\CLE_{\kappa'}$ and then iterate.

The local finiteness of the loops of a $\CLE_{\kappa'}$ (i.e., for every $\epsilon > 0$ there are only finitely many loops with diameter at least $\epsilon$) in the case that $D$ is a bounded Jordan domain was proved as a consequence of the continuity of space-filling $\SLE_{\kappa'}$ in \cite{ms2017ig4}.  As this will be important for us later, we will now explain how this works in more detail.  We can define a path $\Lambda$ which explores all of the loops in a nested $\CLE_{\kappa'}$ in $D$ iteratively as follows.  We start off by defining a path $\Lambda_1$ which explores only the boundary touching loops by considering the path which goes around $\partial D$ counterclockwise starting from $x$ except whenever it hits a loop of $\Gamma$ for the first time it immediately follows it clockwise in its entirety.  In each of the components which are disconnected by $\Lambda_1$ we know that we have an independent $\CLE_{\kappa'}$.  We generate $\Lambda_2$ by starting with $\Lambda_1$ except at each time a component $U$ is disconnected by $\Lambda_1$ we follow an exploration inside of $U$ starting from the last point on $\partial U$ visited by $\Lambda_1$ where we swap the roles of clockwise and counterclockwise if $\partial U$ is not surrounded by a loop of $\Gamma$.  By iterating this procedure $n$ times we obtain the path $\Lambda_n$ and the space-filling curve which is obtained in the limit as $n \to \infty$ is a space-filling $\SLE_{\kappa'}$.  We also note that if we start with $\Lambda_1$ and then target it at $z \in \partial D \setminus \{x\}$ then the resulting path is the branch of the $\CLE_{\kappa'}$ exploration tree from $x$ to $z$.

\subsubsection{A particular exploration of the $\CLE_{\kappa'}$}
\label{subsubsec:exploration}
Let $\Gamma$ be a $\CLE_{\kappa'}$ in $\h$. We now recall from \cite[Section~3.3]{gwynne2021conformal} a certain exploration of $\Gamma$.

Let $P:[0,1] \to \closure{\h}$ be a simple path (deterministic or random) which is independent of $\Gamma$ and let $U \subseteq \h$ be an open set with $P \cap \h \subseteq U$. For $m,\ell \in \N$, we the $(m,\ell)$-exploration process of $\Gamma$ along~$P$ relative to~$U$ is the pair $(\alpha_{m,\ell},\Gamma_m)$ defined as follows. In particular, we will later use a curve $\eta_{m,\ell}$ which arises in the definition of said exploration. (For an illustration of the process in $\D$ rather than in $\h$, see \cite[Figure~5]{gwynne2021conformal}.)

\begin{itemize}
\item Let $t_m$ be the $m$th smallest $t \in [0,1]$ for which the following is true: there exists $\CL \in \Gamma$ such that $\CL \not \subseteq \closure{U}$ and $P$ hits $\CL$ for the first time at time $t$; or let $t_m = 1$ if there are fewer than $m$ such times $t \in [0,1]$.  The local finiteness of $\Gamma$ implies that $t_m$ is well-defined. 
\item If $t_m<1$,  we parameterize the loop $\CL$ in the counterclockwise direction by $[0,1]$ such that $\CL(0) = \CL(1) = P(t_m)$.  Let $\sigma_0 = 0$.  If $j \in \N_0$,  we inductively let $\sigma_{j+1}$ be the first time after $\sigma_j$ at which $\CL$ completes a crossing from $P$ to $\h \setminus U$, i.e.,  the smallest $s \in (\sigma_j,1)$ for which $\CL(s) \notin U$ and there exists $s' \in (\sigma_j,s)$ for which $\CL(s') \in P$.  Let $\sigma_{j+1} = 1$ if $\CL$ does not make any crossings from $P$ to $\h \setminus \closure{U}$ after time $\sigma_j$.
\item Let $\ol{\xi}$ be the last time that the time-reversal of $\CL|_{[\sigma_{\ell},1]}$ completes a crossing from $P$ to $\h \setminus U$,  or let $\ol{\xi}$ be the starting time for this time-reversal if it does not make any such crossings.  Let $\xi$ be the time for $\CL$ corresponding to $\ol{\xi}$ and let $\eta_{m,\ell} = \CL|_{[\sigma_{\ell},\xi]}$.
\item Let $\alpha_{m,\ell}$ be the concatenation of the time-reversal of $\CL|_{[\xi,1]}$ and $\CL|_{[0,\sigma_{\ell}]}$.  Let $\Gamma_m$ be the set of loops in $\Gamma$ which intersect $P$ other than $\CL$. 
\end{itemize}

The reason for this exploration being useful, is that it is Markovian, in a sense which we now describe. We denote by $\Gamma(P;U)$ the set of loops in $\Gamma$ which intersect $P$ and are contained in $U$, we let $\cup \Gamma(P;U)$ denote the union of the loops in $\Gamma(P;U)$ and write $\Gamma(P) = \Gamma(P;\closure{\h})$. For a loop $\CL \in \Gamma$, we say that an arc $\alpha$ of $\CL$ is a $P$-excursion of $\CL$ into $U$ if $\alpha \subseteq \closure{U}$, $\alpha \cap P \neq \emptyset$ and $\alpha$ is not properly contained in any larger arc of $\CL$ with these properties. We say that $\alpha$ is proper if $\alpha \notin \{ \CL,\emptyset\}$. Furthermore, we say that an arc $\alpha'$ of $\CL$ is a complementary $P$-excursion of $\CL$ out of $U$ if $\alpha'$ does not overlap with any $P$-excursion of $\CL$ into $U$ and $\alpha'$ is not contained in any larger arc of $\CL$ with this property.  We denote by $\CS_\Gamma(P;U)$ (resp.\ $\CC_\Gamma(P;U)$) the set of all proper $P$-excursions into $U$ of loops in $\Gamma$ (resp.\ the set of all complementary $P$-excursions of loops in $\Gamma$ out of $U$). 

By \cite[Corollary~3.7]{gwynne2021conformal}, for $P$ and $U$ as above, $\Gamma$ satisfies the Markov property with respect to $(P,U)$. That is, the following holds. Choose $\alpha \in \CS_\Gamma(P;U)$ in a way which is measurable with respect to $\sigma(\Gamma(P;U),\CS_\Gamma(P;U))$, let $x$ be its endpoint and let $\eta_x$ be the complementary $P$-excursion out of $U$ from $x$ to its endpoint $x^*$ and let $\Sigma_x = \sigma(\Gamma(P;U),\CS_\Gamma(P;U),\CC_\Gamma(P;U)\setminus \eta_x)$. Then,
\begin{enumerate}
	\item $\closure{\cup \Gamma(P;U)}$ is almost surely connected,
	\item if $\CS_\Gamma(P;U) \neq \emptyset$ and we condition on $\Sigma_x$, then the conditional law of $\eta_x$ is that of an independent $\SLE_{\kappa'}$ from $x$ to $x^*$ in the connected component of $\h \setminus \closure{\cup \Gamma(P) \setminus \eta_x}$ with $x$ on its boundary,
	\item if we condition on $\Gamma(P)$, then the conditional law of $\Gamma|_{\h \setminus \closure{\cup \Gamma(P)}}$ is that of a collection of independent $\CLE_{\kappa'}$ in the connected components of $\h \setminus \closure{\cup \Gamma(P)}$.
\end{enumerate}

\subsection{Liouville quantum gravity}
\label{subsec:lqg}

Suppose that $h$ is an instance (of some form) of the GFF $h$ on a domain $D$.  The Liouville quantum gravity (LQG) surface described by $h$ is formally given by the metric tensor
\begin{equation}
\label{eqn:lqg_metric}
e^{\gamma h(z)} (dx^2 + dy^2) \quad\text{where}\quad \gamma \in (0,2]
\end{equation}
is a parameter.  This expression does not make literal sense as the GFF is not a function and does not take values at points.  The associated measure was defined in \cite{ds2011lqg} and also the boundary length measure \cite{she2016zipper,dms2021mating}.  The way that the area measure is defined is via a regularization procedure.  Namely, for each $\epsilon > 0$ and $z \in D$ we let $h_\epsilon(z)$ denote the average of $h$ on $\partial B(z,\epsilon)$.  Then we have that $\qmeasure{h}$ is given by the limit as $\epsilon \to 0$ of
\begin{align*}
\epsilon^{\gamma^2/2} e^{\gamma h_\epsilon(z)} dz
\end{align*}
where $dz$ denotes Lebesgue measure.  Due to the choice of normalization, the measure $\qmeasure{h}$ satisfies the following coordinate change rule.  Suppose that $D, \wt{D}$ are domains, $\psi \colon D \to \wt{D}$ is a conformal map, $\wt{h}$ is a GFF on $\wt{D}$, and 
\begin{equation}
\label{eqn:change_of_coordinates}
h = \wt{h} \circ \psi + Q \log |\psi'| \quad\text{where}\quad Q = \frac{2}{\gamma} + \frac{\gamma}{2}.
\end{equation}
Then we have that $\qmeasure{h}(A) = \qmeasure{\wt{h}}(\psi(A))$ for all $A \subseteq D$ Borel.  If $h$, $\wt{h}$ are related as in~\eqref{eqn:change_of_coordinates}, then we say that $(D,h)$, $(\wt{D},\wt{h})$ are equivalent as quantum surfaces.  A quantum surface is an equivalence class under this equivalence relation.  One can also consider quantum surfaces with marked points and in this case the conformal map $\psi$ is required to preserve the marked points.   In the case that $h$ has free boundary conditions on a linear segment $L$, the boundary length measure $\qbmeasure{h}$ is defined to be the limit as $\epsilon \to 0$ of
\begin{align*}
\epsilon^{\gamma^2/4} e^{\gamma h_\epsilon(z)/2} dz
\end{align*}
where $dz$ denotes Lebesgue measure on $L$.  We have that $\qbmeasure{h}$ satisfies the same change of coordinates formula as for $\qmeasure{h}$ and this allows one to define $\qbmeasure{h}$ on boundary segments which are not necessarily linear.

We will now give the definition of two of the quantum surfaces which will be important for this work: the \emph{quantum wedge} and the \emph{quantum cone}.

We let $\CH_1(\h)$ (resp.\ $\CH_2(\h)$) be the subspace of those functions in $H(\h)$ which are constant (resp.\ have mean zero) on $\h \cap \partial B(0,r)$ for all $r \geq 0$.  Then we have that $H(\h) = \CH_1(\h) \oplus \CH_2(\h)$ gives an orthogonal decomposition of $H(\h)$.

We will now give the definition of a quantum wedge.  We note that there are two types of quantum wedge.  A quantum wedge can either be a thick wedge or a thin wedge, the former meaning that it is homeomorphic to $\h$, and the latter meaning that it consists of a Poissonian chain of surfaces each of which is homeomorphic to $\D$.  We will first start with the thick case and then describe the construction in the thin case.

\begin{definition}
\label{def:quantum_wedge_def}
Fix $\alpha < Q$.  An \emph{$\alpha$-quantum wedge} is the doubly marked quantum surface $\CW = (\h,h,0,\infty)$ whose law can be sampled from as follows.
\begin{enumerate}[(i)]
\item Let $A \colon \R \to \R$ be the process defined as follows.  For $s \geq 0$, $A_s = B_{2s} + \alpha s$ where $B$ is a standard Brownian motion with $B_0 = 0$, and for $s \leq 0$, $A_s = \wh{B}_{-2s} + \alpha s$ where $\wh{B}$ is a standard Brownian motion independent of $B$ with $\wh{B}_0 = 0$ conditioned so that $\wh{B}_{2u} + (Q-\alpha)u > 0$ for all $u > 0$.  We take the projection of $h$ onto $\CH_1(\h)$ to be the function whose common value on $\h \cap \partial B(0,e^{-r})$ is $A_r$ for each $r \in \R$.
\item We take the projection of $h$ onto $\CH_2(\h)$ to be the corresponding projection of a GFF with free boundary conditions independently of its projection onto $\CH_1(\h)$.
\end{enumerate}
\end{definition}

The particular embedding here is called the \emph{circle average embedding} because the embedding is such that $\sup\{r \geq 0 : h_r(0) + Q\log r = 0\} = 1$.

Let $\strip = \R \times (0,\pi)$ be the infinite strip.  It is also natural to parameterize a quantum wedge by $\strip$ instead of by $\h$.  Let $\CH_1(\strip)$ (resp.\ $\CH_2(\strip)$) consist of those functions which are constant (resp.\ have mean zero) on lines of the form $r + (0,i \pi)$.  Then we similarly have that $H(\strip) = \CH_1(\strip) \oplus \CH_2(\strip)$ gives an orthogonal decomposition of $H(\strip)$.  

Then we can sample from the law of a quantum wedge $\CW = (\strip,h,-\infty,\infty)$ as follows.
\begin{enumerate}[(i)]
\item Let $\wt{A} \colon \R \to \R$ be the process defined as follows.  For $s \geq 0$, $\wt{A}_s = B_{2s} + (Q-\alpha) s$ where $B$ is a standard Brownian motion with $B_0 = 0$ conditioned so that $B_{2u} + (Q-\alpha)u > 0$ for all $u > 0$, and for $s \leq 0$, $A_s = \wh{B}_{-2s} + (Q-\alpha) s$ where $\wh{B}$ is a standard Brownian motion independent of $B$ with $\wh{B}_0 = 0$.  We take the projection of $h$ onto $\CH_1(\strip)$ to be the function whose common value on $r + (0,i\pi)$ is $A_r$ for each $r \in \R$.
\item We take the projection of $h$ onto $\CH_2(\strip)$ to be the corresponding projection of a GFF with free boundary conditions independently of its projection onto $\CH_1(\strip)$.
\end{enumerate}

We can similarly parameterize the space of quantum wedges by \emph{weight} $W$ instead of by $\alpha$.  Using weight instead of $\alpha$ is often convenient because, as we will recall momentarily, weight is additive under the welding operation.  The relationship between $W$ and $\alpha$ is given by
\begin{equation}
\label{eqn:q_wedge_weight}
 W = \gamma(\gamma/2 + Q - \alpha).
\end{equation}

A compact way to describe the process $\wt{A}$ in the definition of an $\alpha$-quantum wedge parameterized by $\strip$ is as follows.  Let $Z \sim \BES^\delta$ where
\begin{equation}
\label{eqn:bessel_wedge_dimension}
\delta = 2+ \frac{2(Q-\alpha)}{\gamma} = 1 + \frac{2}{\gamma^2} W.
\end{equation}
If we reparameterize the process $2\gamma^{-1} \log Z$ so that its quadratic variation is $2ds$ and then take the horizontal translation so that it last time it hits $0$ is at time $0$ we obtain a process with the same law as $\wt{A}$.  Note that $\alpha < Q$ corresponds to the Bessel process dimension being at least $2$.  This definition makes sense, however, even if the Bessel process dimension is in $(0,2)$.

\begin{definition}
\label{def:thin_quantum_wedge_def}
Fix $\alpha \in (Q,Q+\gamma/2)$.  An \emph{$\alpha$-quantum wedge} is the beaded quantum surface (that is, a countable collection of quantum surfaces) $\CW$ whose law can be sampled from as follows.  Let $\delta = 2+2(Q-\alpha)/\gamma$ and let $Z \sim \BES^\delta$.  For each excursion $e$ that $Z$ makes from $0$ we let $(\strip,h_e,-\infty,\infty)$ be the doubly marked surface where:
\begin{enumerate}[(i)]
\item The projection of $h_e$ onto $\CH_1(\strip)$ is equal to $2\gamma^{-1} \log e$ reparameterized to have quadratic variation $2ds$.
\item The projection of $h_e$ onto $\CH_2(\strip)$ is given by the corresponding projection of a GFF with free boundary conditions independently of its projection onto $\CH_1(\strip)$.
\end{enumerate}
\end{definition}

\begin{remark}
\label{rem:bessel_area_encoding}
In the encoding in Definition~\ref{def:thin_quantum_wedge_def} the conditional law of the quantum area of a bead $(\strip,h_e,-\infty,\infty)$ is equal to that of the length of $e$ times a random variable $X$ which is independent of $e$.  This implies that we can associate with a thin quantum wedge $\CW$ a Bessel process $Z$ with the same dimension as in Definition~\ref{def:thin_quantum_wedge_def} so that the lengths of the excursions of $Z$ from $0$ are exactly equal to the quantum areas of the successive beads in $\CW$.
\end{remark}

The following is the basic cutting and welding result for quantum wedges \cite{she2016zipper,dms2021mating}.

\begin{theorem}
\label{thm:wedge_cutting}
Fix $W > 0$ and suppose that $\CW$ is a quantum wedge of weight $W$.  Suppose that $\rho_1,\rho_2 > -2$ are such that $\rho_1 + \rho_2 + 4 = W$ and that $\eta$ is an $\SLE_\kappa(\rho_1;\rho_2)$ in $\CW$ from $0$ to $\infty$. Set $W_i = \rho_i + 2$ for $i=1,2$. Let~$\CW_1$ (resp.\ $\CW_2$) be the quantum surfaces parameterized by the regions which are to the left (resp.\ right) of $\eta$.  Then $\CW_1, \CW_2$, respectively, are independent quantum wedges of weight $W_1$, $W_2$.
\end{theorem}

It is is also natural to consider $\SLE_{\kappa'}$ processes on top of certain types of quantum wedges.

\begin{theorem}
\label{thm:sle_kappa_prime}
Let $\CW = (\h,h,0,\infty)$ be a quantum wedge of weight $3\gamma^2/2-2$ and let $\eta'$ be an $\SLE_{\kappa'}$ in $\h$ from $0$ to $\infty$ which is sampled independently of $\CW$.  For each $t \geq 0$ we let $T_t^1$ (resp.\ $B_t^1$) be the quantum length of the part of $\partial \h_t \setminus \partial \h$ (resp.\ $\partial \h \setminus \partial \h_t$) which is to the left of $\eta'(t)$ (resp.\ $\eta'(0)$).  Define $T_t^2$ (resp.\ $B_t^2$) in the same way but with left replaced by right.  Set $X_t^i = T_t^i - B_t^i$.  There exists a time parameterization of $\eta'$ so that $(X_t^1,X_t^2)$ are independent $\kappa'/4$-stable L\'{e}vy processes.
\end{theorem}

\begin{remark}\label{rmk:quantum_natural_time}
The time parameterization in Theorem~\ref{thm:sle_kappa_prime} is the so-called \emph{quantum natural time} of $\eta'$. The quantum natural time of $\eta'$ can be recovered as follows. If we denote by $\wt{N}_\epsilon^{h,\eta'}(O)$ the number of bubbles in $O$ with quantum boundary length in $[\epsilon,2\epsilon]$ (with respect to $h$), cut out by $\eta'$, then the amount of quantum natural time that $\eta'$ spends in $O$ is given by $\lim_{\epsilon \to 0} \epsilon^{\kappa'/4} \wt{N}_\epsilon^{h,\eta'}(O)$.
\end{remark}

\begin{remark}
\label{rem:sle_kp_divide}
The law of the left boundary of an $\SLE_{\kappa'}$ process is that of an $\SLE_\kappa(\kappa-4; \kappa/2-2)$ and the conditional law of the right boundary of an $\SLE_{\kappa'}$ given its left boundary is that of an $\SLE_\kappa(-\kappa/2; \kappa-4)$ (see, e.g., \cite[Figure~2.5]{mw2017intersections}).  It thus follows by applying Theorem~\ref{thm:wedge_cutting} twice that if we start with a quantum wedge $\CW = (\h,h,0,\infty)$ of weight $3\gamma^2/2-2$ and let $\eta'$ be an independent $\SLE_{\kappa'}$ on $\h$ from $0$ to $\infty$ then the quantum surfaces $\CW_1$, $\CW_2$, $\CW_3$, which are respectively parameterized by the components which are to the left of $\eta'$, between the left and right boundaries of $\eta'$, and to the right of $\eta'$ are independent quantum wedges with weights $\gamma^2-2$, $2-\gamma^2/2$, and $\gamma^2-2$.
\end{remark}

We let $\CH_1(\C)$ (resp.\ $\CH_2(\C)$) be the subspace of those functions in $H(\C)$ which are constant (resp.\ have mean zero) on $\partial B(0,r)$ for all $r \geq 0$.  Then we have that $H(\C) = \CH_1(\C) \oplus \CH_2(\C)$ gives an orthogonal decomposition of $H(\C)$.

\begin{definition}
\label{def:quantum_cone_def}	
Fix $\alpha < Q$.  An \emph{$\alpha$-quantum cone} is the doubly marked quantum surface $\CC = (\C,h,0,\infty)$ whose law can be sampled from as follows.
\begin{enumerate}[(i)]
\item Let $A \colon \R \to \R$ be the process defined as follows.  For $s \geq 0$, $A_s = B_s + \alpha s$ where $B$ is a standard Brownian motion with $B_0 = 0$, and for $s \leq 0$, $A_s = \wh{B}_{-s} + \alpha s$ where $\wh{B}$ is a standard Brownian motion independent of $B$ with $\wh{B}_0 = 0$ conditioned so that $\wh{B}_u + (Q-\alpha)u > 0$ for all $u > 0$.  We take the projection of $h$ onto $\CH_1(\C)$ to be the function whose common value on $\partial B(0,e^{-r})$ is $A_r$ for each $r \in \R$.
\item We take the projection of $h$ onto $\CH_2(\C)$ to be the corresponding projection of a whole-plane GFF independently of its projection onto $\CH_1(\C)$.
\end{enumerate}
\end{definition}

The particular embedding here is called the \emph{circle average embedding} because the embedding is such that $\sup\{r \geq 0 : h_r(0) + Q\log r = 0\} = 1$.

It is convenient to parameterize the space of quantum cones using weight $W$ rather than $\alpha$.  The relationship between these two quantities is
\begin{equation}
\label{eqn:q_cone_weight}
 W = 2\gamma(Q-\alpha).
\end{equation}

The following is the basic cutting result for quantum cones \cite{dms2021mating}.

\begin{theorem}
\label{thm:cone_cutting}
Fix $W > 0$ and suppose that $\CC$ is a quantum cone of weight $W$.  Let $\eta$ be an $\SLE_\kappa(W-2)$ process independent of $\CC$.  Then the quantum surfaces parameterized by the components in the complement of $\eta$ are a quantum wedge of weight $W$.
\end{theorem}

\subsection{Independence of the GFF across nested annuli}
\label{subsec:good_annuli}

Suppose that $h$ is an instance of the GFF on $\h$, $z \in \h$, and $r \in (0,\im(z)/2)$.  Let~$\Fh_{z,r}$ be the distribution on~$\h$ which is harmonic in $B(z,r)$ so that we can write $h = h_{z,r} + \Fh_{z,r}$ where~$h_{z,r}$ is a zero-boundary GFF on $B(z,r)$ independent of~$\Fh_{z,r}$.  We say that $(z,r)$ is $M$-good for~$h$ if
\[ \sup_{w \in B(z,15r/16)} |\Fh_{z,r}(w) - \Fh_{z,r}(z)| \leq M.\]
It was proved in \cite[Proposition~4.3]{mq2020geodesics} that the $M$-good scales are very likely to be dense among all scales, which we restate below for the convenience of the reader.

\begin{lemma}
\label{lem:good_dense}
Fix $z \in \h$ and $r \in (0,\im(z)/2)$.  For each $k \in \N$ we let $r_k = 2^{-k} r$.  Fix $K \in \N$ and let $N = N(K,M)$ be the number of $1 \leq k \leq K$ so that $B(z,r_k)$ is $M$-good.  For every $a > 0$ and $b \in (0,1)$ there exists $M_0 = M(a,b)$ and $c_0 = c_0(a,b)$ so that for all $M \geq M_0$ we have
\[ \p[ N(K,M) \leq b K ] \leq c_0 e^{- a K}.\]	
\end{lemma}

We let $E_{z,r}^M$ be the event that $B(z,r)$ is $M$-good.  We also recall from \cite{mq2020geodesics} the following Radon-Nikodym derivative estimate for the $M$-good scales (\cite[Lemma~4.1]{mq2020geodesics}).

\begin{lemma}
\label{lem:good_scale_radon}
Fix $z \in \h$ and $r \in (0,\im(z)/2)$.  The conditional law of $h - h_r(z)$ restricted to $B(z,7r/8)$ is mutually absolutely continuous with respect to the law of a zero-boundary GFF in $B(z,r)$ restricted to $B(z,7r/8)$,  where $h_r(z)$ denotes the average of $h$ on $\partial B(z,r)$.  Let $\CZ_{z,r}$ be the Radon-Nikodym derivative of the latter with respect to the former.  For each $p, M > 0$ there exists a constant $c_{p,M}$ depending only on $p,M$ such that
\[ \E[ \CZ_{z,r}^p \giv \CF_{z,r}] \one_{E_{z,r}^M} \leq c_{p,M} \one_{E_{z,r}^M}.\]
\end{lemma}

\subsection{A removability theorem}\label{subsec:removability_theorem}
We next introduce the removability theorem used to prove the conformal removability of $\SLE_4$, namely \cite[Theorem~8.1]{kms2022sle4remov}. Below, we denote by $\disthyp^D(z,w)$ the hyperbolic distance between $z$ and $w$ in the domain $D$ (to be properly defined in Section~\ref{sec:good_cube_lemmas}) and for a square $Q \subseteq \C$, we denote by $\cen(Q)$ its center and by $\len(Q)$ its side length. We now turn to the conditions.

Let $X \subseteq \C$ and suppose that there exists some $a \in (0,1)$ such that for each compact $K \subseteq X$, there exists $M>1$ such that the following hold.
\begin{itemize}
	\item Let $\udimM(K)$ denote the upper Minkowski dimension of $K$. Then $\udimM(K) < 2$ and $0 < a < (2 - \udimM(K))/5$.
	\item There exists a family $\CA = \cup_{k = 1}^\infty \CA_k$ of open subsets of $\C$ such that each $A \in \CA_k$ is a topological annulus with $\diam(A) \leq M 2^{-k}$. For $A \in \CA$, we denote by $\pin A$ (resp.\ $\pout A$) the boundary of the bounded (resp.\ unbounded) connected component of $\C \setminus A$.
	\item There exists $n_0 \in \N$ such that for all $n \geq n_0$ and each $z \in K$, there exist $k \in \N$ such that $(1-a^2)n \leq k \leq n$ and $A \in \CA_k$ so that $B(z,2^{-n})$ is contained in the bounded connected component of $\C \setminus A$ and such that the following hold. 
	\begin{enumerate}[(I)]
		\item\label{it:components} There exist $1 \leq m \leq M$ pairwise disjoint, open and simply connected subsets $U_1, \dots,U_m$ of $A \setminus X$ such that for each $1 \leq i \leq m$, there exists $I_i \subseteq \partial U_i \cap \partial U_{i+1}$ closed (with the convention that $U_0 = U_m$ and $U_{m+1} = U_1$) such that each loop $\gamma : \s^1 \to \closure{ \cup_{i=1}^m U_i}$ hits each $I_i$ exactly once, is disjoint from $(\partial U_i \cup \partial U_{i+1}) \setminus I_i$ and disconnects $\pin A$ from $\pout A$.
		\item\label{it:intersections} For each $1 \leq i \leq m$, there exists $d_i \in (10a,2-10a)$ and a (positive) finite measure $\mu_i$ supported on $I_i$ so that the following hold. Let $\ol{\mu}_i$ be $\mu_i$ normalized to be a probability measure.
		\begin{enumerate}[(a)]
			\item\label{it:intersection_lbd} We have that $\mu_i(I_i) \geq M^{-1} 2^{-d_i k}$.
			\item\label{it:intersection_ubd} For every Borel set $Y \subseteq I_i$, we have that $\mu_i(Y) \leq M \diam(Y)^{d_i - a}$.
			\item\label{it:g_assumptions} For each $1 \leq i \leq m$, let $\CW_i$ be a Whitney square decomposition of $U_i$ and for each $Q \in \CW_i$. There exists a square $Q_i \in \CW_i$ such that the following holds. Let $G_i^-$ be the set of $w \in I_{i-1}$ (where $I_0 = I_m$) such that if $\gamma_{i,w}$ is the hyperbolic geodesic from $\cen(Q_i)$ to $w$,then for each $Q$ which is hit by $\gamma_{i,w}$, $\disthyp^{U_i}(\cen(Q_i),\cen(Q)) \leq M (2^k \len(Q))^{-a}$. We let $G_i^+$ be defined in the same way, but by replacing $I_{i-1}$ with $I_i$. Then
			\begin{enumerate}[(i)]
				\item\label{it:g_meas_lbd} $\ol{\mu}_{i-1}(G_i^-) \geq 3/4$ and $\ol{\mu}_i(G_i^+) \geq 3/4$ and
				\item\label{it:g_meas_ubd} The number of squares in $\CW_i$ with side length $2^{-j}$ which are hit by a hyperbolic geodesic from $\cen(Q_i)$ to a point in $G_i^+$ (resp.\ $G_i^-$) is at most $M 2^{(d_i + a)(j-k)}$ (resp.\ $M 2^{(d_{i-1}+a)(j-k)}$).
			\end{enumerate}
		\end{enumerate}
	\end{enumerate}
\end{itemize}

\begin{figure}[ht!]
\begin{center}
\includegraphics[scale=0.9]{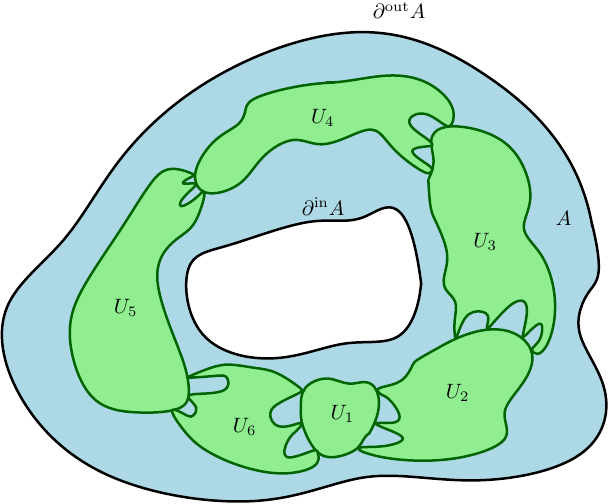}	
\end{center}
\caption{\label{fig:annulus_condition}  Illustration of the assumptions above. The sets $U_1,\ldots,U_m$ (with $m=6$) from~\eqref{it:components} are shown in green.  The main assumption is~\eqref{it:intersections}, which controls the behavior of $\partial U_i \cap \partial U_{i+1}$ for each $1 \leq i \leq m$ (where $U_{m+1} = U_1$).}  
\end{figure}

\begin{theorem}[Theorem~8.1 of \cite{kms2022sle4remov}]
Suppose that $D \subseteq \C$ is open and $X \subseteq D$ is closed in $D$ and satisfies the above assumptions. If $f : D \to \C$ is a homeomorphism onto its image and is conformal on $D \setminus X$, then $f$ is conformal on $D$.
\end{theorem}

\section{Conformally covariant measure on cut points}
\label{sec:natural_measure_on_cut_points}

The purpose of this section is to construct and analyze the properties of a conformally covariant measure on the cut points of an $\SLE_{\kappa'}$ process $\eta'$ with $\kappa' \in (4,8)$, that is, the set of points $\{ \eta(t): \eta((0,t)) \cap \eta((t,\infty)) = \emptyset \}$.  In particular, in Section~\ref{subsec:measure_construction} we will give the construction of the measure.  Next, in Section~\ref{subsec:measure_locally_finite} we will show that the measure that we have constructed is a.s.\ locally finite and in fact that the expected amount of mass that it assigns to any compact set is finite.  We will then prove that the measure is conformally covariant in Section~\ref{subsec:measure_conformally_covariant}.  Finally, in Section~\ref{subsec:measure_moments} we will show that the positive moments of the measure evaluated on any Borel subset of a fixed compact set are all finite and satisfy a certain precise upper bound when we have truncated on an event which only depends on the compact set.

\subsection{Construction of the measure}
\label{subsec:measure_construction}

Fix $\kappa' \in (4,8)$, let $\kappa = 16/\kappa'$, and let $\gamma = \sqrt{\kappa}$. Let $\CW = (\h,h,0,\infty)$ be a quantum wedge of weight $3\gamma^2/2-2$, let $\eta'$ be an independent $\SLE_{\kappa'}$ process in $\h$ from $0$ to $\infty$, and let $\eta_L$ (resp.\ $\eta_R$) denote the left (resp.\ right) outer boundary of $\eta'$.  Let $\CW_1$, $\CW_2$, and $\CW_3$ respectively be the (beaded) quantum surfaces parameterized by the regions which are to the left of~$\eta_L$, between~$\eta_L$ and~$\eta_R$, and to the right of~$\eta_R$.  Then~$\CW_2$ is a quantum wedge of weight $2-\gamma^2/2$ (Remark~\ref{rem:sle_kp_divide}) and we can couple~$\CW_2$ with $X \sim \BES^\delta$, $\delta = \kappa'/4 \in (1,2)$ (recall~\eqref{eqn:bessel_wedge_dimension}), so that the ordered collection of excursions of~$X$ from $0$ are in correspondence with the ordered collection of beads of $\CW_2$ where the length of each excursion is equal to the quantum area of the corresponding bead (Remark~\ref{rem:bessel_area_encoding}).  In particular, in~$T$ units of local time for~$X$ at~$0$ the number of beads of quantum area at least $\epsilon$ in $\CW_2$ is distributed as a Poisson random variable with mean equal to a constant times $T \epsilon^{\delta/2-1}$.  We note that the local time of $X$ at $0$ induces a measure on $\eta_L \cap \eta_R$ (i.e., the cut points of $\eta'$) which we will denote by $\qcutmeasure{h}{\eta'}$. We have included $h$ in the notation because, as we will see in just a moment, it is natural to consider this measure with another choice of field.

For each Borel set $O \subseteq \h$ and $\epsilon > 0$ we let $N_\epsilon^{h,\eta'}(O)$ denote the number of points in $O$ which are the closing points of beads (i.e., the point on the boundary of a bead last visited by $\eta_L$ or $\eta_R$) of quantum area at least $\epsilon$ in $\CW_2$ which are contained in $O$. Then the sequence of measures $(\epsilon^{1-\delta/2} N_\epsilon^{h,\eta'})_{\epsilon > 0}$ converges a.s.\ in the vague topology to a measure $\qcutmeasure{h}{\eta'}$ as $\epsilon \to 0$.  Now let $h^0$ be a GFF on $\h$ with zero boundary conditions which is taken to be independent of $\eta'$.  Since the law of~$h^0$ restricted to any compact set in $\h$ is absolutely continuous with respect to that of $h$, it follows that the sequence of measures $(\epsilon^{1-\delta/2} N_\epsilon^{h^0,\eta'})$ converges a.s.\ in the vague topology to a measure $\qcutmeasure{h^0}{\eta'}$ as $\epsilon \to 0$.  The same is likewise true if $h^0$ is instead a GFF on a simply connected domain $D \subseteq \C$ and $\eta'$ is an $\SLE_{\kappa'}$ in $D$ between distinct boundary points.  In this more general setting, we then define the measure $\cutmeasure{\eta'}$ by
\begin{align}
\label{eqn:cut_measure_definition}
	\cutmeasure{\eta'}(f) =\int f(z)  r_D(z) d\E[ \qcutmeasure{h^0}{\eta'} \giv \eta'](z) \quad\text{where}\quad r_D(z) = \confrad(z,D)^{2-\frac{8}{\kappa'} -\frac{\kappa'}{8}}
\end{align}
for $f$ bounded and Borel.

Throughout, we will make use of the following standard concentration result for Poisson random variables.  If $Z$ is a Poisson random variable with mean $\lambda$ and $\alpha \in (0,1)$ then
\begin{equation}
\label{eqn:poisson_concentration}
\p[ Z \leq \alpha \lambda ] \leq \exp(\lambda (\alpha- \alpha \log \alpha - 1)).
\end{equation}

\subsection{Local finiteness}
\label{subsec:measure_locally_finite}

Although $\qcutmeasure{h^0}{\eta'}$ a.s.\ defines a locally finite measure on $\h$, it is not immediate that $\cutmeasure{\eta'}$ defines a locally finite measure on $\h$ due to the conditional expectation in its definition~\eqref{eqn:cut_measure_definition}.  The purpose of the next proposition is to establish that this is the case.

\begin{proposition}
\label{lem:local_finiteness_cut_measure}
Almost surely, $\cutmeasure{\eta'}$ is locally finite as a measure on $\h$. In fact, for each compact set $K \subseteq \h$ we have that $\E[\cutmeasure{\eta'}](K) < \infty$.
\end{proposition}

The proof of Proposition~\ref{lem:local_finiteness_cut_measure} will require several steps.  First, in Lemma~\ref{lem:counterflow_close_fills_ball} we will collect a result from \cite{ghm2020almost} which gives that whenever an $\SLE_{\kappa'}$ travels distance at least $\epsilon$ it is very likely to disconnect a ball from $\infty$ which has diameter at least $\epsilon^{1+\delta}$.  Next, in Lemma~\ref{lem:disconnection_time_bound} we will show in the case of a quantum wedge $\CW = (\h,h,0,\infty)$ embedded so that $\nu_h([-1,0]) = 1$ and an $\SLE_{\kappa'}$ process $\eta'$ parameterized by quantum natural time that the amount of time $\tau$ necessary to disconnect $\D \cap \h$ from $\infty$ has a polynomial tail.  This will then be used to prove Lemma~\ref{lem:disconnection_local_time_bound}, i.e., that $\qcutmeasure{h}{\eta'}(\eta'([0,\tau]))$ also has a polynomial tail.  

In order to prove Proposition~\ref{lem:local_finiteness_cut_measure}, we fix two compact subsets $K \subseteq K'$ of $\D \cap \h$ with $\dist(K,K') > 0$ (by scale invariance of the law of $h^0$ the results for general $K \subseteq \h$ then follows) and let $0 < \tau_1 < \sigma_1 < \dots < \tau_N < \sigma_N$ be such that $[\tau_k,\sigma_k]$, $1 \leq k \leq N$ are the excursion time intervals of $\eta'$ from $\partial K$ to $\partial K'$. Then,
\begin{align*}
\qcutmeasure{h}{\eta'}(K) \leq \sum_{k=1}^N \qcutmeasure{h}{\eta'}(\eta'([\sigma_k,\tau_k])) \leq \qcutmeasure{h}{\eta'}(\eta'([0,\tau])).
\end{align*}
From Lemma~\ref{lem:counterflow_close_fills_ball} we deduce a bound on the probability that $N$ is large and Lemma~\ref{lem:disconnection_local_time_bound} provides a bound on the probability that $\qcutmeasure{h}{\eta'}(\eta'([0,\tau]))$ is large. Thus, in bounding $\p[ \qcutmeasure{h}{\eta'}(K) \geq x ]$ we may truncate on the events that $N$ and $\qcutmeasure{h}{\eta'}(\eta'([0,\tau]))$ are not too large. The probability of this latter truncated event can then be upper bounded (roughly) by the probability that the $\qmeasure{h}(K')$ is large, which is unlikely, since $\E[ \qmeasure{h}(K')^p]$ is finite for $p \in (0,4/\gamma^2)$. Then, one can bound $\E[ \qcutmeasure{h^0}{\eta'}(K)]$ in terms of $\E[\qcutmeasure{h}{\eta'}(K')]$ using Lemma~\ref{lem:wedge_exponential_moments_unit_boundary_length}.

\begin{lemma}
\label{lem:counterflow_close_fills_ball}
Let $\eta'$ be an $\SLE_{\kappa'}$ process in $\h$ from $0$ to $\infty$.  Fix $K \subseteq \h$ compact.  For each $\epsilon, \delta > 0$ we let $E_{\epsilon,\delta}^K$ be the event that for every $0 < s < t$ with $\eta'(s), \eta'(t) \in K$ and $|\eta'(s)-\eta'(t)| \leq \epsilon$ we have that $\eta'([s,t])$ disconnects from $\infty$ a ball with diameter at least $|\eta'(s) - \eta'(t)|^{1+\delta}$.  Then $\p[ E_{\epsilon,\delta}^K] \to 1$ as $\epsilon \to 0$ (but with $\delta$, $K$ fixed) faster than any power of $\epsilon$.
\end{lemma}
\begin{proof}
This follows from \cite[Lemma~3.6]{ghm2020almost}.
\end{proof}

\begin{lemma}
\label{lem:disconnection_time_bound}
Suppose that $\CW = (\h,h,0,\infty)$ is a quantum wedge of weight $3\gamma^2/2-2$ with the embedding into~$\h$ such that $\qbmeasure{h}([-1,0]) = 1$ and $\eta'$ is an independent $\SLE_{\kappa'}$ in $\h$ from $0$ to $\infty$ which is subsequently parameterized by quantum natural time.  Let $\tau$ be the first time $t$ that $\eta'|_{[0,t]}$ disconnects $\D \cap \h$ from $\infty$.  There exists $\alpha > 0$ so that $\p[ \tau \geq t ] = O(t^{-\alpha})$.
\end{lemma}
\begin{proof}
Let $\CW = (\h,h,0,\infty)$ be a quantum wedge of weight $3\gamma^2/2-2$ with the embedding into~$\h$ so that $\qbmeasure{h}([-1,0]) = 1$.  Let $\eta'$ be an independent $\SLE_{\kappa'}$ process in $\h$ from $0$ to $\infty$ which is subsequently parameterized by quantum natural time.  Let $(X^1,X^2)$ be the left/right boundary length process for $\eta'$.  For $j \in \{1,2\}$ and $t \geq 0$ we let $I_t^j = \inf_{0 \leq s \leq t} X_s^j$.  Let
\[ \tau_1^1 = \inf\{t \geq 0 : I_t^1 \leq -1\} \quad\text{and}\quad \tau_1^2 = \inf\{t \geq \tau_1^1 : X_t^2 = I_t^2\}.\]
Given that we have defined $\tau_1^1,\tau_1^2,\ldots,\tau_n^1,\tau_n^2$, we set
\[ \tau_{n+1}^1 = \inf\{t \geq \tau_n^2 : X_t^1 = I_t^1\} \quad\text{and}\quad \tau_{n+1}^2 = \inf\{t \geq \tau_{n+1}^1 : X_t^2 = I_t^2\}.\]
We note that the successive time-intervals $[\tau_n^1,\tau_n^2]$ and $[\tau_n^2,\tau_{n+1}^1]$ correspond to the successive excursions that~$\eta'$ makes between $\R_-$ and $\R_+$ starting with the first time it hits $(-\infty,-1]$.  Let~$E_n$ be the event that $\eta'|_{[0,\tau_n^2]}$ has disconnected $\D \cap \h$ from $\infty$.  Let also $\CF_t = \sigma(\eta'(s) : s \leq t)$. Next, writing $\tau_0^2 = 0$, we note that there exists $p \in (0,1)$ such that for all $n \geq 0$,
\begin{align}\label{eq:lower_bound_separation}
	\p[ E_{n+1} \giv \CF_{\tau_n^2} ] \geq p.
\end{align}
Indeed, let $(f_t)$ denote the centered Loewner flow of $\eta'$, let $K_t$ denote its hulls, $\ell_t = \inf\{ x: x \in K_t \cap \R\}$, $r_t = \sup \{ x: x \in (K_t \cap \R) \cup \{ 1 \} \}$ and define the family of conformal maps $(\varphi_t)$ by 
\begin{align*}
	\varphi_t(z) = 2 \frac{f_t(z) - f_t(\ell_t)}{f_t(r_t) - f_t(\ell_t)} - 1. 
\end{align*}
Assume we are on the event that $\eta'|_{[0,\tau_n^2]}$ has not disconnected $\h \cap \D$ from $\infty$.  With $\h_t$ denoting the unbounded component of $\h \setminus \eta'([0,t])$, we have that $\varphi_t: \h_t \to \h$, $\varphi(\ell_t) = -1$, $\varphi_t(r_t) = 1$, and $\varphi_t(\infty) = \infty$.  Set $x_n = \varphi_{\tau_n^2}(\eta'(\tau_n^2))$ and note that $x_n \in (-1,1]$.  Let also $\wt{\eta}'$ be the image under $\varphi_{\tau_n^2}$ of the time-reversal of $\eta'|_{[\tau_n^2,\infty)}$.  Note that there exists a universal constant $p_0 \in (0,1)$ such that for every $z \in \D \cap \h_{\tau_n^2}$,  the probability that a Brownian motion starting from $z$ exits $\h_{\tau_n^2}$ in $\eta'([0,\tau_n^2]) \cup [0,1]$ is at least $p_0$.  Thus,  combined with the Beurling estimate,  this implies that there exists $R<\infty$ large enough (depending only on $p_0$) such that $\varphi_{\tau_n^2}(\D \cap \h_{\tau_n^2}) \subseteq \h \cap B(0,R/8)$.  Consider the conformal automorphism~$f$ of~$\h$ given by $f(z) = -R/(z-x_n)$ and set $\wh{\eta}' = f(\wt{\eta}')$.  Then by \cite[Theorem~1.1]{ms2016imag3} we have that~$\wh{\eta}'$ has the law of an $\SLE_{\kappa'}$ in $\h$ from $0$ to $\infty$ conditionally on $\eta'|_{[0,\tau_n^2)}$.  Moreover,  by applying \cite[Lemma~2.5]{mw2017intersections} three times,  we obtain that there exists a constant $p_1 \in (0,1)$ such that with probability at least $p_1$,  the curve $\wh{\eta}'$ hits $(1/2,1)$ before hitting $(-1,-1/2)$ or exiting $\D \cap \h$ and then next hits $\R \setminus (1/2,1)$ in $(-1,-1/2)$ before exiting $\D \cap \h$ and then next hits $\R \setminus (-1,-1/2)$ in $(R,\infty)$.  By taking $R$ sufficiently large, if the latter occurs, we have that $\wt{\eta}'$ first hits $(-4R,-R/4)$ and then $(R/4 , 4R)$ before exiting $\h \setminus B(0,R/4)$ and then it hits $\R \setminus (R/4, 4R)$ next in $(x_n-1,x_n)$.  Note that $\varphi_{\tau_n^2}^{-1}((x_n-1,x_n)) \subseteq \eta'([0,\tau_n^2])$ and so if the latter event occurs,  then the time-reversal of $\eta'|_{[\tau_n^2,\infty)}$ first hits $(-\infty,-1)$ and then $(1,\infty)$ before hitting $\D \cap \h$,  and then next hits $\partial \h_{\tau_n^2} \setminus (1,\infty)$ in $\eta'([0,\tau_n^2])$ (and necessarily to the left of $\eta'(\tau_n^2)$).  On this event, we have that $\eta'([0,\tau_{n+1}^2])$ disconnects $\D \cap \h$ from $\infty$ so that~\eqref{eq:lower_bound_separation} holds.

We let $T_1^1 = \tau_1^1$ and then set
\begin{align*}
	T_n^1 = \frac{\tau_n^1 - \tau_{n-1}^2}{(X_{\tau_{n-1}^2}^1 - I_{\tau_{n-1}^2}^1)^{\kappa'/4}} \quad\text{for} \quad n \geq 2 \quad\text{and}\quad T_n^2 = \frac{\tau_n^2 - \tau_n^1}{(X_{\tau_n^1}^2 - I_{\tau_n^1}^2)^{\kappa'/4}} \quad\text{for} \quad n \geq 1.
\end{align*}
By scaling, the strong Markov property of $(X^1, X^2)$, and the independence of $X^1$ and $X^2$, we have that the sequences of random variables $(T_n^1)$, $(T_n^2)$ are both i.i.d., are independent and have the same distribution as each other.  We claim that
\begin{equation}
\label{eqn:t_n_tail_bound}
\p[T_{n+1}^1 \geq x \giv \CF_{\tau_n^2}] = O(x^{4/\kappa' - 1}) \quad\text{and}\quad \p[T_{n+1}^2 \geq x \giv \CF_{\tau_{n+1}^1}] = O(x^{4/\kappa' - 1}).
\end{equation}
To see this, first note that $\p[T_{n+1}^1 \geq x \giv \CF_{\tau_n^2}] = \p[T_{n+1}^2 \geq x \giv \CF_{\tau_{n+1}^1}] = \p[\tau_1^1 \geq x]$. Then, since $|X_1|$ is not a subordinator,~\eqref{eqn:t_n_tail_bound} follows from~\cite[Chapter~VIII, Proposition~2]{bertoin1996levy} and scaling; note that $\rho = 1 - 4/\kappa'$ in the context of \cite[Chapter~VIII]{bertoin1996levy}).

Next,  fix $u \in (0, 1-4/\kappa')$ and note that $C = \E[ (T_i^j)^{u}] = \E[ (\tau_1^1)^{u} ] < \infty$ for each $1\leq i \leq n$ and $j=1,2$.  Then by the Chernov bound,
\begin{align}
\label{eqn:log_t_n_tail}
\p\left[ \sum_{j=1}^2 \sum_{i=1}^n \log T_i^j \geq c_1 n \right] \leq e^{-c_1 u n} \prod_{j=1}^2 \prod_{i=1}^n \E[  (T_i^j)^{u} ] = C^{2n} e^{-c_1 u n} = e^{-c_2 n},
\end{align}
where $c_2 = c_1 u - 2 \log C$ and $c_1 > 0$ is chosen large enough so that $c_2 > 0$.

We recall that $\tau_0^2 = 0$ and define
\begin{align*}
	R_n^1 = \frac{(X_{\tau_n^2}^1 - I_{\tau_n^2}^1)^{\kappa'/4}}{\tau_n^2 - \tau_n^1} \quad\text{and}\quad R_n^2 = \frac{(X_{\tau_n^1}^2 - I_{\tau_n^1}^2)^{\kappa'/4}}{\tau_n^1 - \tau_{n-1}^2} \quad \text{for} \quad n \geq 1.
\end{align*}
As above, by the strong Markov property of $(X^1,X^2)$ and the independence of $X^1,X^2$, the sequences of random variables $(R_n^1)$, $(R_n^2)$ are both i.i.d., independent and have the same distribution as each other.  In particular,  we have that $R_n^j$ has the same law as $(Z_1^j)^{\kappa'/4}$ where $Z_t^j = X_t^j - I_t^j$ for $j=1,2$,  $n \in \N$,  and $t \geq 0$.  Note that there exist constants $c_3,c_4 > 0$ so that
\begin{equation}
\label{eqn:x_time_tail}
\p[ R_n^1 \geq x \giv \CF_{\tau_n^1}] \leq c_3 e^{-c_4 x^{4/\kappa'}} \quad\text{and}\quad \p[ R_{n+1}^2 \geq x \giv \CF_{\tau_n^2}] \leq c_3 e^{-c_4 x^{4/\kappa'}}.
\end{equation}
Indeed, letting $S_t^j = \sup_{0 \leq s \leq t} X_t^j$, for $t \geq 0$, $j=1,2$, we have by~\cite[Chapter~VI, Proposition~3]{bertoin1996levy} that the processes $Z^j$ and $S^j$ have the same law.  Moreover, \cite[Chapter~VII, Corollary~2]{bertoin1996levy} implies that~$S_1^j$, and hence~$Z_1^j$, has exponential moments of all orders, which gives~\eqref{eqn:x_time_tail}.
By adjusting the values of $c_1, c_2 > 0$ we have in addition to~\eqref{eqn:log_t_n_tail} that 
\begin{equation}
\label{eqn:log_r_n_tail}
\p\left[ \sum_{j=1}^2 \sum_{i=1}^n \log R_i^j \geq c_1 n \right] \leq e^{-c_2 n}.
\end{equation}
Writing $R_0^1 = 1$ and recalling that $T_1^1 = \tau_1^1$, we note for each $j \in \N$ that
\begin{align*}
\log(\tau_{j+1}^1-\tau_j^2) &= \log(T_{j+1}^1) +  \sum_{i=1}^j \left( \log T_{i}^1 + \log T_i^2 + \log R_i^1 + \log R_i^2 \right) \quad \text{and} \\
\log(\tau_j^2-\tau_j^1) &= \sum_{i=1}^j \left( \log T_i^1 + \log T_i^2 + \log R_{i-1}^1 + \log R_i^2 \right).
\end{align*}
It therefore follows from~\eqref{eqn:log_t_n_tail}, \eqref{eqn:log_r_n_tail} that there exist constants $c_5, c_6 > 0$ so that
\begin{align}
\label{eqn:t_i_bound}
\p[ \tau_{j+1}^1 - \tau_j^2 \geq e^{c_5 j} ] \leq e^{-c_6 j} \quad\text{and}\quad \p[ \tau_j^2 - \tau_j^1 \geq e^{c_5 j} ] \leq e^{- c_6 j} \quad\text{for each}\quad j \in \N.
\end{align}
Since $\tau_n^2 = \tau_1^1 + \sum_{j=1}^{n-1}(\tau_{j+1}^1 - \tau_j^2) + \sum_{j=1}^n(\tau_j^2 - \tau_j^1)$ for every $n \geq 2$ and $\tau_1^2 = \tau_1^1 + (\tau_1^2-\tau_1^1)$,  it follows that there exist constants $c_7, c_8 > 0$ so that
\begin{align*}
\p[ \tau_n^2 \geq e^{c_7 n} ]
&\leq \p[\tau_1^1 \geq e^{c_5 n}]+ \sum_{j=1}^{n-1} \p[ \tau_{j+1}^1 - \tau_j^2 \geq e^{c_5 n}] + \sum_{j=1}^n \p[ \tau_j^2 - \tau_j^1 \geq e^{c_5 n}]\\
&= O(e^{-c_8 n}) \quad\text{(by~\eqref{eqn:t_i_bound})}.
\end{align*}
Note that
\[ \p[ \tau \geq t] \leq \p[ E_n, \tau_n^2 \geq t] + \p[E_n^c]\]
so that taking $n = \lfloor c_7^{-1} \log t \rfloor$ and $\alpha = \min\{c_8/c_7 , -\log(1-p)/c_7\}$ completes the proof.
\end{proof}

\begin{lemma}
\label{lem:disconnection_local_time_bound}
Suppose that $\CW = (\h,h,0,\infty)$ is a quantum wedge of weight $3\gamma^2/2-2$ with the embedding into $\h$ such that $\qbmeasure{h}([-1,0]) = 1$ and $\eta'$ is an independent $\SLE_{\kappa'}$ in $\h$ from $0$ to $\infty$ which is subsequently parameterized by quantum natural time.  Let $\tau$ be the first time $t$ that $\eta'|_{[0,t]}$ has disconnected $\D \cap \h$.  There exists $\alpha > 0$ so that
\[ \p[ \qcutmeasure{h}{\eta'}( \eta'([0,\tau])) \geq \ell] = O(\ell^{-\alpha}).\]
\end{lemma}
\begin{proof}
Let $\CW_2$ be the weight $2-\gamma^2/2$ quantum wedge which is between the left and right boundaries of $\eta'$ as described in Section~\ref{subsec:measure_construction} and let $X \sim \BES^{\kappa'/4}$ be the Bessel process encoding $\CW_2$.  For each $\ell \geq 0$ we let $T_\ell$ be the amount of quantum natural time for $\eta'$ which has elapsed in $\ell$ units of local time for~$X$ at~$0$.  We claim that there exist constants $c,b>0$ depending only on $\kappa'$ such that 
\begin{equation}\label{eqn:natural_time_exponential_tail}
\p[T_{\ell} \leq \ell] \leq c \exp(- \ell^b)\quad\text{for every} \quad  \ell >0.
\end{equation}
The idea is that after a certain amount of local time is accumulated, the Bessel process is likely to have generated at least a certain number of beads with lower bounded quantum mass. In each of those beads, $\eta'$ is then likely to accumulate at least a certain amount of quantum natural time in each such bead and together it is very unlikely that less than $\ell$ units of quantum natural time has accumulated in $\ell$ units of local time for $X$. We now spell it out in more detail. Note that by the proof of \cite[Proposition~4.18]{dms2021mating}, the quantum surface parameterized by $\CW_2$ can be sampled by first sampling a p.p.p.  $\Lambda^*$ according to $du \otimes (c t^{\kappa'/4-3}dt)$ (for some $c = c(\kappa')>0$) and then associating with each $(u,e^*)$ a distribution $h_{(u,e^*)}+\frac{2}{\gamma}\log(e^*)$ on $\strip$,  where the $h_{(u,e^*)}$ are i.i.d.\ random variables indexed by $\Lambda^*$ and each one with the law of the distribution as in~\cite[Lemma~4.20]{dms2021mating} with $a = \frac{\sqrt{\kappa'}}{2} - \frac{4}{\sqrt{\kappa'}}$.  Moreover, by~\cite[Proposition~29.12]{kallenberg2021foundations} the quantum natural time that $\eta'$ spends in each bead $B$ of $\CW_2$ is given by the limit as $\epsilon \to 0$ of $c_0 \epsilon^{\kappa'/4}$ times the number of bubbles cut off from $\infty$ by $\eta'$, which have quantum length at least $\epsilon$ and are contained in $B$, for some constant $c_0 = c_0(\kappa) >0$.  Furthermore,  conditional on the left and right outer boundaries of $\eta'$,  the law of $\eta'$ restricted to every bead~$B_{(u,e^*)}$ of $\CW_2$ is that of a chordal $\SLE_{\kappa'}(\kappa'/2 - 4 ; \kappa'/2 -4)$ starting from the opening and ending at the closing point of $B_{(u,e^*)}$,  independently of the parts of $\eta'$ restricted to the rest of the beads of~$\CW_2$.  It follows that for each $(u,e^*) \in \Lambda^*$,  the quantum natural time spent by $\eta'$ in $B_{(u,e^*)}$ is given by $(e^*)^{\kappa'/4} q_{(u,e^*)}$ where $q_{(u,e^*)}$ are i.i.d.  and each one having the law of the limit as $\epsilon \to 0$ of $c_0 \epsilon^{\kappa'/4}$ times the number of bubbles of a chordal $\SLE_{\kappa'}(\kappa'/2 - 4; \kappa'/2 -4)$ in $\strip$ from $-\infty$ to $+\infty$ which are cut off $+\infty$ and with quantum length at least $\epsilon$ with respect to $h_{(u,e^*)}$.

Next we let $N_{\ell}$ be the number of beads in $\CW_2$ which correspond to some $(u,e^*) \in \Lambda^*$ with $u \leq \ell$ and $e^* \geq \ell$.  Then $N_{\ell}$ is a Poisson random variable with mean $c_1 \ell^{\kappa'/4 - 1}$,  for some constant $c_1 >0$. Moreover, by~\eqref{eqn:poisson_concentration} we can find constants $c_2,c_3>0$ such that 
\begin{equation}\label{eqn:number_of_bubbles_concetration}
\p[N_{\ell} \leq c_2 \ell^{\kappa'/4 - 1}] \leq \exp(-c_3 \ell^{\kappa'/4 - 1}) \quad\text{for every}\quad \ell >0.
\end{equation}
We pick $x>0$ such that $\p[q_{(u,e^*)} \geq x] \geq \tfrac{1}{2}$ for $(u,e^*) \in \Lambda^*$.  We also let $\wt{N}_{\ell}$ be the number of $(u,e^*) \in \Lambda^*$ such that $u \leq \ell$,  $e^* \geq \ell$, and $(e^*)^{\kappa'/4} q_{(u,e^*)} \geq x (e^*)^{\kappa'/4}$.  Then on the event $\{ N_\ell \geq c_2 \ell^{\kappa'/4-1} \}$, $\wt{N}_{\ell}$ is stochastically dominated from below by a $\Bin(c_2 \ell^{\kappa'/4-1},\tfrac{1}{2})$ random variable.  It follows from \cite[Lemma~2.6]{mq2020geodesics} that there exist constants $c_4,c_5>0$ such that
\begin{equation}\label{eqn:natural_time_concetration}
\p[\wt{N}_{\ell} \leq c_4 \ell^{\kappa'/4 - 1} \giv N_{\ell} \geq c_2 \ell^{\kappa'/4 - 1}] \leq \exp(-c_5 \ell^{\kappa'/4 - 1})\quad\text{for every}\quad \ell >0.
\end{equation}
Combining~\eqref{eqn:number_of_bubbles_concetration} with~\eqref{eqn:natural_time_concetration} and by possibly taking $c_5$ to be smaller,  we obtain that we can find $c_6>0$ depending only on $\kappa'$ such that
\begin{align*}
\p[\wt{N}_{\ell} \leq c_4 \ell^{\kappa'/4 - 1}] \leq c_6 \exp(-c_5 \ell^{\kappa'/4 - 1})\quad\text{for every} \quad  \ell>0.
\end{align*}
Note that if $\wt{N}_{\ell} \geq c_4 \ell^{\kappa'/4 - 1}$,  we have that $T_{\ell} \geq c_4 x \ell^{\kappa'/2 - 1}$.  Then~\eqref{eqn:natural_time_exponential_tail} follows by taking $b \in (0, (\kappa'/4 - 1)/(\kappa'/2-1))$.

Finally,  let $\alpha > 0$ be the exponent from Lemma~\ref{lem:disconnection_time_bound} and fix $\ell > 0$. Then we have that 
\begin{align*}
\p[ \qcutmeasure{h}{\eta'}( \eta'([0,\tau])) \geq \ell]
&\leq \p[ \qcutmeasure{h}{\eta'}( \eta'([0,\tau])) \geq \ell, \tau \leq \ell] + \p[\tau > \ell]\\
&\leq \p[T_\ell \leq \ell] + O(\ell^{-\alpha}) \leq c \exp(- \ell^b) + O(\ell^{-\alpha}).
\end{align*}
Therefore the assertion of the lemma follows with the same value of $\alpha > 0$ as in Lemma~\ref{lem:disconnection_time_bound}.
\end{proof}

Now we have the ingredients at hand to prove Proposition~\ref{lem:local_finiteness_cut_measure}.

\begin{proof}[Proof of Proposition~\ref{lem:local_finiteness_cut_measure}]

\noindent{\it Step 1. Tail bound for the number of excursions from $K$.}  Let $\CW = (\h,h,0,\infty)$ be a weight $3\gamma^2/2 - 2$ quantum wedge with the embedding into $\h$ such that $\qbmeasure{h}([-1,0]) = 1$, and suppose that $\eta'$ is sampled independently of $\CW$ and then parameterized by quantum natural time.  Let $K \subseteq K' \subseteq \D \cap \h$ be compact with $\dist(\partial K, \partial K') > 0$, $\epsilon \in (0,\dist(\partial K, \partial K'))$, and $\delta > 0$.  We let $E_{\epsilon,\delta}^{K'}$ be the event from Lemma~\ref{lem:counterflow_close_fills_ball}.  Suppose that we are working on $E_{\epsilon,\delta}^{K'}$.  Let $\tau_1 = \inf\{t \geq 0 : \eta'(t) \in K\}$ and $\sigma_1 = \inf\{t \geq \tau_1 : \eta'(t) \in \partial K'\}$.  Given that we have defined $\tau_1,\sigma_1,\ldots,\tau_n,\sigma_n$, we let $\tau_{n+1} = \inf\{t \geq \sigma_n : \eta'(t) \in K\}$ and $\sigma_{n+1} = \inf\{t \geq \tau_{n+1} : \eta'(t) \in \partial K'\}$ (with the convention that a stopping time is infinite if the set of $t$ over which the infimum is taken is empty).  Let $N = \min\{ n \in \N : \tau_{n+1} = \infty\}$ and note that $N$ is a.s.\ finite as $\eta'$ is a.s.\ continuous and transient.  For each $1 \leq k \leq N$ we have that $\eta'([\tau_k,\sigma_k])$ disconnects from $\infty$ a ball of diameter at least $\epsilon^{1+\delta}$.  As these balls must be disjoint for different $k$, there exists a constant $c_0 > 0$ so that $N \leq c_0 \epsilon^{-2-2\delta}$ on $E_{\epsilon,\delta}^{K'}$.  If we choose $\epsilon = (c_0/n)^{1/(2+2\delta)}$ then we have that $\p[ N \geq n ] \leq \p[ (E_{\epsilon,\delta}^{K'})^c]$.  By Lemma~\ref{lem:counterflow_close_fills_ball}, $\p[ (E_{\epsilon,\delta}^{K'})^c] \to 0$ as $\epsilon \to 0$ faster than any power of $\epsilon$.  Therefore $\p[N \geq n] \to 0$ as $n \to \infty$ faster than any negative power of $n$.

\noindent{\it Step 2. Concentration for the quantum area.}  Let $\CW_2$ be the quantum wedge of weight $2-\gamma^2/2$ which is parameterized by the surfaces between the left and right boundaries of $\eta'$ and let $X \sim \BES^{\kappa'/4}$ be the  Bessel process encoding $\CW_2$.  We note that there exists a constant $c_1 = c_1(\kappa') > 0$ so that for any $0 < s < t$ the number of beads in $\CW_2$ with quantum area at least $\epsilon$ which correspond to the interval $[s,t]$ of local time at $0$ for $X$ is distributed as a Poisson random variable with mean $c_1(t-s) \epsilon^{\kappa'/8-1}$.

Fix $\beta > 1$, $\epsilon > 0$, and suppose that $T > 0$ is large.  For each $k \in \N$, we let $I_k = [(k-1) T , k T]$.  Then the number of beads in $\CW_2$ which correspond to the (local time of $X$ at $0$) time interval $I_k$ of quantum area at least $T^{1-\epsilon}$ is distributed as a Poisson random variable with mean
\[ c_1 T \cdot T^{(\kappa'/8-1)(1-\epsilon)} = c_1 T^\zeta \quad\text{where}\quad \zeta = (\kappa'/8-1)(1-\epsilon)+1.\]
By taking $\epsilon > 0$ sufficiently small, we have that $\zeta > 0$.  By~\eqref{eqn:poisson_concentration}, there exists a constant $c_2 > 0$ so that the probability that there does not exist a bead in $\CW_2$ which has quantum area at least $T^{1-\epsilon}$ in the time interval $I_k$ is at most $\exp(-c_2 T^\zeta)$.  By possibly decreasing the value of $c_2>0$ and by taking a union bound over $1 \leq k \leq \lceil T^{\beta-1} \rceil$ we see that the probability that there exists $1 \leq k \leq \lceil T^{\beta-1} \rceil$ for which there does not exist a bead in $\CW_2$ which has quantum area at least $T^{1-\epsilon}$ in the time interval $I_k$ is at most $\exp(-c_2 T^\zeta)$.

\noindent{\it Step 3.  Completion of the proof.}  Let $\tau$ be the first time $t$ that $\eta'$ has disconnected $\D \cap \h$ from $\infty$ and let $L = \qcutmeasure{h}{\eta'}(\eta'([0,\tau]))$.  We have for each $x,y > 0$ and $n \in \N$ that
\begin{align*}
  \p[ \qcutmeasure{h}{\eta'}(K) \geq x]
&\leq \p[ \qcutmeasure{h}{\eta'}(K) \geq x, N \leq n, L \leq y] + \p[ N > n] + \p[L > y].
\end{align*}
Let $\alpha$ be as in Lemma~\ref{lem:disconnection_local_time_bound} and fix $\epsilon > 0$ small.  Let $n = x^\epsilon$ and $y = x^{(4/\gamma^2+1)/\alpha}$.  By Step 1, we have that $\p[ N \geq x^\epsilon] \to 0$ as $x \to \infty$ faster than any negative power of $x$.  By Lemma~\ref{lem:disconnection_local_time_bound}, we have that $\p[L > x^{(4/\gamma^2+1)/\alpha}] = O(x^{-4/\gamma^2-1})$.  Altogether, this yields
\[ \p[ \qcutmeasure{h}{\eta'}(K) \geq x] = \p[ \qcutmeasure{h}{\eta'}(K) \geq x, N \leq x^\epsilon, L \leq x^{(4/\gamma^2+1)/\alpha}] + O(x^{-4/\gamma^2-1}).\]

We note that $L = \qcutmeasure{h}{\eta'}(\eta'([0,\tau])) \geq \sum_{k=1}^N \qcutmeasure{h}{\eta'}(\eta'([\tau_k,\sigma_k])) \geq \qcutmeasure{h}{\eta'}(K)$ so on the event $F_x = \{\qcutmeasure{h}{\eta'}(K) \geq x, N \leq x^{\epsilon}, L \leq x^{(4/\gamma^2+1)/\alpha}\}$ there exists $1 \leq k \leq x^\epsilon$ so that
\[ \qcutmeasure{h}{\eta'}(\eta'([0,\sigma_k])) \leq x^{(4/\gamma^2+1)/\alpha} \quad\text{and}\quad \qcutmeasure{h}{\eta'}(\eta'([\tau_k,\sigma_k])) \geq x^{1-\epsilon}.\]
Assume that we have taken $\beta > 1$ in Step 2 to be at least $(4/\gamma^2+1)/(\alpha (1-\epsilon))$ and $T = x^{1-\epsilon}/2$.  Then we have that $T^\beta > x^{(1+4/\gamma^2)/\alpha}$ for $x>0$ sufficiently large.  By Step 2, off an event whose probability decays to $0$ faster than any power of $x$ we have on $F_x$ that the quantum area disconnected from $\infty$ by $\eta'([\tau_k,\sigma_k])$ is at least $x^{(1-\epsilon)^2}/2^{1-\epsilon}$.  On this event, we have that $\qmeasure{h}(K') \geq x^{(1-\epsilon)^2}/2^{1-\epsilon}$.  It follows that
\begin{align*}
\p[\qcutmeasure{h}{\eta'}(K) \geq x] \leq O(x^{-4/\gamma^2 - 1}) + \p[ \qmeasure{h}(K') \geq x^{(1-\epsilon)^2}/2^{1-\epsilon}] \quad\text{as}\quad x \to \infty.
\end{align*}
Suppose that $h^0$ and $\Fh$ are as in Lemma~\ref{lem:wedge_exponential_moments_unit_boundary_length} and such that $h = h^0+\Fh$.  We fix $\epsilon \in (0,1)$ sufficiently small such that $4(1-\epsilon)^2 / \gamma^2 > 1$ and let $p \in (1,4/\gamma^2)$ be such that $p(1-\epsilon)^2 > 1$.  Then we have that $\qmeasure{h}(K') \leq \exp\left(\gamma \sup_{z \in K'}\Fh(z)\right) \qmeasure{h^0}(K')$ and so combining with Lemma~\ref{lem:wedge_exponential_moments_unit_boundary_length} and \cite[Proposition~3.5]{robert2010gaussian} gives that $\E[\qmeasure{h}(K')^p] < \infty$.  It follows that $\p[\qcutmeasure{h}{\eta'}(K) \geq x] = O(x^{-p(1-\epsilon)^2})$ as $x \to \infty$ and so
\begin{equation}
\label{eqn:q_cut_measure_moment}
\E[\qcutmeasure{h}{\eta'}(K)^{p'}] <\infty \quad\text{for each}\quad p' \in (1,p(1-\epsilon)^2).
\end{equation}
Then, since $N_\epsilon^{h^0,\eta'}|_K \leq N_\epsilon^{h - \inf_{z \in K'} \Fh(z),\eta'}|_K$ it follows that
\begin{equation}
\label{eqn:q_cut_measure_ubd}
\qcutmeasure{h^0}{\eta'}(K) \leq \exp(-\gamma(1-\kappa'/8) \inf_{z \in K'} \Fh(z)) \qcutmeasure{h}{\eta'}(K')	.
\end{equation}
We choose $C<\infty$ such that $\p[E] > 0$ where $E = \{|\inf_{z \in K'}\Fh(z) | \leq C \}$.  Note also that $E \in \sigma(\Fh)$ and so $E$ is independent of $h^0$ and $\eta'$.  In particular,  $E$ is independent of $\qcutmeasure{h^0}{\eta'}(K)$.  Therefore, it follows from Lemma~\ref{lem:wedge_exponential_moments_unit_boundary_length} and H\"older's inequality that
\begin{align*}
 \E[\qcutmeasure{h^0}{\eta'}(K)] \p[E]
&= \E[\qcutmeasure{h^0}{\eta'}(K) \one_E ] \quad\text{(independence of $E$ and $h^0$)}\\
&\leq \E[ \exp(-\gamma(1-\kappa'/8) \inf_{z \in K'} \Fh(z)) \qcutmeasure{h}{\eta'}(K') \one_E] \quad\text{(by~\eqref{eqn:q_cut_measure_ubd})}\\
&\leq e^{\gamma(1-\kappa'/8) C} \E[\qcutmeasure{h}{\eta'}(K')] \quad\text{(definition of $E$)}\\
&< \infty \quad\text{(by~\eqref{eqn:q_cut_measure_moment})}.
\end{align*}
This implies that $\E[ \qcutmeasure{h^0}{\eta'}(K)] < \infty$ since $\p[E] > 0$.

Finally for the general case,  we fix $K\subseteq K'$ arbitrary compact subsets of $\h$ such that $\dist(\partial K ,  \partial K')>0$.  Let $R>0$ be such that $K' \subseteq B(0,R)$ and set $\wh{h}^0 = h^0(R\cdot)$,  $\wh{\eta}' = R^{-1}\eta'$.  Note that $R^{-1}K' \subseteq \D \cap \h$ and that the measures $\qcutmeasure{h^0}{\eta'}$,  $\qcutmeasure{\wh{h}^0}{\wh{\eta}'}$ have the same law.  Therefore it follows that
\begin{align*}
\E[ \qcutmeasure{h^0}{\eta'}(K)] \leq R^{\gamma (1-\kappa'/8)Q} \E[ \qcutmeasure{\wh{h}^0}{\wh{\eta}'}(R^{-1}K')] < \infty.
\end{align*}
The proof of the lemma is then complete since $\sup_{z \in K}r_{\h}(z) < \infty$.
\end{proof}

\subsection{Conformal covariance}
\label{subsec:measure_conformally_covariant}

Recall from \cite[Theorem~1.2]{mw2017intersections} that the cut points of $\eta'$ a.s.\ have dimension
\begin{align*}
	d_{\kappa'}^\cut = 3 - \frac{3 \kappa'}{8}.
\end{align*}
We now prove that $\cutmeasure{\eta'}$ is conformally covariant with exponent $d_{\kappa'}^\cut$. The proof is similar to that of~\cite[Lemma~4.3]{benoist2018natural} and~\cite[Lemmas~3.4 and~4.2]{ms2022volume}.

Before giving the proof of the conformal covariance of $\cutmeasure{\eta'}$, we recall the definition of an image measure.  Suppose that $D, \wt{D} \subseteq \C$ are simply connected domains and $\psi \colon D \to \wt{D}$ is a Borel measurable map.  For a given Borel measure $\nu$ on $D$, the image measure $\nu_\psi$ on $\wt{D}$ is defined by $\nu_\psi(A) = \nu(\psi^{-1}(A))$ for each Borel set $A \subseteq D$.
\begin{proposition}
\label{lem:conformal_covariance_cut_measure}
Suppose that $D, \wt{D} \subseteq \C$ are simply connected domains, $x,y \in \partial D$ are distinct, $\eta'$ is an $\SLE_{\kappa'}$ in $D$ from $x$ to $y$, $\psi \colon D \to \wt{D}$ is a conformal map, and $\wt{\eta}' = \psi(\eta')$.  We a.s.\ have for every continuous function $f$ compactly supported in $\wt{D}$ that
\begin{align}
\label{eqn:conf_cov_formula}
	\int_{\wt{D}} f(z) d \cutmeasure{\wt{\eta}'}(z) = \int_D f \circ \psi(z) |\psi'(z)|^{d_{\kappa'}^\cut} d\cutmeasure{\eta'}(z) = \int_{\wt{D}} f(z) |(\psi^{-1})'(z)|^{-d_{\kappa'}^\cut} d(\cutmeasure{\eta'})_\psi(z).
\end{align}
Equivalently, $(\cutmeasure{\eta'})_\psi$ is absolutely continuous with respect to $\cutmeasure{\wt{\eta}'}$ with Radon-Nikodym derivative $|(\psi^{-1})'|^{-d_{\kappa'}^\cut}$.
\end{proposition}
We note that the second equality in~\eqref{eqn:conf_cov_formula} follows from the definition of an image measure.
\begin{proof}[Proof of Proposition~\ref{lem:conformal_covariance_cut_measure}]
Suppose that $h^0$ and $\wt{h}^0$ are zero-boundary GFFs on $D$ and $\wt{D}$ respectively which are coupled together such that $h^0 = \wt{h}^0 \circ \psi$ and they are both independent of $\eta'$.  Then $\qcutmeasure{\wt{h}^0}{\wt{\eta}'}$ and $(\qcutmeasure{h^0}{\eta'})_\psi$ are both locally finite Borel measures on $\wt{D}$. Since $\psi$ is a homeomorphism it follows that the operation of considering the image measure with respect to $\psi$ preserves vague limits so we have that $\epsilon^{1-\kappa'/8} (N_{\epsilon}^{h^0,\eta'})_\psi$ converges to $(\qcutmeasure{h^0}{\eta'})_\psi$ a.s.\ in the vague topology as $\epsilon \to 0$.  By the invariance of quantum mass under the conformal coordinate change $h^0 \mapsto h^0 \circ \psi^{-1} + Q \log |(\psi^{-1})'| = \wt{h}^0 + Q \log |(\psi^{-1})'|$, we have that $N_\epsilon^{\wt{h}^0 + Q \log|(\psi^{-1})'|,\wt{\eta}'}(A) = (N_\epsilon^{h^0,\eta'})_\psi(A)$ for each Borel set $A \subseteq \wt{D}$ and hence $\epsilon^{1-\kappa'/8} N_\epsilon^{\wt{h}^0 + Q \log|(\psi^{-1})'|,\wt{\eta}'}$ converges to $(\qcutmeasure{h^0}{\eta'})_\psi$ a.s.\ in the vague topology as well. We note that the result follows if we prove that $(\cutmeasure{\eta'})_\psi$ and $\cutmeasure{\wt{\eta}'}$ are mutually absolutely continuous and that $d(\cutmeasure{\eta'})_\psi/d \cutmeasure{\wt{\eta}'} = |(\psi^{-1})'|^{-d_{\kappa'}^\cut}$.  We will proceed by proving the corresponding statements for $(\qcutmeasure{h^0}{\eta'})_\psi$ and $\qcutmeasure{\wt{h}^0}{\wt{\eta}'}$ and we begin with the Radon-Nikodym derivative, that is,
\begin{equation}
\label{eqn:local_density_convergence}
\frac{(\qcutmeasure{h^0}{\eta'})_\psi(B(z,r))}{\qcutmeasure{\wt{h}^0}{\wt{\eta}'}(B(z,r))} \to |(\psi^{-1})'(z)|^{\gamma Q(1 - \kappa'/8)}\quad\text{as}\quad r \to 0\quad\text{for}\quad \qcutmeasure{\wt{h}^0}{\wt{\eta}'} \quad\text{a.e.}\quad  z \in \wt{D}.
\end{equation}
(We note that establishing~\eqref{eqn:local_density_convergence} suffices to identify the Radon-Nikodym derivative by the Lebesgue differentiation theorem \cite[Chapter~3]{folland1999realanalysis}.)  Indeed, a.s.\ the following hold.  Fix $z \in \wt{D}$, $r>0$ such that $\closure{B(z,r)} \subseteq \wt{D}$ and let $(r_m)_{m\in\N}$ be a strictly increasing sequence of positive numbers such that $r_m \to r$ as $m \to \infty$.  For every $m \in \N$,  we fix $f_m \in C_0(\wt{D})$ such that $\one_{\closure{B(z,r_m)}} \leq f_m \leq \one_{B(z,r)}$.  By the monotone convergence theorem we have that 
\begin{align*}
\int_{\wt{D}} f_m d(\qcutmeasure{h^0}{\eta'})_\psi \to (\qcutmeasure{h^0}{\eta'})_\psi(B(z,r)) \quad\text{and}\quad \int_{\wt{D}} f_m d\qcutmeasure{\wt{h}^0}{\wt{\eta}'} \to \qcutmeasure{\wt{h}^0}{\wt{\eta}'}(B(z,r))\quad\text{as} \quad  m \to \infty.
\end{align*}
Note that if $\wt{h}_1 \leq \wt{h}_2$ are two distributions on $\wt{D}$ then $N_\epsilon^{\wt{h}_1,\wt{\eta}'}(A) \leq N_\epsilon^{\wt{h}_2,\wt{\eta}'}(A)$ for each Borel set $A \subseteq \wt{D}$ (provided that the quantities are well-defined). Moreover, for a constant $C > 0$ we have that $N_\epsilon^{\wt{h}^0+C,\wt{\eta}'}(A) = N_{\epsilon e^{-\gamma C}}^{\wt{h}^0,\wt{\eta}'}(A)$ and thus for $f \in C_0(\wt{D})$ we a.s.\ have that
\begin{align}
\label{eq:add_constant_to_field}
	\lim_{\epsilon \to 0} \epsilon^{1-\kappa'/8} N_\epsilon^{\wt{h}^0+C,\wt{\eta}'}(f) = \lim_{\epsilon \to 0} e^{(1-\kappa'/8)\gamma C} (\epsilon e^{-\gamma C})^{1-\kappa'/8} N_{\epsilon e^{-\gamma C}}^{\wt{h}^0,\wt{\eta}'}(f) = e^{(1-\kappa'/8)\gamma C} \qcutmeasure{\wt{h}^0}{\wt{\eta}'}(f).
\end{align}
Thus since $\wt{h}^0 + \inf_{w \in B(z,r)} Q \log |(\psi^{-1})'(w)| \leq \wt{h}^0 + Q \log|(\psi^{-1})'| \leq \wt{h}^0 + \sup_{w \in B(z,r)} Q \log |(\psi^{-1})'(w)|$ on $B(z,r)$ we have that 
\begin{align*}
	N_{\epsilon \inf | (\psi^{-1})'|^{-\gamma Q}}^{\wt{h}^0,\wt{\eta}'}(f_m)  \leq N_\epsilon^{\wt{h}^0 + Q \log|(\psi^{-1})'|}(f_m) \leq N_{\epsilon \sup |(\psi^{-1})'|^{-\gamma Q}}^{\wt{h}^0, \wt{\eta}'}(f_m) 
\end{align*}
(where $\sup$ and $\inf$ are taken over $B(z,r)$) for every $m \in \N$. By taking $\epsilon \to 0$ and using the a.s.\ vague convergence of $(\epsilon^{1-\kappa'/8} N_\epsilon^{\wt{h}^0,\wt{\eta}'})_{\epsilon > 0}$ and $(\epsilon^{1-\kappa'/8} N_\epsilon^{\wt{h}^0+Q\log |(\psi^{-1})'|,\wt{\eta}'})_{\epsilon > 0}$ and~\eqref{eq:add_constant_to_field} we have that
\begin{equation}\label{eqn:local_density_convergence_1}
\begin{split}
&\inf_{w \in B(z,r)}|(\psi^{-1})'(w)|^{\gamma Q (1-\kappa'/8)} \qcutmeasure{\wt{h}^0}{\wt{\eta}'}(f_m)\\
 &\leq (\qcutmeasure{h^0}{\eta'})_\psi(f_m)
\leq \sup_{w \in B(z,r)}|(\psi^{-1})'(w)|^{\gamma Q (1-\kappa'/8)} \qcutmeasure{\wt{h}^0}{\wt{\eta}'}(f_m).
\end{split}
\end{equation}
Hence~\eqref{eqn:local_density_convergence} follows by~\eqref{eqn:local_density_convergence_1}, the continuity of $|(\psi^{-1})'|$, and by letting $m \to \infty$.

We now show that $(\qcutmeasure{h^0}{\eta'})_\psi$ and $\qcutmeasure{\wt{h}^0}{\wt{\eta}'}$ are mutually absolutely continuous a.s.  Indeed,  a.s.\ the following hold.  Fix $A \subseteq \wt{D}$ Borel such that $\qcutmeasure{\wt{h}^0}{\wt{\eta}'}(A) = 0$ and let $K \subseteq \wt{D}$ be compact.  Then $\qcutmeasure{\wt{h}^0}{\wt{\eta}'}(A \cap K) = 0$ and since $\qcutmeasure{\wt{h}^0}{\wt{\eta}'}$ is a locally finite Borel measure,  we can find a decreasing sequence of open sets $(U_m)$ such that $A \cap K \subseteq U_m \subseteq \closure{U_m} \subseteq \wt{D}$ and $\qcutmeasure{\wt{h}^0}{\wt{\eta}'}(\closure{U_m}) \leq 2^{-m}$ for every $m \in \N$.  Moreover, as above, we have that
\begin{align*}
(\qcutmeasure{h^0}{\eta'})_\psi(U_m) &\leq \liminf_{\epsilon \to 0}\epsilon^{1-\kappa'/8}N_{\epsilon}^{h^0,\eta'}(\psi^{-1}(U_m))\\
&= \liminf_{\epsilon \to 0} \epsilon^{1-\kappa'/8}N_{\epsilon}^{\wt{h}^0 + Q \log |(\psi^{-1})'|,\wt{\eta}'}(U_m)\\
&\leq \sup_{z \in U_m} |(\psi^{-1})'(z)|^{\gamma Q(1 - \kappa'/8)} \qcutmeasure{\wt{h}^0}{\wt{\eta}'}(\closure{U_m})\\
&\leq \sup_{z \in U_m} |(\psi^{-1})'(z)|^{\gamma Q (1-\kappa'/8)} / 2^m
\end{align*}
It follows from the continuity of $|(\psi^{-1})'|$ and by taking $m \to \infty$ that $(\qcutmeasure{h^0}{\eta'})_\psi(A \cap K) = 0$.  Since $K$ was arbitrary, we have that $(\qcutmeasure{h^0}{\eta'})_\psi (A) = 0$ and similarly it holds that $\qcutmeasure{\wt{h}^0}{\wt{\eta}'}(A) = 0$ for every Borel set $A \subseteq \wt{D}$ such that $(\qcutmeasure{h^0}{\eta'})_\psi(A) = 0$. Hence $(\qcutmeasure{h^0}{\eta'})_\psi$ and $\qcutmeasure{\wt{h}^0}{\wt{\eta}'}$ are mutually absolutely continuous and by~\eqref{eqn:local_density_convergence} the Radon-Nikodym derivative $d(\qcutmeasure{h^0}{\eta'})_\psi/d\qcutmeasure{\wt{h}^0}{\wt{\eta}'}$ is given by $|(\psi^{-1})'|^{\gamma Q (1-\kappa'/8)}$.

Finally, noting that $\gamma Q (1-\kappa'/8) = 1 - \tfrac{\kappa'}{4} + \tfrac{8}{\kappa'}$, $\confrad(z, \wt{D}) = |(\psi^{-1})'(z)|^{-1} \confrad(\psi^{-1}(z), D)$, and $\sigma(\eta') = \sigma(\wt{\eta}')$, we a.s.\ have for any continuous function $f$ with compact support contained in $\wt{D}$ that
\begin{align*}
\cutmeasure{\wt{\eta}'}(f)
&= \int_{\wt{D}} f(z) r_{\wt{D}}(z) d\E\!\left[ \qcutmeasure{\wt{h}^0}{\wt{\eta}'} \, \middle| \, \wt{\eta}' \right](z)\\
&= \int_{\wt{D}} f(z) r_{\wt{D}}(z) |(\psi^{-1})'(z)|^{-(1 - \frac{\kappa'}{4} + \frac{8}{\kappa'})} d\E\!\left[ (\qcutmeasure{h^0}{\eta'})_\psi \, \middle| \, \eta' \right](z) \quad\text{(by~\eqref{eqn:local_density_convergence})}\\
&= \int_{\wt{D}} f(z) \frac{r_{\wt{D}}(z)}{r_D(\psi^{-1}(z))} |(\psi^{-1})'(z)|^{-(1 - \frac{\kappa'}{4} + \frac{8}{\kappa'})} d (\cutmeasure{\eta'})_\psi(z)\\
&= \int_{\wt{D}} f(z) |(\psi^{-1})'(z)|^{-d_{\kappa'}^\cut} d (\mu_{\eta'}^\cut)_\psi(z).
\end{align*}
This gives~\eqref{eqn:conf_cov_formula}, which concludes the proof.
\end{proof}

\subsection{Finiteness of moments}
\label{subsec:measure_moments}

The main purpose of this section is to prove the following proposition which roughly states that the measure $\mu_{\eta'}^{\text{cut}}$ is a.s.  H\"older continuous with respect to the Euclidean metric when restricted to compact subsets of $\h$.

\begin{proposition}
\label{prop:cut_point_measure_bc}
Suppose that $\eta'$ is an $\SLE_{\kappa'}$ in $\h$ from $0$ to $\infty$.  For every compact set $K \subseteq \h$ and $a > 0$ there a.s.\ exists $C > 0$ so that
\[ \cutmeasure{\eta'}(A) \leq C \diam(A)^{d_{\kappa'}^\cut-a} \quad\text{for all}\quad A \subseteq K \quad\text{Borel}.\]
\end{proposition}

The main input of the proof of Proposition~\ref{prop:cut_point_measure_bc} is the following lemma.

\begin{lemma}
\label{lem:moments_finite}
For every compact set $K \subseteq \h$ and $p_0  \in (0,1)$ there exists $\delta_0 \in (0,1)$ such that the following holds.  For every $\delta \in (0,\delta_0)$ there exists an event $E$ so that $\p[E] \geq p_0$ and for every $p > 0$  there exists a constant $c_0 < \infty$ depending only on $\delta$, $p_0$, $p$, and $K$ such that
\[ \E[ (\cutmeasure{\eta'}(B(z,\epsilon)))^p \one_E ] \leq c_0 \epsilon^{p d_{\kappa'}^\cut(1-\delta)}\]
for every $z \in K$ and $\epsilon \in (0,\dist(K,\partial \h)/2)$.
\end{lemma}

In order to prove  Lemma~\ref{lem:moments_finite},  we will give an upper bound on the density $g(z_1,\dots,z_n)$ of the measure $\E[ \cutmeasure{\eta'} \otimes \cdots \otimes \cutmeasure{\eta'} \one_E]$ (with the product taken~$n$ times) for every fixed $n \in \N$ and an appropriate choice of the event $E$,  and then complete the proof of Lemma~\ref{lem:moments_finite} using H\"older's inequality.  In particular,  we will bound from above the measure on cut points of $\eta'$ in terms of measures on the intersection of interior flow lines of $h$ on the event $E$.  This is achieved using the conformal covariance property (Lemma~\ref{lem:conformal_covariance_cut_measure}) and Lemmas~\ref{lem:gff_flow_lines_hit} and~\ref{lem:intersection_small_ball_measure},  together with the fact that the GFF on $\h$ which is coupled with $\eta'$ so that $\eta'$ is the counterflow line of $h$ from $0$ to $\infty$,  is locally absolutely continuous with respect to a whole-plane GFF with well-controlled Radon-Nikodym derivative (Lemma~\ref{lem:good_scale_radon}).

First,  we state Lemma~\ref{lem:gff_flow_lines_hit}.  We will use Lemma~\ref{lem:gff_flow_lines_hit} in order to localize the event that the cut points of $\eta'$ get close to a point $z$ in order to make use of the independence properties of the GFF.

\begin{lemma}
\label{lem:gff_flow_lines_hit}
Suppose that $h$ is a zero-boundary GFF on $\D$ and $z_0 \in B(0,1/2) \setminus B(0,1/4)$ is sampled from Lebesgue measure (normalized to be a probability measure) independently of $h$.  Let $\eta_1$ (resp.\ $\eta_2$) be the flow line of $h$ starting from $z_0$ with angle $-\pi/2$ (resp.\ $\pi/2$).  For $j=1,2$, let $\tau_j = \inf\{t \geq 0 : \eta_j(t) \notin B(0,3/4)\}$.  Then for all $a>0$ sufficiently small  there exist constants $c_0 < \infty$ and $\epsilon_0 \in (0,1/8)$ depending only on $a$ such that
\begin{align*}
	\p[ \eta_j([0,\tau_j]) \cap B(w,\epsilon) \neq \emptyset \ \text{for} \ j=1,2]  \leq c_0 \epsilon^{2-d_{\kappa'}^\cut-a} \quad\text{for all}\quad w \in B(0,1/8) \quad\text{and}\quad \epsilon \in (0,\epsilon_0).
\end{align*}
\end{lemma}

In order to prove Lemma~\ref{lem:gff_flow_lines_hit}, we will make use of the following inequality.  Suppose that $E$ is an event with $\p[E] > 0$ and $Z$ is a non-negative random variable such that $\E[Z] < \infty$ and $\E[ Z \giv E] \in (0,\infty)$.  Then $\p[E] \leq \E[Z]/\E[Z \giv E]$.  We will use Lemma~\ref{lem:intersection_small_ball_measure} to show that $\E[ Z ] = O(\epsilon^2)$ and Lemma~\ref{lem:pocket_formation_positive_prob} together with conformal covariance to show that $\E[Z \giv E] \gtrsim \epsilon^{d_{\kappa'}^{\text{cut}}(1+2\delta)}$ as $\epsilon \to 0$,  where $\delta \in (0,1)$ is small but fixed.  However,  we will need to compare the law of $h$ with the law of a GFF on $\h$ with different boundary conditions when applying the conformal covariance formula.  For that purpose,  we will have to introduce another event $G$ and use Lemma~\ref{lem:lower_bound_harmonic_measure} to argue that $\p[ E \cap G^c ] = O(\epsilon^2)$ as $\epsilon \to 0$ and so 
\begin{align*}
\p[E] \leq \p[ E \cap G] + \p[G^c] \lesssim \epsilon^{2-d_{\kappa'}^{\text{cut}}-2\delta d_{\kappa'}^{\text{cut}}} + \epsilon^2 \quad \text{as}\,\,\epsilon \to 0,
\end{align*}
which shows the claim of Lemma~\ref{lem:gff_flow_lines_hit}.

We will now define the event $E$ and random variable $Z$ which we will use.  Suppose that $\CC = (\C,h,0,\infty)$ is a quantum cone of weight $4-\gamma^2$ with the circle average embedding and let $h^\IG$ be a whole-plane GFF on $\C$ with values modulo $2\pi \chi$ which is taken to be independent of $h$.  Let $\eta_1$ (resp.\ $\eta_2$) be the flow line of $h^\IG$ from $0$ to $\infty$ with angle $-\pi/2$ (resp.\ $\pi/2$).  Then the collection of quantum surfaces which are parameterized by the components of $\C \setminus (\eta_1 \cup \eta_2)$ which are to the right of $\eta_1$ and to the left of $\eta_2$ form a quantum wedge $\CW$ of weight $2-\gamma^2/2$ (Theorems~\ref{thm:wedge_cutting} and~\ref{thm:cone_cutting}).    Let $X \sim \BES^{\kappa'/4}$ be the Bessel process associated with $\CW$.  As in the case of $\SLE_{\kappa'}$ as described in Section~\ref{subsec:measure_construction}, the local time of $X$ at $0$ induces a measure on $\eta_1 \cap \eta_2$ which we denote by $\qintmeasure{h}{\eta_1}{\eta_2}$, in order to emphasize its dependence on $h$, $\eta_1$, and $\eta_2$.  (Naturally this can in fact be done with other GFF-type distributions in place of the whole-plane GFF.) Now suppose that $h^0$ is a GFF with zero boundary conditions on a simply connected domain $D$ which contains $0$ and let $\tau_i$ be a stopping time for $\eta_i$ so that a.s.\ $\eta_i([0,\tau_i]) \subseteq \closure{D}$ for $i=1,2$.  For $i=1,2$ we let $\eta_i^{\tau_i} = \eta_i|_{[0,\tau_i]}$.  Let $\qintmeasure{h^0}{\eta_1^{\tau_1}}{\eta_2^{\tau_2}}$ be the restriction of $\qintmeasure{h^0}{\eta_1}{\eta_2}$ to $\eta_1([0,\tau_1]) \cap \eta_2([0,\tau_2])$.  Since the law of~$h^0$ restricted to any compact subset~$K$ of~$D$ which has positive distance from $0$ is absolutely continuous with respect to the corresponding restriction for $h$, it follows that $\qintmeasure{h^0}{\eta_1^{\tau_1}}{\eta_2^{\tau_2}}$ is a.s.\ defined.  We let $r_D$ be as in~\eqref{eqn:cut_measure_definition}.  We then define
\begin{equation}
\intmeasure{\eta_1^{\tau_1}}{\eta_2^{\tau_2}}(dz) = r_D(z) \E[ \qintmeasure{h^0}{\eta_1^{\tau_1}}{\eta_2^{\tau_2}} (dz) \giv \eta_1^{\tau_1}, \eta_2^{\tau_2}].	
\end{equation}
We note that $\intmeasure{\eta_1^{\tau_1}}{\eta_2^{\tau_2}}$ is in fact defined if we replace $(\eta_1^{\tau_1}, \eta_2^{\tau_2})$ by any pair $(\wt{\eta}_1, \wt{\eta}_2)$ whose law is locally absolutely continuous with respect to $(\eta_1^{\tau_1}, \eta_2^{\tau_2})$.  The same argument used to prove Proposition~\ref{lem:conformal_covariance_cut_measure} implies that the measure we have defined is conformally covariant with exponent $d_{\kappa'}^\cut$.

In the setting of the proof of Lemma~\ref{lem:gff_flow_lines_hit}, we will take $E$ to be the event that $\eta_i([0,\tau_i]) \cap B(0,\epsilon) \neq \emptyset$ for $i=1,2$ and we will take $Z = \intmeasure{\eta_1}{\eta_2}(B(0,\epsilon))$.  We thus need in particular to show that $\E[Z] < \infty$, which is the content of the following lemma.  We note in particular that due to our choice of definition of $\intmeasure{\eta_1}{\eta_2}$ to prove that $\E[Z] < \infty$ it suffices to prove for each $r \in (0,1)$ that $\qintmeasure{h^0}{\eta_1}{\eta_2}(B(0,r))$ has a finite $p$th moment for some $p > 1$ where $h^0$ is a GFF on $\D$ with zero boundary conditions.

\begin{lemma}
\label{lem:quantum_measure_zero_boundary_finite}
Suppose that $h^\IG$ is a GFF on $\D$ with zero boundary conditions and $z \in B(0,1/2)$ is picked from Lebesgue measure independently of $h^\IG$.  Let $\eta_1$ (resp.\ $\eta_2$) be the flow line of $h^\IG$ starting from $z$ with angle $-\pi/2$ (resp.\ $\pi/2$).  For $r \in (1/2,1)$ and $i=1,2$ let $\tau_i = \inf\{t \geq 0 : \eta_i(t) \notin B(0,r)\}$.  There exists $p > 1$ so that $\E[ (\intmeasure{\eta_1^{\tau_1}}{\eta_2^{\tau_2}}(B(0,r))^p] < \infty$. 
\end{lemma}

The main step in the proof of Lemma~\ref{lem:quantum_measure_zero_boundary_finite} is the following result in the setting of a quantum cone and flow lines of a whole-plane GFF.

\begin{lemma}
\label{lem:quantum_measure_quantum_cone_finite}
Fix $r \in (0,1)$ and $K \subseteq \closure{\D}$ compact which contains $B(0,r)$.  Suppose that $\CC = (\C,h,0,\infty)$ is a quantum cone of weight $4-\gamma^2$ with the circle average embedding.  Let $h^\IG$ be an independent whole-plane GFF modulo $2\pi \chi$ (recall the definition near the end of Section~\ref{subsec:ig}).  Let~$\eta_1$ (resp.\ $\eta_2$) be the flow line of $h^\IG$ starting from $0$ with angle $-\pi/2$ (resp.\ $\pi/2$) and, for $i=1,2$, let $\tau_i = \inf\{t \geq 0 : \eta_i(t) \notin K\}$.  Let $Z =  \qintmeasure{h}{\eta_1^{\tau_1}}{\eta_2^{\tau_2}}(K)$.  For every $p \geq 1$ we have that $\E[ (Z/\qmeasure{h}(K))^p ] = O(1)$ where the implicit constants depends only on $r$, $p$.
\end{lemma}
\begin{proof}
Fix $p \geq 1$.  We begin by noting that $\qmeasure{h}(B(0,r))$ (hence $\qmeasure{h}(K)$) has finite negative moments of all orders.  Indeed,  since $h|_{\D}$ can be expressed as $\wt{h}|_{\D}-\gamma \log |\cdot|$ where $\wt{h}$ is a whole-plane GFF such that $\wt{h}_1(0) = 0$, it suffices to show that $\qmeasure{\wt{h}}(B(0,r))$ has finite negative moments of all orders.  This in turn, follows by writing $\wt{h}|_\D = \wt{h}^0 + \wt{\Fh}$ where $\wt{h}^0$ is a zero-boundary GFF on $\D$ and $\wt{\Fh}$ is a harmonic function on $\D$, independent of $\wt{h}^0$, noting that for fixed $q > 0$,
\begin{align*}
	\E[ \qmeasure{\wt{h}}(B(0,r))^{-q}] \leq \E[ \qmeasure{\wt{h}^0}(B(0,r))^{-q}] \E\Big[ \exp\Big(q\gamma \sup_{z \in B(0,r)}|\wt{\Fh}(z)|\Big) \Big],
\end{align*}
and then applying \cite[Proposition~3.6]{robert2010gaussian} (resp.\ \cite[Lemma~3.11]{ghm2020almost}) to bound the first (resp.\ second) factor on the right side. Moreover, since $\qmeasure{h}(K) \geq \qmeasure{h}(B(0,r))$, it follows that $\qmeasure{h}(K)$ has finite negative moments of all orders as well. In particular, it suffices to bound $\E[ (Z / \qmeasure{h}(K))^p \one_{\{Z \geq 1\}}]$.  We have that
\begin{align}
\label{eqn:z_mu_expectation_bound}
   \E\left[ \left(\frac{Z}{\qmeasure{h}(K)} \right)^p \one_{\{Z \geq 1\}} \right]
&\leq \sum_{m=1}^\infty \sum_{k=-\infty}^{\infty} 2^{(k+1) p} \p[ m \leq Z < m+1,\ 2^k \leq Z/\qmeasure{h}(K) \leq 2^{k+1}] \notag\\
&\leq 2^p + \sum_{m=1}^\infty \sum_{k=1}^\infty 2^{(k+1)p} \p[ m \leq Z < m+1, \ \qmeasure{h}(K) \leq m/2^{k-1}].	
\end{align}
Let~$\CW$ be the weight $2-\gamma^2/2$ quantum wedge parameterized by the connected components of $\C \setminus (\eta_1 \cup \eta_2)$ which are to the right of $\eta_1$ and to the left of $\eta_2$.  Let also $X \sim \BES^{\kappa'/4}$ be the corresponding Bessel process encoding $\CW$.  For each $u \geq 0$, we let~$\CW_u$ be the part of~$\CW$ up to when~$X$ has accumulated $u$ units of local time at~$0$.  Then there exists a constant $c_0 > 0$ so that the number of beads of~$\CW_m$ which have quantum area at least $m/2^{k-1}$ is a Poisson random variable with mean given by
\[ \lambda_{\kappa',m} = c_0 m ( m/2^{k-1})^{\kappa'/8-1} = c_0 m^{\kappa'/8} 2^{(k-1) (1-\kappa'/8)}.\]
Since $\kappa' \in (4,8)$ we have that $1-\kappa'/8 > 0$.  Finally, on the event $\{m\leq Z < m+1 , \ \qmeasure{h}(K) \leq m/2^{k-1} \}$ $\CW_m$ does not contain a bead with quantum area at least $m/2^{k-1}$.  Applying~\eqref{eqn:poisson_concentration} with $\lambda = \lambda_{\kappa',m}$ and $\alpha = 1/2$ we thus see there exists a constant $c_1 > 0$ so that
\begin{equation}
\label{eqn:z_mu_range}
\p[ m \leq Z < (m+1),\ \qmeasure{h}(K) \leq m/2^{k-1}] \leq \exp(-c_1 \lambda_{\kappa',m}) \quad\text{for each}\quad m,k \in \N.
\end{equation}
Inserting~\eqref{eqn:z_mu_range} into~\eqref{eqn:z_mu_expectation_bound} gives the result.
\end{proof}

\begin{proof}[Proof of Lemma~\ref{lem:quantum_measure_zero_boundary_finite}]
Let $h$ be a GFF on $\D$ with zero boundary conditions taken to be independent of $h^\IG$.  Let $Z =  \qintmeasure{h}{\eta_1^{\tau_1}}{\eta_2^{\tau_2}}(B(0,r))$.  Fix $p \in (1,2/\gamma^2+1/2)$.  Let $c_0 = (\E[ \qmeasure{h}(\D)^{2p-1}])^{1/2}$.  Note that $c_0 < \infty$ since $2p-1 < 4/\gamma^2$ \cite[Proposition~3.5]{robert2010gaussian}.  We thus have that
\begin{align*}
\E[ Z^p ]
&= \E\left[ \left( \frac{Z}{\qmeasure{h}(B(0,r))} \right)^p  \qmeasure{h}(B(0,r))^p \right]\\
&\leq \left(\E\left[ \left( \frac{Z}{\qmeasure{h}(B(0,r))} \right)^{2p} \qmeasure{h}(B(0,r)) \right] \right)^{1/2} \left( \E[ \qmeasure{h}(B(0,r))^{2p-1} ] \right)^{1/2}\\
&\leq c_0 \left(\E\left[ \left( \frac{Z}{\qmeasure{h}(B(0,r))} \right)^{2p} \qmeasure{h}(\D) \right] \right)^{1/2}.
\end{align*}
We recall from \cite[Section~3.3]{ds2011lqg} that we can sample from the law of $h$ weighted by $\qmeasure{h}(\D)$ using the following procedure.  We first sample $z$ according to the law with density with respect to Lebesgue measure on $\D$ given by a normalizing constant times $\confrad(z,\D)^{\gamma^2/2}$ and then, given $z$, take the field with law equal to that of $\wt{h} + \gamma G(z,\cdot)$ where $\wt{h}$ is a GFF on $\D$ with zero boundary conditions and $G$ is the Green's function on $\D$ with Dirichlet boundary conditions.

We note that the conditional law of $z$ given $z \in B(0,1/2)$ is absolutely continuous with respect to Lebesgue measure on $B(0,1/2)$ with a Radon-Nikodym derivative which is bounded from above and below by finite and positive constants.  In particular, it suffices to prove the statement of the lemma with~$z$ chosen from this law in place of Lebesgue measure.  Consider $\wh{h} = \wt{h}^w - \gamma \log|z-\cdot|$ where $\wt{h}^w$ is a whole-plane GFF with the additive constant chosen so that the average of $\wh{h}$ on $\partial B(z,2)$ is equal to~$0$.  Since the restriction of this field to $B(z,2)$ has the same law as the corresponding restriction of a quantum cone of weight $4-\gamma^2$ with the marked point taken to be $z$ and with the circle average embedding (with $B(z,1)$ replaced by $B(z,2)$) it follows from Lemma~\ref{lem:quantum_measure_quantum_cone_finite}  that if we let $\wh{Z} = \qintmeasure{\wh{h}}{\eta_1^{\tau_1}}{\eta_2^{\tau_2}}(B(0,r))$ then $\E[ (\wh{Z}/ \qmeasure{\wh{h}}(B(0,r)))^{2p} ] < \infty$ for every $p > 0$.

Let~$\Fh^w$ be the distribution which is harmonic in~$\D$ so that we can write $\wt{h}^w = h^{0,w} + \Fh^w$ where $h^{0,w}$ is a zero-boundary GFF in $\D$ and $h^{0,w}$, $\Fh^w$ are independent.  Let $\wh{h}^w = h^{0,w} + \gamma G(z,\cdot)$ and let
\begin{align*}
Z^{0,w} &= 
\qintmeasure{\wh{h}^w}{\eta_1}{\eta_2}(\eta_1([0,\tau_1]) \cap \eta_2([0,\tau_2])).
\end{align*}
Since $\wh{h}^w = \wh{h} +\gamma (G(z,\cdot) + \log |z-\cdot|) - \Fh^w$ and $|G(x,\cdot)+\log |x-\cdot || \leq \log 2$ whenever $x \in B(0,1/2)$ it follows that there is a constant $c_1 > 1$ such that
\begin{align*}
	Z^{0,w} \leq c_1 \exp\!\Big( \gamma \sup_{x \in B(0,r)} |\Fh^w(x)|\Big) \wh{Z} 
\quad\text{and}\quad \qmeasure{\wh{h}^w}(B(0,r)) \geq c_1^{-1} \exp\!\Big( -\gamma \sup_{x \in B(0,r)} |\Fh^w(x)|\Big) \qmeasure{\wh{h}}(B(0,r)).
\end{align*}
Thus, by the above inequalities and the Cauchy-Schwarz inequality we have that
\begin{align*}
 \E[ (Z^{0,w} / \qmeasure{\wh{h}^w}(B(0,r)))^{2p} ] \leq c_1^{4p}\!\Big( \E\Big[ \exp\Big(8p \gamma \sup_{x \in B(0,r)} |\Fh^w(x)|\Big)\Big] \Big)^{1/2} \left( \E[ (\wh{Z} / \qmeasure{\wh{h}}(B(0,r)))^{4p} ] \right)^{1/2}.
\end{align*}
The latter term was noted above to be finite and by~\cite[Lemma~3.11]{ghm2020almost}, $\E[\exp(8p\gamma \sup_{x \in B(0,r)}|\Fh^w(x)|)]$ is finite as well.  Thus, recalling the discussion in the second paragraph of the proof, we have for a constant $c_2 > 0$ that

\begin{align*}
\E\left[ \left( \frac{Z}{\qmeasure{h}(B(0,r))}\right)^{2p} \qmeasure{h}(\D)\right] = c_2 \E\left[ \left( \frac{Z^{0,w}}{\qmeasure{\wh{h}^w}(B(0,r))}\right)^{2p}\right] <\infty
\end{align*}
and so $\E[Z^p] < \infty$ for every $p \in (1,1/2 + 2/\gamma^2)$.  The proof is then complete since
\begin{align*}
\E[(\intmeasure{\eta_1^{\tau_1}}{\eta_2^{\tau_2}}(K)^p] \leq \sup_{w \in K}(r_{\D}(w)^p) \E[Z^p] < \infty.
\end{align*}
\end{proof}

Next, we show that as long as the harmonic part of a GFF-type distribution is controlled, then the expected mass of the intersection measure of its flow lines in a small ball is small.

\begin{lemma}
\label{lem:intersection_small_ball_measure}
Fix $M > 0$, $r \in (1/2,1)$.  Then there exists $\xi \in (r,1)$ such that the following holds.  Suppose that $h^\IG = h^{\IG,0} + \Fh^{\IG,0}$ where $h^{\IG,0}$ is a GFF on $\D$ with zero boundary conditions and $\Fh^{\IG,0}$ is harmonic on $\D$ with $\sup_{w \in B(0,\xi)} |\Fh^{\IG,0}(w)| \leq M$.  Let $z \in B(0,1/2) \setminus B(0,1/4)$ be picked from Lebesgue measure independently of $h^\IG$.  Let $\eta_1$ (resp.\ $\eta_2$) be the flow line of $h^\IG$ starting from $z$ with angle $-\pi/2$ (resp.\ $\pi/2$).  For $i=1,2$, let $\tau_i = \inf\{t \geq 0 : \eta_i(t) \notin B(0,r)\}$ and write $\ol{\eta}_i$ for $\eta_i|_{[0,\tau_i]}$.  Then
\[ \E[ \intmeasure{\ol{\eta}_1}{\ol{\eta}_2}(B(0,\epsilon))]  = O(\epsilon^2) \quad\text{as}\quad \epsilon \to 0.\]
\end{lemma}
\begin{proof}
Fix $r_0 \in (1/2,1)$ and $q > 1$.  We will set their precise values later.  For each $w \in B(0,r)$ we let $\ol{\eta}_1^w$, $\ol{\eta}_2^w$ be the flow line of $h^\IG$ starting from $w$ with angle $-\pi/2$, $\pi/2$, respectively, and stopped upon exiting $B(0,r)$.  We define $\wt{\eta}_1^w$, $\wt{\eta}_2^w$ in the same way but with $B(0,r)$ replaced by $B(0,r_0)$. Let $A = B(0,1/2) \setminus \closure{B(0,1/4)}$ and $z \in A$ be as in the statement of the lemma.  

Note that if $\wh{\eta}_1^w$ (resp.\ $\wh{\eta}_2^w$) is the flow line of $h^{\IG,0}$ of angle $-\pi/2$ (resp.\ $\pi/2$) from $w \in \D$, stopped at the first time $\wh{\tau}_1^w$ (resp.\ $\wh{\tau}_2^w$) it exits $B(0,r_0)$, then by Lemma~\ref{lem:quantum_measure_zero_boundary_finite} there exists $p>1$ such that $\E[ \intmeasure{\wh{\eta}_1^z}{\wh{\eta}_2^z}(B(0,r_0))^p ] < \infty$.  Moreover, by Fubini's theorem, there exists a $w_0 \in A$ such that $\E[ \intmeasure{\wh{\eta}_1^{w_0}}{\wh{\eta}_2^{w_0}}(B(0,r_0))^p] < \infty$.
In fact, the same holds for some (possibly different) $p>1$ when replacing $\wh{\eta}_j^{w_0}$ by $\wt{\eta}_j^{w_0}$. Indeed, note that $\intmeasure{\wh{\eta}_1^z}{\wh{\eta}_2^z}(B(0,r_0))$ and $\intmeasure{\wt{\eta}_1^z}{\wt{\eta}_2^z}(B(0,r_0))$ are determined by $h^{\IG,0}|_{B(0,r_0)}$ and $h^\IG|_{B(0,r_0)}$ respectively. Choose $r_0 < s_0 < \xi < 1$, fix some $\psi \in C_0^\infty(B(0,s_0))$ such that $0 \leq \psi \leq 1$ and $\psi|_{B(0,r_0)} \equiv 1$, let $F = \psi \Fh^{\IG,0}$ and note that $(h^{\IG,0}+F)|_{B(0,r_0)} = h^\IG|_{B(0,r_0)}$. Since $\| F \|_\nabla \leq C \sup_{x \in B(0,\xi)} | \Fh^{\IG,0}(x)| \leq C M$, where $C$ depends only on $\psi$, it follows as explained in Remark~\ref{rmk:RN_derivative} that $\E[ \CD_{h^{\IG,0},F}^{q}] < \infty$ for all $q>0$ and hence, by H\"older's inequality there is some $q \in (1,p)$ such that 
\begin{align}\label{eq:intmeasure_bound_ball}
	\E[ \intmeasure{\wt{\eta}_1^{w_0}}{\wt{\eta}_2^{w_0}}(\D)^q ] = \E[ \intmeasure{\wt{\eta}_1^{w_0}}{\wt{\eta}_2^{w_0}}(B(0,r_0))^q ] < \infty.
\end{align}
For each $w \in \D$, let $\psi_w:\D \to \D$ be the unique conformal map satisfying $\psi_w(w) = w_0$ and $\psi_w'(w) > 0$. We note that by~\cite[Theorem~3.21]{lawler2008conformally} and the Koebe-1/4 theorem, there exists $r_0 \in (r,1)$ such that $\psi_w(B(0,r)) \subseteq B(0,r_0)$ for each $w \in A$ (and we take the fixed $r_0$ above to be this particular value). We let $h^{\IG,w} = h^\IG \circ \psi_w^{-1} - \chi \arg( \psi_w^{-1})'$ and note that if $\phi \in C_0^\infty(B(0,s_0))$ satisfies $0 \leq \phi \leq 1$ and $\phi|_{B(0,r_0)} \equiv 1$ and we set $F_w =\phi(\Fh^{\IG,0}\circ \psi_w^{-1} -\chi  \arg(\psi_w^{-1})' -\Fh^{\IG,0})$ then the laws of $(h^\IG+F_w)|_{B(0,r_0)}$ and $h^{\IG,w}|_{B(0,r_0)}$ are equal. We write $\CZ_w = \CD_{h^\IG,F_w}$, let $\ol{\eta}_{j,w} = \psi_w(\ol{\eta}_j^w)$ and note that $\ol{\eta}_{1,w}$ (resp.\ $\ol{\eta}_{2,w}$) is the flow line of $h^{\IG,w}$ starting from $w_0$ with angle $-\pi/2$ (resp.\ $\pi/2$). Then by the conformal covariance of $\intmeasure{\ol{\eta}_1}{\ol{\eta}_2}$, we have that
\begin{align}\label{eq:int_measure_epsilon_ball}
	\E[ \intmeasure{\ol{\eta}_1}{\ol{\eta}_2}(B(0,\epsilon)) ] &= \frac{16}{3\pi} \int_A \E\!\left[ \int \one_{B(0,\epsilon)}(u) d\intmeasure{\ol{\eta}_1^w}{\ol{\eta}_2^w}(u) \right] dw \nonumber \\
	&= \frac{16}{3\pi} \int_A \E\!\left[ \int |(\psi_w^{-1})'(u)|^{d_{\kappa'}^\cut} \one_{B(0,\epsilon)}(\psi_w^{-1}(u)) d\intmeasure{\ol{\eta}_{1,w}}{\ol{\eta}_{2,w}}(u) \right] dw \nonumber \\
	&\leq \frac{16}{3\pi} \int_A \E\!\left[ \CZ_w \int |(\psi_w^{-1})'(u)|^{d_{\kappa'}^\cut} \one_{B(0,\epsilon)}(\psi_w^{-1}(u)) d\intmeasure{\wt{\eta}_1^{w_0}}{\wt{\eta}_2^{w_0}}(u) \right] dw \nonumber \\
	&= \frac{16}{3\pi} \E\!\left[ \CZ_w \int \int_{A} |(\psi_w^{-1})'(u)|^{d_{\kappa'}^\cut} \one_{B(0,\epsilon)}(\psi_w^{-1}(u)) dw d\intmeasure{\wt{\eta}_1^{w_0}}{\wt{\eta}_2^{w_0}}(u) \right].
\end{align}
We note that there exists a constant $c_0 > 0$ such that $|(\psi_w^{-1})'(u)| \leq c_0$ for each $w \in A$ and $u \in B(0,r_0)$.  Thus,~\eqref{eq:int_measure_epsilon_ball} is at most
\begin{align}\label{eq:int_measure_ubd}
	c_0^{d_{\kappa'}^\cut} \E\!\left[ \sup_{w \in A} \CZ_w \cdot \intmeasure{\wt{\eta}_1^{w_0}}{\wt{\eta}_2^{w_0}}(\D) \cdot \sup_{u \in B(0,r_0)} \int_A \one_{B(0,\epsilon)}(\psi_w^{-1}(u)) dw \right],
\end{align}
since $\intmeasure{\wt{\eta}_1^{w_0}}{\wt{\eta}_2^{w_0}}$ is supported in $B(0,r_0)$. Noting that
\begin{align}\label{eq:epsilon_square_bound}
	\sup_{u \in B(0,r_0)} \int_A \one_{B(0,\epsilon)} (\psi_w^{-1}(u)) dw = O(\epsilon^2),
\end{align}
it follows that if $q$ is as in~\eqref{eq:intmeasure_bound_ball}, and $p>1$ is such that $p^{-1} + q^{-1} = 1$, then by applying H\"older's inequality and~\eqref{eq:intmeasure_bound_ball}, there exist constants $c_1,c_2 > 0$ so that~\eqref{eq:int_measure_ubd} is at most
\begin{align}
	c_1 \E[ (\sup_{w \in A} \CZ_w)^p ]^{1/p} \E[ \intmeasure{\wt{\eta}_1^{w_0}}{\wt{\eta}_2^{w_0}}(\D)^q] ^{1/q} \epsilon^2 \leq c_2 \E[ (\sup_{w \in A} \CZ_w)^p ]^{1/p} \epsilon^2.
\end{align}
Thus, we are done if we bound the expectation on the right-hand side. Note that there exists a constant $c_3 > 0$ depending only on $\phi$, so that $\sup_{w \in A} \| \phi \arg(\psi_w^{-1})' \|_\nabla \leq c_3$.  For each $w \in A$ we let $Z_w^0 = (h^{\IG,0},-\chi \phi \arg(\psi_w^{-1})')_\nabla$ and $Z_w = (h^\IG,-\chi \phi \arg(\psi_w^{-1})')_\nabla$.  Then $(Z_w^0)_{w \in A}$ is a family of centered Gaussian random variables such that $\sup_{w \in A} \var[Z_w^0] < \infty$.  Thus by the Borell-TIS inequality, there exist constants $c_4,c_5>0$ such that $\p[ \sup_{w \in A} |Z_w^0| - c_4 > x] \leq \exp(-c_5 x^2)$ for each $x \geq 0$.  Since $Z_w-Z_w^0$ is bounded deterministically for $w \in A$, by possibly increasing the values of $c_4,c_5$ we have that $\p[ \sup_{w \in A} |Z_w| - c_4 > x] \leq \exp(-c_5 x^2)$ for each $x \geq 0$.  It follows that $\E[ (\sup_{w \in A} \CZ_w)^p] < \infty$ for all $p>0$. Thus the proof is done.
\end{proof}

Next,  before proving Lemma~\ref{lem:gff_flow_lines_hit},  we state and prove Lemmas~\ref{lem:pocket_formation_positive_prob} and~\ref{lem:lower_bound_harmonic_measure}.  The latter are going to be important inputs in the proof of Lemma~\ref{lem:gff_flow_lines_hit}.

\begin{lemma}
\label{lem:pocket_formation_positive_prob}
Let $\wt{h}$ be a GFF on $\h$ with boundary values given by $-\lambda- \pi \chi/2$ on $(-\infty,-1]$,  $\lambda - \pi \chi/2$ on $(-1,0]$,  $-\lambda + \pi \chi /2$ on $(0,1]$ and $\lambda + \pi \chi /2$ on $(1,\infty)$.  Fix $z \in \h$ and let $\wt{\eta}_1$ (resp.\ $\wt{\eta}_2$) be the flow line of $\wt{h}$ starting from $1$ (resp.\ $-1$) with angle $-\frac{\pi}{2}$ (resp.\ $\frac{\pi}{2}$) and targeted at $\infty$.  Fix $0 < r_1 < \im(z)$, $r_2 > r_1$ and let $E_{r_1,r_2}^z$ be the event that $\wt{\eta}_1$ and $\wt{\eta}_2$ form a pocket which is contained in $B(z,r_2)$ and contains $B(z,r_1)$. Then $\p[E_{r_1,r_2}^z] > 0$.
\end{lemma}
\begin{proof}
The result follows by using first \cite[Lemma~2.3]{mw2017intersections} once and then \cite[Lemma~2.5]{mw2017intersections} twice.
\end{proof}

\begin{lemma}\label{lem:lower_bound_harmonic_measure}
Suppose that we have the set up of the statement of Lemma~\ref{lem:gff_flow_lines_hit}.  Fix $\delta \in (0,1)$.  Then there exists $\wt{\delta} \in (0,1)$ sufficiently small such that the following holds.  For $j=1,2$ and $1 \leq m \leq \delta \log(1/\epsilon)$,  we set $\tau_{j,m} = \inf\{t \geq 0 : \eta_j(t) \in \partial B(w,2^m \epsilon)\} \wedge \tau_j$ and let $F_m$ be the event that the harmonic measure as seen from $w$ of the left and right sides of $\eta_j([0,\tau_{j,m}])$ in the connected component of $\C \setminus \cup_{j=1}^2 \eta_j([0,\tau_{j,m}])$ containing $w$ is at least $\wt{\delta}$.  Then, on the event that both of $\eta_1$ and $\eta_2$ hit $B(w,\epsilon)$, off an event with probability $O(\epsilon^2)$ as $\epsilon \to 0$,  we have that $F_m$ occurs for at least $\delta \log(1/\epsilon)/2$ values of $1 \leq m \leq \delta \log(1/\epsilon)$.
\end{lemma}

\begin{proof}
We set $K = \lfloor \delta \log(1/\epsilon) \rfloor$.  For $M>1$ large but fixed (to be chosen),  and $1 \leq m \leq K$, we let $E_m$ be the event that $B(w,2^m \epsilon)$ is $M$-good (recall Section~\ref{subsec:good_annuli}, and consider $\D$ instead of $\h$ as the domain). Let $\CF_{w,m}$ be the $\sigma$-algebra generated by $h|_{\D \setminus B(w,2^m \epsilon)}$.  Let $\wt{h}_{w,m}$ (resp.\ $\Fh_{w,m}$) be a zero-boundary GFF (resp.\ harmonic function) on $B(w,2^m \epsilon)$ such that $h = \wt{h}_{w,m} + \Fh_{w,m}$ and $\Fh_{w,m}$ is $\CF_{w,m}$-measurable and independent of $\wt{h}_{w,m}$.  We also let $N(M,K)$ be the number of $1 \leq m \leq K$ such that $B(w,2^m \epsilon)$ is $M$-good.  Then Lemma~\ref{lem:good_dense} implies that $\p[N(K,M) \leq 3K/4] = O(\epsilon^2)$ as $\epsilon \to 0$ provided we have taken $M$ to be sufficiently large.  We fix a countable, dense subset $(\theta_{\ell})_{\ell \in \N}$ of $[0,2\pi)$ and for every $\ell \in \N$,  $1 \leq m \leq K$,  we let $\eta_j^{m,\ell}$ be the flow line of~$h$ of angle $-\pi/2$ ($j=1$) or $\pi/2$ ($j=2$) emanating from $w + \frac{3}{4}2^m e^{i\theta_{\ell}} \epsilon$.  We stop the flow lines at the first time that they exit $B(w,\tfrac{7}{8} 2^m \epsilon) \setminus \closure{B(w,2^{m-1}\epsilon)}$ and set $S_m =\closure{ \cup_{j=1}^2 \cup _{\ell \in \N} \eta_j^{m,\ell}}$.  Then Lemma~\ref{lem:finite_strands} implies that $S_m \cap \partial (B(w,\tfrac{7}{8} 2^m \epsilon) \setminus \closure{B(w,2^{m-1}\epsilon)})$ is a finite set a.s.  Fix $p \in (0,1)$ small enough (to be chosen).  We let $A_m$ be the event that the following holds.  The tips of the strands of the flow lines which reach $\partial B(w,\tfrac{7}{8} 2^m \epsilon) \cup \partial B(w,2^{m-1}\epsilon)$ are separated by at least $\delta_* 2^m \epsilon$,  and if $x \in \partial B(w,\tfrac{7}{8} 2^m \epsilon)\cup  \partial B(w,2^{m-1}\epsilon)$ is a tip of such a strand,  then there is only one strand in the ball $B(x,\delta_* 2^m \epsilon/100)$.  Here,  $\delta_* \in (0,1)$ is small but fixed and such that $\p[\wt{A}_m] \geq 1-p$,  where $\wt{A}_m$ is the event defined in the same way as $A_m$ but with the field $h$ replaced by $\wt{h}_{w,m}$.  Also \cite[Remark~4.2]{mq2020geodesics} (see Lemma~\ref{lem:good_scale_radon} above) implies that there exists $c(M)>0$ depending only on $M$ such that
\begin{align*}
\p[A_m^c \,  \giv \CF_{w,m}] \one_{E_m} \leq c(M) \p[\wt{A}_m^c]^{1/2} \leq c(M) p^{1/2}
\end{align*}
for every $1 \leq m \leq K$,  since $A_m$ (resp.\ $\wt{A}_m$) is determined by $h|_{B(w,(7/8) 2^m \epsilon)}$ (resp.\ $\wt{h}_{w,m}|_{B(w,(7/8)2^m \epsilon)}$).  This implies that if $N_K$ is the number of $1 \leq m \leq K$ such that $A_m$ occurs, then on $\{N(K,M) > 3K/4\}$,  $N_K$ is stochastically dominated from below by $X \sim \Bin(3K/4,1- q)$, where $q = c(M) p^{1/2}$. Then \cite[Lemma~2.6]{mq2020geodesics} implies that $\p[X\leq K/2] \leq e^{-C_q K}$ where $C_q>0$ depends only on $q$ and $C_q \to \infty$ as $q \to 0$.  It follows that 
\begin{equation}\label{eq:good_scales}
\p[N_K \leq K/2] = O(\epsilon^2) \quad \text{as} \quad \epsilon \to 0
\end{equation}
if we pick $p$ sufficiently small.  Finally,  we note that for $j=1,2$,  $\eta_j$ merges with $\eta_j^{m+1,\ell}$ for some $\ell \in \N$ before time $\tau_{j,m}$ and so $\eta_j(\tau_{j,m})$ corresponds to the tip of the strand of $\eta_j^{m+1,\ell}$.  It follows that there exists $\wt{\delta} \in (0,1)$ depending only on $\delta_*$ such that if $A_{m+1}$ occurs for some $1 \leq m \leq K-1$,  then $F_m$ occurs as well with the above choice of $\wt{\delta}$.  This completes the proof of the lemma.
\end{proof}

Now we are ready to prove Lemma~\ref{lem:gff_flow_lines_hit}.  See Figure~\ref{fig:cut_measure_positive_illustration} for an illustration of the setup of the proof.

\begin{figure}[ht!]
\begin{center}
\includegraphics[scale=0.85]{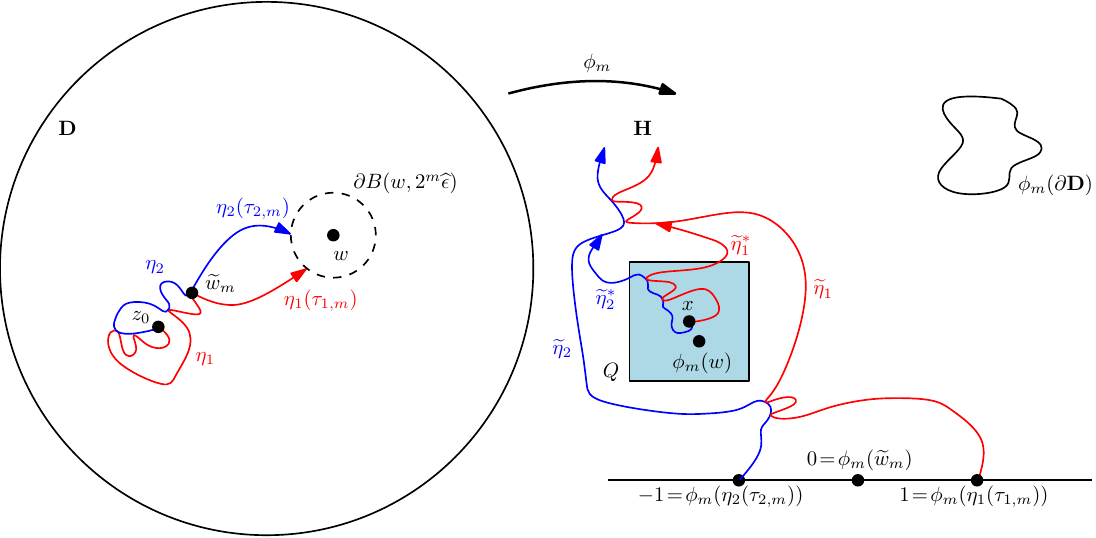}	
\end{center}
\caption{\label{fig:cut_measure_positive_illustration} Illustration of the setup of the proof of Lemma~\ref{lem:gff_flow_lines_hit}.  The map $\phi_m$ takes the component of $\C \setminus (\eta_1([0,\tau_{1,m}]) \cup \eta_2([0,\tau_{2,m}]))$ containing $w$ to $\h$ as shown.  The flow lines on the left side are generated using the field $h$ and the flow lines on the right side are generated using the field $\wt{h}$ on $\h$, which we compare to $h_m = h \circ \phi_m^{-1} - \chi \arg (\phi_m^{-1})'$ in the proof.}
\end{figure}

\begin{proof}[Proof of Lemma~\ref{lem:gff_flow_lines_hit}]
\emph{Step 1. Setup.}
Fix $\delta \in (0,1)$ small and independent of $\epsilon$ and let $\wh{\epsilon} \in (0,1)$ be such that $\epsilon = \wh{\epsilon}^{1-\delta}$.  For $w \in B(0,1/8)$ we let $E_{\wh{\epsilon}}(w)$ be the event that $\eta_j([0,\tau_j]) \cap B(w,\wh{\epsilon}) \neq \emptyset$ for $j=1,2$.  We fix some $w \in B(0,1/8)$,  let $Z = \intmeasure{\eta_1^{\tau_1}}{\eta_2^{\tau_2}}(B(w,\epsilon))$ and note that $Z = 0$ on $E_{\epsilon}(w)^c$.  For $j=1,2$ and $m \in \N$,  we let $\tau_{j,m} = \inf\{t \geq 0 : \eta_j(t) \in B(w,2^m \wh{\epsilon})\} \wedge \tau_j$,  and we also let $\sigma_{1,m}$ (resp.\ $\wt{\sigma}_{1,m}$) be the last time $t$ before time $\tau_{1,m}$ such that at time $t$ the left (resp.\ right) side of~$\eta_1$ hits the right (resp.\ left) side of $\eta_2([0,\tau_{2.m}])$.  Suppose that $\wt{\sigma}_{1,m} < \sigma_{1,m}$.  We let $\wt{U}_m$ be the connected component of $\C \setminus (\cup_{j=1}^2 \eta_j([0,\tau_{j,m}]))$ which contains $w$ and we let $\wt{w}_m$ be the prime end of $\eta_1(\sigma_{1,m})$ in $\wt{U}_m$ which separates the right side of $\eta_2$ from the left side of $\eta_1$.  Then we let $\phi_m$ be the unique conformal transformation mapping $\wt{U}_m$ onto $\h$ satisfying $\phi_m(\eta_2(\tau_{2,m})) = -1$,  $\phi_m(\wt{w}_m) = 0$, and $\phi_m(\eta_1(\tau_{1,m})) = 1$.  Furthermore,  we let $U_m = \wt{U}_m \cap \D$ and $V_m = \phi_m(U_m)$.  Finally,  we let $h_m = h \circ \phi_m^{-1} - \chi \arg(\phi_m^{-1})'$ and let $\Fh_m$ be the function which is harmonic in $V_m$ with the same boundary values as $h_m$ (that is, $\Fh_m$ is the harmonic part of the Markovian decomposition of $h_m$ in $V_m$).  Note that the harmonic function $\arg(\cdot)$ is defined modulo $2\pi n$ for $n \in \Z$,  and so we consider the additive constant such that the boundary values of $\Fh_m$ on $\partial \h$ have the form $\pm \lambda \pm \chi \pi /2$.  

Fix $\wt{\delta} \in (0,1)$ small and $C,\wt{C}>1$ large (to be chosen and independent of $\epsilon$).  For every $1 \leq m \leq \delta \log(1/\wh{\epsilon})$ we let $G_m$ be the event that the following hold.
\begin{enumerate}[(i)]
\item $\tau_{j,m} < \tau_j$ for $j=1,2$.
\item $\sup_{z \in \h \cap B(0,100C \wt{\delta}^{-1})} |\Fh_m(z)| \leq \wt{C}$.
\item\label{it:bound_on_harmonic_measures} The harmonic measure in $\wt{U}_m$ as seen from $w$ of the left and right sides of $\eta_j([0,\tau_{j,m}])$ is at least $\wt{\delta}$ for $j=1,2$.
\end{enumerate}
Let also $G$ be the event that $G_m$ occurs for some $1 \leq m \leq \delta \log(1/\wh{\epsilon})$.

With $G$ at hand, we note that
\begin{equation}
\label{eqn:p_e_bound}
\p[E_{\wh{\epsilon}}(w)] \leq \p[E_{\wh{\epsilon}}(w) \cap G^c] + \p[G]
\end{equation}
so that we are done if we can bound the probabilities on the right side of~\eqref{eqn:p_e_bound}.  We begin by bounding $\p[G]$.

\emph{Step 2. The bound on $\p[G]$.}
On the event $G_m$ we let $\wt{\eta}_{j,m}(t) = \phi_m(\eta_j(t+\tau_{j,m}))$ for $j=1,2$ and $t \in [0,\tau_j - \tau_{j,m}]$.  The key to bounding $\p[G]$ will be to prove the following claim.

\emph{Claim.} There exist $p \in (0,1)$ and $q>0$ depending only on $\wt{\delta}$, $C$, and $\wt{C}$ such that 
\begin{align*}
\p[\intmeasure{\wt{\eta}_{1,m}}{\wt{\eta}_{2,m}}(B(\phi_m(w),\wt{\delta}/40)) \geq q \,  \giv G_m] \geq p.
\end{align*}
In order to deduce the result,  we note that it is indeed true if we replace $\wt{\eta}_{j,m}$,  $j=1,2$ by interior flow lines starting from $\phi_m(w)$.  Indeed,  if we can force~$\wt{\eta}_{1,m}$ and~$\wt{\eta}_{2,m}$ to form a pocket around~$\phi_m(w)$,  then after they exit said pocket,  the flow lines from $\phi_m(w)$ will agree with $\wt{\eta}_{1,m}$ and $\wt{\eta}_{2,m}$ and thus the positive mass of $\intmeasure{\wt{\eta}_{1,m}}{\wt{\eta}_{2,m}}$ follows from that of the corresponding measure for the interior flow lines.  The idea of the proof is that if $\wt{\eta}_{1,m}$ and $\wt{\eta}_{2,m}$ intersect and form a pocket around $B(\phi_m(w),\wt{\delta}/C_1)$ (where $C_1>1$ is a large but fixed universal constant independent of $\wt{\delta}$) and contained in $B(\phi_m(w),\wt{\delta}/50)$,  then this would imply that if we start two flow lines of angles $\pm \pi/2$ from $\phi_m(w)$,  then their intersections outside of the aforementioned pocket are part of $\wt{\eta}_{1,m} \cap \wt{\eta}_{2,m}$,  which thus has a.s.\ positive measure.

\emph{Step 2.1. With positive probability, $\wt{\eta}_{1,m}$ and $\wt{\eta}_{2,m}$ form a pocket around $B(\phi_m(w),\wt{\delta}/C_1)$.}
Suppose that $G_m$ occurs.  By~\eqref{it:bound_on_harmonic_measures} and the conformal invariance of Brownian motion we have that $\im(\phi_m(w)) \geq \wt{\delta}$ and $|\phi_m(w)| \leq C \wt{\delta}^{-1}$ for small enough $\wt{\delta} > 0$ and large enough $C>1$ (but $C$ independent of $\wt{\delta}$).  By the Beurling estimate, we have that  $\dist(\phi_m(w) ,  \phi_m(\C \setminus \closure{\D})) \gtrsim \epsilon^{-1/2}$ as $\epsilon \to 0$.  Note that either $\phi_m(\eta_1(\wt{\sigma}_{1,m})) \in (-\infty,-1)$ or $\phi_m(\eta_1(\wt{\sigma}_{1,m})) \in (1,\infty)$.  We can assume that $\phi_m(\eta_1(\wt{\sigma}_{1,m})) \in (-\infty,-1)$.

Next,  we note that $\wt{\eta}_{1,m}$ (resp.\ $\wt{\eta}_{2,m}$) is the flow line of $h_m$ starting from $1$ (resp.\ $-1$) with angle $-\pi/2$ (resp.\ $\pi/2$).  We will henceforth assume that $\wt{\eta}_{1,m}$, $\wt{\eta}_{2,m}$ are stopped at the first time that they exit $\h \cap B(0,2C\wt{\delta}^{-1})$.  Let $\wh{h}_m$ be a GFF on $\h$ with the same boundary values on $\partial \h$ as $h_m$.  Let also $\wh{\eta}_{1,m}$ (resp.\ $\wh{\eta}_{2,m}$) be the flow line of $\wh{h}_m$ starting from $1$ (resp.\ $-1$) with angle $-\pi/2$ (resp.\ $\pi/2$) stopped at the first time it exits $\h \cap B(0,2C\wt{\delta}^{-1})$.  By possibly taking $\wt{C}$ to be larger,  we can assume that $|\wh{\Fh}_m(z)| \leq \wt{C}$ for every $z \in \h \cap B(0,100 C \wt{\delta}^{-1})$ where $\wh{\Fh}_m$ is the function which is harmonic on $\h$ with the same boundary values as $\wh{h}_m$ on $\partial \h$.  Hence,  it follows from the proofs of \cite[Proposition~3.4]{ms2016imag1} and \cite[Lemma~4.4]{mq2020geodesics} that $\E[\CZ_m^p] < \infty$ for every $p>1$ where $\CZ_m$ is the Radon-Nikodym derivative of the law of $\wh{h}_{m}|_{\h \cap B(0, 2C \wt{\delta}^{-1})}$ with respect to the law of $h_{m}|_{\h \cap B(0,2 C \wt{\delta}^{-1})}$.  This implies that 
\begin{equation}\label{eq:comparing_fields}
\p_{\wh{h}_m}[E]^2 \lesssim \p_{h_m}[E]
\end{equation}
for every event $E$ which is determined by the restriction of the fields to $\h \cap B(0,2C \wt{\delta}^{-1})$ and the implicit constant depends only on $C, \wt{C}, \wt{\delta}$ and $\kappa$,  where $\p_{\wh{h}_m}$ (resp.\ $\p_{h_m}$) denotes the law of $\wh{h}_m$ (resp.\ $h_m$).

Now we fix $C_1 > 1$ large (to be chosen) and let $Q_1,\dots,Q_n$ be the squares with corners in $((\wt{\delta}/C_1) \Z^2) \cap \h$ which intersect $B(0,C \wt{\delta}^{-1}) \cap \{ z \colon \im(z) \geq \wt{\delta} \}$ and let $x_j$ be the center of $Q_j$. Clearly $n$ depends only on $\wt{\delta}$, $C$ and $C_1$. We let $Q$ be a square which contains $\phi_m(w)$ (chosen according to some convention) and $x$ be its center. We note that $Q \subseteq B(x,2\wt{\delta}/C_1)$ and $\dist(Q,\partial \h) \geq \wt{\delta}/4$.  Let $\wt{h}$ be the GFF on $\h$ with boundary values given by $-\lambda - \chi \pi/2$ on $(-\infty,-1]$,  $\lambda - \chi \pi /2$ on $(-1,0]$,  $-\lambda + \chi \pi / 2$ on $(0,1]$ and $\lambda + \chi \pi / 2$ on $(1,\infty)$.  Let $\wt{\eta}_1$ (resp.\ $\wt{\eta}_2$) be the flow line of $\wt{h}$ starting from $1$ (resp.\ $-1$) with angle $-\pi / 2$ (resp.\ $\pi/2$) and targeted at $\infty$.  By Lemma~\ref{lem:pocket_formation_positive_prob} we have that, if $\wt{E}_j$ is the event that~$\wt{\eta}_1$ and~$\wt{\eta}_2$ form a pocket $P$ surrounding $B(x_j,4\wt{\delta}/C_1)$ which is contained in $B(x_j,\wt{\delta}/100)$, then $\p[\wt{E}]>0$. Consequently, since $n$ is finite, it follows that there is a $p_0 > 0$ such that if $\wt{E}$ is the event that $\wt{\eta}_1$ and $\wt{\eta}_2$ form a pocket $P$ surrounding $B(x,4\wt{\delta}/C_1)$ which is contained in $B(x,\wt{\delta}/100)$, then $\p[\wt{E}] \geq p_0$.

\emph{Step 2.2.  Wrapping up the proof of the claim and the bound on $\p[ G]$.}
What is left to do is to deduce that we are likely to have at least a small amount of cut point measure of $\wt{\eta}_1$ and $\wt{\eta}_2$ in $B(x,\wt{\delta}/50)$. Assume that we are working on $\wt{E}$.  If we let $\wt{\eta}_1^*$ (resp.\ $\wt{\eta}_2^*$) denote the flow line of $\wt{h}$ of angle $-\pi/2$ (resp.\ $\pi/2$) starting from $x$, then on the event $\wt{E}$, it follows that once $\wt{\eta}_j^*$ exits $P$, it agrees with $\wt{\eta}_j$. Furthermore, since $\mu_{\wt{\eta}_1^*,\wt{\eta}_2^*}^\cap( B(x,\wt{\delta}/50) \setminus P)$ is a.s.\ positive (since $\wt{\eta}_1^*$ and $\wt{\eta}_2^*$ intersect at the point where they exit $P$),  so that we can find $p^* \in (0,1)$ and $c>0$ such that $\p[ \mu_{\wt{\eta}_1^*,\wt{\eta}_2^*}^\cap( B(x,\wt{\delta}/50) \setminus P) > c \giv \wt{E} ] \geq p^*$.  Consequently, there is a $\wt{p} > 0$ such that $\p[ \mu_{\wt{\eta}_1,\wt{\eta}_2}^\cap(B(x,\wt{\delta}/50)) > c] \geq \wt{p}$. Since the event is determined by $\wt{h}|_{B(x,\wt{\delta}/50)}$ and the Radon-Nikodym derivatives between the law of the latter and that of $\wh{h}_m|_{B(x,\wt{\delta}/50)}$ have finite moments of all orders it follows that there exists $\wh{p} > 0$ such that $\p[ \mu_{\wh{\eta}_{1,m},\wh{\eta}_{2,m}}^\cap(B(x,\wt{\delta}/50)) > c \giv G_m] \geq \wh{p}$. By~\eqref{eq:comparing_fields} and since $B(x,\wt{\delta}/50) \subseteq B(\phi_m(w),\wt{\delta}/40)$, it follows that there exists some $p>0$ such that $\p[ \mu_{\wt{\eta}_{1,m},\wt{\eta}_{2,m}}^\cap(B(\phi_m(w),\wt{\delta}/40)) > c \giv G_m] \geq p$.  Note that by applying a similar argument,  we obtain that the last inequality still holds if we have that $\sigma_{1,m} < \wt{\sigma}_{1,m}$. Suppose that $\intmeasure{\wt{\eta}_{1,m}}{\wt{\eta}_{2,m}}(B(\phi_m(w),\wt{\delta}/40)) \geq c$.  Then \cite[Corollary~3.25]{lawler2008conformally} combined with the Koebe-$1/4$ theorem together imply that $|(\phi_m^{-1})'(z)| \gtrsim 2^m \wh{\epsilon}$ for every $z \in B(\phi_m(w),\wt{\delta}/40)$. Together with the conformal covariance, this implies that $\E[ Z \giv G_m ] \gtrsim \epsilon^{d_{\kappa'}^\cut(1+\delta)}$ for all $1 \leq m \leq M_\delta \coloneqq \lfloor \delta \log(\wh{\epsilon}^{-1}) \rfloor$. Moreover, since $G = \cup_{m=1}^{M_\delta} G_m$, there must exist some $1 \leq m' \leq M_\delta$ such that $\p[G_{m'}] \geq M_\delta^{-1} \p[G]$. Consequently
\begin{align*}
	\E[ Z ] \geq \E[Z \one_{G_{m'}}] = \E[ \E[ Z \giv G_{m'}] \one_{G_{m'}}] \gtrsim \epsilon^{d_{\kappa'}^\cut (1 + \delta)} \p[G_{m'}] \gtrsim \frac{\epsilon^{d_{\kappa'}^\cut (1 + \delta)}}{M_\delta} \p[G] \gtrsim \epsilon^{d_{\kappa'}^\cut (1 + 2\delta)} \p[G],
\end{align*}

as $\epsilon \to 0$ and recalling that $\E[ Z ] = O(\epsilon^2)$ by Lemma~\ref{lem:intersection_small_ball_measure}, it follows that
\begin{align*}
	\p[ G ] \leq C \epsilon^{2 - d_{\kappa'}^\cut(1+2\delta)}.
\end{align*}
Thus, we need only show that $\p[ E_{\wh{\epsilon}}(w) \cap G^c]$ is of at most the same order.

\emph{Step 3.  If $E_{\wh{\epsilon}}(w)$ occurs, $G$ is very likely to occur.}
We shall now prove that $\p[ E_{\wh{\epsilon}}(w) \cap G^c] = O(\epsilon^2)$.

Fix $1 \leq m \leq \delta \log(1/\wh{\epsilon})$,  set $X_m = \eta_1([0,\tau_{1,m}]) \cup \eta_2([0,\tau_{2,m}])$ and suppose that $E_{\wh{\epsilon}}(w)$ occurs.  First we note that a.s.\ on $E_{\wh{\epsilon}}(w)$,  the set $X_m $ does not disconnect $B(w,\wh{\epsilon})$ from $\partial \D$ since otherwise the flow lines $\eta_1$ and $\eta_2$ would have to cross each other and the latter does not occur a.s.\ by \cite[Theorem~1.7]{ms2017ig4}.  Also,  we note that $\sup_{x \in \partial \h}|\Fh_m(x)|  =O(1)$ where the implicit constants depend only on $\kappa$.  We claim that $\Fh_m(z) = Z_m + O(1)$ for $z \in \partial V_m \setminus \partial \h$,  where $Z_m = \chi \arg(\phi_m'(1))$ and the implicit constants in the term $O(1)$ depend only on $\kappa$.  Indeed, we set $u_m(z) = \re(\log(\phi_m'(z))-\log(\phi_m'(1))) = \log\!\left(\frac{|\phi_m'(z)|}{|\phi_m'(1)|}\right)$ and $v_m(z) = \im(\log(\phi_m'(z))-\log(\phi_m'(1))) = \arg(\phi_m'(z)) - \arg(\phi_m'(1))$ for $z \in \wt{U}_m$.  We give upper and lower bounds on $u_m$ first.  To do this,  we note that \cite[Theorem~3.21]{lawler2008conformally} implies that there exists a universal constant $C_2>1$ such that $C_2^{-1} \leq \frac{|\phi_m'(z)|}{|\phi_m'(1)|} \leq C_2$ for every $z \in B(0,7/6) \setminus B(0,5/6)$.  This implies that there exists a universal constant $C_3>0$ such that $|u_m(z)| \leq C_3$ for every $z \in B(0,7/6) \setminus B(0,5/6)$.  Hence,  combining with Lemma~\ref{lem:bound_harmonic_conjugate},  we obtain that $|v_m(z)| \leq C_4$ for every $z \in \partial \D$ for some finite universal constant $C_4$ .  Therefore we have that $\Fh_m(z) = -\chi \arg((\phi_m^{-1})'(z)) = \chi \arg(\phi_m'(\phi_m^{-1}(z)))= Z_m + O(1)$ for every $z \in \partial V_m \setminus \partial \h$.  This proves the claim.  It follows that $\Fh_m(z) = Z_m p(z) + O(1)$ for every $z \in V_m$,  where $p(z)$ is the probability that a Brownian motion starting from $z$ exits $V_m$ on $\partial V_m \setminus \partial \h$.  Note that for $z \in \h \cap B(0,100 C \wt{\delta}^{-1})$ we have that $\dist(z,\partial V_m \setminus \partial \h) \gtrsim \epsilon^{-1/2}$ and so the Beurling estimate implies that $p(z) = O(\epsilon^{1/4})$.  Thus,  if we have that $Z_m = O(\log(1/\epsilon))$,  then it holds that $\Fh_m(z) = O(1)$ for every $z \in \h \cap B(0,100 C \wt{\delta}^{-1})$ and all $\epsilon \in (0,1)$ sufficiently small,  where the implicit constants in the term $O(1)$ depend only on $\kappa$.  

Now we fix an integer $1 \leq m \leq \delta \log(1/\wh{\epsilon})$.  Let $C_m$ be the connected component of $\D \setminus X_m$ whose boundary contains $\partial \D$.  Then we can write $h = h_m^0 + H_m$ where $h_m^0$ (resp.\ $H_m$) is a zero-boundary GFF (resp.\ harmonic function) on $C_m$,  and $h_m^0,H_m$ are conditionally independent given $C_m$.  Similarly,  we can write $h_m^0 = h_m^1 + \Fh_m^1$ where $h_m^1$ (resp.\ $\Fh_m^1$) is a zero-boundary GFF on $A = \D \setminus \closure{B(0,3/4)}$ (resp.\ harmonic function on $A$),  and $h_m^1,\Fh_m^1$ are independent.  Note that $h = h_m^1 + \wt{H}_m$ where $\wt{H}_m = \Fh_m^1 + H_m$.  Fix $\frac{3}{4}<\wt{r}<r<1$.  Then for every $z \in \D \setminus \closure{B(0,\wt{r})}$,  the random variables $\Fh_m^1(z)$ and $\wt{H}_m(z)$ are both mean-zero Gaussians with variances which are bounded from above by a constant which depends only on $\wt{r}$.  Moreover both of $\Fh_m^1$ and $\wt{H}_m$ have zero boundary values on $\partial \D$. Therefore the proof of \cite[Lemma~4.4]{mq2020geodesics} implies that there exists $a>0$ sufficiently small such that both of $\E\left[ \exp\left(a \sup_{z \in \D \setminus \closure{B(0,r)}}|\nabla \Fh_m^1(z)|^2\right) \right]$ and $\E\left[ \exp\left( a \sup_{z \in \D \setminus \closure{B(0,r)}}|\nabla \wt{H}_m(z)|^2\right) \right]$ are finite.  It follows that by possibly taking $a>0$ to be smaller,  we have that 
\begin{equation}\label{eq:bound_on_square_exponential_expectation}
\E\left[ \exp\left( a \sup_{z \in \D \setminus \closure{B(0,r)}}|\nabla H_m(z)|^2 \right) \right] < \infty.
\end{equation}
Note that there exists a universal constant $\wh{C}<\infty$ such that $|v_m(z)| \leq \wh{C}$ for every $z \in B(0,7/6) \setminus B(0,13/16)$,  and so standard estimates for harmonic functions imply that $|\nabla v_m(z)| \leq \wt{C}$ for every $z \in \partial B(0,7/8)$,  where $\wt{C}<\infty$ is a universal constant.  It follows that $H_m - \chi \arg(\phi_m') = D_m + \wt{D}_m$ where $D_m$ is a harmonic function on $C_m$ with boundary values given by $0$ on $X_m$ and $Y_m = -\chi \arg(\phi_m'(1))$ on $\partial \D$, and $\wt{D}_m$ is a harmonic function on $C_m$ whose boundary values are bounded from below and above by constants which depend only on $\kappa$.  Note that there exists a unique $r_m \in (0,1)$ and unique conformal transformation $\psi_m$ mapping $C_m$ onto $A_m = \D \setminus \closure{B(0,r_m)}$ such that $\psi_m(\partial \D) = \partial \D$.  Let $\Phi_m$ be the function which is harmonic on $A_m$ such that $\Phi_m|_{\partial \D} \equiv Y_m$ and $\Phi_m|_{\partial B(0,r_m)} \equiv 0$.  Then $\Phi_m(z) = Y_m(1-\log (|z|)/\log (r_m))$.  Also,  the Cauchy-Riemann equations imply that $|\partial_x D_m|^2 + |\partial_y D_m|^2 = (Y_m/\log(r_m))^2 (|\psi_m'|/|\psi_m|)^2$.  We claim that there exists universal constant $r>0$ such that $r_m\geq r$.  Indeed,  first we note that there exists a universal constant $p>0$ such that for every $z \in \partial B(0,7/8)$ the following holds.  With probability at least $p$,  a Brownian motion starting from $z$ exits $C_m$ on $X_m$.  This implies that $|\psi_m(z)|\leq r_m^p$.  Note that $r_m\leq 3/4$ since $X_m \subseteq B(0,3/4)$.  By possibly taking $p>0$ to be smaller,  we can assume that with probability at least $p$,  a Brownian motion starting from $z \in \partial B(0,7/8)$ exits $C_m$ in $\partial \D$ and so $|\psi_m(z)| \geq r_m^{1-p}$. Thus the Beurling estimate implies that $p \lesssim r_m^{p/2}$ and so the claim follows.  Hence,  combining with the Koebe-$1/4$ theorem,  we obtain that $|\psi_m'(z)| \gtrsim 1$ where the implicit constant depends only on $r$ and $p$,  and so $|\nabla D_m(z)| \gtrsim |Y_m|$ for every $z \in \partial B(0,7/8)$.  Hence,  by taking $\wt{M}>0$ sufficiently small (depending only on $\kappa$), we have that
\begin{align*}
\p[|Y_m| \geq \log(1/\wh{\epsilon}),E_{\wh{\epsilon}}(w)]&\leq \p\left[\sup_{z \in \partial B(0,7/8)}|\nabla D_m(z)| \geq \wt{M}\log(1/\wh{\epsilon}),E_{\wh{\epsilon}}(w)\right]\\
&\leq \p\left[\sup_{z \in \partial B(0,7/8)}|\nabla H_m(z)| \geq \wt{M}\log(1/\wh{\epsilon})/2, E_{\wh{\epsilon}}(w)\right]  \\
&=  \p\left[\exp\!\left(a \sup_{z \in \partial B(0,7/8)}|\nabla H_m(z)|^2 \right) \geq \exp\!\left( \frac{a\wt{M}^2}{4} \log(1/\wh{\epsilon})^2 \right) \! , E_{\wh{\epsilon}}(w)\right] \\
&\lesssim \exp\!\left( - \frac{a\wt{M}^2}{4} \log(1/\wh{\epsilon})^2 \right) \! ,
\end{align*}
for each $1 \leq m \leq \delta \log(1/\wh{\epsilon})$, where we used Markov's inequality and~\eqref{eq:bound_on_square_exponential_expectation} in the last inequality. Taking a union bound implies that on the event $E_{\wh{\epsilon}}(w)$, we have that $|Y_m| \leq \log(1/\wh{\epsilon})$ for all $1 \leq m \leq \delta \log(1/\wh{\epsilon})$ off an event of probability $\delta \log(1/\wh{\epsilon}) \exp( -(a \wt{M}^2 /4) \log(1/\wt{\epsilon})^2) = O(\wh{\epsilon}^2)$ as $\wh{\epsilon} \to 0$. Therefore,  combining with Lemma~\ref{lem:lower_bound_harmonic_measure},  we obtain that on the event $E_{\wh{\epsilon}}(w)$,  off an event with probability $O(\wh{\epsilon}^2)$ as $\wh{\epsilon} \to 0$,  we have that $F_k$ occurs for some $1 \leq k \leq \delta \log(1/\wh{\epsilon})$ and $|Y_m| \leq \log(1/\wh{\epsilon})$ for every $1 \leq m \leq \delta\log(1/\wh{\epsilon})$,  where $F_k$ is defined as in Lemma~\ref{lem:lower_bound_harmonic_measure}. Since on $E_{\wh{\epsilon}}(w)$, $\tau_{j,m} < \tau_j$ for each $1 \leq m \leq \delta \log (1/\wh{\epsilon})$, it follows that $\p[E_{\wh{\epsilon}}(w) \cap G^c] = O(\wh{\epsilon}^2)$ as $\wh{\epsilon} \to 0$ and so 
\begin{align*}
\p[E_{\wh{\epsilon}}(w)] \lesssim \wh{\epsilon}^2 + \epsilon^{2-d_{\kappa'}^{\cut}-2\delta d_{\kappa'}^{\cut}}
\end{align*}
as $\wh{\epsilon} \to 0$.  Since $\delta \in (0,1)$ was arbitrary,  this completes the proof of the lemma.
\end{proof}

We are finally ready to prove the bounds the moments of $\cutmeasure{\eta'}(B(z,\epsilon))$ on an event $E$ whose probability we can make arbitrarily close to $1$ by adjusting some parameters.

\begin{proof}[Proof of Lemma~\ref{lem:moments_finite}]
Fix $K \subseteq \h$ compact,  $p_0 \in (0,1)$, and let $k_0 \in \N$ be so that $\dist(K,\partial \h) \geq 2^{-k_0}$. Let $h$ be a GFF on $\h$ with boundary data $\lambda'$ on $\R_-$ and $-\lambda'$ on $\R_+$, and denote by $\eta'$ its counterflow line from $0$ to $\infty$ (note that $\eta' \sim \SLE_{\kappa'}$) and let $\ol{\eta}'$ be its space-filling variant. Let $\xi \in (5/6,1)$ be the constant of Lemma~\ref{lem:intersection_small_ball_measure} corresponding to the choice $r = 5/6$.  For $M>0$, $z \in \h$ and $r \in (0,\im(z)/2)$,  we say that $B(z,r)$ is $M$-good for $h$ if the following holds.  Let $h = h_{z,r} + \Fh_{z,r}$ where $h_{z,r}$ is a zero-boundary GFF on $B(z,r)$ and $\Fh_{z,r}$ is a distribution which is harmonic on $B(z,r)$ and independent of $h_{z,r}$.  Then we have that $\sup_{w \in B(z,\xi r)}|\Fh_{z,r}(w)-\Fh_{z,r}(z)| \leq M$.  

Fix $\delta \in (0,1/2)$,  and for $M,\wt{M}>1$,  we let $E$ be the event that the following hold.
\begin{enumerate}[(i)]
\item \label{it:filling_a_ball} For every $0 < s < t$ if $\eta'(s),\eta'(t) \in K$ and $|\eta'(s)-\eta'(t)| \leq 2^{-k_0}$ then $\eta'([s,t])$ disconnects from $\infty$ a ball of diameter at least $|\eta'(s)-\eta'(t)|^{1+\delta}$.
\item \label{it:good_condition} For each $k \in \N$ we let $\CD_k$ be the set of $z \in 2^{-k} \Z^2$ such that $\dist(z,K) \leq 2^{-k}$.  For each $k \geq 2k_0+4$ and $z \in \CD_k$ there exists $(1-\delta) k \leq j \leq k$ such that $B(z,2^{-j})$ is $M$-good for $h$ and the following holds.  Suppose that we start flow lines of $h$ with angles $\frac{\pi}{2}$ and $-\frac{\pi}{2}$ from the points in $B(z,\frac{5}{6}2^{-j}) \setminus \closure{B(z,\frac{3}{4}2^{-j})}$ with rational coordinates and stop them upon exiting $B(z,\xi 2^{-j}) \setminus \closure{B(z,\frac{1}{2}2^{-j})}$.  Then,  the number of points in $\partial B(z,\xi 2^{-j}) \cup \partial B(z,2^{-j-1})$ which can be hit by such flow lines is at most $\wt{M}$.
\end{enumerate}

Note that in~\eqref{it:good_condition}, the choice $k \geq 2k_0 + 4$ is made so that each ball $B(z,2^{-j})$ considered has distance at least $2^{-k_0 -1}$ to $\partial \h$. By Lemma~\ref{lem:counterflow_close_fills_ball} the probability that~\eqref{it:filling_a_ball} occurs can be taken to be as close to $1$ as we want, by taking $k_0$ to be larger. We will now describe why~\eqref{it:good_condition} holds with arbitrarily large probability, by picking $M,\wt{M}$ large enough. We let $E_k^1$ be the event that for each $z \in \CD_k$, $B(z,2^{-j})$ is $M$-good for $h$ for no fewer than $\delta k/2$ of the integers $j \in [(1-\delta)k,k]$. We fix $k \geq 2k_0 + 4$ and some $z \in \CD_k$ and note that by Lemma~\ref{lem:good_dense}, with $r = 2^{-\lceil (1-\delta)k \rceil}$, $K = \lceil\delta k \rceil$ and $b = 1/2$, we have that the following holds. For any $a > 0$, there exist $c_0 > 0$ and $M_0 > 0$ such that for all $M \geq M_0$, we have with probability at least $1-c_0 e^{-a\delta k}$, that $B(z,2^{-j})$ is $M$-good for $h$ for no fewer than $\delta k/2$ of the integers $j \in [(1-\delta)k,k]$. Considering $a$ large enough, say, so that $2^{2k} e^{-a \delta k} = O(e^{-k})$ as $k \to \infty$ and taking a union bound over $z \in \CD_k$, we have that $\p[(E_k^1)^c] = O(e^{-k})$. Thus, taking a union bound over $k$, we see that by choosing $M$ and $k_0$ sufficiently large, we have with probability as large as we want, that for all $k \geq 2k_0+4$, $B(z,2^{-j})$ is $M$-good for $h$ for at least $\delta k/2$ integers $j \in [(1-\delta)k,k]$ and all $z \in \CD_k$.

Next, let $\CF_j^z$ be the $\sigma$-algebra generated by $h|_{\h \setminus B(z,2^{-j})}$ and let $h = h_j^z + \Fh_j^z$ be the Markovian decomposition of $h$ into a zero-boundary GFF $h_j^z$ (resp.\ harmonic function $\Fh_j^z$) on $B(z,2^{-j})$. We let $A_j^z$ denote the event that if we start flow lines of $h$ of angles $\pm \tfrac{\pi}{2}$ from all points in $B(z,\tfrac{5}{6} 2^{-j}) \setminus \closure{B(z,\tfrac{3}{4} 2^{-j})}$ with rational coordinates and stop them upon exiting $B(z,\xi 2^{-j}) \setminus \closure{B(z,\tfrac{1}{2} 2^{-j})}$, then the number of points on $\partial B(z,\xi 2^{-j}) \cup \partial B(z,\tfrac{1}{2} 2^{-j})$ which are hit by such flow lines is at most $\wt{M}$. We denote by $\wt{A}_j^z$ the same event but with $h_j^z$ in place of $h$. We fix $p \in (0,1)$ and note that by Lemma~\ref{lem:finite_strands}, we can pick $\wt{M} > 1$ large enough so that $\p[ \wt{A}_j^z ] \geq 1-p$ for all $(1-\delta)k \leq j \leq k$ and $z \in \CD_k$, uniformly in $k \geq 2k_0 + 4$. Moreover, since $A_j^z$ (resp.\ $\wt{A}_j^z$) is determined by $h|_{B(z,\xi 2^{-j})}$ (resp.\ $h_j^z|_{B(z,\xi 2^{-j})}$), we have by Lemma~\ref{lem:good_scale_radon}, that there exists a constant $c(M) > 0$ such that on the event that $B(z,2^{-j})$ is $M$-good, we a.s.\ have that
\begin{align*}
	\p[ (A_j^z)^c \giv \CF_j^z ] \leq c(M) \p[ (\wt{A}_j^z)^c]^{1/2} \leq c(M) p^{1/2}.
\end{align*}
Let $N_k^z$ denote the number of $j \in \Z \cap [(1-\delta)k,k]$ such that $A_j^z$ occurs. On the event that there are at least $\delta k/2$ integers $j \in [(1-\delta)k,k]$ such that $B(z,2^{-j})$ is $M$-good, we have that $N_k^z$ is stochastically dominated from below by $L \sim \Bin(\lceil \delta k/2 \rceil,1-q)$, where $q = c(M) p^{1/2}$. Moreover, by \cite[Lemma~2.6]{mq2020geodesics}, we have that $\p[ L \leq k \delta/4 ] \leq e^{-\wt{c} k}$, where $\wt{c} > 0$ depends only on $q$ and $\delta$, and $\wt{c} \to \infty$ as $q \to 0$. Thus, choosing $p$ small enough, and hence $\wt{M}$ large enough, $q$ small enough and $\wt{c}$ large enough (say, so that $2^{2k} e^{-\wt{c}k} = O(e^{-k})$ as $k \to \infty$), we have by a union bound over $z \in \CD_k$ that, on the event $E_k^1$, the probability that there is a $z \in \CD_k$ such that $N_k^z \leq k \delta/4$ is at most $e^{-k/2}$. Thus, since by the above, $\p[\cup_{k \geq 2k_0 + 4} E_k^1]$ can be taken arbitrarily close to $1$ (by adjusting $M$ and $k_0$), we also have that choosing $\wt{M}$ large enough (and possibly increasing $k_0$ further), that the probability that~\eqref{it:good_condition} holds can be taken to be arbitrarily close to $1$. Hence, for any $p_0 \in (0,1)$, by choosing $M,\wt{M},k_0$ to be large enough, we have that $\p[E] \geq p_0$.

Fix $z_1,\ldots,z_n \in K$ distinct.  Let $n_j$ for $1 \leq j \leq n$ and $s_{j,k},r_{j,k}$ for $1 \leq k \leq n_j$ be as in Lemma~\ref{lem:annuli_algorithm} with $r_0 = 2^{-2k_0 -4}$ and let $m_{j,k} \in \N$ be such that $r_{j,k} = 2^{-m_{j,k}}$ (recall that in Lemma~\ref{lem:annuli_algorithm} we have that $r_{j,k}$ is a negative power of $2$). We let $z_{j,k}$ be the point in $\CD_{\lceil (1+2\delta)m_{j,k} \rceil}$ which is closest to $z_j$ (breaking ties according to some fixed but unspecified convention).  We then let $r_{j,k}^*$ be the largest number of the form $2^{-m}$ in $[2^{-(1+2\delta)m_{j,k}},2^{-(1-\delta)(1+2\delta)m_{j,k}}]$ so that the conditions in~\eqref{it:good_condition} hold for the ball $B(z_{j,k},2^{-m})$. We note that $2^{-(1+2\delta) m_{j,k}} \leq r_{j,k}^* \leq 2^{-(1-\delta)(1+2\delta)m_{j,k}}$, that is, $r_{j,k}^{1+2\delta} \leq r_{j,k}^* \leq r_{j,k}^{(1-\delta)(1+2\delta)}$ and that $|z_j-z_{j,k}| \leq 2^{-1/2} r_{j,k}^{1+2\delta}$ for every $1 \leq j \leq n$ and $1 \leq k \leq n_j$.

\begin{figure}[ht!]
\begin{center}
\includegraphics[scale=0.85]{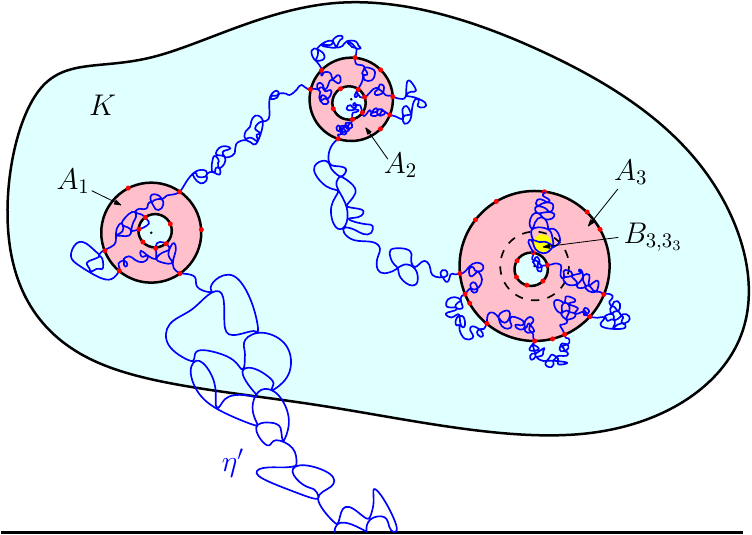}	
\end{center}
\caption{\label{fig:lemma37} Illustration of the proof of Lemma~\ref{lem:moments_finite}. Here, $n = 3$, the pink annuli are $A_j = B(z_{j,3_j},r_{j,3_j}^*) \setminus \closure{B(z_j,r)}$ and the red points on $\partial A_j$ are the points where the interior flow lines of angles $\pm \pi/2$ exit $A_j$ and hence, where $\eta'$ may enter or exit $A_j$. Illustrated in the rightmost annulus, $A_3$, is the yellow ball $B_{3,3_3}$ which the last crossing by $\eta'$ of $B(z_{3,3_3},r_{3,3_3}^*/2) \setminus \closure{B(z_{3,3_3},r_{3,3_3}^*/4)}$ separates from $\infty$. The flow lines started from a point in $B_{3,3_3}$ inevitably will make up part of the outer boundary of $\eta'$, as $\eta'$ separate them from $\infty$ and as such, the cut points of $\eta'$ in $B(z_3,r)$ make up a subset of the intersection of said flow lines.}   
\end{figure}

Next, we shall bound the measure on cut points of $\eta'$ in terms of measures on the intersection of interior flow lines of $h$, as that setting allows us to use Lemma~\ref{lem:intersection_small_ball_measure}. For an illustration, see Figure~\ref{fig:lemma37}. Fix $r > 0$ which is much smaller than $\min_{i \neq j} |z_i - z_j|$; we will eventually take a limit as $r \to 0$. In particular, let $r$ be small so that $B(z_j,r) \subseteq B(z_{j,n_j},\tfrac{1}{2} r_{j,n_j}^*)$. Let~$F$ be the event that $\eta'$ hits $B(z_j,r)$ for each $1 \leq j \leq n$. Moreover,  by~\eqref{it:good_condition},  if $(\wt{z}_\ell^j)_{\ell \geq 1}$ is an enumeration of the points in $A_j' = B(z_{j,n_j},\tfrac{5}{6} r_{j,n_j}^*) \setminus \closure{B(z_{j,n_j},\tfrac{3}{4} r_{j,n_j}^*)}$ with rational coordinates, and $\wt{\eta}_{1}^{j,\ell}$ (resp.\ $\wt{\eta}_{2}^{j,\ell}$) denotes the flow line of $h$ of angle $-\pi/2$ (resp.\ $\pi/2$) emanating from $\wt{z}_\ell^j$ and stopped upon exiting $A_j = B(z_{j,n_j},r_{j,n_j}^*) \setminus \closure{B(z_j,r)}$,  then $\partial A_j \cap \closure{\cup_\ell(\wt{\eta}_{1}^{j,\ell} \cup \wt{\eta}_{2}^{j,\ell})}$ consists of at most $\wt{M}$ points.  Suppose that we are working on the event $E \cap F$. On this event, $\eta'$ makes at least one and at most $\wt{M}$ excursions from $\partial B(z_{j,n_j},r_{j,n_j}^*)$ which intersect $\partial B(z_j,r)$. The lower bound is obvious, so we will explain the upper bound. Upon crossing $A_j'$, $\eta'$ separates from $\infty$ a ball of radius $12^{-(1+\delta)} (r_{j,n_j}^*)^{1+\delta}$ contained in $A_j'$. Each such ball contains at least one $\wt{z}_\ell^j$ (since $(\wt{z}_\ell^j)$ is dense in $A_j'$) and if we run $\wt{\eta}_{k}^{j,\ell}$ for $k \in \{1,2\}$ until exiting $A_j$, their union contains the outer boundary of $\eta'$ in $A_j$, stopped upon separating $\wt{z}_\ell^j$ from $\infty$. Thus, each crossing corresponds to at least one of the points where $\wt{\eta}_{k}^{j,\ell}$ for $k \in \{1,2\}$ exit $A_j$ and hence (since $\eta'$ a.s.\ will not enter or exit $A_j$ through the same point twice) there are at most $\wt{M}$ excursions from $\partial B(z_{j,n_j},r_{j,n_j}^*)$ intersecting $\partial B(z_j,r)$. Next, we note that since there are at most $\wt{M}$ such excursions, say, $\eta'|_{[s_i,t_i]}$, $1 \leq i \leq k$ and $k \leq \wt{M}$, there will be at least one excursion, say $\eta'|_{[s_i,t_i]}$, for which $\cutmeasure{\eta'}(B(z_j,r)) \leq \wt{M} \cutmeasure{\eta'}(\eta'([s_i,t_i]) \cap B(z_j,r))$.  Moreover, on the final crossing of $B(z_{j,n_j},r_{j,n_j}^*/2) \setminus \closure{B(z_{j,n_j},r_{j,n_j}^*/4)}$ by $\eta'|_{[s_i,t_i]}$, a ball $B_{j,n_j}$ of radius at least $4^{-(1+\delta)} (r_{j,n_j}^*)^{1+\delta}$ is separated from $\infty$ by $\eta'|_{[s_i,t_i]}$. Then, for any point $w \in B_{j,n_j}$, if we let $\eta_{1}^w$ and $\eta_{2}^w$ denote the flow lines of $h$ from $w$ of angle $-\pi/2$ and $\pi/2$ respectively, stopped upon exiting $B(z_{j,n_j},\tfrac{3}{4} r_{j,n_j}^*)$, then the cut points of $\eta'|_{[s_i,t_i]}$ in $B(z,r)$ are a subset of $\eta_{1}^w \cap \eta_{2}^w$ and hence $\cutmeasure{\eta'}(\eta'([s_i,t_i]) \cap B(z_j,r)) \leq \intmeasure{\eta_{1}^w}{\eta_{2}^w}(B(z_j,r))$, and consequently
\begin{align}
\label{eq:ub_cut_by_int}
	\cutmeasure{\eta'}(B(z_j,r)) \leq \wt{M} \intmeasure{\eta_{1}^w}{\eta_{2}^w}(B(z_j,r)).
\end{align}
Note further, that if we sample a point, say $w_{j,n_j}$, from Lebesgue measure on $B(z_{j,n_j},r_{j,n_j}^*/2) \setminus \closure{B(z_{j,n_j},r_{j,n_j}^*/4)}$, normalized to be a probability measure, then there is a universal constant $c_0 > 0$ such that with probability at least $c_0 (r_{j,n_j}^*)^{2\delta}$,  $w_{j,n_j} \in B_{j,n_j}$.  Consequently, with conditional probability given $E \cap F$ at least $c_0 (r_{j,n_j}^*)^{2\delta}$, \eqref{eq:ub_cut_by_int} holds with $w$ replaced by $w_{j,n_j}$.

Similarly, for each $1 \leq j \leq n$ such that $n_j \geq 2$ and $1 \leq k \leq n_j -1$, pick $w_{j,k}$ from Lebesgue measure on $B(z_{j,k},r_{j,k}^*/2)\setminus \closure{B(z_{j,k},r_{j,k}^*/4)}$, normalized to be a probability measure independently of everything else. Furthermore, we let $\eta_{1}^{j,k}$ (resp.\ $\eta_{2}^{j,k}$) denote the flow line of $h$ of angle $-\pi/2$ (resp.\ $\pi/2$) emanating from $w_{j,k}$, stopped upon exiting $B(z_{j,k},\tfrac{3}{4}r_{j,k}^*)$. On $E \cap F$, let $H$ be the event that each $w_{j,k}$ lies in a ball $B_{j,k}$ in $B(z_{j,k},r_{j,k}^*/2)\setminus \closure{B(z_{j,k},r_{j,k}^*/4)}$ of radius at least $(r_{j,k}^*/4)^{1+\delta}$ which is separated from $\infty$ by an excursion $\eta'$ makes from $\partial B(z_j,r)$ to $\partial B(z_{j,k},r_{j,k}^*/2)$. With universal implicit constants, we have that
\begin{align}\label{eq:lb_H}
	\p[ H \giv h] \one_{E \cap F} \gtrsim \one_{E \cap F } \prod_{j=1}^n \prod_{k=1}^{n_j} (r_{j,k}^*)^{2\delta} \gtrsim \one_{E \cap F} \prod_{j=1}^n \prod_{k=1}^{n_j} r_{j,k}^{2\delta(1+2\delta)} \quad\text{a.s.}
\end{align}
We let $I_{j,k}$ be the event that $\eta_{1}^{j,k}$ and $\eta_{2}^{j,k}$ hit $\partial B(z_j,s_{j,k})$ if $s_{j,k} \leq r_{j,k}^{1+2\delta}/16$ (and the entire sample space otherwise), $I = \cap_{j=1}^n \cap_{k=1}^{n_j} I_{j,k}$ and note that $(E \cap F  \cap H) \subseteq (E \cap I)$. We note that $I$ is introduced to further localize the event of consideration and enable us to use approximate independence arguments of the GFF (in particular, we need not keep track of the entire curve $\eta'$). Thus,
\begin{align}\label{eq:cut_measure_ball_events}
	\E\!\left[ \prod_{j=1}^n \cutmeasure{\eta'}(B(z_j,r)) \one_{E \cap F} \right] &\lesssim \left( \prod_{j=1}^n \prod_{k=1}^{n_j} r_{j,k}^{-2\delta(1+2\delta)} \right) \E\!\left[ \prod_{j=1}^n \cutmeasure{\eta'}(B(z_j,r)) \one_{E \cap F \cap H} \right] \nonumber \\
	&\lesssim \wt{M}^n \left( \prod_{j=1}^n \prod_{k=1}^{n_j} r_{j,k}^{-2\delta(1+2\delta)} \right) \E\!\left[ \prod_{j=1}^n \intmeasure{\eta_{1}^{w_{j,n_j}}}{\eta_{2}^{w_{j,n_j}}}(B(z_j,r)) \one_{E \cap I} \right]
\end{align}
using~\eqref{eq:lb_H} for the first inequality and~\eqref{eq:ub_cut_by_int} for the second.

We let $\CF_{j,k}$ be the $\sigma$-algebra generated by $h|_{\h \setminus B(z_{j,k},r_{j,k}^*)}$ and note that by Lemmas~\ref{lem:good_scale_radon} and~\ref{lem:gff_flow_lines_hit} and H\"older's inequality, we have for $a \in (0,2-d_{\kappa'}^\cut)$ small that
\begin{align}\label{eq:bound_I_jk}
	\p[I_{j,k} \giv \CF_{j,k}] \one_E \lesssim \Big( \frac{s_{j,k}}{r_{j,k}^{1+2\delta}} \Big)^{2-d_{\kappa'}^\cut-a}
\end{align}
where the implicit constant is universal. Next, we note that we may order the the balls $B(z_{j,k},r_{j,k})$ in a convenient way.  In particular, we let $(j_i,k_i)$ be such that $1 \leq j_i \leq n$, $1 \leq k_i \leq n_{j_i}$ satisfy that either $B(z_{j_i,k_i},r_{j_i,k_i}) \subseteq B(z_{j_{i+1},k_{i+1}},r_{j_{i+1},k_{i+1}})$ or $B(z_{j_i,k_i},r_{j_i,k_i}) \cap B(z_{j_{i+1},k_{i+1}},r_{j_{i+1},k_{i+1}}) = \emptyset$ for each $i$.  For each $i$, we then let $B_i^* = B(z_{j_i,k_i},r_{j_i,k_i}^*)$.  Let $N$ be the number of such indices $i$.  For $m \leq i-1$ we have that $B(z_{j_i,k_i},r_{j_i,k_i}^*) \setminus B(z_{j_i,k_i},s_{j_i,k_i})$ is disjoint from $B_m^*$ and $I_{j_m,k_m} \in \sigma(\CF_{j_m,k_m})$, we have that
\begin{align}
\label{eq:bound_E_I}
	\p[ E \cap I] &= \E[ \one_{I_{j_1,k_1}} \cdots \one_{I_{j_N,k_N}} \one_E] = \E[ \E[ \one_{I_{j_1,k_1}} \giv \CF_{j_1,k_1}] \one_{I_{j_2,k_2}} \cdots \one_{I_{j_N,k_N}} \one_E] \nonumber \\
	&\lesssim \Big( \frac{s_{j_1,k_1}}{r_{j_1,k_1}^{1+2\delta}} \Big)^{2-d_{\kappa'}^\cut-a} \E[ \one_{I_{j_2,k_2}} \cdots \one_{I_{j_N,k_N}} \one_E] \quad\text{(by~\eqref{eq:bound_I_jk})} \nonumber \\
	&\lesssim \dots \lesssim \prod_{j=1}^n \prod_{k=1}^{n_j-1} \Big( \frac{s_{j,k}}{r_{j,k}^{1+2\delta}} \Big)^{2-d_{\kappa'}^\cut-a} \p[E] \quad\text{(by~\eqref{eq:bound_I_jk})} \nonumber \\
	&\lesssim \prod_{j=1}^n \prod_{k=1}^{n_j-1} \Big( \frac{s_{j,k}}{r_{j,k}^{1+2\delta}} \Big)^{2-d_{\kappa'}^\cut-a}.
\end{align}
Moreover, by Lemmas~\ref{lem:good_scale_radon},~\ref{lem:conformal_covariance_cut_measure} (the analogous version for the measure $\intmeasure{\eta_{1}^{w_{j,n_j}}}{\eta_{2}^{w_{j,n_j}}}$) and Lemma~\ref{lem:intersection_small_ball_measure} we have that
\begin{align*}
	\E\!\left[ \intmeasure{\eta_{1}^{w_{j,n_j}}}{\eta_{2}^{w_{j,n_j}}}(B(z_j,r)) \, \middle| \, \CF_{j,n_j} \right] \one_{E \cap I} \lesssim \Big( \frac{r}{r_{j,n_j}^*} \Big)^2 (r_{j,n_j}^*)^{d_{\kappa'}^\cut} \leq r^2 \cdot r_{j,n_j}^{(1+\delta)(d_{\kappa'}^\cut - 2)}.
\end{align*}
Similarly to~\eqref{eq:bound_E_I} (since $\intmeasure{\eta_{1}^{w_{j,n_j}}}{\eta_{2}^{w_{j,n_j}}}(B(z_j,r))$ is $\CF_{i,n_i}$-measurable whenever $i \neq j$) and by the inequality~\eqref{eq:bound_E_I}, the above implies that
\begin{align}\label{eq:product_of_measures}
	\E\!\left[ \prod_{j=1}^n \intmeasure{\eta_{1}^{w_{j,n_j}}}{\eta_{2}^{w_{j,n_j}}}(B(z_j,r)) \one_{E \cap I} \right] \lesssim r^{2n} \prod_{j=1}^n \Bigg( \prod_{k=1}^{n_j-1} \Big( \frac{s_{j,k}}{r_{j,k}^{1+2\delta}} \Big)^{2 - d_{\kappa'}^\cut - a} \Bigg) r_{j,n_j}^{(1+\delta)(d_{\kappa'}^\cut-2)}.
\end{align}
Hence,
\begin{align}
	&\E\!\left[ \prod_{j=1}^n \cutmeasure{\eta'}(B(z_j,r)) \one_{E \cap F} \right] \notag \\
\lesssim& \wt{M}^n \left( \prod_{j=1}^n \prod_{k=1}^{n_j} r_{j,k}^{-2\delta(1+2\delta)} \right) \E\!\left[ \prod_{j=1}^n \intmeasure{\eta_{1}^{w_{j,n_j}}}{\eta_{2}^{w_{j,n_j}}}(B(z_j,r)) \one_{E \cap I} \right] \quad\text{(by~\eqref{eq:cut_measure_ball_events})} \notag\\
\lesssim& \wt{M}^n r^{2n} \left( \prod_{j=1}^n \prod_{k=1}^{n_j} r_{j,k}^{-2\delta(1+2\delta)} \right)  \Bigg(  \prod_{j=1}^n\prod_{k=1}^{n_j-1} \Big( \frac{s_{j,k}}{r_{j,k}^{1+2\delta}} \Big)^{2 - d_{\kappa'}^\cut - a} \Bigg) r_{j,n_j}^{(1+\delta)(d_{\kappa'}^\cut-2)} \quad\text{(by~\eqref{eq:product_of_measures})}. \label{eqn:cut_measure_ubd}
\end{align}
Combining~\eqref{eqn:cut_measure_ubd} with Lemma~\ref{lem:annuli_algorithm} implies that by choosing $a,\delta > 0$ sufficiently small and using that $r_{j,k} \gtrsim \min_{i \neq \ell} |z_i - z_\ell|$, we have that
\begin{align*}
\E\!\left[ \prod_{j=1}^n \cutmeasure{\eta'}(B(z_j,r)) \one_{E \cap F} \right]
&\lesssim r^{2n} \Bigg( \prod_{j=1}^n \prod_{k=1}^{n_j} r_{j,k}^{-c\delta} \Bigg) \prod_{j=1}^n (\min_{i < j}|z_j - z_i|)^{d_{\kappa'}^\cut - 2 - \delta} \\
&\lesssim r^{2n} (\min_{i \neq j} | z_i - z_j|)^{-c_n \delta} \prod_{j=1}^n (\min_{i < j}|z_j - z_i|)^{d_{\kappa'}^\cut - 2 - \delta},
\end{align*}
where $c>0$ is a universal constant and $c_n$ is a constant which depends only on $n$ and the implicit constant depends only on $n$, $K$, $\delta$ and $k_0$,  and $z_0=0$. Hence, dividing both sides by $r^{2n}$ and then taking the limit as $r \to 0$ implies that the measure $\E[ \cutmeasure{\eta'} \otimes \cdots \otimes \cutmeasure{\eta'} \one_E]$ (with the product taken~$n$ times) has a density $g(z_1,\dots,z_n)$ with respect to Lebesgue measure. Moreover, there exists a constant $C=C_{n,K,\delta,k_0}$ such that
\begin{align}\label{eq:density}
	g(z_1,\dots,z_n) \leq C (\min_{i \neq j} |z_i-z_j|)^{-c_n \delta} \prod_{j=1}^n (\min_{i < j}|z_j - z_i|)^{d_{\kappa'}^\cut - 2 - \delta}
\end{align}
Moreover, we note that if $f(z_1,\dots,z_n) = (\min_{i \neq j} | z_i - z_j |)^{-1}$, then $f \in L^a(K^n)$ for all $a > 0$ sufficiently small. Thus, by fixing $p',q'>1$ such that $\frac{1}{p'}+\frac{1}{q'}=1$, integrating~\eqref{eq:density} and applying H\"{o}lder's inequality, we have that
\begin{align*}
	\E[ \cutmeasure{\eta'}(B(z,\epsilon))^n \one_E ] \lesssim \epsilon^{n(d_{\kappa'}^\cut - 2-\delta + 2/q')}
\end{align*}
for each $z \in K$ and $\epsilon \in (0,\dist(K,\partial \h) /2)$ where the implicit constant depends only on $n$, $K$, $\delta$, $p'$ and $k_0$. The result for $p=n$ follows since $p'>1$ was arbitrary and $\delta \in (0,1)$ can be chosen to be arbitrarily small. The result for $p \neq n$ follows by using H\"older's inequality.
\end{proof}

\begin{proof}[Proof of Proposition~\ref{prop:cut_point_measure_bc}]
This follows from Lemma~\ref{lem:moments_finite} and the Borel-Cantelli lemma.
\end{proof}

\section{Harmonic measure estimates}
\label{sec:good_cube_lemmas}

The purpose of this section is to prove two deterministic estimates related to harmonic measure, the hyperbolic metric, and Whitney square decompositions.  They will be applied in Section~\ref{sec:pocket_argument} to show that the range of an $\SLE_{\kappa'}$ with $\kappa' \in \adjcon$ satisfies the conditions of \cite[Theorem~8.1]{kms2022sle4remov}.  For the purpose of the rest of this article, one only needs the statements and can in fact defer reading them until they are needed in Section~\ref{sec:pocket_argument}. For similar estimates of the hyperbolic metric in H\"{o}lder domains, see \cite[Section~4.6]{pommerenke1992boundary} and for an illustration of the situation we rule out, see Figure~\ref{fig:section4}.

\begin{figure}[ht!]
\begin{center}
\includegraphics[scale=0.85,page=1]{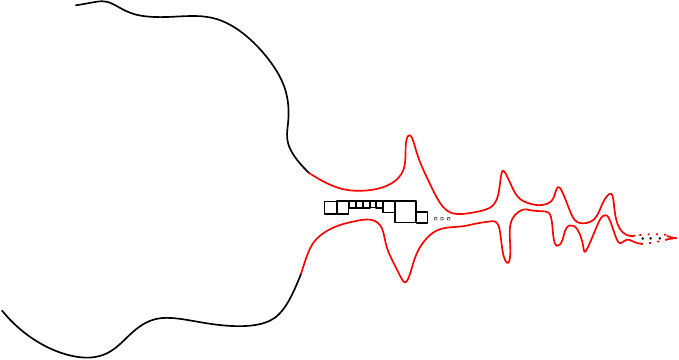}	
\end{center}
\caption{\label{fig:section4} The figure represents the situation that we spend Section~\ref{sec:good_cube_lemmas} ruling out. More precisely, two results of this section, Lemmas~\ref{lem:good_cubes} and~\ref{lem:not_too_many_large_cubes}, describe the same phenomenon, that a part of the boundary that locally looks like that of a H\"{o}lder domain (pictured in red in the figure), is not too ``inaccessible''. It is a priori obvious that such a boundary is somewhat nice on each scale, but the content of said lemmas essentially guarantees that it is uniformly nice across scales and does not end up slowly getting worse and breaking the estimates.}
\end{figure}

Suppose that $D \subseteq \C$ is a simply connected domain.  We let $\disthyp^D$ denote the hyperbolic metric in~$D$ and we fix the normalization in $\D$ to be so that
\begin{align}
\label{eqn:hyperbolic_distance}
	\disthyp^\D(z_1,z_2) = \inf_\gamma \int_\gamma \frac{|dz|}{1-|z|^2},
\end{align}
where the infimum is taken over all curves $\gamma : [0,1] \to \D$ with $\gamma(0) = z_1$ and $\gamma(1) = z_2$. Then, $\disthyp^D(z_1,z_2) = \disthyp^\D(\varphi^{-1}(z_1),\varphi^{-1}(z_2))$ for any conformal map $\varphi: \D \to D$. For a simply connected domain $D$ and two points $z,w \in \closure{D}$ (viewed as prime ends if in $\partial D$), we denote by $\gamma_{z,w}^D$ the hyperbolic geodesic in $D$ from $z$ to $w$ parameterized to have unit speed with respect to $\disthyp^D$, that is, so that the hyperbolic length of $\gamma_{z,w}^{D}|_{[s,t]}$ is $t-s$. With our choice of normalization in~\eqref{eqn:hyperbolic_distance}, we have that
\begin{align}
\label{eq:hyperbolic_geodesic}
	\gamma_{0,e^{i\theta}}^\D(t) = e^{i\theta} \! \left(1 - \frac{2}{1+e^{2t}} \right) \! .
\end{align}
We recall that for any conformal map $\varphi: \D \to D$ we have $\gamma_{z,w}^D(t) = \varphi(\gamma_{\varphi^{-1}(z),\varphi^{-1}(w)}^\D(t))$ for all $t \geq 0$.

Suppose that $\CW$ is a Whitney square decomposition of $D$. For any $Q \in \CW$, we denote by $\cen(Q)$ (resp.\ $\len(Q)$) the center (resp.\ side length) of $Q$. We also denote by $\partial^\UR Q$ the union of the upper and right boundaries of $Q$.  We also note that
\begin{equation}
\label{eqn:whitney_diameter}
\disthyp^D(z,w) \leq 1 \quad\text{for all}\quad z,w \in Q.	
\end{equation}

\begin{figure}[ht!]
\begin{center}
\includegraphics[scale=0.85,page=1]{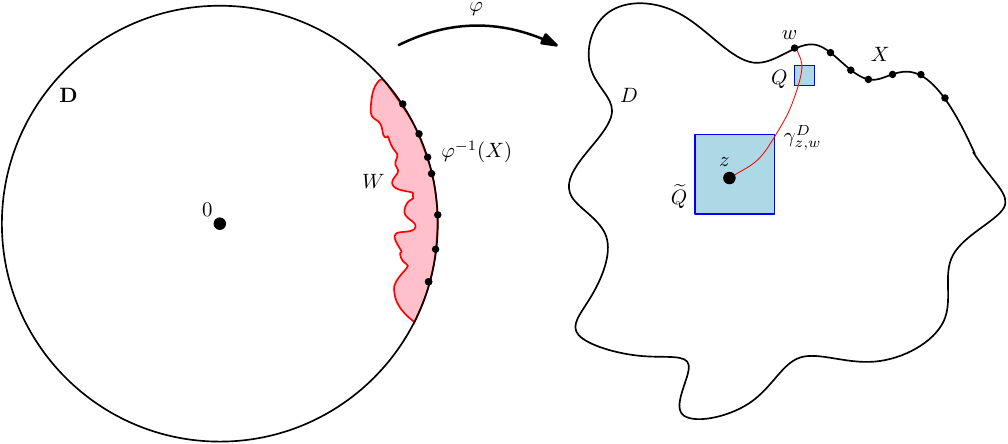}	
\end{center}
\caption{\label{fig:good_cubes_setup}  Illustration of the setup of Lemma~\ref{lem:good_cubes}.  We note that Lemma~\ref{lem:good_cubes} does not assume that $X$ is connected and in fact when we use Lemma~\ref{lem:good_cubes} to prove Theorem~\ref{thm:sle_removable} we will be in the situation that $X$ is totally disconnected as shown.}
\end{figure}

\begin{lemma}
\label{lem:good_cubes}
Let $D \subseteq \C$ be a bounded Jordan domain and let $X \subseteq \partial D$ be non-empty.  Fix $z \in D$ and let $\varphi: \D \to D$ be the conformal map satisfying $\varphi(0) = z$ and $\varphi'(0) > 0$.  Suppose that there exist $\delta, \beta \in (0,1)$ such that with $W = \{w \in \D : \dist(w,\varphi^{-1}(X)) < \delta\}$ we have that $\varphi|_W$ is $\beta$-H\"{o}lder continuous.  Furthermore,  let $\CW$ be a Whitney square decomposition of $D$ and let $\wt{Q}$ be a square in $\CW$ with $z \in \wt{Q}$.  Fix $a \in (0,1)$. Then there exists $M >0$ depending only on $a$, $\beta$, $\delta$, $D$, and $z$ such that if $w \in X$ and $\gamma_{z,w}^D$ passes through $Q \in \CW$ then 
\begin{align}\label{eq:dist_len_bound}
\disthyp^D(\cen(\wt{Q}),\cen(Q)) \leq M \len(Q)^{-a}.
\end{align}
\end{lemma}

\begin{proof}
See Figure~\ref{fig:good_cubes_setup} for an illustration of the setup. Fix $w \in X$.  We consider $Q \in \CW$ which is hit by $\gamma_{z,w}^D$ and let $t>0$ be such that $\gamma_{z,w}^D(t) \in Q$.  If $t \in (0,1)$ then~\eqref{eq:dist_len_bound} holds for large enough $M$, depending only on $a$ and $D$, since $\len(Q)^{-a} \geq \diam(D)^{-a} > 0$ (since $\diam(D) < \infty$) and because in this case, $\disthyp^D(\cen(Q),\cen(\wt{Q})) \leq 1 + t < 2$ (recall~\eqref{eqn:whitney_diameter}). Assume that $t \geq 1$.  Then we have that $\disthyp^D(\cen(Q),\cen(\wt{Q})) \leq 1 + t \leq 2t$, and so it suffices to give a similar bound on $t$.  Let $\theta \in [0,2\pi)$ be such that $\varphi(e^{i\theta}) = w$ (such a value exists since by \cite[Theorem~I.3.1]{gm2005harmonic} $\varphi$ extends to a continuous bijective function $\varphi: \closure{\D} \to \closure{D}$) and note that $\gamma_{z,w}^D(t) = \varphi(\gamma_{0,e^{i\theta}}^\D(t))$. Since $\dist(\gamma_{z,w}^D(t), \partial D) \geq \len(Q)$, we have by the Koebe-$1/4$ theorem and~\eqref{eq:hyperbolic_geodesic} that $\dist(\gamma_{z,w}(t),\partial D) \leq 8 e^{-2t} | \varphi'(\gamma_{0,e^{i\theta}}^\D(t))|$, so that
\begin{align}\label{eq:len_upper_bound}
	\len(Q)^{-a} \geq \left(8 e^{-2t} | \varphi'(\gamma_{0,e^{i\theta}}^\D(t))| \right)^{-a}.
\end{align}
We are done if we find an $M = M(a,\beta,\delta,D,z)>0$ such that $M \len(Q)^{-a} \geq 2t$ for $t \geq 1$.  It follows by~\eqref{eq:len_upper_bound} and rearranging that it suffices to find $M > 0$ such that
\begin{align}
\label{eq:sufficient_bound}
	| \varphi'( \gamma_{0,e^{i\theta}}^\D(t)) | \leq M^{1/a} 2^{-3-1/a} t^{-1/a} e^{2t}.
\end{align}
Note that there exists a constant $M_0 = M_0(D,z) > 0$ such that $|\varphi(x)-\varphi(y)| \leq M_0 |x-y|^{\beta}$ for all $x,y \in W$. Moreover, since $\varphi$ is conformal on $\D$, there exists a constant $M_1 = M_1(\delta,D,z) > 0$ such that $| \varphi(x) - \varphi(y)| \leq M_1 |x-y|^\beta$ for all $x,y \in B(0,1-\delta/2)$.  By the Koebe-$1/4$ theorem, there thus exists $M_2$ depending only on $M_0$, $M_1$ so that
\begin{align}
	| \varphi'(e^{i \theta}(1-s)) | \leq M_2 s^{\beta-1} \quad\text{for all}\quad s \in (0,1).
\end{align}
Moreover, there exists $t_0 = t_0(a,\beta)>0$ such that $e^{2(1-\beta)t} \leq t^{-1/a}e^{2t}$ for all $t \geq t_0$. Thus, noting that the minimum of $t \mapsto t^{-1/a} e^{2t}$ is attained at $t = 1/(2a)$, it follows that there exists $M$ depending only on $M_2$, $t_0$ so that
\begin{align}
| \varphi'(\gamma_{0,e^{i\theta}}^\D(t))| \leq M_2 e^{2(1-\beta)t} \leq M^{1/a} 2^{-3-1/a} t^{-1/a} e^{2t} \quad\text{for all}\quad t \geq 1.
\end{align}
Thus choosing $M$ as above and noting that $M$ depends only on $a$, $\beta$, $\delta$, $D$ and $z$, the result follows for $t \geq 1$. Thus, the proof is complete.
\end{proof}

Before stating the next lemma, we recall a useful fact about hyperbolic geodesics. If $D \subseteq \C$ is a simply connected domain, $z \in D$ and $w_1,w_2 \in \partial D$ are distinct points (prime ends), then the function $f(t) = \disthyp^D(\gamma_{z,w_1}^D(t),\gamma_{z,w_2}^D(t))$ is strictly increasing. Moreover, we have that $\disthyp^D(\gamma_{z,w_1}^D(t),\gamma_{z,w_2}^D(t)) \leq \disthyp^D(\gamma_{z,w_1}^D(t) ,  \gamma_{z,w_2}^D(t+s))$ for every $t,s\geq 0$. (These facts are trivial to check by conformally mapping to $\D$ and sending $z$ to $0$.)

\begin{figure}[ht!]
\begin{center}
\includegraphics[scale=0.85,page=2]{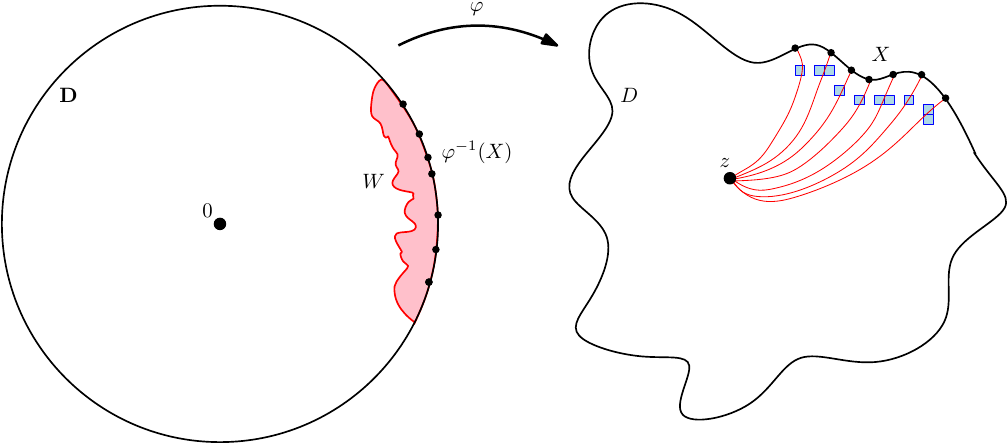}	
\end{center}
\caption{\label{fig:good_cubes_setup}  Illustration of the setup of Lemma~\ref{lem:not_too_many_large_cubes}, which bounds the number of squares in a Whitney square decomposition of $D$ of side length $2^{-j}$ which are hit by the collection of hyperbolic geodesics in $D$ from $z$ to a point in $X$ (such geodesics are shown in red and the squares of side length $2^{-j}$ which are hit are in blue) in terms of the upper Minkowski dimension of $X$.}
\end{figure}

\begin{lemma}\label{lem:not_too_many_large_cubes}
Suppose that $D \subseteq \C$ is a bounded Jordan domain and fix $z \in D$ and $\beta \in (0,1)$.  Let $\CW$ be a Whitney square decomposition of $D$.  Fix $d \in (0,2)$,  suppose that $X \subseteq \partial D$ is closed and has upper Minkowski dimension at most $d$, and let $\CW_j = \{ Q \in \CW: \len(Q) = 2^{-j},  \gamma_{z,w}^{D}([0,\infty)) \cap Q \neq \emptyset \ \text{for some} \ w \in X \}$.  Let also $\varphi: \D \to D$ be the conformal map satisfying $\varphi(0) = z$ and $\varphi'(0) > 0$ and let $\delta \in (0,1)$ be such that if $W = \{w \in \D : \dist(w,\varphi^{-1}(X)) < \delta\}$ then $\varphi|_W$ is H\"{o}lder continuous with exponent $\beta$.  Then for every $a \in (0,1)$,  there exists $c = c(a,\beta,\delta,X,D,z) > 0$ such that $|\CW_j| \leq c 2^{(d+a)j}$.
\end{lemma}

\begin{proof}
\emph{Outline and setup.} See Figure~\ref{fig:good_cubes_setup} for an illustration of the setup.  The strategy of the proof is to first consider a small neighborhood of $X$ in $D$ contained in $\varphi(W)$ which contains at most a constant times $2^{(d+3a)j}$ squares of $\CW$ of side length at least $2^{-(1+a)j}$. Next, we show that each geodesic from $z$ to a point in $X$ hits a square of $\CW$ of side length at least $2^{-(1+a)j}$ in said neighborhood and then show that if we fix such a square $Q \in \CW$, then the set of hyperbolic geodesics from $z$ to $X$ which hit $Q$ can together hit at most a constant times $2^{aj}$ squares of side length $2^{-j}$ before $Q$. The way to deduce this is to first consider a single geodesic and show that it cannot intersect more than $2^{aj}$ squares of side length $2^{-j}$ before hitting~$Q$. This follows since if the geodesic were to hit too many squares, then this would violate the H\"{o}lder continuity of $\varphi|_W$.  Then using that any two geodesics from $z$ which intersect $Q$ will stay close to each other (in the sense of $\disthyp^D$) at least until hitting $Q$, we have that the collection of squares of side length $2^{-j}$ hit by any geodesic from $z$ before hitting $Q$ is at most a constant times the number of such squares hit by any of the geodesics.

We now turn to the setup. Fix $j \in \N$, $0 < M_0 < M_1 < M_2 < 1$ small and $M_3 > 1$ large (to be chosen later) and let $Q_1,\dots,Q_n$ denote the closed squares with side length $M_0 2^{-j}$ and corners in $M_0 2^{-j} \Z^2$ such that $X \cap Q_k \neq \emptyset$ for $1 \leq k \leq n$ and we write $Q_k^* = Q_k \setminus \partial^\UR Q_k$. For $1 \leq k \leq n$ and $\ell \in \{1,2,3\}$, we let $Q_k^\ell$ be the square with the same center as $Q_k$ and side length $M_\ell 2^{-j}$. Finally, for $x \in X$, we let $\sigma_x = \sup \{ t \geq 0: \gamma_{z,x}^D(t) \notin Q_k^2 \}$ and $\tau_x = \inf\{ t \geq \sigma_x: \gamma_{z,x}^D(t) \in \partial Q_k^1 \}$, where $k = k(x)$ is such that $x \in Q_k^*$. We shall show that there exists $j_0 = j_0(a,\beta,\delta,X,D,z) \in \N$ such that each of the following steps holds for all $j \geq j_0$. We note that there exists $j_1 = j_1(D,z,\delta,X) \in \N$ such that if $j \geq j_1$, then $\varphi^{-1}( \cup_{m=1}^n Q_m^3 \cap D) \subseteq W$.  We also note that by taking $M_2>0$ sufficiently small,  we have that $Q \cap Q_m^2 = \emptyset$ for every $1 \leq m \leq n$ and every $Q \in \CW$ such that $\len(Q) = 2^{-j}$ for $j \in \N$.

\emph{Step 1. Bound on number of squares with side length at least $2^{-(1+a)j}$ which intersect $\cup_{i=1}^n Q_i^2$.} Note that we can choose $M_3$ large enough so that the number of Whitney squares of side length at least $2^{-(1+a)j}$ that intersect $Q_m^2$ is at most $M_3^2 2^{2aj}$ for every $1 \leq m \leq n$. Moreover, since $\ol{\dim}_\CM X \leq d$ it follows that $n \leq C 2^{(d+a)j}$ where $C = C(a,X,M_0)$. Consequently, there are at most $C M_3^2 2^{(d+3a)j}$ Whitney squares of side length at least $2^{-(1+a)j}$ intersecting $\cup_{m=1}^n Q_m^2$. Thus, it remains to bound the number of squares of side length $2^{-j}$ that a geodesic can hit before hitting a square with side length at least $2^{-(1+a)j}$ intersecting $Q_m^2$ for some $1 \leq m \leq n$.

\emph{Step 2. A fixed hyperbolic geodesic hits a Whitney square of side length at least $2^{-(1+a)j}$.}
Fix $x \in X$ and assume that $j \geq j_1$.  We note that $\gamma_{z,x}^D([\sigma_x,\tau_x)) \subseteq Q_k^2 \setminus Q_k^1$. We assume for the sake of contradiction that $\gamma_{z,x}^D|_{[\sigma_x,\tau_x]}$ only intersects squares in $\CW$ which have side length less than $2^{-(1+a)j}$. We let $\wt{Q}_1,\dots,\wt{Q}_N \in \CW$ be the squares which $\gamma_{z,x}^D|_{[\sigma_x,\tau_x]}$ intersects. It follows that
\begin{align*}
	(M_2 - M_1) 2^{-j} \leq | \gamma_{z,x}^D(\tau_x) - \gamma_{z,x}^D(\sigma_x) | \leq \sum_{m=1}^N \diam(\wt{Q}_m) \leq N 2^{-(1+a) j + 1/2}
\end{align*}
and thus $N \gtrsim 2^{aj}$ where the implicit constant depends only on $M_1$ and $M_2$.  Let $z_1 = \varphi^{-1}(\gamma_{z,x}^D(\sigma_x))$, $z_2 = \varphi^{-1}(\gamma_{z,x}^D(\tau_x))$, and $e^{i\theta_x} = \varphi^{-1}(x)$, and note that by \cite[Lemma~A.4]{kms2022sle4remov} there exists $c_1 = c_1(M_1,M_2) > 0$ such that $\disthyp^\D(z_1,z_2) \geq c_1 2^{aj}$. Moreover, since $\disthyp^\D(z_1,z_2) \leq \disthyp^\D(0,z_2) = \tfrac{1}{2} \log((1+|z_2|)/(1-|z_2|))$, it follows that $1-|z_2| \leq 2 \exp(-c_1 2^{aj+1})$. However, we can find a lower bound which does not match and hence get a contradiction. Indeed, since $z_1,z_2 \in W$, we have that 
\begin{align*}
	(M_1 - M_0) 2^{-j} \leq | \gamma_{z,x}^D(\tau_x) - x| \leq M |z_2 - e^{i\theta_x}|^\beta = M(1-|z_2|)^\beta,
\end{align*}
where $M$ is the $\beta$-H\"{o}lder seminorm of $\varphi|_{W}$.  Thus, $1-|z_2| \geq c_2 2^{-j/\beta}$ for some $c_2 = c_2(\beta,M,M_0,M_1) > 0$. For large enough $j$ (depending only on $a,\beta,M,M_0,M_1$ and $M_2$), this contradicts the upper bound proved above and thus there exists $j_2 = j_2(a,\beta,\delta,M,M_0,M_1,M_2,X,D,z) \geq j_1$ such that if $j \geq j_2$, then $\gamma_{z,x}^D|_{[\sigma_x,\tau_x]}$ intersects a square with side length at least $2^{-(1+a)j}$.

\emph{Step 3. Bound on the number of Whitney squares hit before hitting one with side length at least $2^{-(1+a)j}$.}
Let $Q$ be a Whitney square of side length at least $2^{-(1+a)j}$ hit by $\gamma_{z,\wt{x}}^D|_{[\sigma_{\wt{x}},\tau_{\wt{x}}]}$ for some $\wt{x} \in X \cap Q_k^*$. Let $X_k^Q = \{ w \in X \cap Q_k^* \colon \gamma_{z,w}^D([\sigma_w,\tau_w]) \cap Q \neq \emptyset \}$. For $w \in X_k^Q$, we let $t_w = \inf\{ t \geq \sigma_w \colon  \gamma_{z,w}^D(t) \in Q \}$ and for $j \geq j_2$ we define $\CW_{k,j}^Q = \{ \wt{Q} \in \CW: \len(\wt{Q}) = 2^{-j},  \wt{Q} \cap \gamma_{z,w}^D([0,t_w]) \neq \emptyset \ \text{for some} \ w \in X_k^Q\}$, so that $\CW_{k,j}^Q$ is the set of Whitney squares of side length $2^{-j}$ which are hit before $Q$ by geodesics which hit $Q$ on their last excursion from $\partial Q_k^2$ to $\partial Q_k^1$. We shall bound $|\CW_{k,j}^Q|$ (and hence get a bound on the size of the subset of squares of side length $2^{-j}$). This is done by showing that all geodesics $\gamma_{z,y}^D$ for $y \in X_k^Q$, run up until they hit a fixed square $\wt{Q} \in \CW$ with side length $2^{-j}$ belong to a hyperbolic neighborhood of (either) one of the geodesics run until hitting $Q$.

Fix $x,y \in X_k^Q$ and note that $\disthyp^D(\gamma_{z,x}^D(t_x),\gamma_{z,y}^D(t_y)) \leq 1$.  Let $\wt{Q} \in \CW$ be such that $\len(\wt{Q}) = 2^{-j}$ and $\gamma_{z,y}^D(s) \in \wt{Q}$ for some $s \geq 0$,  where $j \geq j_2$.  Note that $s \leq t_y$.  Suppose that $t_y \leq t_x$.  Then,  by the paragraph above the statement of the lemma,  we have that $\disthyp^D(\gamma_{z,x}^D(s) ,  \gamma_{z,y}^D(s)) \leq \disthyp^D(\gamma_{z,x}^D(t_x) ,  \gamma_{z,y}^D(t_y)) \leq 1$.  Suppose that $t_x < t_y$.  By possibly taking $M_2 > 0$ to be smaller we can assume that $s \leq t_x$.  Indeed,  suppose that $t_x < s$.  Then we have that $t_y \leq \disthyp^D(z,\gamma_{z,x}^D(t_x)) + \disthyp^D(\gamma_{z,x}^D(t_x) ,  \gamma_{z,y}^D(t_y)) \leq t_x +1 < s+1$ and so $\disthyp^D(\gamma_{z,y}^D(s) ,  \gamma_{z,y}^D(t_y)) = t_y - s < 1$.  It follows from \cite[Lemma~A.1]{kms2022sle4remov} that there exists a universal constant $\wh{C}>1$ such that $\len(Q) \geq \wh{C}^{-1} 2^{-j}$.  But since $Q \cap Q_k^2 \neq \emptyset$,  we obtain a contradiction if we take $M_2>0$ sufficiently small (depending only on $\wh{C}$).  Thus we have that $s \leq t_x$ and so $\disthyp^D(\gamma_{z,x}^D(s),\gamma_{z,y}^D(s)) \leq \disthyp^D(\gamma_{z,x}^D(t_x) ,  \gamma_{z,y}^D(t_y)) \leq 1$ in every case which implies that $\gamma_{z,y}^D(s) \in \Bhyp^D(\gamma_{z,x}^D(s),1)$,  where for a set $A \subseteq D$ we let $\Bhyp^D(A,r)$ denote the hyperbolic $r$-neighborhood of $A$ in $D$.  Note that \cite[Lemma~A.1]{kms2022sle4remov} implies that there exists a universal constant $\wt{M}>0$ such that $\Bhyp^D(\gamma_{z,x}^D(s),1)$ can be covered by at most $\wt{M}$ squares in $\CW$ and this implies that there exists a universal constant $\wt{C}>1$ such that $\len(\wh{Q}) \geq \wt{C}^{-1} 2^{-j}$,  where $\wh{Q} \in \CW$ is such that $\gamma_{z,x}^D(s) \in \wh{Q}$.  Combining we obtain that $|\CW_{k,j}^Q| \leq \wt{M} |\CW_{k,j}^{Q,x}|$,  where $\CW_{k,j}^{Q,x}$ is the set of squares $\wt{Q} \in \CW$ that $\gamma_{z,x}^D|_{[0,t_x]}$ intersects and such that $\len(\wt{Q}) \geq \wt{C}^{-1} 2^{-j}$.

\emph{Step 4. Upper bound on the number of squares that $\gamma_{z,x}^D|_{[0,t_x]}$ intersects for $x \in X_k^Q$.} This argument is the same as in Step 2. Assume for the sake of contradiction that $\gamma_{z,x}^D|_{[0,t_x]}$ intersects at least $2^{aj}$ squares in $\CW$.  Then,  as in Step 2,  \cite[Lemma~A.4]{kms2022sle4remov} implies that there exists a constant $\wh{c}>0$ such that $\disthyp^D(z,\gamma_{z,x}^D(t_x)) \geq \wh{c}2^{aj}$. Then, letting $w_x = \varphi^{-1}(\gamma_{z,x}^D(t_x))$, we have (as above) that $|w_x - \exp(i \theta_x)| = 1 - |w_x| \leq 2 \exp(-\wh{c} 2^{aj+1})$. Moreover, by the H\"{o}lder continuity of $\varphi|_{W}$ and the inequality $| \gamma_{z,x}^D(t_x) - x | \geq \dist(Q,\partial D) > 2^{-(1+a)j}$, we have that $|w_x - \exp(i \theta_x)| \geq M^{-1/\beta} | \gamma_{z,x}^D(t_x) - x |^{1/\beta} > M^{-1/\beta} 2^{-j(1+a)/\beta}$. Clearly, there is a $j_3 = j_3(M,\beta,\wh{c}) \in \N$ such that $M^{-1/\beta} 2^{-j(1+a)/\beta} > 2 \exp(-\wh{c}2^{aj+1})$ for all $j \geq j_3$ and hence we have a contradiction. It follows that $| \CW_{k,j}^Q | \leq \wt{M} 2^{aj}$ if $j \geq \max(j_2,j_3)$.

\emph{Step 5. Conclusion.} By the above steps, there exists $j_0 = j_0(a,\beta,\delta,X,D,z) \in \N$ such that the following holds. If $j \geq j_0$, then a hyperbolic geodesic from $z$ to a point in $X$ can hit at most $\wt{M} 2^{aj}$ Whitney squares of side length $2^j$ before hitting a Whitney square of side length at least $2^{-(1+a)j}$ intersecting $\cup_{m=1}^n Q_m^2$.  Moreover, the number of squares of side length at least $2^{-(1+a)j}$ which intersect $\cup_{m=1}^n Q_m^2$ is at most $C M_3^2 2^{(d+3a)j}$. Together, these facts imply that there is a constant $c = c(a,\beta,\delta,X,D,z) > 0$ such that $| \CW_j | \leq c 2^{(d+4a)j}$ for all $j \in \N$. Since $a \in (0,1)$ was arbitrary, the result follows.
\end{proof}

\section{Strong connectivity and good annuli}
\label{sec:pocket_argument}

The purpose of this section is to complete the proof of Theorem~\ref{thm:sle_removable}.  We will start in Section~\ref{subsec:basic_removability_conditions} by upgrading the connectivity statement for the adjacency graph of complementary components of an $\SLE_{\kappa'}$ proved in \cite{gp2020adj} and then establish a local version of it which will be phrased in terms of the flow and counterflow lines of a GFF in an annulus.  We will also use the results from Section~\ref{sec:natural_measure_on_cut_points} to show that there exists a measure on the intersection of the boundary of two adjacent components in the localized setting.  This will put us into a position to apply \cite[Theorem~8.1]{kms2022sle4remov} in Section~\ref{subsec:completion_of_the_proof} to complete the proof of Theorem~\ref{thm:sle_removable}.

\subsection{Localized adjacency graph}
\label{subsec:basic_removability_conditions}

\begin{figure}[ht!]
\begin{center}
\includegraphics[scale=0.85]{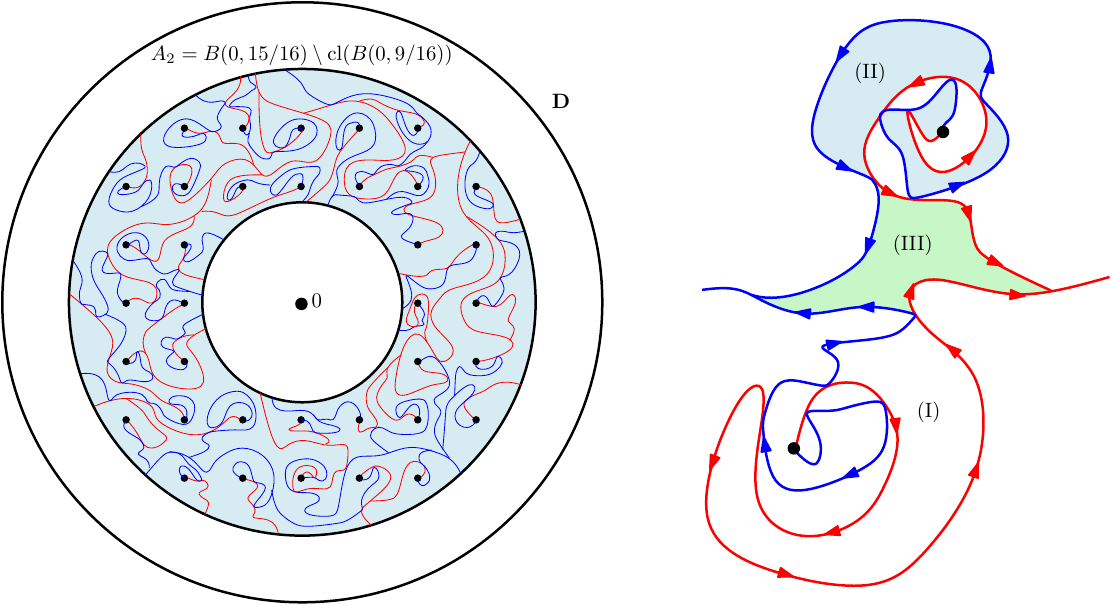}
\end{center}
\caption{\label{fig:xa_ya_illustration} {\bf Left:} Illustration of the set $X_a$.  Shown in red (resp.\ blue) are the flow lines with angle $-\pi/2$ (resp.\ $\pi/2$) starting from the grid of points $A_2 \cap (a \Z)^2$, $a > 0$ small, stopped upon exiting $A_2$.  These paths together make up $X_a$.  {\bf Right:} Illustration of the construction of $Y_a$ from $X_a$.  Shown in light green is a component of $A_2 \setminus X_a$ whose boundary consists of four arcs, given by the left/right sides of flow lines of angle $\pi/2$ and the left/right sides of flow lines of angle $-\pi/2$.  The components in blue have boundary with two marked arcs, consisting of the left side of a flow line with angle $\pi/2$ and the right side of a flow line of angle $-\pi/2$. To build $Y_a$ from $X_a$, we add a counterflow line into each component of $A_2 \setminus X_a$ which is of the same type as the green component or the blue components, from the opening point to the closing point.  In both cases, the law of the counterflow line is that of an $\SLE_{\kappa'}(\kappa'/2-4;\kappa'/2-4)$ from the opening to the closing point of the component.  In the case of a green component, the force points are located at where the boundary arcs with angle $\pi/2$ (resp.\ $-\pi/2$) merge.  In the case of a blue component, the force points are located at the opening point of the component.}
\end{figure}

As mentioned just above, the purpose of this subsection is to upgrade and localize the connectivity statement for the adjacency graph  of complementary components of an $\SLE_{\kappa'}$ established in \cite{gp2020adj}.  In order to do so, we will primarily work in the setting of a GFF on the unit disk and prove results about the behavior of its flow and counterflow lines in an annulus centered at $0$ which is away from both $0$ and $\partial \D$.  We will later use translation, scaling, and the independence properties of the GFF across annuli (Section~\ref{subsec:good_annuli}) to show that the event we define here occurs in a suitably dense collection of annuli at all scales.

We will now make the setup concrete.  See Figure~\ref{fig:xa_ya_illustration} for an illustration of the setup.  Let $A_0 = B(0,3/4) \setminus \closure{B(0,11/16)}$,  $A_1 = B(0,7/8) \setminus \closure{B(0,5/8)}$, and $A_2 = B(0,15/16) \setminus \closure{B(0,9/16)}$.  Note that $A_0 \subseteq A_1 \subseteq A_2$.  Fix $a > 0$ small.  Let $h$ be a GFF on $\D$ with zero boundary conditions and let~$X_a$ be the union of the flow lines of $h$ with angles $\pm \pi/2$ starting from all of the points in $(a \Z)^2 \cap A_2$ and stopped upon exiting $A_2$. We note that each component of $A_2 \setminus X_a$ whose boundary is completely contained in $X_a$ is naturally marked by its opening and closing points.  In other words, the opening (resp.\ closing) point is the first (resp.\ last) point on the component boundary according to the induced ordering from the definition of space-filling $\SLE_{\kappa'}$ (as described at the end of Section~\ref{subsec:ig}). 
Let $C$ be a connected component of $A_2 \setminus X_a$ whose boundary is entirely contained in $X_a$.  Then there are three possibilities:
\begin{enumerate}[(I)]
\item $\partial C$ consists of part of the right side of a flow line with angle $\frac{\pi}{2}$ and the left side of a flow line of angle $-\frac{\pi}{2}$,
\item\label{it:caseII} $\partial C$ consists of the left side of a flow line with angle $\frac{\pi}{2}$ and the right side of a flow line with angle $-\frac{\pi}{2}$, or
\item\label{it:caseIII} $\partial C$ consists of the union of the concatenation of two segments of flow lines with angle $\frac{\pi}{2}$ and the concatenation of two segments of flow lines with angle $-\frac{\pi}{2}$.
\end{enumerate}
Suppose that one of the last two cases holds.  Then,  in each such connected component we start the counterflow line of $h$ from its opening point to its closing point.  In both cases, the conditional law of the counterflow line is that of an $\SLE_{\kappa'}(\kappa'/2-4;\kappa'/2-4)$ from the opening point of the component to the closing point.  In~\eqref{it:caseII}, the force points are located at the opening point of the pocket while in~\eqref{it:caseIII} one is located at the point on the boundary where the flow lines of angle $\pi/2$ merge and the other at the corresponding point for the flow lines of angle $-\pi/2$.     We then let $Y_a$ be the closure of the union of the above counterflow lines of $h$ with the boundaries of the connected components of $A_2 \setminus X_a$ whose boundaries are entirely contained in $X_a$.

\begin{figure}[ht!]
\begin{center}
\includegraphics[scale=0.85]{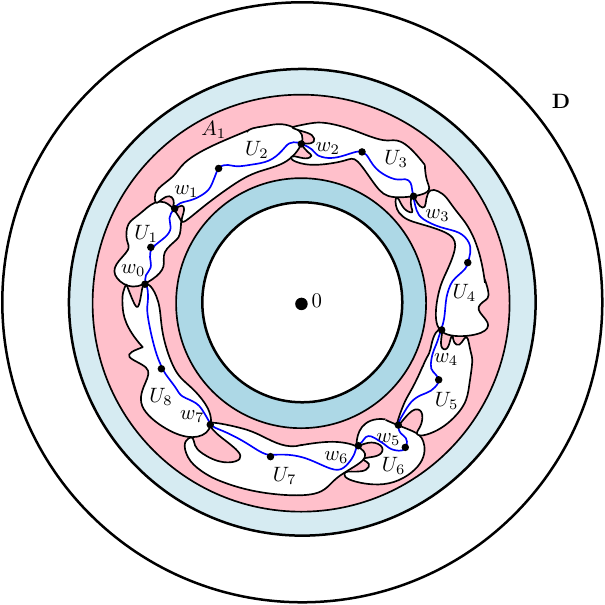}	
\end{center}
\caption{\label{fig:graph_connected} Illustration of the statement of Lemma~\ref{lem:graph_connected}.  Shown in red (resp.\ blue) is $A_1$ (resp.\ $A_2$).  The points $z_j$ and paths $\gamma_j$ are shown but unlabeled in the figure.}
\end{figure}

\begin{lemma}
\label{lem:graph_connected}
Suppose that $\kappa' \in \adjcon$.  Fix $p \in (0,1)$ and $b \in (0,d_{\kappa'}^{\cut})$.  There exists $a_0 \in (0,1)$ and $C_0 > 0$ such that whenever $a \in (0,a_0)$ and $C \geq C_0$ the following holds with probability at least~$p$.  There exist connected components $U_1,\dots,U_m$ of $A_2 \setminus Y_a$ such that $\partial U_{j-1} \cap \partial U_j \neq \emptyset$ for each $1 \leq j \leq m$ (with the convention that $U_0 = U_m$ and $U_1 = U_{m+1}$),  $U_j \cap A_0 \neq \emptyset$ and $U_j \subseteq A_1$ for each $1 \leq j \leq m$,  and the following holds.  Suppose that for each $1 \leq j \leq m$ we have points $z_j \in U_j$ and $w_{j-1} \in \partial U_{j-1} \cap \partial U_j$ and set $w_m = w_0$ and let $\gamma_j$ be a path in $\closure{U_j}$ starting and ending at $w_{j-1}$ and $w_j$ respectively and which passes through $z_j$ and does not otherwise hit $\partial U_j$.  If $\gamma$ is the concatenation of the paths $\gamma_1,\dots,\gamma_m$ then $\gamma$ disconnects $\partial B(0,1/2)$ from $\partial \D$. Moreover,  for each $1 \leq j \leq m$,  there exists a measure $\mu_j$ supported on $\partial U_{j-1} \cap \partial U_j$ and the following hold.
\begin{enumerate}[(i)]
\item \label{it:cut_point_measure_diameter_bound} $\mu_j (X) \leq C \diam(X)^{d_{\kappa'}^{\cut}-b}$ for every $X \subseteq \partial U_{j-1} \cap \partial U_j$ Borel.
\item \label{it:cut_point_measure_local_positive} $\mu_j (B(x,r) \cap \partial U_{j-1} \cap \partial U_j) > 0$ for every $x \in \partial U_{j-1} \cap \partial U_j$ and $r > 0$.
\item \label{it:upper_bound_minkowski}The upper Minkowski dimension of $\partial U_{j-1} \cap \partial U_j$ is at most $d_{\kappa'}^{\cut}$.
\item \label{it:holder_continuity} There exists $\alpha = \alpha_{\kappa'} \in (0,1)$ such that the following holds. For each conformal transformation $\phi: \D \to O$,  mapping $\D$ onto $O$ with $O \in \{U_{j-1},U_j \}$, there is an open set $W \subseteq \D$ which contains a neighborhood of $\phi^{-1}(\partial U_{j-1} \cap \partial U_j)$ in $\D$ such that $\phi|_W$ is H\"{o}lder continuous with exponent $\alpha$.
\end{enumerate}
\end{lemma}

See Figure~\ref{fig:graph_connected} for an illustration of the statement of Lemma~\ref{lem:graph_connected}.  We will prove the assertions of Lemma~\ref{lem:graph_connected} related to the connectivity in Section~\ref{subsubsec:connectivity} using the following strategy.  First, in Lemma~\ref{lem:radial_connected} we will prove a strong version of the statement that the adjacency graph of components associated with the complement of a radial $\SLE_{\kappa'}(\kappa'-6)$ is a.s.\ connected.  We will then consider a GFF $h$ on $\D$ whose boundary conditions are compatible with a coupling with a $\CLE_{\kappa'}$ instead of having zero boundary conditions and prove the result in this setting.  We note that the law of the restriction of such a field to $A_2$ is absolutely continuous with respect to the case of having zero boundary conditions, so it suffices to consider this setting.  Each of the counterflow lines considered in the definition of $Y_a$ then corresponds to part of the branch of the $\CLE_{\kappa'}$ exploration tree targeted at the corresponding point in $(a \Z)^2 \cap A_2$.  Working in this setting will allow us to use some of the resampling properties of $\CLE_{\kappa'}$ to reduce Lemma~\ref{lem:graph_connected} to the connectivity of the adjacency graph of the complementary components of an $\SLE_{\kappa'}$ and an $\SLE_{\kappa'}(\kappa'-6)$.  In Sections~\ref{subsubsec:cut_point_support} and~\ref{subsubsec:measure_on_the_intersection}, we will turn to establish the assertions of Lemma~\ref{lem:graph_connected} related to the measure on the intersections of the boundaries of the components.  This will draw on the results of Section~\ref{sec:natural_measure_on_cut_points} and the setup from Section~\ref{subsubsec:connectivity}.  We will combine all of the ingredients to complete the proof of Lemma~\ref{lem:graph_connected} in Section~\ref{subsubsec:proof_of_graph_connected}.

\subsubsection{Connectivity}
\label{subsubsec:connectivity}

\begin{lemma}
\label{lem:radial_connected}
Suppose that $\kappa' \in \adjcon$ and that $\eta'$ is a radial $\SLE_{\kappa'}(\kappa'-6)$ in $\D$ from $-i$ to $0$.  Then a.s.\ the following holds.  Let $W \subseteq \D$ be a Jordan domain and $z,w \in W \cap \Q^2$. Let $U$ (resp.\ $V$) be the connected component of $\D \setminus \eta'$ containing $z$ (resp.\ $w$).  Then there exist connected components $U_1,\ldots,U_n$ of $\D \setminus \eta'$ such that $U = U_1$, $V=U_n$ and $U_i \cap W \neq \emptyset$ for every $1 \leq i \leq n$ and $\partial U_i \cap \partial U_{i+1} \neq \emptyset$ for every $1 \leq i \leq n-1$.  Moreover,  there exist simple paths $\gamma_1,\dots,\gamma_n$ parameterized by $[0,1]$ such that $\gamma_i \subseteq W$ and $\gamma_i((0,1)) \subseteq U_i$ for every $1 \leq i \leq n$,  $\gamma_{i-1}(1) = \gamma_i(0) \in \partial U_{i-1} \cap \partial U_i$ for every $2 \leq i \leq n$ and $\gamma_1(0) = z$,  $\gamma_n(1) = w$.
\end{lemma}

The main step in the proof of Lemma~\ref{lem:radial_connected} is Lemma~\ref{lem:cle_away_from_boundary_connected}, which states that the graph of complementary components of the boundary touching loops of a $\CLE_{\kappa'}$ satisfy a stronger version of the connectivity of its adjacency graph.  Roughly speaking, this states that if we have any fixed Jordan domain~$W$ and two complementary components $U,V$ which intersect $W$ then we can take the chain of complementary components connecting $U$ and $V$ to all intersect $W$.  This will allow us to show that the adjacency graph of the complementary components of a chordal $\SLE_{\kappa'}(\kappa'-6)$ is a.s.\ connected in a strong sense from which Lemma~\ref{lem:radial_connected} will follow by using the target invariance of radial $\SLE_{\kappa'}(\kappa'-6)$.

More precisely, as described in Step~1 of the proof of Lemma~\ref{lem:radial_connected}, we can decompose the radial $\SLE_{\kappa'}(\kappa'-6)$ curve $\eta'$ in terms of a collection of chordal $\SLE_{\kappa'}(\kappa'-6)$ curves $\eta_j'$ from $-i$ to $i$ in $\D$, stopped upon disconnecting $0$ from $i$, so that there is a sequence $\sigma_j$ of random times for $\eta'$, with $\sigma_j \to \infty$ as $j \to \infty$, such that $\eta'([0,\sigma_j])$ is given in terms of $\eta_1',\dots,\eta_j'$. This is done so that if the result of Lemma~\ref{lem:radial_connected} holds with chordal $\SLE_{\kappa'}(\kappa'-6)$ --- which is the content of Lemma~\ref{lem:chordal_connected} --- in place of radial $\SLE_{\kappa'}(\kappa'-6)$, then the result holds for $\eta'([0,\sigma_j])$ for each $j$ and thus Lemma~\ref{lem:radial_connected} follows by an inductive argument. In order to achieve the result for chordal $\SLE_{\kappa'}(\kappa'-6)$, it is useful to first prove the corresponding result for the boundary-touching loops of a non-nested $\CLE_{\kappa'}$ (Lemma~\ref{lem:cle_away_from_boundary_connected}), as the chordal $\SLE_{\kappa'}(\kappa'-6)$ can be coupled as a subset of those loops. The convenient route to that result is similarly, via the corresponding result for boundary-touching loops of a nested $\CLE_{\kappa'}$. The reason for this being convenient, is that one can use the exploration of such loops described in Section~\ref{subsubsec:exploration}.

Thus, expanding a bit, the route to Lemma~\ref{lem:radial_connected} is as follows. First, we consider a ``denser'' gasket consisting of the loops of a nested $\CLE_{\kappa'}$ which intersect the boundary of a Jordan domain $W$. It follows from \cite[Corollary~3.7]{gwynne2021conformal} that this gasket has a connected adjacency graph, even when restricting to components intersecting $W$. This is the content of Lemma~\ref{lem:nested_cle_away_from_boundary_connected}. We then extend this in Lemma~\ref{lem:cle_away_from_boundary_connected} to the case of the less dense gasket consisting of the loops which touch $\partial \h$. This is done in the natural way: any component in the complement of the former gasket will be contained in a component in the complement of the latter gasket, hence we may just consider the sequence of components of the latter gasket which contains the sequence that is provided by Lemma~\ref{lem:nested_cle_away_from_boundary_connected}. Moreover, we will also extract a family of simple curves which together can be concatenated into one continuous curve passing through finitely many components (of the components just considered) which stays in $W$. After that, the same argument is then used to extend the result to the components of the complement of a chordal $\SLE_{\kappa'}(\kappa'-6)$ and the curves in the complement of this $\SLE_{\kappa'}(\kappa'-6)$ are constructed by concatenating the loops which arise in Lemma~\ref{lem:cle_away_from_boundary_connected}. This then proves Lemma~\ref{lem:chordal_connected} and from this we can deduce Lemma~\ref{lem:radial_connected}.

We will now recall a certain exploration defined in \cite{gwynne2021conformal}, which we described in Section~\ref{subsubsec:exploration}.  Let $\Gamma$ be a nested $\CLE_{\kappa'}$ in $\h$.  We fix an open set $U \subseteq \h$ and a path $P : [0,1] \rightarrow \closure{\h}$ which is either deterministic or it is random and independent of $\Gamma$ and such that $P(0) \in \partial \h$,  $P \cap \h \subseteq U$.  We let $\Gamma(P)$ be the set of loops in $\Gamma$ which intersect $P$ and let $\Gamma(P;U)$ be the set of loops in $\Gamma(P)$ which are contained in $U$.  Also,  we say that an arc $\alpha$ of a loop $\gamma$ in $\Gamma$ is a $P$-excursion of $\gamma$ into $U$ if $\alpha \subseteq \closure{U}$,  $\alpha \cap P \neq \emptyset$,  and $\alpha$ is not properly contained in any larger arc of $\gamma$ with these properties.  Moreover,  we let $S_{\Gamma}(P;U)$ be the set of $P$-excursions into $U$ of loops in $\Gamma$. Finally, we abuse notation slightly and write $\cup \Gamma$ or $\cup \Gamma(P)$ to denote the union of the loops in $\Gamma$ or $\Gamma(P)$ and use the same notation even if $P$ is a path in $\closure{\h}$ which does not start at $0$.

\begin{lemma}
\label{lem:nested_cle_away_from_boundary_connected}
Fix $\kappa' \in \CK$ and let $\Gamma$ be a nested $\CLE_{\kappa'}$ in $\h$.  Let $W \subseteq \h$ be a Jordan domain such that $\closure{W} \subseteq \h$.  Then a.s.\ the following holds.  Let $U,V$ be connected components of $\h \setminus \closure{\cup \Gamma(\partial W)}$ which both intersect $W$.  Then we can find connected components $U_1,\dots,U_n$ of $\h \setminus \closure{\cup \Gamma(\partial W)}$ such that $U=U_1$, $V=U_n$, $\partial U_i \cap \partial U_{i+1} \neq \emptyset$ for every $1 \leq i \leq n-1$,  and $U_j \cap W \neq \emptyset$ for every $1 \leq j \leq n$.
\end{lemma}
We will deduce this lemma from \cite[Corollary~3.7]{gwynne2021conformal} using, roughly,  a result that tells us that the parts in $W$ of the loops of $\Gamma$ which intersect $W$ can be realized as chordal $\SLE_{\kappa'}$ curves for which the connectedness of the adjacency graph of connected components follows by the choice $\kappa' \in \CK$. In the proof, we also recall the definition of the exploration considered in \cite[Corollary~3.7]{gwynne2021conformal}.
\begin{proof}[Proof of Lemma~\ref{lem:nested_cle_away_from_boundary_connected}]
Fix $w_0 \in W$ and for every $n \in \N$,  we set $\wt{W}_n = \{z \in W : \dist(z,\partial W) > 1/n \}$. If~$n$ is such that $\dist(w_0,\partial W) > 1/n$, then we let $W_n$ be the connected component of $\wt{W}_n$ which contains $w_0$ and otherwise we let $W_n = \emptyset$.  Then we have that $W = \cup_{n\geq 1}W_n$.  Let also $\wt{\gamma} : [0,1] \rightarrow \partial W$ be a parameterization of $\partial W$ such that $\wt{\gamma}(0) = \wt{\gamma}(1) = z_0$ and $z_0$ is the point in $\partial W$ with the smallest imaginary part and ties broken according to the smallest real part.  We parameterize $\wt{\gamma}$ in the clockwise direction and let $P: [0,1] \to \closure{\h}$ be the concatenation of the segment $[0,z_0]$ and $\wt{\gamma}$,  where the parameterization of $P$ along $\wt{\gamma}$ is in the clockwise direction.  For every $k \in \N$,  we let $P_k:[0,1] \to \closure{\h}$ be given by $P_k(t) = P(t(1-1/k))$, and note that $P_k$ is simple. 
Furthermore, we set $O_n = \h \setminus \closure{W_n}$ and note that $P_k \cap \h \subseteq O_n$.

For $k,\ell,m,n \in \N$, we recall the $(m,\ell)$-exploration process described in Section~\ref{subsubsec:exploration}. Denote by $(\alpha_{m,\ell}^{k,n},\Gamma_m^{k,n})$ the $(m,\ell)$-exploration process along $P_k$ relative to $O_n$ and let $\eta_{m,\ell}^{k,n}$ be the curve that arises in the definition of the exploration (in Section~\ref{subsubsec:exploration} denoted as $\eta_{m,\ell}$). Let $\CE_{m,\ell}^{k,n} = \sigma(\Gamma(P_k;O_n),S_{\Gamma}(P_k;O_n),\CC_\Gamma(P_k;O_n) \setminus \eta_{m,\ell}^{k,n})$ and let $x_{m,\ell}^{k,n}$ (resp.\ $x_{m,\ell}^{*,k,n}$) be the starting (resp.\ ending) point of $\eta_{m,\ell}^{k,n}$,  and let $D_{m,\ell}^{k,n}$ be the connected component of $\h \setminus \closure{\cup \Gamma(P_k) \setminus \eta_{m,\ell}^{k,n}}$ with $x_{m,\ell}^{k,n}$ and $x_{m,\ell}^{*,k,n}$ on its boundary.   Note that the $P_k$-excursion of $\gamma$ into $O_n$ whose terminal endpoint is given by $x_{m,\ell}^{k,n}$ is measurable with respect to $\sigma(\Gamma(P_k;O_n),S_{\Gamma}(P_k;O_n))$.

Since $\Gamma$ satisfies the Markov property with respect to $(P_k,O_n)$ (by \cite[Corollary~3.7]{gwynne2021conformal}, described in Section~\ref{subsubsec:exploration}) we have that conditionally on $\CE_{m,\ell}^{k,n}$ and on the event that $S_{\Gamma}(P_k;O_n) \neq \emptyset$,  the curve $\eta_{m,\ell}^{k,n}$ has the law of a chordal $\SLE_{\kappa'}$ in $D_{m,\ell}^{k,n}$ from $x_{m,\ell}^{k,n}$ to $x_{m,\ell}^{*,k,n}$.  Note that $\{\eta_{m,\ell}^{k,n} : m,\ell \in \N\}$ is the set of all complementary $P_k$-excursions of loops in $\Gamma$ out of $O_n$.  Since $\kappa' \in \CK$,  it follows that a.s.,  for every $k,\ell,m \in \N$,  the adjacency graph of the connected components of $D_{m,\ell}^{k,n} \setminus \eta_{m,\ell}^{k,n}$ is connected.  Note also that a.s.\ $z_0$ is surrounded by a loop in $\Gamma$,  which implies that a.s.\ there exists $k_0 \in \N$ such that for every $k \geq k_0$ and $\ell, m \in \N$,  we have that $D_{m,\ell}^{k,n} = D_{m,\ell}^n$ and $\eta_{m,\ell}^{k,n} = \eta_{m,\ell}^n$,  where $D_{m,\ell}^n$ and $\eta_{m,\ell}^n$ are defined in the same way as $D_{m,\ell}^{k,n}$ and $\eta_{m,\ell}^{k,n}$, respectively, but with $P_k$ replaced by $P$.  It follows that a.s.\ on the event that $S_{\Gamma}(P;O_n) \neq \emptyset$,  for every $m,\ell \in \N$ such that $\eta_{m,\ell}^n$ is non-trivial,  the adjacency graph of connected components of $D_{m,\ell}^n \setminus \eta_{m,\ell}^n$ is connected.  Note also that the definition of the $\eta_{m,\ell}^n$'s implies that $\eta_{m,\ell}^n \subseteq \closure{W}$ and so every connected component of $D_{m,\ell}^n \setminus \eta_{m,\ell}^n$ intersects $W$.

\begin{figure}[ht!]
\begin{center}
\includegraphics[scale=0.85]{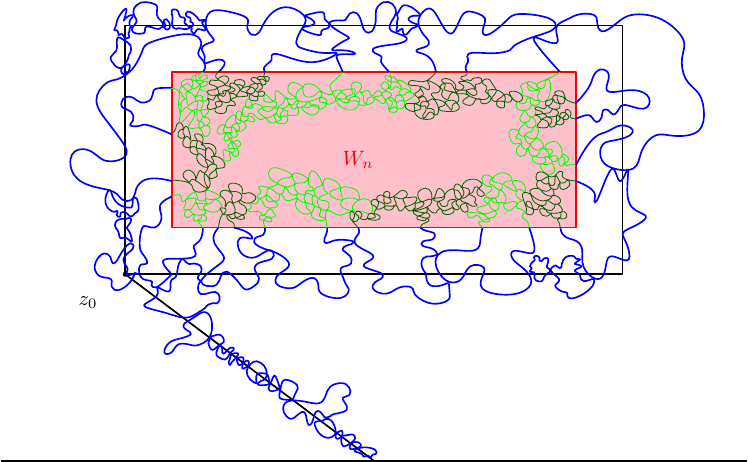}	
\end{center}
\caption{\label{fig:lemma53} Illustration of the proof of Lemma~\ref{lem:nested_cle_away_from_boundary_connected}. The curves $\eta_{m,\ell}^n$ are drawn in green (two shades of green, to make them easier to distinguish). The laws of these curves are those of chordal $\SLE_\kappa$. Since $\kappa \in \CK$, the adjacency graphs of bubbles of the $\eta_{m,\ell}^n$ curves are connected, hence for any two different connected components, one can find a finite chain of components whose boundaries intersect, connecting said components. Since that holds for each $\eta_{m,\ell}^n$, two bubbles in complements of different $\eta_{m,\ell}^n$ can be connected in the same way, if their boundaries intersect. Similarly, one can connect bubbles cut out by different $\eta_{m,\ell}^n$ farther apart. The blue loops and parts of loops are also self-intersecting (and locally look like the green ones), but are drawn as simple for simplicity.}  
\end{figure}

Let $U,V$ be distinct connected components of $\h \setminus \closure{\cup \Gamma(P)}$ which both intersect $W_n$.  Note that this implies that $S_{\Gamma}(P;O_n) \neq \emptyset$.  Also there exist $m_1,\ell_1,\ldots,m_j,\ell_j \in \N$ such that $U$ (resp.\ $V$) is a connected component of $D_{m_1,\ell_1}^n \setminus \eta_{m_1,\ell_1}^n$ (resp.\ $D_{m_j,\ell_j}^n \setminus \eta_{m_j,\ell_j}^n$) and $\partial D_{m_i,\ell_i}^n \cap \partial D_{m_{i+1},\ell_{i+1}}^n \neq \emptyset$ for each $1 \leq i \leq j-1$.  Moreover, for each $1 \leq i \leq j-1$ there exist connected components $\wt{U}_i$ and $\wt{V}_i$ of $D_{m_i,\ell_i}^n \setminus \eta_{m_i,\ell_i}^n$ and $D_{m_{i+1},\ell_{i+1}}^n \setminus \eta_{m_{i+1},\ell_{i+1}}^n$, respectively, such that $\partial \wt{U}_i \cap \partial \wt{V}_i \neq \emptyset$.  Since the adjacency graph of components of $D_{m_i,\ell_i}^n \setminus \eta_{m_i,\ell_i}^n$ for each $1 \leq i \leq j$ is connected,  we obtain that there exist connected components $U_{i,1}^n,\dots,U_{i,q_i}^n$ of $D_{m_i,\ell_i}^n \setminus \eta_{m_i,\ell_i}^n$ for each $1 \leq i \leq j$ such that $U = U_{1,1}^n$, $\wt{U}_1 = U_{1,q_1}^n$, $\wt{V}_i = U_{i+1,1}^n$ for each $1 \leq i \leq j-1$, $\wt{U}_i = U_{i,q_i}^n$ for each $1 \leq i \leq j-1$, $U_{j,q_j}^n = V$, and $\partial U_{i,k}^n \cap \partial U_{i,k+1}^n \neq \emptyset$ for each $1 \leq k \leq q_i-1$ and $1 \leq i \leq j$. See Figure~\ref{fig:lemma53} for an illustration.

Since every connected component of $\h \setminus \closure{\cup \Gamma(P)}$ which intersects $W$ must intersect $W_n$ for some $n \in \N$,  we obtain that a.s.\ the following is true.  Let $U,V$ be connected components of $\h \setminus \closure{\cup \Gamma(P)}$ which both intersect $W$.  Then there exist connected components $U_1,\ldots,U_N$ of $\h \setminus \closure{\cup \Gamma(P)}$ such that $\partial U_i \cap \partial U_{i+1} \neq \emptyset$ for every $1 \leq i \leq N-1$,  $U_i \cap W \neq \emptyset$ for every $1 \leq i \leq N$,  $U=U_1$ and $V = U_N$.  To finish the proof,  we note that if we fix a (deterministic) countable and dense subset of $\partial W$, $(a_j)_{j\geq 1}$,  then a.s.\ for every $j \in \N$,  we can find a loop in $\Gamma$ surrounding $a_j$.  This implies that a.s.\ every connected component of $\h \setminus \closure{\cup \Gamma(P)}$ intersecting $W$ is also a connected component of $\h \setminus \closure{\cup \Gamma(\partial W)}$ intersecting $W$ and vice-versa.  This completes the proof of the lemma.
\end{proof}

We now turn to the proof of the (strong form of the) connectivity of the adjacency graph of the boundary intersecting loops of a $\CLE_{\kappa'}$.
\begin{lemma}
\label{lem:cle_away_from_boundary_connected}
Suppose that $\kappa' \in \adjcon$.  Let $W \subseteq \h$ be a Jordan domain such that $\closure{W} \subseteq \h$ and fix $z,w \in W \cap \Q^2$.  Let $\Gamma$ be a non-nested $\CLE_{\kappa'}$ in $\h$ and let~$\Gamma'$ be the loops of~$\Gamma$ which intersect~$\partial \h$.  Furthermore, let $U$ (resp.\ $V$) be the connected component of $\h \setminus \closure{\cup \Gamma'}$ containing $z$ (resp.\ $w$). Then a.s.\ the following holds.  There exist connected components $U_1,\dots,U_n$ of $\h \setminus \closure{\cup \Gamma'}$ with $U_1 = U$, $U_n = V$, $\partial U_i \cap \partial U_{i+1} \neq \emptyset$ for each $1 \leq i \leq n-1$, and $U_i \cap W \neq \emptyset$ for each $1 \leq i \leq n$.  Also,  there exist simple paths $\gamma_1,\dots,\gamma_n$ parameterized by $[0,1]$ such that $\gamma_i \subseteq W$ and $\gamma_i((0,1)) \subseteq U_i$ for every $1 \leq i \leq n$,  $\gamma_1(0) = z,  \gamma_n(1) = w$,  and $\gamma_i(0) = \gamma_{i-1}(1) \in \partial U_{i-1} \cap \partial U_i$ for every $2 \leq i \leq n$.  In particular,  the adjacency graph of connected components of $\h \setminus \closure{\cup \Gamma'}$ is connected a.s.
\end{lemma}
\begin{proof}
Let $U,V$ be as in the statement of the lemma and let $\wt{\Gamma}$ be a nested $\CLE_{\kappa'}$ constructed by sampling a nested $\CLE_{\kappa'}$ in each loop of $\Gamma$. Since $U \cap W \neq \emptyset$ and $V \cap W \neq \emptyset$, there exist connected components $\wt{U},\wt{V}$ of $\h \setminus \closure{\cup \wt{\Gamma}(\partial W)}$ such that $\wt{U} \subseteq U \cap W$ and $\wt{V} \subseteq V \cap W$.  Lemma~\ref{lem:nested_cle_away_from_boundary_connected} implies that a.s.\ there exists connected components $\wt{U}_1,\dots,\wt{U}_n$ of $\h \setminus \closure{\cup \wt{\Gamma}(\partial W)}$ such that $\partial \wt{U}_i \cap \partial \wt{U}_{i+1} \neq \emptyset$ for every $1 \leq i \leq n-1$,  $\wt{U}_i \cap W \neq \emptyset$ for every $1 \leq i \leq n$,  $\wt{U} = \wt{U}_1$ and $\wt{V} = \wt{U}_n$.  Note that $\wt{U}_i \subseteq W$ for every $1 \leq i \leq n$,  since $\partial W \subseteq \closure{\cup \wt{\Gamma}(\partial W)}$.  We now construct simple paths in $\wt{U}_i$. We fix a point $x_i \in \partial \wt{U}_i \cap \partial \wt{U}_{i+1}$ for $1 \leq i \leq n-1$ and for $2 \leq i \leq n-1$ we let $\wt{\gamma}_i$ be a simple path in $\closure{\wt{U}_i}$ parameterized by $[0,1]$ such that $\wt{\gamma}_i((0,1)) \subseteq \wt{U}_i$,  $\wt{\gamma}_i(0) = x_{i-1}$ and $\wt{\gamma}_i(1) = x_i$.  We also fix $z_1 \in \wt{U}_1,  z_n \in \wt{U}_n$ and let $\wt{\gamma}_1$ (resp.\ $\wt{\gamma}_n$) be a path in $\closure{\wt{U}_1}$ (resp.\ $\closure{\wt{U}_n}$) parameterized by $[0,1]$ such that $\wt{\gamma}_1((0,1)) \subseteq \wt{U}_1$ (resp.\ $\wt{\gamma}_n((0,1)) \subseteq \wt{U}_n$),  $\wt{\gamma}_1(0) = z_1$ (resp.\ $\wt{\gamma}_n(0) = x_{n-1}$) and $\wt{\gamma}_1(1) = x_1$ (resp.\ $\wt{\gamma}_n(1) = z_n$).  

Next, we deduce the corresponding properties, but with connected components of $\h \setminus \closure{ \cup \Gamma'}$ in place of $\wt{U}_1,\dots,\wt{U}_n$. Note that for every $1 \leq i \leq n$,  there exists a connected component $\wh{U}_i$ of $\h \setminus \closure{\cup \Gamma'}$ such that $\wt{U}_i \subseteq \wh{U}_i$.  Note also that $U = \wh{U}_1,  V = \wh{U}_n$.  We define $1= i_1 < \dots <i_m \leq n$ as follows.  We set $i_1 = 1$.  Assuming that we have defined $i_1,\dots,i_k$ and $\wh{U}_{i_k} \neq \wh{U}_n$, we let $i_{k+1} = \min\{ i > i_k: \wh{U}_i \neq \wh{U}_{i_k} \}$ and we let $m = \min\{ k \in \N: \wh{U}_{i_k} = \wh{U}_n \}$. We then set $U_j = \wh{U}_{i_j}$ for each $1 \leq j \leq m$. Moreover,  we define the paths $\wh{\gamma}_1,\dots,\wh{\gamma}_m$ as follows.  For every $1 \leq j \leq m-1$,  we let $\wh{\gamma}_j$ be the concatenation of the paths $\wt{\gamma}_k$, $i_j \leq k \leq i_{j+1}-1$ (in the order of increasing $k$),  and we parameterize $\wh{\gamma}_j$ by $[0,1]$.  Furthermore, since $U_j = \wh{U}_k$ for all $i_j \leq k \leq i_{j+1}-1$, it follows that $\wh{\gamma}_j \subseteq \closure{U_j} \cap W$. For $j=m$,  we let $\wh{\gamma}_m$ be the concatenation of the paths $\wt{\gamma}_{i_m},\dots,\wt{\gamma}_n$ in this order,  where we parameterize $\wh{\gamma}_m$ by $[0,1]$.  Similarly,  we have that $\wh{\gamma}_m \subseteq \closure{U_m} \cap W$. Finally, by possibly extending $\wh{\gamma}_1$ (resp.\ $\wh{\gamma}_m$) to start at $z$ (resp.\ end at $w$), removing any loops made by the loops $\wh{\gamma}_j$, for $1 \leq j \leq m$ (recall here that while the paths $\wt{\gamma}_i$ are simple, their concatenations might not be) and adjusting the curve slightly near the boundary (as to avoid boundary intersections outside of the starting and ending points), we have collection of simple paths $\gamma_j$ for $1 \leq j \leq m$, with $\gamma_1(0) = z$, $\gamma_m(1) = w$, $\gamma_j(0) = \wh{\gamma}_j(0)$ for $2 \leq j \leq m$, $\gamma_j(1) = \wh{\gamma}_j(1)$ for $1 \leq j \leq m-1$, so that in particular, $\gamma_j(0) = \gamma_{j-1}(1) \in \partial U_{j-1} \cap \partial U_j$ for each $2 \leq j \leq m$ and $\gamma_j((0,1)) \subseteq U_j$ for each $1 \leq j \leq m$. Hence, $U_1,\dots,U_m$ and $\gamma_1,\dots,\gamma_m$ are as in the statement of the lemma and the proof of the first part is complete.

In order to deduce that $\h \setminus \closure{ \cup \Gamma'}$ is connected,  we set $W_n = (-n,n) \times (1/n,n)$ for each $n \in \N$.  Note that $W_n$ is a Jordan domain and that for every connected component $U$ of $\h \setminus \closure{\cup \Gamma'}$,  there exists $n \in \N$ such that $U \cap W_n \neq \emptyset$.  Therefore, the second part of the lemma follows from the first part.
\end{proof}

Next, we prove the (strong form of the) connectivity of the adjacency graph of a chordal $\SLE_{\kappa'}(\kappa'-6)$.

\begin{lemma}
\label{lem:chordal_connected}
Suppose that $\kappa' \in \adjcon$ and let $\eta'$ be an $\SLE_{\kappa'}(\kappa'-6)$ in $\h$ from $0$ to $\infty$.  Then the adjacency graph of components of $\h \setminus \eta'$ is a.s.\ connected. Moreover, if $W \subseteq \h$ is a Jordan domain and $z,w \in W \cap \Q^2$, then a.s.\ the following holds.  Let $U$ (resp.\ $V$) be the connected component of $\h \setminus \eta'$ which contains $z$ (resp.\ $w$). Then there exist connected components $U_1,\ldots,U_n$ of $\h \setminus \eta'$ such that $U_1 = U$, $U_n = V$, $\partial U_i \cap \partial U_{i+1} \neq \emptyset$ for each $1 \leq i \leq n-1$, and $U_i \cap W \neq \emptyset$ for each $1 \leq i \leq n$.  Also,  there exist simple paths $\gamma_1,\dots,\gamma_n$ parameterized by $[0,1]$ such that $\gamma_i \subseteq W$ and $\gamma_i((0,1)) \subseteq U_i$ for every $1 \leq i \leq n$,  $\gamma_{i-1}(1) = \gamma_i(0) \in \partial U_{i-1} \cap \partial U_i$ for every $2 \leq i \leq n$,  $\gamma_1(0) = z$ and $\gamma_n(1) = w$.
\end{lemma}
\begin{proof}
It suffices to prove the claim of the lemma when $\closure{W} \subseteq \h$.  Let $\Gamma,\Gamma'$ be as in Lemma~\ref{lem:cle_away_from_boundary_connected} and suppose that $\eta'$ is the branch of the associated exploration tree from $0$ to $\infty$.  Then we have that $\eta' \subseteq \closure{\cup \Gamma'}$.  It follows that a.s.\ every connected component of $\h \setminus \closure{\cup \Gamma'}$ is contained in some connected component of $\h \setminus \eta'$.  Suppose that $\wt{U},\wt{V}$ are the connected components of $\h \setminus \closure{\cup \Gamma'}$ which contain $z$ and $w$ respectively and let $U$ (resp.\ $V$) be the connected component of $\h \setminus \eta'$ such that $\wt{U} \subseteq U$ (resp.\ $\wt{V} \subseteq V$).  Then Lemma~\ref{lem:cle_away_from_boundary_connected} implies that we can find connected components $\wt{U}_1,\dots,\wt{U}_n$ of $\h \setminus \closure{\cup \Gamma'}$ such that $\wt{U} = \wt{U}_1$,  $\wt{V} = \wt{U}_n$,  and $\partial \wt{U}_i \cap \partial \wt{U}_{i+1} \neq \emptyset$ for every $1 \leq i \leq n-1$, as well as a family of simple paths $\wt{\gamma}_1,\dots,\wt{\gamma}_n$, contained in $W$ and such that $\wt{\gamma}_i((0,1)) \subseteq \wt{U}_i$ for each $1 \leq i \leq n$, $\wt{\gamma}_1(0) = z$, $\wt{\gamma}_n(1) = w$ and $\wt{\gamma}_i(0) = \wt{\gamma}_{i-1}(1) \in \partial \wt{U}_{i-1} \cap \partial \wt{U}_i$ for each $2 \leq i \leq n$. We proceed as in the proof of Lemma~\ref{lem:cle_away_from_boundary_connected}. Let $\wh{U}_i$ be the connected component of $\h \setminus \eta'$ containing $\wt{U}_i$, let $i_1 = 1$ and given $i_1,\dots,i_k$ define $i_{k+1} = \min\{ i > i_k: \wh{U}_i \neq \wh{U}_{i_k} \}$ and let $m = \min\{k \in \N: \wh{U}_{i_k} = \wh{U}_n \}$. Then, for each $1 \leq j \leq m$ letting $U_j = \wh{U}_{i_j}$ and defining the curve $\gamma_j$ to be the concatenation of $\gamma_k$ for $i_j \leq k \leq i_{j+1} - 1$, then removing any loops and adjusting the curves near the boundary if necessary, and finally parameterizing by $[0,1]$, we have that $U_1,\dots,U_m$ and $\gamma_1,\dots,\gamma_m$ are as in the statement of the lemma. This completes the proof.
\end{proof}

Finally, we use Lemma~\ref{lem:chordal_connected} to prove Lemma~\ref{lem:radial_connected}.

\begin{proof}[Proof of Lemma~\ref{lem:radial_connected}]
\emph{Step 1. Setup.} We will prove the claim of the lemma by iteratively applying Lemma~\ref{lem:chordal_connected} and the target invariance of $\SLE_{\kappa'}(\kappa'-6)$.  To this end, we let $\eta_1' = \eta'$, $\tau_1$ be the first time $t$ that $\eta_1'$ disconnects $0$ from $i$, and let $\varphi_1$ be the unique conformal transformation mapping the connected component of $\D \setminus \eta'([0,\tau_1])$ containing~$0$ onto~$\D$ which fixes $0$ and sends $\eta_1'(\tau_1)$ to $-i$.  Given that we have defined $\eta_j'$, $\tau_j$, $\varphi_j$ for $1 \leq j \leq k$, we let $\eta_{k+1}' = \varphi_k(\eta_k'|_{[\tau_k,\infty)})$, $\tau_{k+1}$ be the first time that $\eta_{k+1}'$ disconnects $i$ from $0$, and $\varphi_{k+1}$ be the unique conformal transformation mapping the connected component of $\D \setminus \eta_{k+1}'([0,\tau_{k+1}])$ containing $0$ onto $\D$ which fixes $0$ and sends $\eta_{k+1}'(\tau_{k+1})$ to $-i$.  For each $j$, we let $\sigma_j$ be the time for $\eta'$ which corresponds to $\tau_j$ (when $\eta'$ is parameterized by log-conformal radius as seen from $0$).  Then we have that $\sigma_j \to \infty$ as $j \to \infty$. For each $k \in \N$ we denote by $G_k$ the connected component of $\D \setminus \eta'([0,\sigma_k])$ containing $0$ and we write $\psi_k = \varphi_k \circ \dots \circ \varphi_1$, so that $\psi_k$ conformally maps $G_k$ onto $\D$, fixing $0$ and sending $\eta'(\sigma_k)$ to $-i$.

\emph{Step 2. Proof with $\eta'([0,\sigma_j])$ in place of $\eta'([0,\infty))$.} First we will show that the claim of the lemma holds with $\eta'$ replaced by $\eta'|_{[0,\sigma_j]}$ for every $j \in \N$ a.s.  We will show this using induction on $j$.  We treat the $j=1$ case first.  It suffices to prove the claim when $W \subseteq \D$ is a fixed Jordan domain and $z,w \in W \cap \Q^2$ are fixed.

\emph{Case 1. $j=1$.} Note that we can extend $\eta'|_{[0,\sigma_1]}$ after time $\sigma_1$ to a curve $\wt{\eta}'$ which terminates at $i$ and such that $\wt{\eta}'$ has the law of a chordal $\SLE_{\kappa'}(\kappa'-6)$ in $\D$ from $-i$ to $i$.  Let $\wt{U}^1$ (resp.\  $\wt{V}^1$) be the connected component of $\D \setminus \wt{\eta}'$ containing $z$ (resp.\ $w$).  Then Lemma~\ref{lem:chordal_connected} implies that a.s.\ there exist connected components $\wt{U}_1^1,\dots,\wt{U}_{\wt{n}_{1,1}}^1$ of $\D \setminus \wt{\eta}'$ and simple paths $\wt{\gamma}_1^1,\dots,\wt{\gamma}_{\wt{n}_{1,1}}^1$ with the same properties as in the statement of Lemma~\ref{lem:chordal_connected}.  For each $1 \leq \ell \leq \wt{n}_{1,1}$, let $\wh{U}_\ell^1$ be the connected component of $\D \setminus \eta'([0,\sigma_1])$ such that $\wt{U}_\ell^1 \subseteq \wh{U}_\ell^1$.  Note that $U^1 = \wh{U}_1^1$ and $V^1 = \wh{U}_{\wt{n}_{1,1}}^1$ where $U^1$ (resp.\ $V^1$) is the connected component of $\D \setminus \eta'([0,\sigma_1])$ containing $z$ (resp.\ $w$).  We define $1 = \ell_1 < \dots < \ell_{n_1} \leq \wt{n}_{1,1}$ as follows.  We set $\ell_1 = 1$ and for $k \geq 2$, we inductively define $\ell_k = \min\{ \ell > \ell_{k-1}: \wh{U}_\ell \neq \wh{U}_{\ell_{k-1}}\}$ and we define $n_1 = \min\{ m \in \N: \wh{U}_{\ell_m} = \wh{U}_{\wt{n}_{1,1}} \}$. Then, for $1 \leq j \leq n_1$, we let $U_j^1 = \wh{U}_{\ell_j}^1$. Furthermore, for each $1 \leq j \leq n_1-1$, we let $\wh{\gamma}_j^1$ be the concatenation of the paths $\wt{\gamma}_{\ell_j}^1,\dots,\wt{\gamma}_{\ell_{j+1}-1}^1$ and we let $\wh{\gamma}_{n_1}^1$ be concatenation of $\wt{\gamma}_{n_1}^1, \dots, \wt{\gamma}_{\wt{n}_{1,1}}^1$.  For all $j$ we parameterize $\wh{\gamma}_j^1$ by $[0,1]$. For each $1 \leq j \leq n_1$,  we can form the simple path $\gamma_j^1$ by removing from $\wh{\gamma}_j^1$ the loops, parameterizing it by $[0,1]$, and modifying it slightly by shifting it away from the boundary, if the concatenation hits the boundary other than at the start and endpoints (this can only happen at a finite number of points). Then, $\gamma_j^1((0,1)) \subseteq U_j^1$ and $\gamma_j^1(0) = \wh{\gamma}_j^1(0)$,  $\gamma_j^1(1) = \wh{\gamma}_j^1(1)$ and clearly the paths $\gamma_1^1,\dots,\gamma_{n_1}^1$ satisfy the desired properties and so this completes the proof of the case $j=1$.  

\emph{Case 2. $j \geq 2$.} Suppose that the claim is true for $j$ a.s. Let $U_1^j,\dots,U_{n_j}^j$ be the chain of connected components corresponding to $\D \setminus \eta'([0,\sigma_j])$ such that $U_1^j$ (resp.\ $U_{n_j}^j$) is the connected component of $\D \setminus \eta'([0,\sigma_j])$ containing $z$ (resp.\ $w$) and let $\gamma_1^j,\dots,\gamma_{n_j}^j$ be the simple paths corresponding to $W$. If $G_j \neq U_m^j$ for every $1 \leq m \leq n_j$, then each connected component $U_m^j$ is also a connected component of $\D \setminus \eta'([0,\sigma_{j+1}])$ and hence we are done by setting $n_{j+1} = n_j$,  $\gamma_\ell^{j+1} = \gamma_\ell^j$, and $U_\ell^{j+1} = U_\ell^j$ for every $1 \leq \ell \leq n_j$. Thus, we assume that $G_j = U_m^j$ for some $1 \leq m \leq n_j$ and let $1 \leq i_1 < \dots <i_{k_j} \leq n_j$ be such that $\{i_1,\dots,i_{k_j}\} = \{1\leq i \leq n_j :  U_i^j=G_j\}$. We now need to show that we can replace $\gamma_{i_1}^j$ with a sequence of simple curves passing through finitely many connected components of $\D \setminus \eta'([0,\sigma_{j+1}])$ which starts at $\gamma_{i_1}^j(0)$ and ends at $\gamma_{i_1}^j(1)$. We note that $\eta_{j+1}'|_{[0,\tau_{j+1}]} = \psi_j(\eta'|_{[\sigma_j,\sigma_{j+1}]})$ has the law of a radial $\SLE_{\kappa'}(\kappa'-6)$ in $\D$ from $-i$ to $0$, with force point $(-i)^+$, stopped upon disconnecting $i$ from~$0$ and hence we can extend $\eta_{j+1}'$ as we extended $\eta'|_{[0,\sigma_1]}$ in the $j=1$ case to a chordal $\SLE_{\kappa'}(\kappa'-6)$ targeted at $i$. Denote this curve by $\wt{\eta}_{j+1}'$. If $W_1^j$ is the connected component of $W \cap G_j$ containing $\gamma_{i_1}^j((0,1))$, then $W_1^j$ is a Jordan domain and hence, so is $\wt{W}_1^j = \psi_j(W_1^j)$. Thus, if $\wt{U}_1^j$ (resp.\ $\wt{V}_1^j$) is the connected component of $\D \setminus \wt{\eta}_{j+1}$ which has $x_1^j = \psi_j(\gamma_{i_1}^j(0))$ (resp.\ $y_1^j =\psi_j(\gamma_{i_1}^j(1))$) on its boundary, then picking points $z_1^j \in \wt{U}_1^j$ and $w_1^j \in \wt{V}_1^j$ close to $x_1^j$ and $y_1^j$, respectively, we can use Lemma~\ref{lem:chordal_connected} and reason as in the case $j=1$ to find sequences of connected components $\wt{U}_1^{j,1},\dots,\wt{U}_{n_{j,1}}^{j,1}$ of $\D \setminus \wt{\eta}'_{j+1}$ and simple curves $\wt{\gamma}_1^{j,1},\dots,\wt{\gamma}_{n_{j,1}}^{j,1}$ with $\wt{U}_1^j = \wt{U}_1^{j,1}$, $\wt{V}_1^j = \wt{U}_{n_{j,1}}^{j,1}$, $\wt{U}_m^{j,1} \cap \wt{W}_1^j \neq \emptyset$, $\wt{\gamma}_m^{j,1}((0,1)) \subseteq \wt{U}_m^{j,1}$ and $\wt{\gamma}_m^{j,1} \subseteq \wt{W}_1^j$ for $1 \leq m \leq n_{j,1}$, $\wt{\gamma}_1^{j,1}(0) = z_1^j$, $\wt{\gamma}_{n_{j,1}}^{j,1}(1) = w_1^j$, $\wt{\gamma}_m^{j,1}(0) = \wt{\gamma}_{m-1}^{j,1}(1) \in \partial \wt{U}_m^{j,1} \cap \partial \wt{U}_{m-1}^{j,1}$ for $2 \leq m \leq n_{j,1}$. Moreover, we extend (and reparameterizing by $[0,1]$) $\wt{\gamma}_1^{j,1}$ (resp.\ $\wt{\gamma}_{n_{j,1}}^{j,1}$) to start at $x_1^j$ (resp.\ end at $y_1^j$). Repeating this procedure for $i_k$, $2 \leq k \leq k_j$, we get for each such $k$ families of connected components and simple curves as above. Furthermore, the preimage of each such connected component under $\psi_j^{-1}$ is a connected component of $\D \setminus (\eta'([0,\sigma_{j}]) \cup \psi_j^{-1}(\wt{\eta}'_{j+1}))$ 
and the preimage of each such simple curve under $\psi_j^{-1}$ is a simple curve in the aforementioned connected component. Thus, by repeating the last argument from the proof of the case $j=1$ we get a finite sequence of connected components and simple curves as in the statement of the lemma (but with $\eta'([0,\sigma_j])$ in place of $\eta'([0,\infty))$). This concludes the induction step.

\emph{Step 3. Wrapping up the proof for $\eta'$.}
Now,  we prove the claim for $\eta'$.  Fix a Jordan domain $W \subseteq \D$ such that $0 \notin \closure{W}$ and let $z,w \in W \cap \Q^2$.  Let $U$ (resp.\  $V$) be the connected component of $\D \setminus \eta'$ containing $z$ (resp.\  $w$).  Then a.s.\ there exists $j \in \N$ such that $G_j \cap \closure{W} = \emptyset$ and let $U_1^j,\dots,U_{n_j}^j$ be the corresponding connected components of $\D \setminus \eta'([0,\sigma_j])$ and $\gamma_1^j,\dots,\gamma_{n_j}^j$ be the corresponding simple paths.  Since $U_m^j \cap W \neq \emptyset$ for every $1 \leq m \leq n_j$,  it follows that $U_1^j,\dots,U_{n_j}^j$ are connected components of $\D \setminus \eta'$ with $U = U_1^j$ and $V = U_{n_j}^j$.  This completes the proof of the lemma since $W$, $z$, and $w$ were arbitrary.
\end{proof}

We let $\wh{h}$ be a GFF on $\h$ with boundary conditions $\lambda'$ on $\R_-$ and $\lambda'-2\pi\chi$ on $\R_+$ and let $\varphi: \D \to \h$ be the conformal transformation such that $\varphi(-i) = 0$, $\varphi(i) = \infty$ and, say, $\varphi(1) = 1$. We then let $\wt{h} = \wh{h} \circ \varphi - \chi \arg \varphi'$ and note that $\wt{h}$ is a GFF in $\D$ with boundary conditions chosen so that the counterflow line of $\wt{h}$ from $-i$ to $i$ is a chordal $\SLE_{\kappa'}(\kappa'-6)$ with force point located at $(-i)^+$. For each $z \in \D$, we let $\eta_z'$ denote the counterflow line of $\wt{h}$ from $-i$ and targeted at $z$ (and recall that this is equivalent to first sampling interior flow lines $\eta_z^{1},\eta_z^{2}$ of angle difference $\pi$ from $z$ and then sampling the corresponding counterflow lines of $\wt{h}$ restricted to each connected component of $\D \setminus (\eta_z^{1} \cup \eta_z^{2})$ which is bounded by the left side of $\eta_z^{2}$, the right side of $\eta_z^{1}$ and possibly $\partial \D$ --- see \cite[Theorem~4.1 and Remark~4.4]{ms2017ig4}). Then $\eta_z'$ has the law of a radial $\SLE_{\kappa'}(\kappa'-6)$ in $\D$ from $-i$ to $z$. Moreover, we let $(z_j)_{j=1}^n = (a\Z)^2 \cap A_2$ and let $\wt{X}_a$ (resp.\ $\wt{Y}_a$) be defined in the same way as $X_a$ (resp.\ $Y_a$) but with the field $\wt{h}$ in place of $h$.

Above, we have proved the connectivity results for the connected components of the complement of a single radial $\SLE_{\kappa'}(\kappa'-6)$ and we wanted each of the components in our path to intersect a given Jordan domain. In what follows, however, we will consider the setting of many radial $\SLE_{\kappa'}(\kappa'-6)$ processes and need to construct a chain of components which disconnect the inside and the outside of the annulus.  Thus, we prove that we can connect the top and the bottom halves of an annulus, each of which is a Jordan domain, so that our previous results apply to them.  To this end, we let $A_0^T = A_0 \cap \h$ and $A_0^B = A_0 \cap (\C \setminus \closure{\h})$ and let $x$ (resp.\ $y$) be the midpoint of $\partial A_0^T \cap \R_-$ (resp.\ $\partial A_0^T \cap \R_+$).

\begin{lemma}\label{lem:chain_of_components_radial_curves}
Suppose that we have the setup of the above paragraph.  Then the following holds for each $1 \leq j \leq n$ and $q \in \{T,B\}$ a.s.  Let $z,w \in A_0^q \cap \Q^2$.  Let $U^{j,q}$ (resp.\ $V^{j,q}$) be the connected component of $\D \setminus \cup_{i=1}^j \eta_{z_i}'$ containing $z$ (resp.\ $w$).  There exist connected components $U_1^{j,q},\ldots,U_{n_j}^{j,q}$ of $\D \setminus \cup_{i=1}^j \eta_{z_{i}}'$ such that $U^{j,q} = U_1^{j,q}$,  $V^{j,q} = U_{n_j}^{j,q}$,  $\partial U_\ell^{j,q} \cap \partial U_{\ell+1}^{j,q} \neq \emptyset$ for every $1 \leq \ell \leq n_j-1$,  and $U_\ell^{j,q} \cap A_0^q \neq \emptyset$ for every $1 \leq \ell \leq n_j$.  Moreover,  there exist simple paths $\gamma_1^{j,q},\dots,\gamma_{n_j}^{j,q}$ parameterized by $[0,1]$ such that $\gamma_1^{j,q}(0) = z,  \gamma_{n_j}^{j,q}(1) = w,  \gamma_i^{j,q} \subseteq A_0^q$ and $\gamma_\ell^{j,q}((0,1)) \subseteq U_\ell^{j,q}$ for every $1 \leq \ell \leq n_j$,  and $\gamma_{\ell-1}^{j,q}(1) = \gamma_\ell^{j,q}(0) \in \partial U_{\ell-1}^{j,q} \cap \partial U_\ell^{j,q}$ for every $2 \leq \ell \leq n_j$.
\end{lemma}

\begin{proof}
We will prove the claim of the lemma using induction on $j$.  The case $j=1$ follows from Lemma~\ref{lem:radial_connected} since both of $A_0^T$ and $A_0^B$ are Jordan domains.  Suppose that the claim is true a.s.\ for $1 \leq j \leq n-1$ and fix $q \in \{T,B\}$.  Fix points $z,w \in A_0^q \cap \Q^2$ and let $U^{j+1,q}$ (resp.\  $V^{j+1,q}$) be the connected component of $\D \setminus \cup_{\ell=1}^{j+1}\eta_{z_\ell}'$ containing $z$ (resp.\  $w$).  Let $U_1^{j,q},\dots,U_{n_j}^{j,q}$ be the corresponding connected components of $\D \setminus \cup_{\ell=1}^j \eta_{z_\ell}'$ and $\gamma_1^{j,q},\dots,\gamma_{n_j}^{j,q}$ be the corresponding simple paths satisfying the desired properties.  Suppose that $z_{j+1} \notin U_\ell^{j,q}$ for every $1 \leq \ell \leq n_j$.  Then all of the $U_1^{j,q},\dots,U_{n_j}^{j,q}$ are connected components of $\D \setminus \cup_{i=1}^{j+1}\eta_{z_i}'$.  Then we set $U_i^{j+1,q} = U_i^{j,q}$ and $\gamma_\ell^{j+1,q} = \gamma_\ell^{j,q}$ for every $1 \leq \ell \leq n_j$.  Suppose now instead that there exists $1 \leq \ell_0 \leq n_j$ such that $U_{\ell_0}^{j,q} = G_j$, where $G_j$ is the connected component of $\D \setminus \cup_{\ell=1}^j \eta_{z_\ell}'$ containing $z_{j+1}$.  Note that the law of $\eta_{z_{j+1}}'$ restricted to $G_j$ is that of a radial $\SLE_{\kappa'}(\kappa'-6)$ in $G_j$ targeted at $z_{j+1}$.  Therefore by applying Lemma~\ref{lem:radial_connected} and arguing as in the proof of Lemma~\ref{lem:radial_connected},  we obtain the connected components $U_1^{j+1,q},\dots,U_{n_{j+1}}^{j+1,q}$ of $\D \setminus \cup_{\ell=1}^{j+1}\eta_{z_\ell}'$ and the simple paths $\gamma_1^{j+1,q},\dots,\gamma_{n_{j+1}}^{j+1,q}$ satisfying the desired properties.  This completes the proof of the induction step and hence the proof of the lemma.
\end{proof}

Next, we will need to know that with high probability, there is no large component of $A_2 \setminus \wt{X}_a$.  The following lemma is an immediate consequence of \cite[Lemma~3.6]{ghm2020almost}, but we include the proof for the convenience of the reader.

\begin{lemma}\label{lem:small_components}
Fix $p \in (0,1)$.  Then there exists $a \in (0,1)$ small enough such that with probability at least $p$,  we have that there is no connected component of $A_2 \setminus \wt{X}_a$ which intersects both $A_0$ and $\partial A_1$.
\end{lemma}

\begin{proof}
Fix $p \in (0,1)$, let $d_0 = \dist(\partial A_0, \partial A_1)$ and let $\ol{\eta}'$ be the space-filling $\SLE_{\kappa'}$ starting from $-i$ coupled with $\wt{h}$. By \cite[Lemma~3.6]{ghm2020almost} and absolute continuity between $\wt{h}$ and a whole-plane GFF, there exists $a \in (0,1)$ small such that with probability at least $p$, we have that whenever $s<t$ are such that $\{\ol{\eta}'(s),\ol{\eta}'(t)\} \subseteq A_2$ and $|\ol{\eta}'(s)-\ol{\eta}'(t)|\geq d_0/100$, then $\ol{\eta}'|_{[s,t]}$ fills in a ball of radius $2a$.  Suppose that we are working on this event.  Let $\wt{Z}_a$ be the closure of the union of the flow lines of $\wt{h}$ starting from points in $(a\Z)^2 \cap A_2$ with angles $-\frac{\pi}{2}$ and $\frac{\pi}{2}$ and run up until they hit $\partial \D$ so that $\wt{X}_a \subseteq \wt{Z}_a$.  Suppose that there exists a connected component $P$ of $A_2 \setminus \wt{X}_a$ which intersects both of $A_0$ and $\partial A_1$.  In particular,  we have that $\diam(P) \geq d_0$.  Therefore there exists a flow line $\eta_P$ with angle either $-\frac{\pi}{2}$ or $\frac{\pi}{2}$ such that the part of it which is contained in $A_2 \cap \partial P$ has diameter at least $d_0 / 4$.  Let $Q$ be the connected component of $\D \setminus \wt{Z}_a$ which is contained in $P$ and such that $\partial Q$ contains the part of $\eta_P$ which is contained in $A_2 \cap \partial P$.  In particular, we have that $\diam(Q) \geq d_0/4$ and moreover there exist points $u,v \in \partial Q \cap \eta_P$ whose distance to $\partial A_0 \cup \partial A_1$ is at least $d_0/10$ and $||u|-|v|| \geq d_0/10$.  Suppose that $u$ comes before $v$ in the space-filling ordering.  Suppose that $|u| < |v|$ and let $\ol{\eta}'|_{[s,t]}$ be the last excursion made by $\ol{\eta}'$ from $\partial B(0,|u|)$ into $\partial B(0,|v|)$ before it hits $v$ for the first time at time $t$.  Then we have that $\ol{\eta}'([s,t]) \subseteq A_2$ and $|\ol{\eta}'(s)-\ol{\eta}'(t)|\geq ||v|-|u||\geq d_0/10$.  Hence $\ol{\eta}'|_{[s,t]}$ fills a ball of radius $2a$.  A similar argument works when $|v|<|u|$.  Therefore the above ball must be contained in $P$.  But in that case,  the former must contain a point in $(a\Z)^2 \cap A_2$ and so $P \cap (a\Z)^2 \cap A_2 \neq \emptyset$.  But that is a contradiction by the definition of $\wt{X}_a$.  Hence,  there is no such component $P$ and so this completes the proof of the lemma.
\end{proof}

We now deduce the connectivity of $A \setminus \wt{Y}_a$.

\begin{lemma}\label{lem:chain_of_components_disconnecting_origin_from_boundary}
Fix $p \in (0,1)$.  Then there exists $a \in (0,1)$ sufficiently small such that with probability at least $p$ the following holds.  There exist connected components $U_1,\ldots,U_m$ of $A_2 \setminus \wt{Y}_a$ such that $\partial U_{\ell-1} \cap \partial U_\ell \neq \emptyset$ for every $1 \leq \ell \leq m$ (with the convention that $U_0 = U_m$ and $U_1 = U_{m+1}$),  $U_\ell \cap A_0 \neq \emptyset$ and $U_\ell \subseteq A_1$ for every $1 \leq \ell \leq m$, and the following holds.  Suppose that for each $1 \leq j \leq n$ we have points $z_j \in U_j$ and $w_{j-1} \in \partial U_{j-1} \cap \partial U_j$ and set $w_m = w_0$ and let $\gamma_j$ be a path in $\closure{U_j}$ starting and ending at $w_{j-1}$ and $w_j$ respectively and which passes through $z_j$ and does not otherwise hit $\partial U_j$.  If $\gamma$ is the concatenation of the paths $\gamma_1,\dots,\gamma_m$ then $\gamma$ disconnects $\partial B(0,1/2)$ from $\partial \D$.
\end{lemma}

\begin{proof}
Let $a \in (0,1)$ be as in Lemma~\ref{lem:small_components} and suppose that we are working on the event that every connected component of $A_2 \setminus \wt{X}_a$ intersecting $A_0$ is contained in $A_1$.  First we note that $\wt{X}_a \subseteq \cup_{j=1}^n \eta_{z_j}'$ since the flow line with angle $\frac{\pi}{2}$ (resp.\ $-\frac{\pi}{2}$) of $\wt{h}$ starting from $z_j$ is the left (resp.\ right) outer boundary of $\eta_{z_j}'$. Recalling that one may sample $\eta_{z_j}'$ by first sampling $\eta_{z_j}^1$, $\eta_{z_j}^2$ and then independently a counterflow line in each of the connected components of $\D \setminus (\eta_{z_j}^{1} \cup \eta_{z_j}^{2})$ lying between $\eta_{z_j}^{1}$ and $\eta_{z_j}^{2}$ as above, it follows that $\wt{Y}_a \subseteq \cup_{i=1}^n \eta_{z_i}'$.

Next,  we let $U$ (resp.\ $V$) be the connected component of $\D \setminus \cup_{j=1}^n \eta_{z_j}'$ containing $y$ (resp.\ $x$) and let $z^q,w^q$ be points in $U \cap A_0^q,  V \cap A_0^q$ respectively with rational coordinates for $q \in \{L,R\}$.  Then Lemma~\ref{lem:chain_of_components_radial_curves} implies that for the above choice of points,  a.s.\ we can find connected components $U_1,\dots,U_k,V_1,\dots,V_m$ of $\D \setminus \cup_{j=1}^n \eta_{z_j}'$ such that the following hold:  $U = U_1 = V_1$, $V = U_k = V_m$, $U_i \cap A_0^T \neq \emptyset$ for every $1 \leq i \leq k$,  $\partial U_j \cap \partial U_{j+1} \neq \emptyset$ for every $1 \leq j \leq k-1$,  $V_j \cap A_0^B \neq \emptyset$ for every $1 \leq j \leq m$,  and $\partial V_j \cap \partial V_{j+1} \neq \emptyset$ for every $1 \leq j \leq m-1$.  We claim that all of $U_1,\dots,U_k,V_1,\dots,V_m$ are connected components of $A_2 \setminus \wt{Y}_a$ which are contained in $A_1$.  Indeed,  fix $1 \leq j \leq k$.  Since $\wt{X}_a \subseteq \cup_{j=1}^n \eta_{z_j}'$,  it follows that there exists a connected component $P$ of $\D \setminus \wt{X}_a$ such that $U_j \subseteq P$.  Then,  $P \cap A_0 \neq \emptyset$ and so we must have that $P \subseteq A_1$ and hence $U_j \subseteq A_1$.  Also,  $\partial P$ consists of parts of flow lines with angles either $\frac{\pi}{2}$ or $-\frac{\pi}{2}$ starting from points in $(a\Z)^2 \cap A_2$ and so we obtain that $U_j$ is a connected component of the complement in $P$ of a counterflow line of $\wt{h}$ in $P$ starting and ending at the opening and closing point of $P$ respectively (with respect to the ordering induced by the space-filling $\SLE_{\kappa'}(\kappa'-6)$).  It follows that $U_j$ is a connected component of $A_2 \setminus \wt{Y}_a$ for every $1 \leq j \leq k$ and a similar argument works for the $V_j$'s. Moreover,  arguing as in the proof of Lemma~\ref{lem:small_components},  we can assume in addition that $\diam(U_i) \leq \frac{1}{100}$ for every $1 \leq i \leq k$ and $\diam(V_j) \leq \frac{1}{100}$ for every $1 \leq j \leq m$,  possibly by taking $a \in (0,1)$ to be smaller.  Then,  we consider the connected components $U_1,\dots,U_k,V_1,\dots,V_m$.  Let $z_j,w_{j-1},w_j,\gamma_j$ and $\gamma$ be as in the statement of the lemma.  We parameterize $\gamma$ by $[0,1]$ in the counterclockwise way such that $\gamma(0) = \gamma(1) = z_1 \in U_1$ and let $t \in (0,1)$ be such that $\gamma(t) = z_k \in U_k$. Then,  the construction of $\gamma$ implies that $\gamma|_{[0,t]}$ (resp.\ $\gamma|_{[t,1]}$) lies in the $\frac{1}{100}$-neighborhood of $A_0^T$ (resp.\ $A_0^B$) and hence standard topological considerations imply that $\gamma$ disconnects $\partial B(0,1/2)$ from $\partial \D$.  This completes the proof of the lemma.
\end{proof}

\subsubsection{Support of the cut point measure}
\label{subsubsec:cut_point_support}

We will now show that a.s.\ each small ball around each cut point of an $\SLE_{\kappa'}$ process $\eta'$ has positive cut point measure.  This will be used to show that part~\eqref{it:cut_point_measure_local_positive} of Lemma~\ref{lem:graph_connected} holds.

\begin{lemma}
\label{lem:cut_point_measure_local_positive}
Let $\eta'$ be an $\SLE_{\kappa'}$ in $\h$ from $0$ to $\infty$ for $\kappa' \in (4,8)$.  Then a.s.\ we have that $\cutmeasure{\eta'}(B(p,r)) > 0$ for every $r>0$ and every cut point $p$ of $\eta'$.
\end{lemma}

In order to prove Lemma~\ref{lem:cut_point_measure_local_positive}, we first need to collect the the following result which gives that a Bessel process with dimension $\delta \in (0,2)$ accumulates a positive amount of local time in every neighborhood of every time where it hits $0$.
\begin{lemma}\label{lem:local_time_positive}
Let $X$ be a Bessel process of dimension $\delta \in (0,2)$.  Then a.s.\ the following holds.  Fix $\epsilon > 0$ and let $t > 0$ be such that $X_t = 0$.  Then the amount of local time at $0$ that $X$ accumulates in $(t-\epsilon,t+\epsilon)$ is  positive.
\end{lemma}
\begin{proof}
Fix $\epsilon > 0$,  $k \in \N$, and set $\tau_{\epsilon,k} = \inf\{t \in [k\epsilon,(k+1)\epsilon] : X_t = 0\}$.  It suffices to show that on the event that $\tau_{\epsilon,k} \in [k\epsilon,(k+1)\epsilon)$,  the amount of local time at $0$ that $X$ accumulates in $[\tau_{\epsilon,k},(k+1)\epsilon)$ is positive a.s. 
By the strong Markov property,  it suffices to show that the amount of local time at $0$ accumulated by $X|_{[0,s]}$ is positive for every $s>0$ a.s.  This follows because scaling implies that the probability of the event that the $X|_{[0,s]}$ accumulates a positive amount of local time at $0$ does not depend on $s$ and the local time that $X|_{[0,s]}$ accumulates at $0$ a.s.\ tends to $\infty$ as $s \to \infty$.
\end{proof}

We recall that since the left (resp.\ right) boundary $\eta_L$ (resp.\ $\eta_R$) of an $\SLE_{\kappa'}$ process in $\h$ from~$0$ to~$\infty$ is an $\SLE_\kappa(\kappa-4;\tfrac{\kappa}{2}-2)$ (resp.\ $\SLE_\kappa(\tfrac{\kappa}{2}-2;\kappa-4)$) in $\h$ from $0$ to $\infty$ with force points at~$0^-$ and~$0^+$ (see Remark~\ref{rem:sle_kp_divide}), it follows that $\eta_L$ (resp.\ $\eta_R$) a.s.\ does not hit $\R_+$ (resp.\ $\R_-$). This is because the weight of the force point has to be smaller than $\tfrac{\kappa}{2}-2$ in order for an $\SLE_{\kappa}(\rho)$ to hit the interval following it. This in turn implies that a.s., all of the cut points of $\eta'$ lie in $\h$.

\begin{proof}[Proof of Lemma~\ref{lem:cut_point_measure_local_positive}]
Let $\eta_L$ (resp.\ $\eta_R$) be the left (resp.\ right) outer boundary of $\eta'$. Let $h^0$ be a zero-boundary GFF on $\h$ and $(\h,h,0,\infty)$ be a weight $3\gamma^2/2 - 2$ quantum wedge, both independent of~$\eta'$.  By the comment just above, it suffices to prove the result for the measure $\qcutmeasure{h^0}{\eta'}$.  Note also that for every compact set $K \subseteq \h$ (hence $\dist(K,\partial \h) > 0$),  the laws of the restrictions of $h$ and $h^0$ to $K$ are mutually absolutely continuous and so it suffices to prove the claim for the measure~$\qcutmeasure{h}{\eta'}$.

Let $X$ be the Bessel process of dimension $\kappa'/4$ which encodes the weight $2-\gamma^2/2$ wedge parameterized by the regions between $\eta_L$ and $\eta_R$ and is such that the length of every excursion made by $X$ from $0$ is equal to the quantum area of the bead that it encodes.  Fix $0<t<T$ and let $\tau$ (resp.\ $\sigma$) be the time for $\eta'$ which corresponds to the opening (resp.\ closing) point of the bead encoded by the excursion that $X$ makes away from $0$ at time $t$ (resp.\ $T$).  As we remarked just before the proof of the lemma, the distance from $\partial \h$ of the set of cut points of $\eta'$ which are visited by $\eta'$ in the time interval $[\tau,\sigma]$ is positive a.s.  Let $E_\delta$ be the event that each cut point of $\eta'$ which lies in $\eta'([\tau,\sigma])$ has distance at least $\delta$ to $\partial \h$. We fix $p \in (0,1)$ and let $\delta_0>0$ be such that $\p[ E_{\delta_0}] \geq 1-p/2$.  Fix $\delta \in (0,\delta_0)$ small and let $\ol{\eta}'$ be a space-filling $\SLE_{\kappa'}$ in~$\h$ from~$0$ to~$\infty$ so that parameterized by capacity as seen from~$\infty$ is equal to~$\eta'$.  We assume that $\ol{\eta}'$ is parameterized by quantum area with respect to $h$.  Note that $\ol{\eta}'$ is a.s.\ continuous under the above parameterization.  Let $\ol{\tau}$ (resp.\ $\ol{\sigma}$) be the time for $\ol{\eta}'$ which corresponds to $\tau$ (resp.\ $\sigma$).  For $\epsilon \in (0,1)$ we let $\omega(\epsilon) = \sup_{\ol{\tau}\leq s_1 \leq s_2 \leq \ol{\sigma},|s_2-s_1|\leq \epsilon}|\ol{\eta}'(s_2)-\ol{\eta}'(s_1)|$ and note that $\omega(\epsilon) \to 0$ as $\epsilon \to 0$ a.s.  Thus, if we define the event $F_{\epsilon_*}^\delta = \{ \omega(\epsilon) \leq \delta/100 \ \text{for each} \ \epsilon \in (0,\epsilon_*) \}$, then there exists $\epsilon_0 \in (0,1)$ such that $\p[F_{\epsilon_0}^\delta] \geq 1-p/2$.  From now on, we suppose that we are working on $E_{\delta_0} \cap F_{\epsilon_0}^\delta$ and note that $\p[ E_{\delta_0} \cap F_{\epsilon_0}^\delta] \geq 1-p$.

Let $q$ be a cut point of $\eta'$ which is part of a bead encoded by an excursion made by $X$ from $0$ in $(t,T)$ and let $s \in (t,T)$ be the time for $X$ which corresponds to $q$.  Let $\epsilon \in (0,\epsilon_0)$ be small such that $[k\epsilon,(k+1)\epsilon] \subseteq [t,T]$ and $s \in (k\epsilon,(k+1)\epsilon)$ for some $k \in \N$.  Let $a$ (resp.\ $b$) be the first (resp.\ last) time in $[k\epsilon,(k+1)\epsilon]$ that $X$ hits $0$ (by Lemma~\ref{lem:local_time_positive}, the local time at $0$ that $X$ accumulates in any open neighborhood of $s$ is positive a.s.\ and thus $a$ and $b$ are a.s.\ distinct) Let $A$ (resp.\ $B$) be the time for $\eta'$ which corresponds to the cut point associated with the zero of $X$ at $a$ (resp.\ $b$).  Let also $\ol{A}$ (resp.\ $\ol{B}$) be the time for $\ol{\eta}'$ which corresponds to $A$ (resp.\ $B$).  Note that in order for $\ol{\eta}'|_{[\ol{A},\ol{B}]}$ to disconnect from $\infty$ any of the cut points of $\eta'$ hit in the time interval $[A,B]$,  it must first traverse distance at least $\delta_0$.

We shall now show that $\eta'([A,B]) \subseteq B(q,\delta/100)$. This follows if we show that $\ol{B} - \ol{A} \leq \epsilon$, since as we are working on $E_{\delta_0} \cap F_{\epsilon_0}^\delta$, this implies that $\diam(\eta'([A,B])) \leq \diam(\ol{\eta}'([\ol{A},\ol{B}])) \leq \delta/100$ and hence that $\eta'([A,B]) \subseteq B(q,\delta/100)$. Thus, we turn to showing that $\ol{B}-\ol{A} \leq \epsilon$.  We note that $\ol{\eta}'(\ol{A})$ lies between $\eta_L$ and $\eta_R$ and hence $\ol{\eta}'|_{[\ol{A},\ol{B}]}$ can only leave the closure of the union of the components bounded between $\eta_L$ and $\eta_R$ by hitting $\partial \h$.  If $\ol{\eta}'|_{[\ol{A},\ol{B}]}$ does not hit $\partial \h$, it is thus between $\eta_L$ and $\eta_R$, so that $\ol{B} - \ol{A} = b -a \leq \epsilon$ (recall that the quantum area of the components between $\eta_L$, $\eta_R$ encoded by $X|_{[a,b]}$ is equal to $b-a$). Let $\ol{C} = \inf\{ t \geq \ol{A}: \ol{\eta}'(t) \in \partial \h\}$. We assume by that $\ol{C} \leq \ol{B}$ (in order to reach a contradiction).  Note that $\ol{\eta}'([\ol{A},\ol{C}])$ lies between $\eta_L$ and $\eta_R$. Then $\mu_h(\ol{\eta}'([\ol{A},\ol{C}])) = \ol{C}-\ol{A} \leq b-a \leq \epsilon$. Thus, since we are working on the event $E_{\delta_0} \cap F_{\epsilon_0}^\delta$, it follows that $|\ol{\eta}'(\ol{C})- \ol{\eta}'(\ol{A})| \leq \omega(\epsilon) \leq \delta/100$. Since $\dist(\ol{\eta}'(\ol{A}),\partial \h) \geq \delta_0 > \delta/100$, this is a contradiction. As stated above, we thus have that $\eta'([A,B]) \subseteq B(q,\delta/100)$.

Finally,  since $\qcutmeasure{h}{\eta'}(\eta'([A,B]))$ is the amount of local time accumulated by $X$ in the time interval $[a,b]$, this is a.s.\ positive by Lemma~\ref{lem:local_time_positive}. Thus $\qcutmeasure{h}{\eta'}(B(q,\delta/100)) > 0$ a.s.\ as well.  Hence, we have that for every $\delta \in (0,\delta_0 /100)$,  with probability at least $1-p$ we have that $\qcutmeasure{h}{\eta'}(B(q,\delta)) > 0$ for every cut point $q$ of $\eta'$ which corresponds to a bead lying between $\eta_L$ and $\eta_R$,  and said bead is encoded by an excursion made by $X|_{[t,T]}$ away from $0$.  The claim of the lemma then follows since $p \in (0,1), \delta \in (0,\delta_0)$ and $0<t < T<\infty$ were arbitrary.
\end{proof}

\subsubsection{Measure on the intersection of adjacent component boundaries}
\label{subsubsec:measure_on_the_intersection}

Next, we introduce a number of conditions on the intersection of the boundaries of two connected components of the complement of a closed set $K \subseteq \closure{\D}$.  We will  consider distinct connected components $U$ and $V$ of $\D \setminus K$ with $\partial U \cap \partial V \neq 0$ which satisfy the following for some $d>0$ and $\alpha \in (0,1)$.
\begin{enumerate}
	\item[(B1)]\label{it:upper_bound_minkowski_b} The upper Minkowski dimension of $\partial U \cap \partial V$ is at most $d$.
	\item[(B2)]\label{it:holder_continuity_b} For each conformal transformation $\phi: \D \to O$, $O \in \{U,V\}$ mapping $\D$ onto $O$,  there exists $\delta \in (0,1)$ such that if $W = \{z \in \D : \dist(z,\partial U \cap \partial V) < \delta\}$,  then $\phi|_W$ is H\"{o}lder continuous with exponent $\alpha$.
\end{enumerate}
If $U,V$ are two distinct connected components of $\D \setminus K$, then for some $d^* > 0$, we will consider a measure $\mu = \mu^{U,V}$ that satisfies the following.
\begin{enumerate}
	\item[(M1)]\label{it:measure_support_m} $\mu$ is supported on $\partial U \cap \partial V$.
	\item[(M2)]\label{it:diameter_bound_m} There exists a constant $C > 0$ such that for each Borel set $X \subseteq \partial U \cap \partial V$, we have $\mu(X) \leq C \diam(X)^{d^*}$.
	\item[(M3)]\label{it:measure_local_positive_m} $\mu(B(p,r)) > 0$ for each $p \in \partial U \cap \partial V$ and $r > 0$.
\end{enumerate}

\begin{figure}
\begin{center}
\includegraphics[scale=0.85]{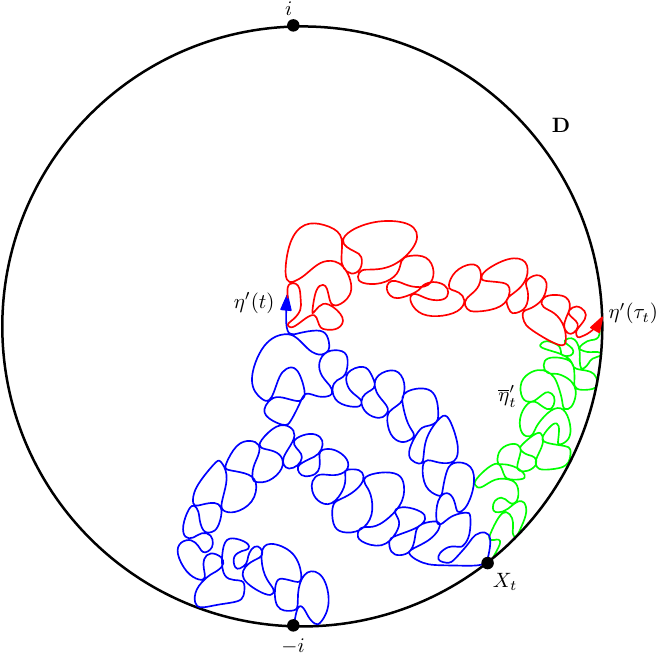}
\end{center}
\caption{\label{fig:eta_etabar_illustration} Illustration of the setup of the proof of Lemma~\ref{lem:good_measures_chordal_case}.  Shown in blue is a chordal $\SLE_{\kappa'}(\kappa'-6)$ process in $\D$ from $-i$ to $i$ up to some time $t$.  The point $X_t$ is the most recent intersection of $\eta'|_{[0,t]}$ with the counterclockwise arc of $\partial \D$ from $-i$ to $i$.  Shown in red is the part of $\eta'$ starting from time $t$ up until $\tau_t$, which is the first time that it disconnects~$X_t$ from $i$.  The law of $\eta'|_{[t,\tau_t]}$ is equal to that of a chordal $\SLE_{\kappa'}$ from $\eta'(t)$ to $X_t$ stopped at the corresponding time; the rest of this chordal $\SLE_{\kappa'}$ is shown in green.  In the proof of Lemma~\ref{lem:good_measures_chordal_case}, the time-reversal of this $\SLE_{\kappa'}$ (going from $X_t$ to $\eta'(t)$) is called $\ol{\eta}_t'$.}
\end{figure}

\begin{lemma}
\label{lem:good_measures_chordal_case}
Fix $\kappa' \in (4,8)$, $b \in (0,d_{\kappa'}^{\cut})$ and let $\eta'$ be a chordal $\SLE_{\kappa'}(\kappa'-6)$ in $\D$ from $-i$ to $i$ with force point located at $(-i)^+$.  Fix points $z,w \in \D \cap \Q^2$ and let $U$ (resp.\ $V$) be the connected component of $\D \setminus \eta'$ containing $z$ (resp.\ $w$).  Then a.s.\ on the event $\{\partial U \cap \partial V \neq \emptyset,  U \neq V\}$,  there exists a measure $\mu_{\eta'}^{z,w}$ satisfying \hyperref[it:measure_support_m]{(M1)}, \hyperref[it:diameter_bound_m]{(M2)}, and \hyperref[it:measure_local_positive_m]{(M3)} with $d^* = d_{\kappa'}^\cut - b$.
\end{lemma}

\begin{proof}
See Figure~\ref{fig:eta_etabar_illustration} for an illustration of the setup of the proof.  For each $t\geq 0$ we let $X_t$ be the place where $\eta'$ last intersected the counterclockwise arc of $\partial \D$ from $-i$ to $i$ before time $t$ and $(K_t)$ be the family of hulls corresponding to $\eta'$. If $\eta'(t) \neq X_t$ and $\tau_t$ is the first time after $t$ that $\eta'$ disconnects $X_t$ from $i$, then the law of $\eta'|_{[t,\tau_t]}$ is that of a chordal $\SLE_{\kappa'}$ from $\eta'(t)$ to $X_t$ in $\D \setminus K_t$, stopped upon disconnecting $X_t$ from $i$. We can then continue the curve past time $\tau_t$, targeting $X_t$ instead. We denote the resulting process by $\eta_t'$ and out of convenience we choose the parameterization so that $\eta_t'(s) = \eta'(s)$ for all $s \leq \tau_t$ and in a fixed (but unspecified) way for $s > \tau_t$. We let $\ol{\eta}_t'$ be the time-reversal of $\eta_t'$, $(\ol{K}_{\ol{t}})_{\ol{t}\geq 0}$ the hulls of $\ol{\eta}_t'$ and define $\zeta_t(\ol{t}) = \inf\{ s > 0: \eta_t'(s) = \ol{\eta}_t'(\ol{t}) \}$.  For every $t,\ol{t}\geq 0$,  the conditional law of $\eta_t'|_{[t,\zeta_t(\ol{t})]}$ given $\eta|_{[0,t]}$ and $\ol{\eta}_t'|_{[0,\ol{t}]}$ is that of a chordal $\SLE_{\kappa'}$ from $\eta'(t)$ to $\ol{\eta}_t'(\ol{t})$ in $\D \setminus (K_t \cup \ol{K}_{\ol{t}})$ by \cite[Theorem~1.1]{ms2016imag3}. Thus, we can define the cut point measure $\mu_{t,\ol{t}}^\cut$ of $\eta_t'|_{[t,\zeta_t(\ol{t})]}$ for all $t, \ol{t} \in \Q_+$ simultaneously off a common set of zero probability. We note that if $\ol{\tau}_t = \inf\{ s \geq 0: \ol{\eta}_t'(s) = \eta'(\tau_t)\}$, then for each $\ol{t} > \ol{\tau}_t$, the measure $\mu_{t,\ol{t}}^\cut$ is supported on $\eta'$. We emphasize that $\mu_{t,\ol{t}}^\cut$ (for $\ol{t} > \ol{\tau}_t$) is a measure on the cut points of the curve $\eta'|_{[t,\zeta_t(\ol{t})]}$ and not on the cut points of $\eta'$. Indeed, the curve $\eta'$ might revisit points $\eta'(s)$ for $s < t$ at a time after $\zeta_t(\ol{t})$.

Fix $z,w \in \D \cap \Q^2$ distinct and let $U$ (resp.\ $V$) be the connected component of $\D \setminus \eta'$ containing~$z$ (resp.\ $w$). 
Let $\tau$ (resp.\ $\sigma$) be the first (resp.\ last) time that $\eta'$ hits a point in $\partial U \cap \partial V$.  Let also $\wt{\tau}$ be the time at which $\eta'$ finishes tracing $\partial U$ and suppose that $\eta'$ finishes tracing $\partial V$ after time $\wt{\tau}$. Note that $\dist(\eta'([\tau , \sigma])  ,  \partial \D)>0$ and $\sigma < \wt{\tau}$ a.s.\ on the event that $\partial U \cap \partial V \neq \emptyset$ and $U \neq V$.
For every $n \in \N$,  we let $T_n$ be the largest number of the form $k 2^{-n}$ for $k \in \N$,  so that $k2^{-n} < \tau$ and $\eta'(k2^{-n}) \notin \partial \D$.  Let also $\ol{T}_n$ be the largest number of the form $\ell 2^{-n}$ for $\ell \in \N$,  so that $\zeta_{T_n} (\ell 2^{-n}) \in (\sigma ,  \wt{\tau})$ and $\dist(\eta'([T_n ,  \zeta_{T_n}(\ell 2^{-n})]) ,  \partial \D) > 0$.  We then set $S_n = \zeta_{T_n}(\ol{T}_n)$ and note that $\ol{T}_n > \ol{\tau}_{T_n}$. Then the measure $\mu_{T_n,\ol{T}_n}^{\cut}$ is well-defined for all $n \in \N$ sufficiently large a.s. We note that a.s.\ on the event that $\partial U \cap \partial V \neq \emptyset$ and $U \neq V$,  we have that $\diam(\partial U \cap \partial V) > 0$ and that there are uncountably many points in $\partial U \cap \partial V$.  Note also that as $n \to \infty$,  we have that $T_n$ increases to $\tau$ and $S_n$ decreases to $\sigma$,  and so the support of $\mu_{T_n,\ol{T}_n}^{\cut}$ contains $\partial U \cap \partial V$ for all sufficiently large $n \in \N$.  We also let $M$ be the smallest $n \in \N$ such that both of $T_n$ and $\ol{T}_n$ are finite.  We note that on $\{\partial U \cap \partial V \neq \emptyset,  U \neq V\}$, we a.s.\ have that $M < \infty$ and that the support of $\mu_{T_M,\ol{T}_M}^{\cut}$ contains $\partial U \cap \partial V$.  Then we let $\mu_{\eta'}^{z,w}$ be the measure $\mu_{T_M,\ol{T}_M}^{\cut}$ restricted to $\partial U \cap \partial V$.  Note that Proposition~\ref{prop:cut_point_measure_bc} implies that \hyperref[it:diameter_bound_m]{(M2)} holds a.s.\ on the event that $\partial U \cap \partial V \neq \emptyset$ and $U \neq V$.

We claim that \hyperref[it:measure_local_positive_m]{(M3)} holds a.s.\ on the event that $\partial U \cap \partial V \neq \emptyset$ and $U \neq V$.  Indeed,  suppose that $\partial U \cap \partial V \neq \emptyset$ and fix $x \in \partial U \cap \partial V$ and $r > 0$.  First,  we note that all of the points of $\partial U \cap \partial V$ are cut points of the curve $\eta'|_{[T_M,S_M]}$ when the latter is considered in the connected component of $\D \setminus (\eta'([0,T_M]) \cup \ol{\eta}_{T_M}'([0,\ol{T}_M]))$ whose boundary contains $\eta'(T_M)$ and $\ol{\eta}_{T_M}'(\ol{T}_M)$.  Suppose that $x \notin \{\eta'(\tau),\eta'(\sigma)\}$.  Then there exist $0<s<r$ such that the cut points of $\eta'|_{[T_M,S_M]}$ contained in $B(x,s)$,  are points of $\partial U \cap \partial V$.  Since the cut point measure gives positive mass to every neighborhood of every cut point a.s.\ (Lemma~\ref{lem:cut_point_measure_local_positive}),  it follows that $\mu_{\eta'}^{z,w}(B(x,r)) \geq \mu_{T_M,\ol{T}_M}^{\cut}(B(x,s)) > 0$.  Suppose that $x \in \{\eta'(\tau),\eta'(\sigma)\}$.  Then there exists $y \in \partial U \cap \partial V,  s>0$ such that $y \notin \{\eta'(\tau),\eta'(\sigma)\}$ and $B(y,s) \subseteq B(x,r)$.  It follows that $\mu_{\eta'}^{z,w}(B(x,r)) \geq \mu_{\eta'}^{z,w}(B(y,s)) > 0$.  This proves the claim.
\end{proof}

\begin{lemma}\label{lem:good_boundaries_maps}
Fix $\kappa' \in (4,8)$ and let $\eta'$ be a chordal $\SLE_{\kappa'}(\kappa'-6)$ in $\D$ from $-i$ to $i$ with the force point located at $(-i)^+$.  Fix points $z,w \in \D \cap \Q^2$ and let $U$ (resp.\ $V$) be the connected component of $\D \setminus \eta'$ containing $z$ (resp.\ $w$).  Then a.s.\ on the event $\{\partial U \cap \partial V \neq \emptyset,  U \neq V\}$,  we have that \hyperref[it:upper_bound_minkowski_b]{(B1)} holds with $d = d_{\kappa'}^{\cut}$ and that \hyperref[it:holder_continuity_b]{(B2)} holds for some constant $\alpha \in (0,1)$ depending only on $\kappa'$.
\end{lemma}
\begin{proof}
Assume that we are on the event $\{\partial U \cap \partial V \neq \emptyset,  U \neq V\}$ and recall the notation from Lemma~\ref{lem:good_measures_chordal_case}. This lemma follows by similar arguments to those of Lemma~\ref{lem:good_measures_chordal_case}. Indeed, since the conditional law of $\eta'$ given $\eta'|_{[0,t]}$ and $\ol{\eta}_t'|_{[0,\ol{t}]}$ is that of a chordal $\SLE_{\kappa'}$ from $\eta'(t)$ to $\ol{\eta}_t'(\ol{t})$ in the connected component of $\D \setminus (\eta'([0,t]) \cup \ol{\eta}_t'([0,\ol{t}]))$ with $\eta'(t)$ and $\ol{\eta}_t'(\ol{t})$ on its boundary, it follows by applying \cite[Theorem~1.2]{mw2017intersections} to all $t, \ol{t} \in \Q_+$ (which includes $T_M, \ol{T}_M$) that the upper Minkowski dimension of the set of cut points of $\eta'|_{[T_M,S_M]}$ is a.s.\ at most $d_{\kappa'}^\cut$. Moreover, since $\partial U \cap \partial V$ is a subset of the cut points of $\eta'|_{[T_M,S_M]}$ we have that~\hyperref[it:upper_bound_minkowski_b]{(B1)} holds with $d = d_{\kappa'}^{\cut}$. Similarly,  \cite[Theorem~5.2]{rs2005basic} implies that \hyperref[it:holder_continuity_b]{(B2)} holds for some constant $\alpha \in (0,1)$ depending only on $\kappa'$.
\end{proof}

\begin{lemma}\label{lem:good_measures_radial_case_up_to_disconnection_time}
Fix $\kappa' \in (4,8)$, $b \in (0,d_{\kappa'}^\cut)$ and let $\eta'$ be a radial $\SLE_{\kappa'}(\kappa'-6)$ in $\D$ from $-i$ to $0$ with the force point located at $(-i)^+$.  Let also $\sigma_1$ be the time at which $\eta'$ disconnects $0$ from $i$.  Fix $z,w \in \D \cap \Q^2$ and let $U$ (resp.\ $V$) be the connected component of $\D \setminus \eta'([0,\sigma_1])$ containing $z$ (resp.\ $w$).  Then a.s.\  on the event $\{\partial U \cap \partial V \neq \emptyset,  U \neq V\}$, there exists a measure $\mu_1^{z,w}$ satisfying \hyperref[it:measure_support_m]{(M1)}, \hyperref[it:diameter_bound_m]{(M2)}, and \hyperref[it:measure_local_positive_m]{(M3)} with $d^* = d_{\kappa'}^\cut - b$.
\end{lemma}

\begin{proof}
Fix $z,w \in \D \cap \Q^2$ and let $U$ (resp.\ $V$) be the connected component of $\D \setminus \eta'([0,\sigma_1])$ containing~$z$ (resp.\ $w$).  Suppose that we are working on the event that $U \neq V$ and $\partial U \cap \partial V \neq \emptyset$.  We extend $\eta'|_{[0,\sigma_1]}$ after time $\sigma_1$ to a curve $\wt{\eta}'$ which terminates at $i$ and such that $\wt{\eta}'$ has the law of a chordal $\SLE_{\kappa'}(\kappa'-6)$ in $\D$ from $-i$ to $i$ with the force point located at $(-i)^+$.  Then the measures $\mu^{\cut}_{t,\ol{t}}$ defined in Lemma~\ref{lem:good_measures_chordal_case} which correspond to $\wt{\eta}'$ are well-defined for all $t,\ol{t} \in \Q_+$ simultaneously a.s.  Let $G$ be the connected component of $\D \setminus \eta'([0,\sigma_1])$ whose boundary contains $i$.  If both of $U$ and $V$ are disjoint from $G$,  then they are both connected components of $\D \setminus \wt{\eta}'$.  In this case,  we set $\mu_1^{z,w}$ to be the measure $\mu_{\wt{\eta}'}^{z,w}$ corresponding to $\wt{\eta}'$.

Suppose now that exactly one of $U$ and $V$ is equal to $G$.  We can assume that $V = G$ and let $\tau$ (resp.\ $\sigma$) be the first (resp.\ last) time that $\wt{\eta}'$ hits a point in $\partial U \cap \partial V$ before time $\sigma_1$.  Suppose first that $U \neq G_1$,  where $G_1$ is the connected component of $\D \setminus \eta'([0,\sigma_1])$ which contains $0$.  Then we have that $\dist(\eta'([\tau,\sigma]) ,  \partial \D) > 0$ a.s.  As in Lemma~\ref{lem:good_measures_chordal_case},  for $n \in \N$,  we let $T_n$ be the largest number of the form $k2^{-n}$ for $k \in \N$  so that $k2^{-n} < \tau$ and $\eta'(k2^{-n}) \notin \partial \D$.  We also let $\ol{T}_n$ be the largest number of the form $\ell 2^{-n}$ for $\ell \in \N$  so that $\zeta_{T_n}(\ell 2^{-n}) \in (\sigma ,  \wt{\tau})$ and $\dist(\eta'([T_n ,  \zeta_{T_n}(\ell 2^{-n})]) ,  \partial \D) >0$,  where $\wt{\tau}$ is the time at which $\eta'$ finishes tracing $\partial U$.  We then set $S_n = \zeta_{T_n}(\ol{T}_n)$ and note that $S_n < \tau_{T_n}$.  It follows that $T_n$ and $\ol{T}_n$ are well-defined for all $n$ sufficiently large a.s.  Then we let $M$ be the smallest $n \in \N$ such that both of $T_n$ and $\ol{T}_n$ are finite.  As in Lemma~\ref{lem:good_measures_chordal_case},  it follows that the points in $\partial U \cap \partial V$ are cut points of the curve $\eta'|_{[T_M ,  S_M]}$ and $\dist(\partial U \cap \partial V ,  \eta'([0,T_M]) \cup \ol{\eta}_M'([0,\ol{T}_M]) \cup \partial \D)>0$ a.s.  Then we let $\mu_1^{z,w}$ be the measure $\mu_{T_M ,  \ol{T}_M}^{\cut}$ restricted to $\partial U \cap \partial V$.  Therefore,  as in Lemma~\ref{lem:good_measures_chordal_case},  we have that the measure $\mu_1^{z,w}$ satisfies the desired properties in that case a.s.

Finally, it remains to treat the case when $U = G_1$.  Then for all $n \in \N$, we let $T_n$ be as before but we let $\ol{T}_n$ be the smallest number of the form $\ell 2^{-n}$ for $\ell \in \N$,  so that $\zeta_{T_n}(\ell 2^{-n}) < \sigma_1$.  For each $n \in \N$, we let $C_n$ be the maximum between $\mu_{T_n, \ol{T}_n}^\cut(\partial U \cap \partial V)$ and the smallest constant $C > 0$ so that~\hyperref[it:diameter_bound_m]{(M2)} holds for the restriction of $\mu_{T_n, \ol{T}_n}^\cut$ to $\partial U \cap \partial V$.  Let 
\[ \mu_1^{z,w} = \sum_{n=1}^\infty \frac{1}{2^n C_n} \mu_{T_n, \ol{T}_n}^\cut.\]
Then $\partial U \cap \partial V$ is contained in the support of $\mu_1^{z,w}$ and it is clear that $\mu_1^{z,w}$ satisfies \hyperref[it:diameter_bound_m]{(M2)} with $d^* = d_{\kappa'}^\cut - b$.  Since $\eta'|_{[T_n,S_n]}$ has cut points which get arbitrarily close to~$\eta'(\sigma_1)$ as $n \to \infty$, it follows that~$\mu_1^{z,w}$ restricted to $\partial U \cap \partial V$ satisfies \hyperref[it:measure_local_positive_m]{(M3)} as well.  This completes the proof of the lemma.
\end{proof}

\begin{lemma}\label{lem:good_measures_radial_case}
Fix $\kappa' \in \adjcon$,  $b \in (0,d_{\kappa'}^{\cut})$ and let $\eta'$ be a radial $\SLE_{\kappa'}(\kappa'-6)$ in $\D$ from $-i$ to $0$ with the force point located at $(-i)^+$.  Then a.s.\  the following holds.  Let $W \subseteq \D$ be a Jordan domain and $z,w \in W \cap \Q^2$.  Let $U$ (resp.\ $V$) be the connected component of $\D \setminus \eta'$ containing $z$ (resp.\ $w$).  Then, there exist connected components $U_1,\dots,U_n$ of $\D \setminus \eta'$ such that $U = U_1$,  $V = U_n$, and $U_i \cap W \neq \emptyset$ for every $1 \leq i \leq n$ and $\partial U_i \cap \partial U_{i+1} \neq \emptyset$ for every $1 \leq i \leq n-1$.  Moreover,  there exist simple paths $\gamma_1,\dots,\gamma_n$ parameterized by $[0,1]$ such that $\gamma_i \subseteq W$ and $\gamma_i((0,1)) \subseteq U_i$ for every $1 \leq i \leq n$,  $\gamma_{i-1}(1) = \gamma_i(0) \in \partial U_{i-1} \cap \partial U_i$ for every $2 \leq i \leq n$,  and $\gamma_1(0) = z,  \gamma_n(1) = w$.  Furthermore,  on the event that $U \neq V$,  we can find measures $\mu_1,\dots,\mu_n$ such that for every $1 \leq j \leq n-1$,  the measure $\mu_j$ is supported on $\partial U_j \cap \partial U_{j+1}$ and satisfies \hyperref[it: diameter_bound_m]{(M2)} with $d^* = d_{\kappa'}^{\cut} - b$,  and $\mu_j(B(\gamma_j(1),s)) > 0$ for every $s>0$.
\end{lemma}

\begin{proof}
The strategy of this proof is as follows.  The existence of the chain of connected components and the simple paths will follow from the proof of Lemma~\ref{lem:radial_connected}.  As for the construction of the measures,  a key observation here is that at the time $\eta'$ first separates $0$ from $i$, we may continue growing the curve in the connected component with $i$ on its boundary, so that the law of the resulting curve is that of a chordal $\SLE_{\kappa'}(\kappa'-6)$ from $-i$ to $i$. Thus, for this new curve, we may use Lemma~\ref{lem:good_measures_chordal_case} to find suitable measures. Moreover, by continuing the original curve in the domain containing $0$, we again have a radial $\SLE_{\kappa'}(\kappa'-6)$. Thus, we can repeat this procedure (after applying a suitable conformal map) to get a family of measures on an increasingly large portion of the (local) cut points of the various parts of $\eta'$. If at some step, neither of the two points lies in the connected component containing $0$, then this means that both $U$ and $V$ are cut out by $\eta'$ at that step and hence we can stop updating the chains of measures for the pair $z,w$, as the measures that were chosen from Lemma~\ref{lem:good_measures_radial_case_up_to_disconnection_time} satisfy the conditions of the lemma. The procedure continues for other pairs $z,w \in \D \cap \Q^2\setminus \{0\}$.

We assume that we have the setup of the proof of Lemma~\ref{lem:radial_connected} and for $z \in \D \cap \Q^2$, we let $U_j^z$ be the connected component of $\D \setminus \eta'([0,\sigma_j])$ containing $z$ and for each $j \in \N$, we let $G_j = U_j^0$.  First,  we will prove the claim of the lemma with $\eta'([0,\sigma_j])$ in place of $\eta'([0,\infty))$ using induction on $j \in \N$ a.s.\  on the event that $U_j^z \neq U_j^w$.

\emph{Step 1. The case $j=1$.}  The case $j=1$ follows by combining Lemma~\ref{lem:good_measures_radial_case_up_to_disconnection_time} with the proof of Lemma~\ref{lem:radial_connected}.

\emph{Step 2. When $j \geq 2$.} Suppose that the claim of the lemma is true $j$ a.s.  Let $U_1^j,\dots,U_{n_j}^j$ be the chain of connected components corresponding to $\D \setminus \eta'([0,\sigma_j])$ such that $U_1^j = U_j^z$ (resp.\ $U_{n_j}^j = U_j^w$) and let $\gamma_1^j,\dots,\gamma_{n_j}^j$ be the simple paths corresponding to $W$.  Let also $\mu_1^j,\dots,\mu_{n_j}^j$ be the corresponding measures.  In particular,  we will show that there exists connected components $U_1^{j+1},\dots,U_{n_{j+1}}^{j+1}$ of $\D \setminus \eta'([0,\sigma_{j+1}])$,  simple paths $\gamma_1^{j+1},\dots,\gamma_{n_{j+1}}^{j+1}$ and measures $\mu_1^{j+1},\dots,\mu_{n_{j+1}}^{j+1}$ satisfying the desired properties.

\emph{Case 1.  $G_j \neq U_m^j$ for every $1 \leq m \leq n_j$}.  Then each connected component $U_m^j$ is also a connected component of $\D \setminus \eta'([0,\sigma_{j+1}])$ and hence the claim follows by setting $n_{j+1} = n_j,  \gamma_{\ell}^{j+1} = \gamma_{\ell}^j,U_{\ell}^{j+1} = U_{\ell}^j$ and $\mu_{\ell}^{j+1} = \mu_{\ell}^j$ for every $1 \leq \ell \leq n_j$.

\emph{Case 2.  $G_j = U_m^j$ for some $1 \leq m \leq n_j$.} We let $1 \leq i_1 < \dots < i_{k_j} \leq n_j$ be such that $\{i_1,\dots,i_{k_j}\} = \{1 \leq i \leq n_j : U_i^j = G_j \}$.  As in the proof of Lemma~\ref{lem:radial_connected},  we have that $\eta_{j+1}'|_{[0,\tau_{j+1}]} = \psi_j(\eta'|_{[\sigma_j,\sigma_{j+1}]})$ has the law of a radial $\SLE_{\kappa'}(\kappa'-6)$ in $\D$ from $-i$ to $0$,  with the force point located at $(-i)^+$ and stopped upon disconnecting $i$ from $0$.  If $W_1^j$ is the connected component of $W \cap G_j$ containing $\gamma_{i_1}^j((0,1))$,  then $W_1^j$ is a Jordan domain and hence,  so is $\wt{W}_1^j = \psi_j(W_1^j)$.  Thus,  if $\wt{U}_1^j$ (resp.\ $\wt{V}_1^j$) is the connected component of $\D \setminus \eta'_{j+1}([0,\tau_{j+1}])$ which has $x_1^j = \psi_j(\gamma_{i_1}^j(0))$ (resp.\ $y_1^j = \psi_j(\gamma_{i_1}^j(1))$) on its boundary,  then picking points $z_1^j \in \wt{U}_1^j$ and $w_1 \in \wt{V}_1^j$ close to $x_1^j$ and $y_1^j$ respectively,  by applying the $j=1$ case to find sequences of connected components $\wt{U}_1^{j,1},\dots,\wt{U}_{n_{j,1}}^{j,1}$ of $\D \setminus \eta'_{j+1}([0,\tau_{j+1}])$ and simple paths $\wt{\gamma}_1^{j,1},\dots,\wt{\gamma}_{n_{j,1}}^{j,1}$ with $\wt{U}_1^j = \wt{U}_1^{j,1}$,  $\wt{V}_1^j = \wt{U}_{n_{j,1}}^{j,1}$,  $\wt{U}_m^{j,1} \cap \wt{W}_1^j \neq \emptyset$, $\wt{\gamma}_m^{j,1}((0,1)) \subseteq \wt{U}_m^{j,1}$, and $\wt{\gamma}_m^{j,1} \subseteq \wt{W}_1^j$ for every $1 \leq m \leq n_{j,1}$,  $\wt{\gamma}_1^{j,1}(0) = z_1^j$, $\wt{\gamma}_{n_{j,1}}^{j,1}(1) = w_1^j$,  $\wt{\gamma}_m^{j,1}(0) = \wt{\gamma}_{m-1}^j(1) \in \partial \wt{U}_m^{j,1} \cap \partial \wt{U}_{m-1}^{j,1}$ for every $2 \leq m \leq n_{j,1}$.  Moreover,  we extend (and reparameterize by $[0,1]$) $\wt{\gamma}_1^{j,1}$ (resp.\ $\wt{\gamma}_{n_{j,1}}^{j,1}$) to start at $x_1^j$ (resp.\ end at $y_1^j$).  Furthermore,  we can find measures $\wt{\mu}_1^j,\dots,\wt{\mu}_{n_{j,1}-1}^j$ such that for every $1 \leq i \leq n_{j,1}-1$,  the measure $\wt{\mu}_i^j$ is supported on $\partial \wt{U}_i^{j,1} \cap \partial \wt{U}_{i+1}^{j,1}$,  it satisfies \hyperref[it:diameter_bound_m]{(M2)} and $\wt{\mu}_i^j(B(\wt{\gamma}_i^{j,1}(1),s)) > 0$ for every $s>0$.  Note that $\dist(\partial \wt{U}_i^{j,1} \cap \partial \wt{U}_{i+1}^{j,1} ,  \partial \D) >0$ for every $1 \leq i \leq n_{j,1}-1$ a.s.\  and so the measures $|\psi_j'|^{d_{\kappa'}^{\cut}} \wt{\mu}_1^{j,1} \circ \psi_j ,  \dots,  |\psi_j'|^{d_{\kappa'}^{\cut}} \wt{\mu}_{n_{j,1}-1}^j \circ \psi_j$ satisfy the desired properties,  i.e.,  for every $1 \leq i \leq n_{j,1}-1$,  the measure $|\psi_j'|^{d_{\kappa'}^{\cut}} \wt{\mu}_i^j \circ \psi_j$ is supported on $\partial \psi_j^{-1}(\wt{U}_i^{j,1}) \cap \partial \psi_j^{-1}(\wt{U}_{i+1}^{j,1})$,  it satisfies \hyperref[it:diameter_bound_m]{(M2)} and $\int_{B(\psi_j^{-1}(\wt{\gamma}_i^{j,1}(1)),s)} |\psi_j'(z)|^{d_{\kappa'}^{\cut}}(\wt{\mu}_i^j \circ \psi_j)(dz) > 0$ for every $s>0$.  Note that the preimage of each such connected component under $\psi_j^{-1}$ is a connected component of $\D \setminus (\eta'([0,\sigma_j]) \cup \psi_j^{-1}(\eta_{j+1}'([0,\tau_{j+1}])))$ and the preimage of each such simple path under $\psi_j^{-1}$ is a simple path in the aforementioned connected component.  Note also that chordal $\SLE_{\kappa'}(\kappa'-6)$ does not hit fixed points a.s.\  and so we can find open intervals $I_1^j$ and $J_1^j$ of $\partial G_j$ such that $I_1^j \subseteq \partial \psi_j^{-1}(\wt{U}_1^j)$ (resp.\ $J_1^j \subseteq \partial \psi_j^{-1}(\wt{V}_1^j)$) and $\gamma_{i_1}^j(0) \in I_1^j$ (resp.\ $\gamma_{i_1}^j(1) \in J_1^j$).  Thus,  there exists $r>0$ sufficiently small such that $B(\gamma_{i_1}^j(0) ,  r) \cap \partial U_{i_1-1}^j \cap \partial U_{i_1}^j \subseteq I_1^j$ (resp.\ $B(\gamma_{i_1}^j(1) ,  r) \cap \partial U_{i_1}^j \cap \partial U_{i_1+1}^j \subseteq J_1^j$) and so the restriction of $\mu_{i_1-1}^j$ (resp.\ $\mu_{i_1}^j$) to $\partial U_{i_1-1}^j \cap \partial \psi_j^{-1}(\wt{U}_{i_1}^j)$ (resp.\ $\partial U_{i_1}^j \cap \partial \psi_j^{-1}(\wt{V}_{i_1}^j)$) is a non-trivial measure.  Repeating this procedure for $i_k,  2 \leq k \leq k_j$,  we obtain for each such $k$ families of connected components,  simple paths and measures as above.

\emph{Step 3.  The claim for $\eta'([0,\infty))$.} It follows from the same argument as in Step 3 of the proof of Lemma~\ref{lem:radial_connected}. This completes the proof of the lemma.
\end{proof}

\begin{lemma}\label{lem:good_boundaries_maps_radial_case}
Fix $\kappa' \in (4,8)$ and let $\eta'$ be a radial $\SLE_{\kappa'}(\kappa'-6)$ in $\D$ from $-i$ to $0$ with the force point located at $(-i)^+$. Fix $z,w \in \D \cap \Q^2 \setminus \{0\}$ and let $U$ (resp.\ $V$) be the connected component of $\D \setminus \eta'$ containing $z$ (resp.\ $w$).  Then a.s.\ on the event $\{\partial U \cap \partial V \neq \emptyset,  U \neq V\}$,  we have that \hyperref[it:upper_bound_minkowski_b]{(B1)} holds with $d = d_{\kappa'}^{\cut}$ and that \hyperref[it:holder_continuity_b]{(B2)} holds for some constant $a \in (0,1)$ depending only on $\kappa'$.
\end{lemma}
\begin{proof}
This follows by the arguments in the proofs of Lemmas~\ref{lem:good_measures_radial_case_up_to_disconnection_time} and~\ref{lem:good_measures_radial_case} combined with Lemma~\ref{lem:good_boundaries_maps}.
\end{proof}

\begin{lemma}\label{lem:good_measures_multiple_radial_case}
Fix $\kappa' \in \adjcon$ and $b \in (0,d_{\kappa'}^{\cut})$,  $n \in \N$ and $z_1,\dots,z_n \in \D$.  For $1 \leq j \leq n$,  we let $\eta_{z_j}'$ be a radial $\SLE_{\kappa'}(\kappa'-6)$ in $\D$ from $-i$ to $z_j$ with the force point located at $(-i)^+$,  and suppose that $\eta_{z_1}',\dots,\eta'_{z_n}$ are coupled so that they are branches of the exploration tree corresponding to a nested $\CLE_{\kappa'}$ in $\D$.  Let also $A_0^q$ be as in Lemma~\ref{lem:chain_of_components_radial_curves} for $q \in \{T,B\}$.  Then for every $1 \leq j \leq n$ and $q \in \{T,B\}$,  the following is true.  Fix $z,w \in A_0^q \cap \Q^2$ and let $U^{j,q}$ (resp.\ $V^{j,q}$) be the connected component of $\D \setminus \cup_{i=1}^j \eta_{z_i}'$ containing $z$ (resp.\ $w$).  Then,  a.s.\  on the event that $U^{j,q} \neq V^{j,q}$ the following holds. There exist connected components $U_1^{j,q},\dots,U_{n_j}^{j,q}$ of $\D \setminus \cup_{i=1}^j \eta_{z_i}'$ such that $U^{j,q} = U_1^{j,q},  V^{j,q} = U_{n_j}^{j,q},  \partial U_{\ell}^{j,q} \cap \partial U_{\ell + 1}^{j,q} \neq \emptyset$ for every $1 \leq \ell \leq n_j-1$,  and $U_{\ell}^{j,q} \cap A_0^q \neq \emptyset$ for every $1 \leq \ell \leq n_j$.  Moreover,  there exist simple paths $\gamma_1^{j,q},\dots,\gamma_{n_j}^{j,q}$ parameterized by $[0,1]$ such that $\gamma_1^{j,q}(0) = z$,  $\gamma_{n_j}^{j,q}(1) = w$, $\gamma_{\ell}^{j,q} \subseteq A_0^q$, and $\gamma_{\ell}^{j,q}((0,1)) \subseteq U_{\ell}^{j,q}$ for every $1 \leq \ell \leq n_j$,  and $\gamma_{\ell-1}^{j,q}(1) = \gamma_{\ell}^{j,q}(0) \in \partial U_{\ell-1}^{j,q} \cap \partial U_{\ell}^{j,q}$ for every $2 \leq \ell \leq n_j$.  Furthermore,  there exist measures $\mu_1^{j,q},\dots,\mu_{n_j}^{j,q}$ such that for every $1 \leq \ell \leq n_j-1$,  the measure $\mu_{\ell}^{j,q}$ is supported on $\partial U_{\ell}^{j,q} \cap \partial U_{\ell+1}^{j,q}$,  it satisfies \hyperref[it:diameter_bound_m]{(M2)} with $d^* = d_{\kappa'}^{\cut}-b$,  and $\mu_{\ell}^{j,q}(B(\gamma_{\ell}^{j,q}(1),s)) > 0$ for every $s>0$.
\end{lemma}

\begin{proof}
The claim of the lemma follows by induction on $1 \leq j \leq n$.  The existence of the chains of connected components and simple paths follows from Lemma~\ref{lem:chain_of_components_radial_curves}.  The existence of the measures with the desired properties follows by combining the arguments in the proofs of Lemmas~\ref{lem:chain_of_components_radial_curves} and~\ref{lem:good_measures_radial_case}.
\end{proof}

\begin{lemma}\label{lem:good_boundary_maps_multiple_radial_case}
Fix $\kappa' \in (4,8)$, $n \in \N$ and $z_1,\dots,z_n \in \D$.  For $1 \leq j \leq n$,  we let $\eta_{z_j}'$ be a radial $\SLE_{\kappa'}(\kappa'-6)$ in $\D$ from $-i$ to $z_j$ and suppose that $\eta_{z_1}',\dots,\eta_{z_n}'$ are coupled such that they are branches of the exploration tree corresponding to a nested $\CLE_{\kappa'}$ in $\D$.  Fix $z,w \in \D \cap \Q^2 \setminus \cup_{i=1}^n \{z_i\}$ and let $U$ (resp.\ $V$) be the connected component of $\D \setminus \cup_{i=1}^n \eta_{z_i}'$ containing $z$ (resp.\ $w$).  Then a.s.\ on the event $\{\partial U \cap \partial V \neq \emptyset,  U \neq V\}$,  we have that \hyperref[it:upper_bound_minkowski_b]{(B1)} holds for $d = d_{\kappa'}^{\cut}$ and that \hyperref[it:holder_continuity_b]{(B2)} holds for some constant $a \in (0,1)$ depending only on $\kappa'$. 
\end{lemma}
\begin{proof}
This follows by combining Lemmas~\ref{lem:good_boundaries_maps} and~\ref{lem:good_boundaries_maps_radial_case} with the arguments in the proof of Lemma~\ref{lem:good_measures_multiple_radial_case}
\end{proof}

\subsubsection{Proof of Lemma~\ref{lem:graph_connected}}
\label{subsubsec:proof_of_graph_connected}
Suppose that we have the setup of the proofs of Lemmas~\ref{lem:chain_of_components_disconnecting_origin_from_boundary},~\ref{lem:good_measures_multiple_radial_case} and~\ref{lem:good_boundary_maps_multiple_radial_case}.  First we note that the event in the statement of the lemma is determined by $h|_{A_1}$.  Also,  the laws of $h|_{A_1}$ and $\wt{h}|_{A_1}$ are mutually absolutely continuous and the corresponding Radon-Nikodym derivatives have finite moments of all orders.  Hence,  it suffices to prove the claim when~$Y_a$ is replaced by~$\wt{Y}_a$.  The latter follows by combining Lemmas~\ref{lem:chain_of_components_disconnecting_origin_from_boundary},~\ref{lem:good_measures_multiple_radial_case} and~\ref{lem:good_boundary_maps_multiple_radial_case} and since the connected components 
$U_1,\dots,U_m$ of $A_2 \setminus \wt{Y}_a$ constructed in the proof of Lemma~\ref{lem:chain_of_components_disconnecting_origin_from_boundary} are also connected components of $\D \setminus \cup_{j=1}^n \eta_{z_j}'$. \qed

\subsection{Completion of the proof}
\label{subsec:completion_of_the_proof}

We now turn to complete the proof of Theorem~\ref{thm:sle_removable} by combining the previous estimates to show that the hypotheses of \cite[Theorem~8.1]{kms2022sle4remov} are satisfied.  For each $z \in \h$ and $k \in \N$ such that $B(z,2^{-k+1}) \subseteq \h$,  we let $A_{z,k} = B(z,2^{-k}) \setminus \closure{B(z,2^{-k-1})}$.  We also let $A_{z,k}^0 = B(z,(3/4) 2^{-k}) \setminus \closure{B(z,(11/16) 2^{-k})}$,  $A_{z,k}^1 = B(z,(7/8) 2^{-k}) \setminus \closure{B(z, (5/8) 2^{-k})}$ and $A_{z,k}^2 = B(z,(15/16)2^{-k}) \setminus \closure{B(z,(9/16)2^{-k})}$ so that $A_{z,k}^0 \subseteq A_{z,k}^1 \subseteq A_{z,k}^2 \subseteq A_{z,k}$.  Suppose that $h$ is a GFF on $\h$ with boundary conditions given by $-\lambda'$ (resp.\ $\lambda'$) on $\R_+$ (resp.\ $\R_-$) so that $h$ is compatible with a coupling with an $\SLE_{\kappa'}$ process in $\h$ from $0$ to $\infty$ as its counterflow line.

Fix $\wt{a}, M > 0$.  Let $X_{z,k}^{\wt{a}}$ be the union of the flow lines of $h$ with angles $\pm \pi/2$ starting from all of the points in $A_{z,k}^2 \cap (\wt{a} 2^{-k} \Z)^2$ and stopped upon exiting $A_{z,k}^2$.  Let $C$ be a connected component of $A_{z,k}^2 \setminus X_{z,k}^{\wt{a}}$ such that $\partial C$ consists of either the left side of a flow line with angle $\frac{\pi}{2}$ and the right side of a flow line of angle $-\frac{\pi}{2}$,  or it consists of the union of the concatenation of two segments of flow lines with angle $\frac{\pi}{2}$ and the concatenation of two segments of flow lines with angle $-\frac{\pi}{2}$.  Then,  in each such connected component,  we start the counterflow line of $h$ from its opening point to its closing point.  We then let $Y_{z,k}^{\wt{a}}$ be the closure of the union of the above counterflow lines with the boundaries of the connected components of $A_{z,k}^2 \setminus X_{z,k}^{\wt{a}}$ whose boundary is entirely contained in $X_{z,k}^{\wt{a}}$.

\begin{definition}
\label{def:good_annulus}
Fix $a,\wt{a},M>0$.  We say that $A_{z,k}$ is $(a,\wt{a},M)$-good if there exists $1 \leq n \leq M$ so that the following is true.  There are connected components $U_1,\ldots,U_n$ of $A_{z,k}^2 \setminus Y_{z,k}^{\wt{a}}$ whose boundary is entirely contained in $Y_{z,k}^{\wt{a}}$ so that $\partial U_i \cap \partial U_{i+1} \neq \emptyset$,  $U_i \cap A_{z,k}^0 \neq \emptyset$ and $U_i \subseteq A_{z,k}^1$ for every $1 \leq i \leq n$ (with the convention $U_{n+1} = U_1$ and $U_0 = U_n$) and the following holds.  Suppose that for each $1 \leq j \leq n$ we have points $z_j \in U_j$ and $w_{j-1} \in \partial U_{j-1} \cap U_j$ and set $w_m = w_0$ and let $\gamma_j$ be a path in $\closure{U_j}$ starting and ending at $w_{j-1}$ and $w_j$ respectively and which passes through $z_j$ and does not otherwise hit $\partial U_j$.  If $\gamma$ is the concatenation of the paths $\gamma_1,\dots,\gamma_n$ then $\gamma$ disconnects $\partial B(z,2^{-k-1})$ from $\partial B(z,2^{-k})$. Moreover,  for every $1 \leq j \leq n$,  there exists a measure $\mu_j$ such that the following hold.
\begin{enumerate}[(i)]
\item \label{it:cut_support} $\mu_j$ is supported on $\partial U_{j-1} \cap \partial U_j$.
\item \label{it:cut_below} $\mu_j(\partial U_{j-1} \cap \partial U_j) \geq M^{-1} 2^{-d_{\kappa'}^{\cut}k}$
\item  \label{it:cut_above} $\mu_j (X) \leq M \diam(X)^{d_{\kappa'}^{\cut}-a}$ for every $X\subseteq \partial U_{j-1} \cap \partial U_j$ Borel.
\item \label{it:good_fraction} We let $\CW_j$ be a Whitney square decomposition of $U_j$.  There exists $Q_j \in \CW_j$ such that the following is true.  Fix $w \in \partial U_{j-1} \cap \partial U_j$ (resp.\ $w \in \partial U_j \cap \partial U_{j+1}$) and let $\gamma_{j,w}$ be the hyperbolic geodesic in $U_j$ from $\cen(Q_j)$ to $w$.  Then for every square $Q \in \CW_j$ that $\gamma_{j,w}$ intersects,  we have that the hyperbolic distance in $U_j$ between $\cen(Q_j)$ and $\cen(Q)$ is at most $M(2^k \len(Q))^{-a}$.
\item \label{it:good_cubes_hit} For every $\ell \in \N_0$,  the number of elements of $\CW_j$ with side length $2^{-\ell}$ which are hit by a hyperbolic geodesic in $U_j$ from $\cen(Q_j)$ to a point in $\partial U_{j-1} \cap \partial U_j$ (resp.\ $\partial U_j \cap \partial U_{j+1}$) is at most $M2^{(d_{\kappa'}^{\cut}+a)(\ell-k)}$.
\end{enumerate}

\end{definition}

Let $E_{z,k}^{a,\wt{a},M}$ be the event that $A_{z,k}$ is $(a,\wt{a},M)$-good.  Let $\CF_{z,k}$ be the $\sigma$-algebra generated by the restriction of $h$ to $\h \setminus B(z,2^{-k})$, and let $\Fh_{z,k}$ be the $\CF_{z,k}$-measurable distribution which is harmonic in $B(z,2^{-k})$ so that we can write $h = h_{z,k} + \Fh_{z,k}$ where $h_{z,k}$ is a GFF in $B(z,2^{-k})$ with zero boundary conditions independent of $\CF_{z,k}$.  For each $N > 0$ we let
\[ G_{z,k}^N = \left\{ \sup_{w \in B(z,(31/32) 2^{-k})} |\Fh_{z,k}(w) - \Fh_{z,k}(z)| \leq N \right\}.\]
Note that $G_{z,k}^N \in \CF_{z,k}$.

\begin{proposition}
\label{prop:good_probability_lbd}
There exists $b \in (0,1)$ such that the following holds. For each $p \in (0,1)$,  $N > 0$ and $a \in (0,b)$ there exist $M > 0$ and $\wt{a} \in (0,1)$ so that for every $z \in \h$ and $k \in \N$ with $B(z,2^{-k+1}) \subseteq \h$ we have
\[ \p[ E_{z,k}^{a,\wt{a},M} \giv \CF_{z,k} ] \one_{G_{z,k}^N} \geq p \one_{G_{z,k}^N}.\]
\end{proposition}
\begin{proof}
Fix $z \in \h$ and $k \in \N$ so that $B(z,2^{-k+1}) \subseteq \h$.  Note that $h_{2^{-k}}(z) = \Fh_{z,k}(z)$.  Let $R = \lfloor h_{2^{-k}}(z)/(2\pi \chi) \rfloor$ and note that $R$ is $\CF_{z,k}$-measurable.  On $G_{z,k}^N$,  the proof of Lemma~\ref{lem:good_scale_radon} implies that the law of the restriction of $h-2\pi \chi R$ to $A_{z,k}^2$ is absolutely continuous with respect to the law of a zero-boundary GFF on $B(z,2^{-k})$ restricted to $A_{z,k}^2$ where the Radon-Nikodym derivative has finite moments of all orders which can each be bounded by a constant that depends only on $N$.  Since the flow lines of $h-2\pi \chi R$ in $A_{z,k}^2$ are the same as the flow lines of $h$ in $A_{z,k}^2$, we obtain that the conditional law given $\CF_{z,k}$ and on $G_{z,k}^N$ of the flow lines of $h$ in $A_{z,k}^2$ stopped at the first time that they exit $A_{z,k}^2$ is absolutely continuous with respect to the law of the flow lines of a zero-boundary GFF on $B(z,2^{-k})$ restricted to $A_{z,k}^2$ where the Radon-Nikodym derivative has finite moments of all orders, each of which can be bounded by a constant that depends only on $N$.  Hence,  since the event $E_{z,k}^{a,\wt{a},M}$ is determined by the restriction of $h$ to $A_{z,k}^2$,  it suffices to prove the claim of the lemma when $E_{z,k}^{a,\wt{a},M}$ is replaced by $\wt{E}_{z,k}^{a,\wt{a},M}$,  where $\wt{E}_{z,k}^{a,\wt{a},M}$ is defined in the same way as $E_{z,k}^{a,\wt{a},M}$ but with the field $h$ replaced by $h_{z,k}$.  Similarly,  we define $\wt{X}_{z,k}^{\wt{a}}$ and $\wt{Y}_{z,k}^{\wt{a}}$,  where $\wt{a} \in (0,1)$ is sufficiently small (to be chosen).

We are going to deduce the result from Lemma~\ref{lem:graph_connected} and the results of Sections~\ref{sec:natural_measure_on_cut_points} and~\ref{sec:good_cube_lemmas}.  Fix $q \in (p,1)$.  We consider the field $\wt{h}_{z,k} = h_{z,k} \circ \phi_{z,k}^{-1}$ where $\phi_{z,k}(w) = 2^k(w-z)$ for every $w \in B(z,2^{-k})$ and consider the sets $X_{\wt{a}}$ and $Y_{\wt{a}}$ as in Section~\ref{subsec:basic_removability_conditions}.  We note that $\wt{h}_{z,k}$ has the law of a zero-boundary GFF on $\D$ and that $\phi_{z,k}(A_{z,k}^j) = A_j$ for every $j=0,1,2$ where $A_0,A_1,A_2$ are defined as in Section~\ref{subsec:basic_removability_conditions}.  Note that $\eta$ is a flow line of $h_{z,k}$ if and only if $\phi_{z,k}(\eta)$ is a flow line of $\wt{h}_{z,k}$ and so we obtain that $X_{\wt{a}} = \phi_{z,k}(\wt{X}_{z,k}^{\wt{a}})$ and $Y_{\wt{a}} = \phi_{z,k}(\wt{Y}_{z,k}^{\wt{a}})$.  Then Lemma~\ref{lem:graph_connected} implies that if we take $b$ sufficiently small (depending only on $\kappa'$) and $a \in (0,b)$,  then we can find $\wt{a} \in (0,1)$ sufficiently small and $C>0$ sufficiently large such that with probability at least $q$ the following holds.  We can find connected components $\wt{U}_1,\dots,\wt{U}_n$ of $A_2 \setminus Y_{\wt{a}}$ and measures $\wt{\mu}_1,\dots,\wt{\mu}_n$ such that~\eqref{it:cut_point_measure_diameter_bound}--\eqref{it:holder_continuity} of Lemma~\ref{lem:graph_connected} hold.  Moreover,  we have that $\wt{U}_j \cap A_0 \neq \emptyset$ and $\wt{U}_j \subseteq A_1$ for every $1 \leq j \leq n$,  and the following holds.  Suppose that for each $1 \leq j \leq n$ we have points $\wt{z}_j \in \wt{U}_j$ and $\wt{w}_{j-1} \in \partial \wt{U}_{j-1} \cap \partial \wt{U}_j$ and set $\wt{w}_{n} = \wt{w}_0$ and let $\wt{\gamma}_j$ be a path in $\closure{\wt{U}_j}$ starting and ending at $\wt{w}_{j-1}$ and $\wt{w}_j$ respectively and which passes through $\wt{z}_j$ and does not hit $\partial \wt{U}_j$ otherwise.  If $\wt{\gamma}$ is the concatenation of the paths $\wt{\gamma}_1,\dots,\wt{\gamma}_n$ then $\wt{\gamma}$ disconnects $\partial B(0,1/2)$ from $\partial \D$.  Note that $\gamma_j = \phi_{z,k}^{-1}(\wt{\gamma}_j)$ is a path in the closure of $U_j = \phi_{z,k}^{-1}(\wt{U}_j)$ starting and ending at $w_{j-1} = \phi_{z,k}^{-1}(\wt{w}_{j-1})$ and $w_j = \phi_{z,k}^{-1}(\wt{w}_j)$ respectively and which passes through $z_j = \phi_{z,k}^{-1}(\wt{z}_j)$ and does not otherwise hit $\partial U_j$.  It follows that the path $\gamma = \phi_{z,k}^{-1}(\wt{\gamma})$ is the concatenation of the paths $\gamma_1,\dots,\gamma_n$ and it disconnects $\partial B(z,2^{-k-1})$ from $\partial B(z,2^{-k})$. Furthermore, by~\eqref{it:cut_point_measure_local_positive} and Lemmas~\ref{lem:good_cubes} and~\ref{lem:not_too_many_large_cubes},  we can find $M>0$ large such that with probability at least $p$,  we have in addition that $n\leq M$,  $\wt{\mu}_j(\partial \wt{U}_{j-1} \cap \partial \wt{U}_j) \geq M^{-1}$, and $\wt{\mu}_j(X) \leq M \diam(X)^{d_{\kappa'}^{\cut}-a}$ for every $X \subseteq \partial \wt{U}_{j-1} \cap \partial \wt{U}_j$ Borel and every $1 \leq j \leq n$,  and the following hold.
\begin{itemize}
\item Fix $1 \leq j \leq n$ and let $\wt{\CW}_j$ be a Whitney square decomposition of $\wt{U}_j$ chosen in a measurable way.  Pick $\wt{z}_j \in \wt{U}_j$ according to Lebesgue measure in $\wt{U}_j$.  Let also $\wt{Q}_j$ be a square in $\wt{\CW}_j$ which contains $\wt{z}_j$ (chosen in some arbitrary but fixed measurable way).  Then we have that $\disthyp^{\wt{U}_j}(\cen(\wt{Q}_j),\cen(\wt{Q})) \leq M\len(\wt{Q})^{-a}$ for every square $\wt{Q}$ in $\wt{\CW}_j$ that $\wt{\gamma}_{j,\wt{w}}$ intersects, where $\wt{\gamma}_{j,\wt{w}}$ is the hyperbolic geodesic in $\wt{U}_j$ from $\cen(\wt{Q}_j)$ to $\wt{w}$.
\item For every $i \in \N_0$,  we let $\wt{\CN}_{i,j}$ be the set of squares $\wt{Q}$ in $\wt{\CW}_j$ such that $\len(\wt{Q}) = 2^{-i}$ and $\wt{\gamma}_{j,\wt{w}} \cap \wt{Q} \neq \emptyset$ for some $\wt{w} \in (\partial \wt{U}_{j-1} \cap \partial \wt{U}_j) \cup (\partial \wt{U}_j \cap \partial \wt{U}_{j+1})$.  Then we have that $|\wt{\CN}_{i,j}| \leq M 2^{(d_{\kappa'}^{\cut}+a)i}$.
\end{itemize}
Suppose that we are working on the event that the above hold.  Then for every $1 \leq j \leq n$, we consider the collection of squares $\CW_j$ defined by the relation $Q \in \CW_j$ if and only if $Q = \phi_{z,k}^{-1}(\wt{Q})$ for some $\wt{Q} \in \wt{\CW}_j$, and we consider the measure $\mu_j$ defined by $\mu_j(A) = 2^{-d_{\kappa'}^{\cut}k} \wt{\mu}_j(\phi_{z,k}(A))$ for every $A \subseteq B(z,2^{-k})$ Borel.  It is clear that the measures $(\mu_j)_j$ satisfy~\eqref{it:cut_support},~\eqref{it:cut_below}, and~\eqref{it:cut_above}.  Note that $\gamma_{j,w} = \phi_{z,k}^{-1}(\wt{\gamma}_{j,\wt{w}})$ for every $1 \leq j \leq n$ and every $\wt{w} \in \partial \wt{U}_j$,  where $w = \phi_{z,k}^{-1}(\wt{w})$.  Therefore, we have that both of~\eqref{it:good_fraction} and~\eqref{it:good_cubes_hit} hold as well since $\CW_j$ is a Whitney square decomposition of~$U_j$ and $\len(\phi_{z,k}^{-1}(\wt{Q})) = 2^{-k} \len(\wt{Q})$ for every $\wt{Q} \in \wt{\CW}_j$,  where~$Q_j$ is a square in~$\CW_j$ which contains~$z_j = \phi_{z,k}^{-1}(\wt{z}_j)$.  This completes the proof of the proposition.
\end{proof}

\begin{proof}[Proof of Theorem~\ref{thm:sle_removable}]
It suffices to show that the conditions of  \cite[Theorem~8.1]{kms2022sle4remov} are satisfied for $\eta'$ a.s.  Fix $K \subseteq \h$ compact and such that $\dist(K,\partial \h) > 0$.  Let $b>0$ be as in Proposition~\ref{prop:good_probability_lbd}.  Note that the upper Minkowski dimension of $\eta' \cap K$ is at most $d_{\kappa'} = 1+ \frac{\kappa'}{8}$ and so following  \cite[Theorem~8.1]{kms2022sle4remov},  we fix $a \in (0,1)$ small enough such that $a< \min(b, (2-d_{\kappa'})/5)$ and $d_{\kappa'}^\cut \in (10a ,  2 - 10a)$.  For each $k \in \N$,  we set $\CA_k = \{B(z,2^{-k}) \setminus \closure{B(z,2^{-k-1})},  z \in \h\}$ and we let $\CD_k$ be the set of $z \in 2^{-k} \Z^2$ with $\dist(z,K) \leq 2^{-k}$.  Let $k_0 \in \N$ be such that $2^{-(1-a^2) k_0} \leq \dist(K, \partial \h)/2$.  The proof of Lemma~\ref{lem:good_dense} implies that we can choose $N$ sufficiently large so that if $J_k = \{ (1-a^2)k \leq j \leq k: G_{z,j}^N \ \text{occurs}\}$, then for every $k \geq k_0$ and $z \in \CD_k$ we have $\p[ |J_k| \geq a^2 k /2] = 1-O(2^{-3k})$.  Moreover,  combining with Proposition~\ref{prop:good_probability_lbd},  we obtain that we can choose $\wt{a} \in (0,1)$ sufficiently small and $M>0$ sufficiently large such that the probability that there does not exist $(1-a^2)k\leq j \leq k-3$ such that $E_{z,j}^{a,\wt{a},M}$ occurs is $O(2^{-3k})$ for every $k\geq k_0$ uniformly in $z \in \CD_k$. Thus by applying a union bound and the Borel-Cantelli lemma, we see that there a.s.\ exists $n_0 \in \N$ with $n_0 \geq k_0$ so that if $n \geq n_0$ then for every $z \in \CD_n$ there exists  $(1-a^2) n \leq k \leq n-3$ so that $E_{z,k}^{a,\wt{a},M}$ occurs.  Suppose that we are working on that event.  Fix $z \in \eta' \cap K$ and let $w \in \CD_n$ be such that $|z-w| \leq 2^{-n}$ for $n \geq n_0$.  Let $(1-a^2)n \leq k \leq n-3$ be such that $E_{z,k}^{a,\wt{a},M}$ occurs and set $A = B(w,2^{-k}) \setminus \closure{B(w,2^{-k-1})} \in \CA_k$.  Then $B(z,2^{-n})$ is contained in the bounded connected component of $\C \setminus A$.  Let $U_1,\dots,U_m$ be the connected components in the definition of $E_{z,k}^{a,\wt{a},M}$.  It is immediate from the definition of $E_{z,k}^{a,\wt{a},M}$ that all the conditions of \cite[Theorem~8.1]{kms2022sle4remov} --- except possibly that $U_1,\dots,U_m$ are subsets of $\h \setminus \eta'$ --- are satisfied. Therefore,  in order to complete the proof of the theorem,  we need to show that the connected components $U_1,\dots,U_m$ are contained in $\h \setminus \eta'$.  To prove this,  we argue as follows.  Let $U$ be a connected component of $A_{z,k}^2 \setminus Y_{z,k}^{\wt{a}}$ such that $U \subseteq A_{z,k}^1$.  Then there exists a unique connected component $C$ of $A_{z,k}^2 \setminus X_{z,k}^{\wt{a}}$ such that $U \subseteq C$.  Then,  there are three possibilities.
\begin{enumerate}[(i)]
\item $\partial C$ consists of part of the right side of a flow line with angle $\frac{\pi}{2}$ and the left side of a flow line with angle $-\frac{\pi}{2}$.  Then,   $\eta'$ cannot enter $U$ due to the flow lines interaction rules \cite[Theorem~1.7]{ms2017ig4}.
\item $\partial C$ consists of the left side of a flow line with angle $\frac{\pi}{2}$ and the right side of a flow line with angle $-\frac{\pi}{2}$.  Then,  by the definition of $Y_{z,k}^{\wt{a}}$,  we have added the counterflow line of $h$ in $C$ from its opening to its closing point.  Then $U$ must be a connected component of $C \setminus \eta_C'$,  where $\eta_C'$ is the counterflow line of $h$ in $C$.  Let $w$ be the closing point of $P$ and let $\eta_w'$ be the counterflow line targeted at $w$.  Since $\eta_C'$ is part of $\eta_w'$ and the counterflow line from $0$ to $\infty$ cannot enter into any connected component disconnected from $\infty$ by $\eta_w'$,  it follows that it cannot enter $U$ in that case.
\item $\partial C$ consists of four arcs of flow lines.  Since the counterflow line visits points according to the space-filling ordering,  the only place that it can enter $C$ is at its opening point.  Once the counterflow line enters $C$,  it has to respect the space-filling ordering and therefore it has to agree with the counterflow line from the opening to the closing point of $C$.  Therefore $\eta'$ cannot enter $U$.
\end{enumerate}
This completes the proof of the theorem.
\end{proof}

\appendix

\section{General estimates}

In this appendix we collect a few estimates needed in Sections~\ref{sec:natural_measure_on_cut_points} and~\ref{sec:pocket_argument}. Appendix~\ref{app:deterministic} will focus on two deterministic estimates, Appendix~\ref{app:flow_lines} will prove an estimate for GFF flow lines, and Appendix~\ref{app:wedges} will prove several estimates for quantum wedges.

\subsection{Deterministic estimates}
\label{app:deterministic}

We begin by bounding the harmonic conjugate of a bounded harmonic function on $\D$.

\begin{lemma}
\label{lem:bound_harmonic_conjugate}
Let $u$ be harmonic on $\D$ with $u(0) = 0$ and assume that there exists a constant $c > 0$ such that $|u(z)| \leq c$ for all $z \in \D$. Let $v$ be the harmonic conjugate of $u$ with $v(0) = 0$. Then, for each $r \in (0,1)$ we have that $|v(z)| \leq 2 c/(1-r)$ for all $z \in \closure{B(0,r)}$.
\end{lemma}
\begin{proof}
Note that
\begin{align*}
	u(z) + iv(z) = \int_{-\pi}^\pi u(e^{i\theta}) \frac{e^{i\theta}+z}{e^{i\theta}-z} \frac{d\theta}{2\pi}.
\end{align*}
If $|z| \leq r$ then $|e^{i\theta} + z| / |e^{i\theta} - z| \leq 2/(1-r)$ and we thus we have that
\begin{align*}
	|v(z)| \leq \int_{-\pi}^\pi |u(e^{i\theta})| \frac{|e^{i\theta}+z|}{|e^{i\theta}-z|} \frac{d\theta}{2\pi} \leq \frac{2 c}{1-r}.
\end{align*}
\end{proof}

Next, we prove the following lemma, stating that we can choose a family of annuli separating any $n$ distinct points in a compact set $K$ in a controlled manner.  This estimate is key in bounding the moments of~$\cutmeasure{\eta'}$ for $\eta' \sim \SLE_{\kappa'}$.
\begin{lemma}
\label{lem:annuli_algorithm}
Fix $n \in \N$, $r_0>0$ and $K \subseteq \h$ compact.  There exists a constant $C=C(n,K,r_0)$ so that the following is true.  Suppose that $z_1,\ldots,z_n \in K$ are distinct.  For each $1 \leq j \leq n$ there exists $n_j \in \N$ and $s_{j,k},r_{j,k} \in (0,r_0)$ such that $0 < s_{j,k} < r_{j,k}$ for $1 \leq k \leq n_j-1$ and $0=s_{j,n_j} < r_{j,n_j}$ if $n_j \geq 2$ and $0=s_{j,1}<r_{j,1}$ if $n_j=1$ so that the annuli $B(z_j,r_{j,k}) \setminus B(z_j,s_{j,k})$ are pairwise disjoint,  $B(z_j,4r_{j,1}) \subseteq \h$ for all $1 \leq j \leq n$ and
\begin{align*}
\prod_{j=1}^n \frac{\prod_{k=1}^{n_j-1} s_{j,k}}{ \prod_{k=1}^{n_j} r_{j,k}} \leq C \prod_{j=1}^n \frac{1}{\min_{i<j} |z_j-z_i|}
\end{align*}
with the convention $\prod_{k=1}^{n_j-1}s_{j,k} = 1$ if $n_j=1$.
\end{lemma}

\begin{remark}
Note that the order of the points is important in the above inequality. Indeed, if the minimum had been over $i \neq j$, then the result would trivially hold true by choosing $n_j = 1$ and $r_{j,1} = \tfrac{1}{2}(\tfrac{\dist(K,\partial \h)}{4} \wedge \min_{i \neq j} |z_j - z_i|)$ for all $j$. However, since the right-hand side of the inequality has the minimum over $i < j$, a pattern of annuli being disjoint appears when points with larger indices are closer together than some with smaller indices.
\end{remark}

\begin{proof}[Proof of Lemma~\ref{lem:annuli_algorithm}]
Let $r_{1,1}$ be the largest number $r > 0$ of the form $2^{-m}$ for $m \in \N$ so that $r\leq r_0$,  $B(z_1,4r) \subseteq \h$, and there are no points $z_j$ for $2 \leq j \leq n$ in $B(z_1,4r) \setminus B(z_1,r/4)$.  By the pigeon hole principle, we note that $r_{1,1} \asymp \dist(z_1,\partial \h)$ where the implicit constants depend only on $n$,  $r_0$, and $K$.  There are two possibilities to consider.
\begin{itemize}
\item $B(z_1,r_{1,1})$ contains one of the points $z_j$ for $2 \leq j \leq n$.  We let $\wt{s}_{1,1}$ be the smallest number $r \in (0,r_{1,1})$ of the form $2^{-m}$ for $m \in \N$ so that $B(z_1,r_{1,1}) \setminus B(z_1,r)$ does not contain any of the points $z_j$ for $2 \leq j \leq n$ and then we let $s_{1,1} = 2 \wt{s}_{1,1}$.
\item $B(z_1,r_{1,1})$ does not contain one of the points $z_j$ for $2 \leq j \leq n$.  In this case, we set $s_{1,1} = \wt{s}_{1,1} = 0$.
\end{itemize}

If $s_{1,1} > 0$, we then let $r_{1,2}$ be the largest number $r \in (0,s_{1,1})$ of the form $2^{-m}$ for $m \in \N$ so that there are no points $z_j$ for $2 \leq j \leq n$ in $B(z_1,4r) \setminus B(z_1,r/4)$.  We then define $s_{1,2}$ in the same manner as above and continue to define $r_{1,k}$, $s_{1,k}$ until the first $k$ that $s_{1,k} = 0$.  We then set $n_1=k$.  Assume that for some $1 \leq j \leq n-2$, we have defined radii $r_{\ell,k},s_{\ell,k}$ for each pair $(\ell,k)$ such that $1 \leq \ell \leq j$ and $1 \leq k \leq n_\ell$.  We now define the radii $r_{j+1,k},s_{j+1,k}$ as follows.  We let $r_{j+1,1}$ be the largest number $r > 0$ of the form $2^{-m}$ for $m \in \N$ so that $r \leq r_0$,  $B(z_{j+1},4r) \subseteq \h$,  $B(z_{j+1},4r)$ is disjoint from the annuli $B(z_\ell,r_{\ell,k}) \setminus B(z_\ell,s_{\ell,k})$ for $1 \leq \ell \leq j$ and $1\leq k \leq n_{\ell}$ and there are no points $z_\ell$ for $j+2 \leq \ell \leq n$ in $B(z_{j+1},4r) \setminus B(z_{j+1},r/4)$. As before, there are two possibilities to consider.
\begin{itemize}
\item $B(z_{j+1},r_{j+1,1})$ contains one of the points $z_\ell$ for $j+2 \leq \ell \leq n$.  We let $\wt{s}_{j+1,1}$ be the smallest number $r \in (0,r_{j+1,1})$ of the form $2^{-m}$ for $m \in \N$ so that $B(z_{j+1},r_{j+1,1}) \setminus B(z_{j+1},r)$ does not contain any of the points $z_\ell$ for $j+2 \leq \ell \leq n$ and then we let $s_{j+1,1} = 2 \wt{s}_{j+1,1}$.
\item $B(z_{j+1},r_{j+1,1})$ does not contain one of the points $z_\ell$ for $j+2 \leq \ell \leq n$.  In this case, we set $s_{j+1,1} = \wt{s}_{j+1,1} = 0$.
\end{itemize}

If $s_{j+1,1} > 0$, we then let $r_{j+1,2}$ be the largest number $r \in (0,s_{j+1,1})$ of the form $2^{-m}$ for $m \in \N$ so that there are no points $z_\ell$ for $j+2 \leq \ell \leq n$ in $B(z_{j+1},4r) \setminus B(z_{j+1},r/4)$.  We then define $s_{j+1,2}$ in the same manner as above and continue defining $r_{j+1,k}$, $s_{j+1,k}$ until the first $k$ so that $s_{j+1,k} = 0$ and set $n_{j+1}=k$. Finally, assuming that $r_{j,k}$ and $s_{j,k}$ have been defined for $1 \leq j \leq n-1$ and $1 \leq k \leq n_j$ we let $r_{n,1}$ be the largest number $r$ of the form $2^{-m}$ for $m \in \N$ such that $r \leq r_0$, $B(z_n, 4r) \subseteq \h$ and $B(z_n,4r)$ is disjoint from $B(z_\ell,r_{\ell,k}) \setminus B(z_\ell,s_{\ell,k})$ for all $1 \leq \ell \leq n-1$ and $1 \leq k \leq n_\ell$.  Then, since $B(z_n,r_{n,1})$ does not contain $z_\ell$ for any $1 \leq \ell \leq n-1$, we let $s_{n,1} = \wt{s}_{n,1} = 0$.

Next we claim that $s_{j,k}/r_{j,k+1} \lesssim 1$ for every $1 \leq j \leq n$ and every $1 \leq k \leq n_j-1$ such that $n_j \geq 2$,  where the implicit constant depends only on $n$.  Indeed,  fix $j,k$ as above and let $m_0 \in \N$ be such that $s_{j,k} = 2^{-m_0}$.  Set $m_{\ell} = m_0 + 4\ell$ for $\ell \in \N_0$.  Then the annuli $B(z_j,4 \cdot 2^{-m_{\ell}}) \setminus B(z_j ,  2^{-m_{\ell}}/4)$ for $\ell = 0,\dots,n-1$ are pairwise disjoint and so there exists $0 \leq \ell \leq n-1$ such that $B(z_j ,  4\cdot 2^{-m_{\ell}}) \setminus B(z_j  ,2^{-m_{\ell}}/4)$ does not contain any of the points $z_{b}$ for $1 \leq b \leq n$.  Therefore in every case we have that $r_{j,k+1} \geq 2^{-m_\ell}\geq s_{j,k} 2^{-4(n-1)}$ and so this proves the claim.  

Next, we shall prove that $r_{j,1} \gtrsim \min_{i < j} |z_j - z_i|$ for $1 \leq j \leq n$ where the implicit constant depends only on $n$, $r_0$, and $K$,  which together with the above implies the result of the lemma.  Indeed,  first we note that we have already shown the claim when $j=1$.  Suppose that $2\leq j \leq n$. We first consider the case that $z_j \notin B(z_\ell,r_{\ell,1})$ for all $1 \leq \ell < j$ and note that then $z_j \notin B(z_\ell,4 r_{\ell,1})$ and hence $r_{\ell,1} < |z_j - z_\ell|/4$ for all $1 \leq \ell < j$. Set $\wh{r}_j = \min_{i < j} |z_j - z_i|/8$ and note that then $B(z_j, 4 \wh{r}_j) \cap B(z_\ell, r_{\ell,1}) = \emptyset$ for all $1 \leq \ell < j$. Moreover, writing 
\begin{align*}
	\ol{r}_j = \wh{r}_j \times \frac{r_0 \wedge 1}{\diam(K) \vee 1} \times \frac{\dist(K,\partial \h) \wedge 1}{4},
\end{align*}
letting $m_{j,0} \in \N$ be such that $\wt{r}_j \coloneqq 2^{-m_{j,0}} \in (\ol{r}_j/2,\ol{r}_j]$, it follows that $B(z_j,4 \wt{r}_j) \cap B(z_\ell,r_{\ell,1}) = \emptyset$ for all $1 \leq \ell < j$, $\wt{r}_j \leq r_0$, $B(z_j, 4 \wt{r}_j) \subseteq \h$ -- and these continue to hold if we replace $\wt{r}_j$ by something smaller -- and that $\wt{r}_j \geq c \min_{i < j} | z_j - z_i|$ for some constant $c = c(r_0,K)$. Finally, letting $m_{j,\ell} = m_{j,0} + 4 \ell$ we have similarly to the above that there is at least one $0 \leq \ell \leq n-j$ such that the annulus $B(z_j,4 \cdot 2^{-m_{j,\ell}}) \setminus B(z_j, 2^{- m_{j,\ell}}/4)$, does not contain any of the points $z_{j+1},\dots,z_n$. Let $\ell^*$ be the largest such $\ell$. Then $r_{j,1} \geq 2^{-m_{j,\ell^*}} \geq c 2^{-4(n-j)} \min_{i < j} | z_j - z_i|$, as was to be shown. In the case of $r_{n,1}$, note that $r_{n,1} \geq \wt{r}_n$.

Next, we consider the case where there is some $1 \leq \ell < j$ such that $z_j \in B(z_\ell,r_{\ell,1})$.  Then there exist $1 \leq \ell' < j$ and $1 \leq k' \leq n_{\ell'}-1$ such that $z_j \in B(z_{\ell'},s_{\ell',k'}) \setminus B(z_{\ell'},r_{\ell',k'+1})$ and we let $\ell$ and $k$ be chosen so that $B(z_\ell,s_{\ell,k}) \setminus B(z_\ell,r_{\ell,k+1})$ is the smallest such annulus. We have by the choice of $r_{\ell,k+1}$ that $z_j \notin B(z_\ell,4 r_{\ell,k+1})$ so that $B(z_j,r_{\ell,k+1}) \cap B(z_\ell,r_{\ell,k+1}) = \emptyset$.  Moreover, by the choice of $s_{\ell,k}$, $z_j \in B(z_\ell,s_{\ell,k}/2)$ so that $B(z_j,s_{\ell,k}/2) \subseteq B(z_\ell,s_{\ell,k})$. Since $r_{\ell,k+1} \leq s_{\ell,k}/2$,  it follows that $B(z_j,r_{\ell,k+1}) \subseteq B(z_\ell, s_{\ell,k}) \setminus B(z_\ell,r_{\ell,k+1})$ and that $B(z_j,r_{\ell,k+1})$ is disjoint of $B(z_{\ell'},r_{\ell',k'}) \setminus B(z_{\ell'},s_{\ell',k'})$ for all $1 \leq \ell' < j$ and $1 \leq k' \leq n_{\ell'}$. It follows by the same argument as above that $r_{j,1} \geq 2^{-4(n-j)} r_{\ell,k+1}/4$.  Moreover, since $s_{\ell,k} > |z_j - z_\ell |$ and $s_{\ell,k} / r_{\ell,k+1} \lesssim1$, it follows that $r_{j,1} \gtrsim \min_{i < j} |z_j-z_i|$ where the implicit constant depends only on $n$, $K$ and $r_0$.

Putting together the results of the above three paragraphs, it follows that
\begin{align*}
	\frac{\prod_{k=1}^{n_j-1} s_{j,k}}{\prod_{k=1}^{n_j} r_{j,k}} = \frac{1}{r_{j,1}} \prod_{k=1}^{n_j-1} \frac{s_{j,k}}{r_{j,k+1}} \leq \wt{c}^{n_j}\frac{1}{r_{j,1}} \leq \wh{c} \frac{\wt{c}^{n_j}}{\min_{i<j} |z_j-z_i|}
\end{align*}
for some constants $\wt{c} = \wt{c}(n,K,r_0)>0$ and $\wh{c} = \wh{c}(n,K,r_0)$. Therefore,  to complete the proof of the lemma,  it suffices to show that $n_j\leq n$ for every $1 \leq j \leq n$.  But this can be easily seen since for every $1 \leq j \leq n$ such that $n_j\geq 2$ and every $1\leq k \leq n_j-1$,  the definitions of $s_{j,k}$ and $r_{j,k+1}$ imply that there exists $1 \leq \ell \leq n$ such that $z_{\ell} \in B(z_j,s_{j,k}) \setminus B(z_j,r_{j,k+1})$ and the annuli $B(z_j,s_{j,k}) \setminus B(z_j,r_{j,k+1})$ are pairwise disjoint.
\end{proof}

\subsection{Exit points of flow lines}
\label{app:flow_lines}

We now prove that the flow lines started at interior points of a set leaves said set at a finite number of boundary points.

\begin{lemma}
\label{lem:finite_strands}
Let $h$ be a zero-boundary GFF in $\D$. Fix $0 \leq r<r_0<R<1$ and let $A = B(0,R) \setminus B(0,r)$. Let $(\theta_j)_{j \in \N}$ be a countable dense subset of $[0,2\pi)$, $\eta_1^j$ (resp.\ $\eta_2^j$) be the flow line of $h$ of angle $-\tfrac{\pi}{2}$ (resp.\ $\tfrac{\pi}{2}$) that emanates from $z_j = r_0 e^{i \theta_j}$ and $\tau_{k}^j = \inf\{ t \geq 0: \eta_{k}^j(t) \notin \closure{A} \}$ for $k \in \{1,2\}$. Let $K = \closure{\cup_j ( \eta_{1}^j([0,\tau_{1}]) \cup \eta_{2}^j([0,\tau_{2}^j]))}$. Then $K \cap \partial \closure{A}$ is a.s.\ a finite set. Moreover, the same result holds if we replace $\eta_{k}^j$ for $k \in \{1,2\}$ by the flow lines started from any countable set of points in $B(0,R_1) \setminus B(0,r_1)$ for any $r < r_1 < R_1 < R$.
\end{lemma}
\begin{proof}
By absolute continuity, it suffices to prove the lemma with any fixed choice of boundary conditions for the GFF $h$ on $\D$.  Suppose that $\wt{h}$ is a GFF on $\h$ with boundary conditions given by $\lambda'$ on $\R_-$ and $\lambda'-2\pi \chi$ on $\R_+$.  Let $\varphi \colon \D \to \h$ be a conformal map so that $\varphi(-i) = 0$, $\varphi(1) = 1$, and $\varphi(i) = \infty$.  We then take $h = \wt{h} \circ \varphi - \chi \arg \varphi'$.  Let $\ol{\eta}'$ be the space-filling $\SLE_{\kappa'}$ in $\D$ from $-i$ to $-i$ coupled with $h$ in the usual way.  For each $z \in \D$, we let $\ol{\tau}_z = \inf \{ t \geq 0: z \in \ol{\eta}'([0,t]) \}$ (note that $\ol{\eta}'$ does hit every point $z$ in $\D$). We note that if we run $\eta_{1}^j$ and $\eta_{2}^j$ until hitting $-i$, then they will form the outer boundary of $\ol{\eta}'([0,\ol{\tau}_{z_j}])$. It follows that $\eta_{k}^j$ for $k \in \{1,2\}$ can only exit $A$ at points in which $\ol{\eta}'$ starts an excursion from $\partial A$ to $\partial B(0,r_0)$. Moreover, by the continuity of space-filling $\SLE_{\kappa'}$ \cite[Theorem~4.12]{ms2017ig4}, it follows that $\ol{\eta}'$ can only make finitely many such excursions. Thus, the result follows in the first case. The proof of the second case is the same.
\end{proof}

\subsection{Wedge estimates}
\label{app:wedges}

This section is dedicated to the proof that we can decompose a weight $3 \gamma^2/2 - 2$ quantum wedge into the sum of a zero-boundary GFF and a harmonic function (provided the embedding of the wedge is chosen in a certain way) and the upper tail of said harmonic function in each compact subset of $\h$ has Gaussian decay. We begin by the proof that we can have such a decomposition.

\begin{lemma}\label{lem:wedge_harmonic_expression}
Let $\CW = (\h,h,0,\infty)$ be a quantum wedge of weight $3\gamma^2/2 - 2$ with the embedding into $\h$ such that $\qbmeasure{h}([-1,0]) = 1$.  Then we can write $h = h^0 + \Fh$ where $h^0$ is a zero-boundary GFF on $\h$ and $\Fh$ is harmonic on $\h$ and independent of $h^0$.
\end{lemma}
\begin{proof}
First we assume that $\CW$ has the circle average embedding.  Recall that $h|_{\D \cap \h} = h^f - \alpha \log| \cdot |$, where $h^f$ is a free boundary GFF and $\alpha = 4/\gamma - \gamma/2$. It follows that we can write $h|_{\D \cap \h} = h^0+\Fh$ where $h^0$ is a zero-boundary GFF on $\D \cap \h$ and $\Fh$ is harmonic on $\D \cap \h $ and independent of $h^0$.  Fix $\epsilon \in (0,1)$ and let $x_{\epsilon}>0$ be such that $\qbmeasure{h}([-x_{\epsilon},0]) = \epsilon$.  Note that $h_\epsilon = h(x_{\epsilon}\cdot) + Q \log(x_{\epsilon}) - (\gamma/2) \log \epsilon$ has the law of a quantum wedge $(\h,h_{\epsilon},0,\infty)$ of weight $3\gamma^2/2-2$ embedded such that $\qbmeasure{h_{\epsilon}}([-1,0]) = 1$.  Moreover, $x_\epsilon$ and $h^0$ are independent (as $\qbmeasure{h}$ is determined by $\Fh$) and thus by conformal invariance, $h^0(x_\epsilon \cdot)$ is a zero-boundary GFF on $\h$ for each $\epsilon > 0$. Furthermore, since the law of $h_{\epsilon}$ does not change in $\epsilon$,  it follows that the pair $(h_{\epsilon},h^0(x_{\epsilon} \cdot))$ is tight in $\epsilon$.  Thus for every fixed subsequence,  we can find a further subsequence such that we have a joint convergence of the laws of $h_{\epsilon}$ and $h^0(x_{\epsilon} \cdot)$ and consequently of $h_{\epsilon}-h^0(x_{\epsilon} \cdot)$.  Since $h_{\epsilon}-h^0(x_{\epsilon} \cdot)$ is a.s.\  harmonic on $x_{\epsilon}^{-1}\D \cap \h$,  the limit is harmonic on $\h$.  Moreover the joint law of any subsequential limit is always the same,  and so it follows that the limit of $h_\epsilon$ (which is a quantum wedge of weight $3\gamma^2/2-2$ embedded into $\h$ such that $\qbmeasure{h}([-1,0])=1$) can be expressed as $h_*^0 + \Fh$ where $h_*^0$ is a zero-boundary GFF on $\h$ and $\Fh$ is harmonic on $\h$ and independent of $h_*^0$.  This completes the proof of the lemma.
\end{proof}

We note that Lemma~\ref{lem:wedge_harmonic_expression} implies that if $\CW = (\h,h,0,\infty)$ is quantum wedge with the circle average embedding then we can write $h = \Fh + h^0$ where~$\Fh$ is the harmonic extension of the boundary values of~$h$ from $\partial \h$ to~$\h$ and~$h^0$ is a GFF-like distribution (this decomposition comes from applying Lemma~\ref{lem:wedge_harmonic_expression} and then applying the change of coordinates formula for quantum surfaces.  In contrast to the choice of embedding from Lemma~\ref{lem:wedge_harmonic_expression}, when we use the circle average embedding, $\Fh$ and~$h^0$ are not independent and~$h^0$ does not have the law of a GFF with zero boundary conditions on~$\h$.  We now bound the exponential moments of the harmonic part when the wedge has the circle average embedding.

\begin{lemma}\label{lem:wedge_exponential_moments_circle_average}
Let $\CW = (\strip,h,-\infty,\infty)$ be a quantum wedge of weight $3\gamma^2/2-2$ with the first exit parameterization (i.e.,  parameterized such that the projection of $h$ to $H_1(\strip)$ first hits $0$ at time $u=0$) and let $\Fh$ be the harmonic extension of the values of $h$ from $\partial \strip$ to $\strip$.  Fix $y \in \R$ and $\delta \in (0,\pi/2)$.  Then there exist constants $c_1,c_2>0$ depending only on $y$, $\delta$, and $\kappa'$ such that
\begin{align*}
\p\left[ \sup_{z \in (-\infty,y] \times [\delta ,  \pi - \delta]}\Fh(z) \geq x \right] \leq c_1 \exp(-c_2 x^2) \quad \text{for every} \quad x>0.
\end{align*}
\end{lemma}
\begin{proof}
We will prove the claim of the lemma in the case that $y=0$.  The general case follows from a similar argument.  We note that we can couple $h$ with a free boundary GFF $\wt{h}$ on $\strip$ with the additive constant taken so that its average of $\wt{h}$ on $\{0\} \times (0,\pi)$ is equal to $0$ and such that their projections to $H_1(\strip)$ are independent and their projections to $H_2(\strip)$ are the same.  Let $h_1$ (resp.\ $h_2$) be the projection of $h$ to $H_1(\strip)$ (resp.\ $H_2(\strip)$).  We note that $(h_1(u))_{u \geq 0}$ has the same law as $(B_{2u}+(Q-\alpha)u)_{u \geq 0}$,  where $(B_t)$ is a standard Brownian motion with $B_0 = 0$ and $\alpha = 4/\gamma - \gamma/2$.  Also,  $(h_1(-u))_{u \geq 0}$ has the law of $(\wh{B}_{2u}-(Q-\alpha)u)_{u \geq 0}$,  where $\wh{B}$ is a standard Brownian motion with $\wh{B}_0 = 0$ and conditioned so that $\wh{B}_{2u}-(Q-\alpha)u < 0$ for every $u<0$,  and such that $\wh{B}$ is independent of $B$.  Let $\Fh_1$ (resp.\ $\Fh_2$) be the harmonic extension of $h_1$ (resp.\ $h_2$) from $\partial \strip$ to $\strip$.  Note that if $\wt{h}_1$ is the projection of $\wt{h}$ to $H_1(\strip)$,  we have that $\Fh_2 = \wt{\Fh} - \wt{\Fh}_1$ where $\wt{\Fh}$ (resp.\ $\wt{\Fh}_1$) is the harmonic extension of $\wt{h}$ (resp.\ $\wt{h}_1$) from $\partial \strip$ to $\strip$.  Fix $K \subseteq \strip$ compact.  We claim that there exist constants $c_1,c_2>0$ depending only on $K$ such that for every $x>0$,
\begin{equation}\label{eqn:lateral_harmonic_part_tail_bound}
\p\left[ \sup_{z \in K}|\Fh_2(z)| \geq x \right] \leq c_1 \exp(-c_2 x^2).
\end{equation}
Indeed,  it will be more convenient to work on $\h$ instead of $\strip$.  Let $\phi(z) = e^z$ and $\wh{h} = \wt{h} \circ \phi^{-1}$.  Then $\wh{h} \circ \phi^{-1}$ is a free boundary GFF on $\h$ with the additive constant taken so that $\wh{h}_1(0) = 0$.  Moreover,  if $(\wt{X}_t)$ is the average value of $\wh{h}$ on $\h \cap \partial B(0,e^{-t})$,  then $\wt{X}_t = \wt{B}_{2t}$ where $\wt{B}$ is a two-sided Brownian motion.  Then,  it is easy to see that
\begin{equation}\label{eqn:free_boundary_upper_bound}
\int_{\R} |\wt{X}_{-\log |t|}| p(z,t) \lesssim 1 + \sup_{-T \leq t \leq T}|\wt{B}_t| + \sup_{|t| \geq 1} |\wt{B}_t| / |t|
\end{equation}
for every $z \in \phi(K)$,  where both $T>0$ and the implicit constant depend only on $K$ and $p(x+iy,t) = \frac{y}{\pi ( y^2 + (x-t)^2)}$ is the Poisson kernel on $\h$.  Note that $\wt{\Fh}_1 \circ \phi^{-1}$ is given by the harmonic extension of the function $\wt{X}^1$ on $\h$ defined by $\wt{X}^1(z) = \wt{X}_t$ for every $z \in \h$,  where $t = -\log |z|$.  Also,  we note that if $W$ is a standard Brownian motion,  then $(uW_{1/u})_{u \geq 0}$ has the law of a Brownian motion and since $\sup_{0 \leq u \leq T}|W_u|$ has Gaussian tail,  it follows that there exist universal constants $c_1,c_2>0$ such that $\p[ \sup_{u \geq 1} |W_u|/u \geq b] \leq c_1 \exp(-c_2 b^2)$ for every $b>0$.  Thus,  combining with~\eqref{eqn:free_boundary_upper_bound},  we obtain that $\p[\sup_{z \in K} |\wt{\Fh}_1(z)| \geq x ] \leq c_1 \exp(-c_2 x^2)$ for every $x>0$ possibly by taking $c_1$ to be larger and $c_2$ smaller.  Moreover,  it follows from \cite[Lemma~A.2]{kms2021regularity} that $\sup_{z \in K} |\wt{\Fh}(z)|$ has exponential moments of all orders.  Combining,  we obtain that $\E[ \sup_{z \in K}|\Fh_2(z)| ] < \infty$.  Since $(\Fh_2(z))_{z \in K}$ is a family of  Gaussians,  \eqref{eqn:lateral_harmonic_part_tail_bound} follows from the Borell-TIS inequality.

Next we show that possibly by taking $c_1$ to be larger and $c_2$ smaller,  we have for every $x \geq 0$ that
\begin{equation}\label{eqn:wedge_radial_part_upper_bound}
\p\left[ \sup_{z \in \R_- \times (0,\pi)} \Fh_1(z) \geq x \right] \leq c_1 \exp(-c_2 x^2).
\end{equation}
Indeed,  first we note that by conformally mapping onto $\h$ and using the Beurling estimate,  we obtain that there exists a universal constant $C>0$ such that the following holds.  For every $r >0,  z \in \R_- \times (0,\pi)$,  the probability that a complex Brownian motion starting from $z$ exits $\strip$ on $[r,\infty) \times \{0,\pi\}$ is at most $C e^{-r/2}$.  Since $h_1$ takes negative values on $\R_-$,  we obtain that for every $z \in \R_- \times (0,\pi)$,  
\begin{align*}
\Fh_1(z) \lesssim 1 + \sum_{m \geq 1} e^{-m/2} \sup_{2(m-1) \leq t \leq 2m} B_t
\end{align*}
where the implicit constant depends only on $\kappa'$.  Therefore,  summing over $m \in \N$, \eqref{eqn:wedge_radial_part_upper_bound} follows from the Gaussian tail of the maximum of Brownian motion combined with the Markov property.  

Now,  for every $a \in \R$,  we set $T_a = \inf\{u \in \R : h_1(u) = a\}$ and for every $n \in \N$ we set $K_n = [T_{-n} ,  T_{-n+1}] \times [\delta,\pi-\delta]$.  We claim that by adjusting the constants $c_1,c_2$,  we have for every $x >0$ that
\begin{equation}\label{eqn:wedge_harmonic_part_upper_bound}
\p\left[ \sup_{z \in K_1} \Fh(z) \geq x \right] \leq c_1 \exp(-c_2 x^2).
\end{equation}
Indeed,  first we note that $-T_{-1}$ has the same law as the first time that $B_{2t} + (Q-\alpha)t$ hits $1$.  Hence,  Girsanov's theorem implies that $-T_{-1}$ has density given by $f(t) = (2\pi t^3)^{-1/2}\exp(-(1-(Q-\alpha)t/2)^2/2t)$.  It follows that $\E[\exp(-pT_{-1})] < \infty$ for every $p<(Q-\alpha)^2/8$.  Moreover,  the law of $\Fh_2$ is translation invariant.  Therefore,  if we set $R_m = [-m,0] \times [\delta ,  \pi - \delta]$ and $S_m = [-m,-m+1] \times [\delta ,  \pi - \delta]$ for every $m \geq 0$,  combining~\eqref{eqn:lateral_harmonic_part_tail_bound} with~\eqref{eqn:wedge_radial_part_upper_bound} gives that for every $x > 1$,
\begin{align*}
\p\left[ \sup_{z \in K_1} \Fh(z) \geq x \right]& \leq \p[T_{-1} \leq -x^2] + \p\left[ \sup_{z \in R_{x^2}} \Fh_2(z) \geq x/2 \right] + \p\left[ \sup_{z \in R_{x^2}} \Fh_1(z) \geq x/2 \right]\\
&\leq \p[T_{-1} \leq -x^2] + 2x^2 \p\left[ \sup_{z \in S_1}\Fh_2(z) \geq x/2 \right] + \p\left[ \sup_{z \in \R_- \times (0,\pi)} \Fh_1(z) \geq x/2 \right]\\
& \leq c_1 \exp(-c_2 x^2)
\end{align*}
possibly by taking $c_1$ to be larger and $c_2$ smaller. This proves~\eqref{eqn:wedge_harmonic_part_upper_bound}.

Finally,  we note that $\sup_{z \in K_n} \Fh(z)$ and $-n+\sup_{z \in K_1} \Fh(z)$ have the same law for every $n \in \N$,  and so~\eqref{eqn:wedge_harmonic_part_upper_bound} implies that
\begin{align*}
\p\left[ \sup_{z \in \R_- \times [\delta ,  \pi - \delta]} \Fh(z) \geq x \right] &\leq \sum_{n\geq 1}\p\left[ \sup_{z \in K_n} \Fh(z) \geq x \right] = \sum_{n \geq 1} \p\left[ \sup_{z \in K_1} \Fh(z) \geq x+n \right]\\
&\leq c_1 \sum_{n \geq 1} \exp(-c_2 (x+n)^2) \lesssim \exp(-c_2 x^2).
\end{align*}
This completes the proof of the lemma.
\end{proof}

\begin{lemma}\label{lem:wedge_exponential_moments_unit_boundary_length}
Let $\CW = (\h , h , 0,\infty)$ be a quantum wedge of weight $3\gamma^2/2-2$ embedded into $\h$ such that $\qbmeasure{h}([-1,0]) = 1$ and let $h = h^0 + \Fh$ be the expression of $h$ as in Lemma~\ref{lem:wedge_harmonic_expression}.  Fix $K \subseteq \h$ compact. Then there exist constants $c_1,c_2>0$ depending only on $K$ and $\kappa'$ such that 
\begin{align*}
\p\left[ \sup_{z \in K}\Fh(z) \geq x \right] \leq c_1 \exp(-c_2 x^2) \quad \text{for every}\quad x > 0.
\end{align*}
\end{lemma}
\begin{proof}
Let $\phi(z) = e^z$ and $\wh{h} = h \circ \phi + Q \log|\phi'|$. Then the quantum surface $\wt{\CW}=(\strip,\wh{h},-\infty,+\infty)$ is equivalent to $\CW$ and is a quantum wedge of weight $3\gamma^2/2-2$ embedded into $\strip$ such that $\qbmeasure{\wh{h}}((-\infty,0] \times \{\pi\}) = 1$ and let $\wh{h} = \wh{h}^0 + \wh{\Fh}$,  where $\wh{h}^0$ is a zero-boundary GFF on $\strip$ and $\wh{\Fh}$ is harmonic on $\strip$ and independent of $\wh{h}^0$.  Let also $\wh{T}$ be the first time that the projection of $\wh{h}$ onto $H_1(\strip)$ hits $0$.  Then $(\strip,\wh{h}(\cdot+\wh{T}),-\infty,+\infty)$ is a quantum wedge of weight $3\gamma^2/2-2$ parameterized so that its projection onto $H_1(\strip)$ first hits $0$ at time $t=0$ and so $\wt{h}=\wh{h}(\cdot+\wh{T})$ can be decomposed as $\wt{h}^0+\wt{\Fh}$ as in Lemma~\ref{lem:wedge_exponential_moments_circle_average} where $\wt{\Fh}$ is harmonic on $\strip$.  We fix a compact set $K \subseteq \h$, let $\wh{K} = \phi^{-1}(K) \subseteq \strip$, and let $T \in \R$ be such that $\qbmeasure{\wt{h}}((-\infty,T] \times \{\pi\}) = 1$.  Then we have that $\sup_{z \in \wh{K}}\wh{\Fh}(z)$ and $\sup_{z \in (\wh{K}+T)}\wt{\Fh}(z)$ have the same law.  Also, for each $C \in \R$,  we let $T_C$ be the first time that the projection of $\wt{h}$ to $H_1(\strip)$ hits $C$.  We note that $\p[ T \geq T_C+1]$ decays to zero faster than any exponential of $-C$ as $C \to \infty$.  Indeed, $\qbmeasure{\wt{h}}([T_C,T_C+1] \times \{\pi\})$ and $e^{\gamma C/2} \qbmeasure{\wt{h}}([0,1] \times \{\pi\})$ have the same law and the proof of \cite[Lemma~4.5]{ds2011lqg} implies that there exist constants $c_1,c_2>0$ such that $\p[\qbmeasure{\wt{h}}([0,1] \times \{\pi\}) \leq e^{-\gamma C/2}] \leq c_1 e^{-c_2 \gamma^2 C^2/4}$ for every $C>0$.  Thus we have that
\begin{align*}
\p[T\geq T_C+1] \leq \p[\qbmeasure{\wt{h}}([T_C,T_C+1] \times \{\pi\}) \leq 1]= \p[ \qbmeasure{\wt{h}}([0,1] \times \{\pi\}) \leq e^{-\gamma C/2}] \leq c_1 e^{-c_2 \gamma^2 C^2 / 4}
\end{align*}
for every $C>0$.  Fix $y\in \R$ and $\delta \in (0,\pi/2)$ so that $\wh{K} \subseteq (-\infty,y] \times [\delta ,  \pi- \delta]$. We then have  for every $x,C>0$ that 
\begin{align*}
	 \p\!\left[ \sup_{z \in \wh{K}}\wh{\Fh}(z) \geq x \right] &=\p\!\left[ \sup_{z \in (\wh{K}+T)}\wt{\Fh}(z) \geq x \right]\\
	 &\leq \p\!\left[ \sup_{z \in (-\infty,y+1+T_C] \times [\delta ,  \pi - \delta]}\wt{\Fh}(z) \geq x \right] + \p[ T\geq T_C+1]\\
	&\leq \p\!\left[ \sup_{z \in (-\infty,y+1] \times [\delta ,  \pi - \delta]}\wt{\Fh}(z) \geq x-C\right] + \p[T\geq T_C+1]
\end{align*}
since $\wt{\Fh}(\cdot + T_C)$ and $C + \wt{\Fh}(\cdot)$ have the same law.  The proof is then complete by setting $C=x/2$ and combining with Lemma~\ref{lem:wedge_exponential_moments_circle_average} since $\Fh = \wh{\Fh} \circ \phi^{-1} + Q \log |(\phi^{-1})'|$.
\end{proof}

\bibliographystyle{abbrv}
\bibliography{references}

\end{document}